\documentclass[reqno]{amsart}
\usepackage{amsmath,amsthm,amssymb,enumitem,times,xcolor,bm,esint}
\usepackage{leftidx}
\usepackage[mathscr]{eucal}
\usepackage{harpoon}
\usepackage{tikz}
\usepackage[titletoc]{appendix} 
\usepackage[
linktocpage=true,colorlinks,citecolor=magenta,linkcolor=blue,urlcolor=magenta]{hyperref}
\allowdisplaybreaks[4]
\providecommand{\U}[1]{\protect \rule{.1in}{.1in}}

\usepackage{geometry}
\makeatletter
\newcommand\@dotsep{4.5}
\def\@tocline#1#2#3#4#5#6#7{\relax
  \ifnum #1>\c@tocdepth 
  \else
    \par \addpenalty\@secpenalty\addvspace{#2}%
    \begingroup \hyphenpenalty\@M
    \@ifempty{#4}{%
      \@tempdima\csname r@tocindent\number#1\endcsname\relax
    }{%
      \@tempdima#4\relax
    }%
    \parindent\z@ \leftskip#3\relax \advance\leftskip\@tempdima\relax
    \rightskip\@pnumwidth plus1em \parfillskip-\@pnumwidth
    #5\leavevmode\hskip-\@tempdima{#6}\nobreak
    \leaders\hbox{$\m@th\mkern \@dotsep mu\hbox{.}\mkern \@dotsep mu$}\hfill
    \nobreak
    \hbox to\@pnumwidth{\@tocpagenum{\ifnum#1=1\bfseries\fi#7}}\par
    \nobreak
    \endgroup
  \fi}
\AtBeginDocument{%
\expandafter\renewcommand\csname r@tocindent0\endcsname{0pt}
}
\def\l@subsection{\@tocline{2}{0pt}{2.5pc}{5pc}{}}
\makeatother

\newtheorem{theorem}{Theorem}[section] 
\newtheorem{corollary}[theorem]{Corollary} 
\newtheorem{definition}[theorem]{Definition} 
\newtheorem{lemma}[theorem]{Lemma} 
\newtheorem{proposition}[theorem]{Proposition}
\newtheorem{theoremA}{Theorem} 
\theoremstyle{remark}
\newtheorem{remark}[theorem]{Remark} 
\numberwithin{equation}{section}  

\def\be{\begin{equation}}
\def\ee{\end{equation}}
\def\BR{{\mathcal B}}
\def\CR{{\mathcal C}}
\def\DR{{\mathcal D}}
\def\ER{{\mathcal E}}

\def\IR{{\mathcal I}}

\def\OR{{\mathcal O}}

\def\QR{{\mathcal Q}}
\def\SR{{\mathcal S}}

\def\XR{{\mathcal X}}

\def\R{\mathbb R}
\def\N{\mathbb N}

\DeclareMathAlphabet{\mathpzc}{OT1}{pzc}{m}{it}
\usepackage{bm}
\def\p{\mathpzc{p}}
\def\al{\alpha}
\def\ga{\gamma}
\def\Ga{\Gamma}
\def\ep{\varepsilon}
\def\iy{\infty}
\def\la{\lambda}
\def\vp{\varphi}
\def\pa{\partial}
\def\ti{\tilde}
\def\wt{\widetilde}
\def\ov{\overline}
\def\lab{\label}
\def\f{\frac}

\def\sbr#1{\left(#1\right)} 
\def\mbr#1{\left[#1\right]} 
\def\lbr#1{\left\{#1\right\}} 
\def\abr#1{\left\langle#1\right\rangle} 
\def\abs#1{\left\lvert#1\right\rvert} 
\def\nm#1{\left\|#1\right\|} 
\def\vs{\vskip .2in} 
\def\rd{\mathrm d} 
\def\loc{{\operatorname{\rm loc}}} 
\def\dist{\operatorname{dist}} 

\begin{document}
\theoremstyle{plain}

\title[Compactness on singular surface]{Critical points of the Moser-Trudinger functional on conical singular surfaces, I: compactness}
\author{Zhijie Chen}\address{Department of Mathematical Sciences, Yau Mathematical Sciences Center, Tsinghua University, Beijing, China}\email{zjchen2016@tsinghua.edu.cn}
\author{Houwang Li}\address{Beijing Institute of Mathematical Sciences and Applications \& Yau Mathematical Sciences Center, Tsinghua University, Beijing,  China}\email{lhwmath@bimsa.cn}



\begin{abstract}
Let $(\Sigma, g_1)$ be a compact Riemann surface with conical singularites of angles in $(0, 2\pi)$, and $f: \Sigma\to\mathbb R$ be a positive smooth function.
In this paper, by establishing a sharp quantization result, we prove the compactness of the set of positive critical points for the Moser-Trudinger functional 
$$F_1(u)=\int_{\Sigma}(e^{u^2}-1)f\rd v_{g_1}$$constrained to $u\in\ER_\beta:=\{u\in H^1(\Sigma,g_1) : \nm{u}_{H^1(\Sigma,g_1)}^2=\beta\}$ for any $\beta>0$.
This result is a generalization of the compactness result for the Moser-Trudinger functional on regular compact surfaces, proved by  De Marchis-Malchiodi-Martinazzi-Thizy (Inventiones Mathematicae, 2022, 230: 1165-1248). The presence of conical singularities brings many additional difficulties and we need to develop different ideas and techniques.
The compactness lays the foundation for proving the existence of critical points of the Moser-Trudinger functional on conical singular surfaces in a sequel work.

\end{abstract}
\maketitle

\tableofcontents

\section{Introduction}\lab{section-1}
\subsection{Motivation}\ 

Let $(\Sigma,g)$ be a compact Riemann surface without boundary. We endow the usual Sobolev space $H^1(\Sigma,g)$ with the standard norm $\nm{~\cdot~}_{H^1(\Sigma,g)}$ given by
\be\lab{norm-1} \nm{u}_{H^1(\Sigma,g)}^2=\int_\Sigma(|\nabla_{g} u|^2+u^2)\rd v_{g}. \ee

If $g=g_0$ is a smooth metric, then building up on previous works, see e.g. \cite{MT-1,MT-2,MT-3}, Li \cite{MT-4} proved the following Moser-Trudinger inequality
\be\lab{MT} 
	\sup\Big\{ \int_\Sigma e^{u^2}\rd v_{g_0}:~u\in H^1(\Sigma,g_0), \nm{u}_{H^1(\Sigma,g_0)}^2=\beta\Big\}<+\iy\quad\Longleftrightarrow\quad \beta\le 4\pi,
\ee
and  that there is an extremal function for \eqref{MT} even in the critical case $\beta=4\pi$ (see also \cite[Remark 5.1]{MT-blowup-7}). Let $F_0$ be defined by
	$$F_0(u):=\int_\Sigma (e^{u^2}-1)f\rd v_{g_0}, \quad\text{for}~u\in\Big\{u\in H^1(\Sigma,g_0) :  \nm{u}_{H^1(\Sigma,g_0)}^2=\beta\Big\},$$
where $f$ is a smooth positive function, then the critical points of $F_0$ are solutions of the following Moser-Trudinger equation
\be\lab{equ-1} -\Delta_{g_0}u+u=2\la f ue^{u^2} \quad\text{in}~(\Sigma,g_0),\ee
where $\Delta_{g_0}$ is the Laplace-Beltrami operator and $\lambda>0$ appears as the Lagrange multiplier and is given by
	$$2\la\int_\Sigma u^2e^{u^2}f\rd v_{g_0}=\beta=\nm{u}_{H^1(\Sigma,g)}^2.$$
For $\beta<4\pi$, by using Moser-Trudinger inequality \eqref{MT}, finding critical points of $F_0$ reduces to a standard maximization argument. Finding such critical points for $\beta>4\pi$ is a more challenging problem, since upper bounds on the functional fail. Some existence results of critical points for planar domains and slightly supercritical regions $0<\beta-4\pi\ll1$ were obtained in \cite{EQ-2,exist-2}. 
Moreover, as observed in \cite{MT-blowup-1}, $F_0$ fails to satisfy the global Palais-Smale condition, and instead, sequences of solutions of equation \eqref{equ-1}, as well as Palais-Smale sequences for $F_0$, may present a lack of compactness and develop some concentration phenomena. This prevents one from using the methods of  \cite{EQ-2,exist-2} for large $\beta$. 

Keeping in mind the well-known concentration compactness results \cite{MF-1,MF-2,MF-3} for Liouville type equations, it seems to be reasonable to recover the compactness for functional $F_0$. 
However, the problem becomes more difficult for Moser-Trudinger equations due to the critical increasing nonlinearity $ue^{u^2}$ in dimension two, and there have been many outstanding contributions over the past two decades. 
In \cite{MT-blowup-2}, Druet obtained an energy quantization result for solutions of equation \eqref{equ-1} on Euclidean bounded domains. Later, this quantization result was extended to multi-harmonic equations \cite{EQ-1,EQ-3}, and to elliptic equations on smooth closed surfaces \cite{MT-blowup-3}. Recently, up on Druet's quantization result, Druet and Thizy \cite{MT-blowup-6} showed that when $\Sigma=\Omega$ is an Euclidean bounded domain, for any solution sequence $u_n$ satisfying
	$$-\Delta u_n=\la_nfu_ne^{u_n^2},\quad u_n>0,\quad\text{in}~\Omega,\quad u_n=0,\quad\text{on}~\pa\Omega,$$
where $\Delta$ is the Laplace operator in $\R^2$, $\la_n>0$ and $f$ is a smooth positive function, if $u_n$ is uniformly bounded in $H^1_0(\Omega)$, then there holds
	$$\text{either}~\sup_{n}\nm{u_n}_{L^\iy(\Omega)}<\iy,\quad \text{or}~\nm{u_n}_{H^1_0(\Omega)}^2=4\pi N+o(1),\quad\text{where }N\in\N_+.$$
That is, the lack of compactness can only occur at the levels $4\pi\N_+$. 
Using this important result, and combining a sharper quantization based on the accurate error estimates in \cite{MT-blowup-4,MT-blowup-5}, Thizy \cite{exist-1} proved the existence and non-existence of extremal functions for some general Moser-Trudinger inequalities. 
More recently, in a remarkable work \cite{MT-blowup-7}, by considering a class of subcritical functionals $I_{p,\beta}$ (see \eqref{Ip} with $g_1$ replaced by $g_0$), and combining a minimax scheme together with compactness and sharp quantization results, De Marchis-Malchiodi-Martinazzi-Thizy succeeded to prove the existence of positive critical points of the functional $F_0$ for any $\beta>0$. Actually, they proved a more general result:

\begin{theoremA}(\cite{MT-blowup-7})
Let $(\Sigma,g_0)$ be a compact surface without boundary, where $g_0$ is a smooth metric. Let $f$ be a smooth positive function. Let $p\in[1,2]$ and $\beta>0$. Then, the following statements hold.
\begin{itemize}
\item[(1)] If $p\in(1,2)$, the set $\CR_{p,\beta}$ of the positive critical points of $I_{p,\beta}$ is compact in $C^2(\Sigma)$;
\item[(2)] If $p=1$ and $\beta\not\in4\pi\N_+$, the set $\CR_{1,\beta}$ of the positive critical points of $I_{1,\beta}$ is compact in $C^2(\Sigma)$;
\item[(3)] If $p=2$, the set $\CR_{2,\beta}$ of the positive critical points of $F_0$ constrained to $\{u\in H^1(\Sigma,g_0) : \nm{u}_{H^1(\Sigma,g_0)}^2=\beta\}$ is  compact in $C^2(\Sigma)$.
\end{itemize}
In all the above cases, the set $\CR_{p,\beta}$ is nonempty.
\end{theoremA}

We explain roughly their strategy: firstly, by a minimax scheme based on the so called {\it barycenters}, they proved the existence of positive critical points of $I_{p,\beta}$ for $p\in(1,2)$ and for a.e. $\beta>0$, where the subcriticality of the nonlinearity, due to $p<2$, plays an essential role;  
secondly, by using the crucial compactness result in (1)-(3) of Theorem A (see \cite[Theorem 4.1, 5.1]{MT-blowup-7} or our Theorem \ref{thm0} below for a general case), they extended the existence, on the one hand, for the subcritical case $p<2$ to the critical case $p=2$, and on the other hand, for a.e. $\beta$ to every $\beta\in(0,+\iy)$.
This process to get general existence results has been widely used for many problems, see for example \cite{barycenter-1,barycenter-2,bg-2,barycenter-3,barycenter-4,barycenter-5,barycenter-6,barycenter-7} and the references therein. 
Apart from the barycenter approach, there is another approach — replacing the first step in the above process with counting the topological degree — to get the general existence. For the topological degree approach, we refer the readers to \cite{degree-1,degree-2,degree-3,degree-5,degree-4} and the references therein, and to a recent paper \cite{MT-blowup-8} by Malchiodi-Martinazzi-Thizy on equation \eqref{equ-1} over Euclidean bounded domains.
In any case, compactness is always the base to get the general existence result in both approaches.  

\vskip0.1in
Our main goal is to generalize this remarkable Theorem A to singular surfaces. Precisely, we consider the following natural questions: Does the compactness and  existence results still hold on singular surfaces? Or in other words, when blow-up occurs, will the energy still quantize on singular surfaces? And if so, what kind of asymptotic behaviors does the solution sequence exhibit? To the best of our knowledge, these questions have only been studied in a very special case \cite{exist-s-1, exist-s-2, exist-s-3} (see the paragraph below \eqref{617-2}), and remains largely open.

We try to answer these questions in a series of papers. To start with, we consider surfaces with conical singularities. Roughly speaking, such a surface is a compact Riemann surface with a smooth metric (i.e., of class $C^2$) everywhere except at finitely many points. 
To be precise, we follow Troyanov's notations (\cite{Singular-1}) .  

\vskip0.1in

\noindent{\bf Definition.}
Let $\Sigma$ be a compact surface without boundary. A {\it divisor} on $\Sigma$ is a formal sum
	$$\bm{\al}=\sum_{a\in\DR}\al_{a}a,$$
where $\DR$ is a finite set consisting of distinct points in $\Sigma$ and $\al_{a}>-1$ for each $a\in\DR$. A metric $g$ on $\Sigma$ is said to {\it represent} the divisor $\bm{\al}$, if $g$ is a smooth  Riemannian metric on $\Sigma\setminus \DR$ such that under a local coordinate $x$ defined near $a\in\DR$, there exists a continuous function $\phi_a$ satisfying
	$$ g=e^{2\phi_a}|x|^{2\al_{a}}|\rd x|^2, \quad \text{where $x(a)=0$.}$$
The point $a\in\DR$ is then said to be a {\it conical singularity} of order $\al_{a}$, and the {\it conical angle} of $a$ is $\theta_a=2\pi(1+\al_{a})$. 

\vskip0.1in
For example, a (somewhat idealized) American football has two singularities of equal angles, while a teardrop has only one singularity. Both these examples correspond to the case $\theta\in(0,2\pi)$. Such singularities also appear in orbifolds and branched coverings. They can also describe the ends of complete Riemann surfaces with finite total curvature. In the last decades, conical singularities have been widely studied in many problems, see for example \cite{clas-0,bg-1,bg-2,bg-3,bg-4,bg-5} and the references therein. 

Let $g$ be a conical singular metric on $\Sigma$ representing the divisor $\bm{\al}$. It has been proved in \cite{Singular-1,Singular-2} (see \cite{SMT-0} for planar domains) that the following singular Moser-Trudinger inequality holds:
\be \begin{aligned} 
	\sup\Big\{ \int_\Sigma e^{u^2}\rd v_{g}:~u\in H^1(\Sigma,g), \int_\Sigma u\rd v_g=0, \int_\Sigma|\nabla_gu|^2\rd v_g=\beta\Big\}<+\iy\\
	 \quad\iff\quad \beta\le 4\pi(1+\al_{\DR}),   
\end{aligned}\ee
where 
\be\lab{almin} \al_{\DR}=\min\lbr{0~,~\min_{a\in\DR}\al_{a}}. \ee 
Then, exactly as in \cite{MT-4}, one can obtain that
\be\lab{SMT}\begin{aligned} 
	\sup\Big\{ \int_\Sigma e^{u^2}\rd v_{g}:~u\in H^1(\Sigma,g), \nm{u}_{H^1(\Sigma,g)}^2=\beta\Big\}<+\iy\\
	 \quad\iff\quad \beta\le 4\pi(1+\al_{\DR}). 
\end{aligned}\ee

\subsection{Setting on conical singular surface}\ 

In this paper, we assume that $(\Sigma,g)$ has conical singularities of angles in $(0,2\pi)$. Namely, we take $g=g_1$ which represents the following divisor
\be\lab{divisor} \bm{\al_1}=\sum_{a\in\DR}\al_{a} a,\quad\text{with}~\al_{a}\in\sbr{-1,0}~\text{for each}~a\in\DR,\ee
where $\DR$ is a finite set of singularities. 

\begin{remark}\label{remark1-1}
The generalization of the compactness result of Theorem A to conical singular surfaces is quite difficult, and in this paper we need to assume the technical condition $\al_{a}\in (-1,0)$ for all $a\in\DR$, which will be used in several places of our proofs (see Section 1.4 for more explanations). We believe that the compactness should also hold for the general case that $\al_{a}>0$ is allowed for $a\in\DR$. We should study this general case elsewhere.
\end{remark}

Now given a smooth positive function $f: \Sigma\to (0,+\infty)$, we let the Moser-Trudinger functional $F_1$ on the conical singular surface $(\Sigma, g_1)$ be defined by
\be\lab{F1} F_1(u):=\int_\Sigma (e^{u^2}-1)f\rd v_{g_1},\ee
constrained to
\be u\in\ER_\beta:=\Big\{u\in H^1(\Sigma,g_1) : \nm{u}_{H^1(\Sigma,g_1)}^2=\beta\Big\}.\ee
Then the critical points of $F_1$ are solutions of the following singular Moser-Trudinger equation
\be\lab{617-1} -\Delta_{g_1}u+u=2\la f ue^{u^2} \quad\text{in}~\Sigma,\ee
where $\Delta_{g_1}$ is the Laplace-Beltrami operator and $\lambda>0$ is the Lagrange multiplier given by
\be\lab{617-2} 2\la\int_\Sigma u^2e^{u^2}f\rd v_{g_1}=\beta=\nm{u}_{H^1(\Sigma,g_1)}^2. \ee
It is clear that $4\pi(1+\al_{\DR})<4\pi$ by \eqref{divisor}. 
For $\beta<4\pi(1+\al_{\DR})$, by using singular Moser-Trudinger inequality \eqref{SMT}, finding critical points of $F_1|_{\ER_\beta}$ reduces to a standard maximization argument. 
For $\Sigma\subset\R^2$ being a planar domain, $f\equiv1$ and $g_1=|x|^{2\al}|\rd x|^2$ with $\al\in(-1,0)$, critical points of $F_1|_{\ER_\beta}$ with $\beta=4\pi(1+\al)$ (i.e., extremal functions of \eqref{SMT}) has been obtained in \cite{exist-s-1} (see also \cite{exist-s-2} for a simpler proof and also \cite{exist-s-3} for a general version on conical singular surfaces), and their proof relies on the method of Carleson–Chang \cite{extremal-0} which can not work for $\beta>4\pi(1+\al_{\DR})$. 
To the best of our knowledge, there is no other existence results of critical points of $F_1|_{\ER_\beta}$. 
We point out that finding such critical points for $\beta>4\pi(1+\al_\DR)$ is a more challenging problem, since the upper bound on the functional fails. In this paper, we succeed to prove new sharp quantization and compactness results on the conical singular surface $(\Sigma, g_1)$, and in a sequel work, we will prove the general existence result by applying this compactness result. 

We follow the strategy of De Marchis-Malchiodi-Martinazzi-Thizy \cite{MT-blowup-7} to interpolate between a singular Liouville type problem and the critical singular Moser-Trudinger problem. More precisely, given $p\in[1,2)$ and $\beta>0$, we let $I_{p,\beta}: H^1(\Sigma,g_1)\to\mathbb R\cup\{+\infty\}$ be given by
\be\lab{Ip} I_{p,\beta}(u):=\f{2-p}{2}\sbr{\f{p\nm{u}_{H^1(\Sigma,g_1)}^2}{2\beta}}^{\f{p}{2-p}}-\ln\int_\Sigma\sbr{e^{u_+^p}-1}f\rd v_{g_1}, \ee
where $u_+=\max\{u,0\}$ and we set $I_{p,\beta}(u)=+\iy$ if $u\le 0$. It is easy to see that positive critical points of $I_{p,\beta}$ are solutions of
\be\lab{equ-2} -\Delta_{g_1}u+u=p\la fu^{p-1}e^{u^p},\quad u>0\quad\text{in}~\Sigma, \ee
where $\la>0$ is given by
\be\lab{la-0} \f{\la p^2}{2}\sbr{\f{p\nm{u}_{H^1(\Sigma,g_1)}^2}{2\beta}}^{\f{2(p-1)}{2-p}}\int_\Sigma\sbr{e^{u^p}-1}f\rd v_{g_1}=\beta. \ee
Multiplying \eqref{equ-2} by $u$ and integrating by parts in $\Sigma$, \eqref{la-0} may be rewritten as
\be\lab{la-1}   \f{\la p^2}{2}\sbr{\int_\Sigma\sbr{e^{u^p}-1}f\rd v_{g_1}}^{\f{2-p}{p}}\sbr{\int_\Sigma u^pe^{u^p}f\rd v_{g_1}}^{\f{2(p-1)}{p}}=\beta. \ee
By \eqref{SMT} and Young inequality, $I_{p,\beta}$ is bounded from below for all $\beta\le4\pi(1+\al_{\DR})$, and finding critical points of $I_{p,\beta}$ reduces to a standard maximization argument. As we shall prove, the compactness and quantization will give that, as $p$ approaches the borderline case $p_0=2$, the critical points of $I_{p,\beta}$ converge to critical points of the functional $F_1$ in \eqref{F1} constrained to $\ER_\beta$, at least when $\beta>0$ is given out of the set of critical values
\be\lab{OR}\begin{aligned} 
	\OR:=\Big\{\beta:~\beta=4\pi m+4\pi\sum_{a\in\DR'}(1+\al_{a})~&\text{for some}~m\in\N,~\DR'\subset\DR~\\
		&\text{such that}~m+|\DR'|>0  \Big\},  
\end{aligned}\ee
where $|\DR'|$ denotes the number of the elements of $\DR'$.

\subsection{Conformal to smooth metric}\ 

To keep on computing, it is convenient to translate the conical singular metric $g_1$ to some smooth metric $g_0$, so that $g_1$ is also called a conformal conical metric. 
We take a smooth basic metric $g_0$ in $\Sigma$ such that, there exists a function $h:\Sigma\to\R$ satisfying that 
\be\lab{h} g_1=h\cdot g_0,\ee
and for each $a\in\DR$, 
\be\lab{hh} h\sim d_{g_0}(~\cdot~,a)^{2\al_{a}}\quad\text{near}~a. \ee
Additionally, like in many papers (see e.g. \cite{clas-0,bg-1,bg-3,bg-4}) concerning conical singularities in PDEs, we require that 
\be\lab{hhh} d_{g_0}(~\cdot~,a)^{-2\al_{a}}h~\text{is $C^1$ near each singularity}~a\in\DR. \ee
The existence of such conical metrics has been given in \cite{bg1-1,barycenter-3}. Generally speaking, the function $d_{g_0}(~\cdot~,a)^{-2\al_{a}}h$ is just continuous and positive near $a\in\DR$, and is smooth outside $\DR$. 
By the elliptic theory (see \cite[Lemma 3]{Singular-1}), we get that for each $a\in\DR$, the function $h$ would be of class $C^1$ near $a$ if $\al_a\in\big(-\f{1}{2},0\big)$; while the function $h$ would be of class $C^{0,\delta}$ near $a$ if $\al_a\in\big(-1,-\f{1}{2}\big]$. 
By assuming \eqref{hhh}, we mainly focus on a singular Moser-Trudinger equation, with the singularities reflected in the coefficient function (see \eqref{equ-3} below), which can be seen as a generalization on the nonlinearity of the singular Liouville type equation  studied in \cite{bg-1,bg1-1}.

We introduce the norm $\nm{~\cdot~}_{h}$ given by
\be  \nm{u}_h^2=\int_\Sigma |\nabla_{g_0}u|^2+hu^2\rd v_{g_0}. \ee
Keeping then the notation in \eqref{norm-1}, we have $$\nm{u}_h=\nm{u}_{H^1(\Sigma,g_1)} \quad\text{for all $u\in H^1(\Sigma,g_1)$},$$ thanks to
	$$\nabla_{g_1}=h^{-1/2}\nabla_{g_0},\quad \rd v_{g_1}=h\rd v_{g_0}.$$
By \cite[Proposition 3]{Singular-1}, it holds $H^1(\Sigma,g_1)=H^1(\Sigma,g_0)$ in the sense that $\nm{~\cdot~}_{H^1(\Sigma,g_1)}\sim\nm{~\cdot~}_{H^1(\Sigma,g_0)}$, so we denote $$H^1(\Sigma)=H^1(\Sigma,g_1)=H^1(\Sigma,g_0),$$ with $\nm{~\cdot~}_h$ as its norm. Besides, since $\Delta_{g_1}=h^{-1}\Delta_{g_0}$ by the conformal covariance of the Laplacian, we obtain that $u$ solves \eqref{equ-2} if and only if it solves
\be\lab{equ-3}  
	-\Delta_{g_0}u+hu=p\la hfu^{p-1}e^{u^p}, \quad u>0 \quad\text{in}~\Sigma.
\ee 
Since $h\sim d_{g_0}(\cdot,a)^{2\al_{a}}$ near each singularity $a$, we call equation \eqref{equ-3} a singular Moser-Trudinger equation, which can be seen as a generalization of the Moser-Trudinger equation studied in \cite{MT-blowup-7}.

For all $p\in[1,2)$, by conformal to the smooth metric $g_0$, we get that the aforementioned critical points $u$ of $I_{p,\beta}$ solving \eqref{equ-2}-\eqref{la-1} are exactly those of the functional $J_{p,\beta}$ given by
\be\lab{Jp} J_{p,\beta}(u):=\f{2-p}{2}\sbr{\f{p\nm{u}_h^2}{2\beta}}^{\f{p}{2-p}}-\ln\int_\Sigma hf\sbr{e^{u_+^p}-1}\rd v_{g_0}, \ee
solving \eqref{equ-3} with $\la>0$ be given by
\be\lab{la-2}   \f{\la p^2}{2}\sbr{\int_\Sigma hf\sbr{e^{u^p}-1}\rd v_{g_0}}^{\f{2-p}{p}}\sbr{\int_\Sigma hfu^pe^{u^p}\rd v_{g_0}}^{\f{2(p-1)}{p}}=\beta. \ee
Now, even for $p=2$, we also have  that $u\in H^1(\Sigma)$ solves our problem \eqref{equ-2} if and only if $u$ solves \eqref{equ-3} for $p=2$, namely
\be\lab{equ-4} -\Delta_{g_0}u+hu=2\la hfue^{u^2}, \quad u>0 \quad\text{in}~\Sigma,\ee
with $\la>0$ given by \eqref{la-2}.

\subsection{Main results}\

First, we get the following sharp quantization result, which determines in a precise way the possible blow-up energy levels.
\begin{theorem}\lab{thm0}
Let $h$ be given by \eqref{h} satisfying \eqref{hhh} and $f$ be a smooth positive function. Let $(\la_n)_{n}$ be a sequence of positive real numbers and $(p_n)_n$ be any sequence of numbers in $[1,2]$. Let $(u_n)_n$ be a sequence of solutions solving 
\be\lab{equ-5} -\Delta_{g_0}u_n+hu_n=p_n\la_n hfu_n^{p_n-1}e^{u_n^{p_n}},\quad u_n>0\quad\text{in}~\Sigma, \ee
for any $n$. Let $(\beta_n)_n$ be given by
\be\lab{betan} \beta_n=\f{\la_n p_n^2}{2}\sbr{\int_\Sigma hf\sbr{e^{u_n^{p_n}}-1}\rd v_{g_0}}^{\f{2-p_n}{p_n}}\sbr{\int_\Sigma hfu_n^{p_n}e^{u_n^{p_n}}\rd v_{g_0}}^{\f{2(p_n-1)}{p_n}} \ee
for any $n$. If we assume the energy bound
\be\lab{bound} \lim_{n\to+\iy}\beta_n=\beta\in(0,+\iy), \ee
but the pointwise blow-up of the $u_n$, namely
\be\lab{blow-up} \lim_{n\to+\iy}\max_\Sigma u_n=+\iy, \ee
then, up to a subsequence, there holds $\lim_{n\to+\iy}\la_n=0$ and 
\be\lab{level-1} \beta\in\OR, \ee
where $\OR$ is given by \eqref{OR}.
Moreover, if $p_n\equiv p\in(1,2]$ for all $n$, then 
\be\lab{level-2} \beta_n>\beta \ee 
for large $n$.
\end{theorem}

The basic idea of the proof of Theorem \ref{thm0} mainly follows \cite[Section 2-5]{MT-blowup-7}, but the appearance of the conical singularities brings many more intrinsic difficulties, and we need to develop different ideas and techniques. Let us briefly explain our strategy and point out that our method highly and nontrivially relies on our condition $\al_{a}<0$, as already mentioned in Remark \ref{remark1-1}. 

Since we assume $u_n$ blows up, i.e., \eqref{blow-up} holds, we first locate the concentration points. 
To achieve this purpose, the strategy in \cite{MT-blowup-2} seems to be inefficient due to the appearance of singularities, and also the strategy in \cite{MF-1} seems to be inefficient due to the possible critical nonlinearity for $p_n\to2$. Instead, we introduce a selection process suitable for the singular Moser-Trudinger equation, which has been widely used in many problems, for example, in prescribing curvature-type equations \cite{Selection-1,Selection-2}, in singular Toda systems \cite{Selection-3} and in Chern-Simons-Higgs equations \cite{Selection-4}. 
Here, we modify it to select the concentration points one by one. 
Then, the energy bound \eqref{bound} implies the finiteness of concentration points, and a rough blow-up picture is gathered in Proposition \ref{bubble}. 
We observe that around each concentration point, there are some (maybe more than one) bubbles concentrating at this point. By a suitable scaling of $u_n$, if it leads to the non-singular Liouville equation
\be\lab{Liouville-1} -\Delta \eta=e^{2\eta}\ \ \text{in}\ \R^2,\quad \int_{\R^2}e^{2\eta}\rd x<+\iy, \ee
then we say there is a non-singular bubble; if it leads to the singular Liouville equation
\be\lab{Liouville-2} -\Delta \eta=|\cdot|^{2\al}e^{2\eta}\ \ \text{in}\ \R^2,\quad \int_{\R^2}|x|^{2\al}e^{2\eta}\rd x<+\iy,\ee
then we say there is a singular bubble. It is well known that solutions of \eqref{Liouville-1} have been classified by \cite{clas-2}, and solutions of \eqref{Liouville-2} have been classified by \cite{clas-0,clas-1}. We find that if the concentration point is not in $\DR$, then all the bubbles concentrating at this point are non-singular bubbles, and the local energy quantization has been actually given by \cite[Theorem 4.1]{MT-blowup-7}. 

The main task of this paper is to study the new situation that the concentration point is in $\DR$, for which we need to develop new ideas and techniques.  For each concentration point $a\in\DR$, we find that apart from non-singular bubbles, there appears at most one singular bubble, so that we have two types of bubbles (see Section \ref{sec2-4} for details):
\begin{itemize}[fullwidth,itemindent=1em]
\item[{\bf Type I } ] {\it There are only non-singular bubbles concentrating at $a\in\DR$.} 
\item[{\bf Type II}] {\it There are a singular bubble and some non-singular bubbles concentrating at $a\in\DR$.} 
\end{itemize}

To get the quantization for Type I, we compare in small balls our blow-up solutions $u_n$ with some standard non-singular bubbles, some radially symmetric functions solving the same PDE without singularity, and we directly show that the difference is ``small'' and the limit of its gradient is ``good''. From these nice estimates, we get that the union of these separate balls is ``large'' in the sense that the complementary region can not contribute energy in the quantization, and this process relies on our condition $\al_a<0$, see Remark \ref{614-1}. 

To get the quantization for Type II, apart from non-singular bubbles, we have to compare in small balls our blow-up solutions $u_n$ with some standard singular bubbles, some radially symmetric functions solving the same PDE with singularity. However, due to $\al_a<0$, there is no gradient estimates of the difference near the singularity, and to overcome this difficulty, we use the H\"older semi-norm and again get the pointwise smallness of the difference, see Remark \ref{515-2}. Another difficulty appears when showing the largeness of the union of these separate balls, since we have no gradient convergences like in Type I. To solve this difficulty, we use the so called {\it local Pohozaev identity} technique, which also relies on $\al_{a}<0$, see Remark \ref{611-1}. 

One delicate consequence is that each ball associated to a non-singular bubble brings the minimal energy $4\pi$, while each ball associated to a singular bubble brings the minimal energy $4\pi(1+\al_a)$ in \eqref{level-1}. 
Finally, the sharp estimate \eqref{level-2} relies on the delicate energy expansion of $u_n$ carried out in \cite[Section 5]{MT-blowup-7}, originated in \cite{MT-blowup-4,MT-blowup-5}.

We shall compare the local Pohozaev identity here with those of Liouville type equations. For example, for the Chern-Simons-Higgs equation in \cite{Selection-4}, the local Pohozaev identity is obtained by multiplying the equation with $(x-x_{n,i})\cdot\nabla u_n$ and integrating over a small ball around $x_{n,i}$, where $x_{n,i}$ is some bubble center, and then by computing energy in neck domains by some decay estimates, one can get energy identities for the Chern-Simons-Higgs equation. However, for the singular Moser-Trudinger equation here, apart from the same process to the Chern-Simons-Higgs equation, we have to balance the Pohozaev identity by multiplying $\f{p_n}{2}u_n(x_{n,j})^{p_n-1}$ near each bubble center $x_{n,j}$ due to our blow-up scaling $\f{p_n}{2}u_n(x_{n,j})^{p_n-1}(u_n(x_{n,j}+\mu_{n,j}\cdot)-u_n(x_{n,j}))$. As a result, the local Pohozaev identity will lead to
  $$\sbr{(1+\al_a)+\sum_{j\in J_i\setminus\{i\}}\theta_j}^2=(1+\al_a)\sbr{(1+\al_a)+\sum_{j\in J_i\setminus\{i\}}\theta_j^2},$$
where $\theta_j>0$ are limits of the ratio of the bubble heights $u_n(x_{n,j})$ and $u_n(x_{n,i})$. 
The problem is that we don't know apriori the relationship between different bubble heights $u_n(x_{n,i})$ and $u_n(x_{n,j})$, i.e., we have no $\theta_j=1$. 
That's why we can only use this identity to obtain a contradiction for the case $\al_a\in(-1,0)$.

\vskip0.1in


By applying Theorem \ref{thm0}, we can obtain the following compactness result.
\begin{theorem}\lab{thm1}
Let $(\Sigma,g_1)$ be a compact surface without boundary, where $g_1$ is a conformal conical metric on $\Sigma$ representing divisor $\bm{\al_1}$ given by \eqref{divisor}, such that $g_1$ satisfies \eqref{h}-\eqref{hhh}. Let $f$ be a smooth positive function. Let $p\in[1,2]$ and $\beta>0$. Then the following hold.
\begin{itemize}
\item[(1)] If $p\in(1,2)$, the set $\CR_{p,\beta}$ of the positive critical points of $I_{p,\beta}$ is compact in $C^0(\Sigma)$;
\item[(2)] If $p=1$ and $\beta\not\in\OR$, the set $\CR_{1,\beta}$ of the positive critical points of $I_{1,\beta}$ is compact in $C^0(\Sigma)$;
\item[(3)] If $p=2$, the set $\CR_{2,\beta}$ of the positive critical points of $F_1$ constrained to $\ER_\beta$ is compact in $C^0(\Sigma)$.
\end{itemize}
\end{theorem}

Theorem \ref{thm1} generalizes the compactness result of Theorem A to singular surfaces.  Like Theorem A,
a notable fact in Theorem \ref{thm1} is that, except for $p=1$, the full range $\beta>0$ is covered and in particular also the case $\beta\in\OR$. In fact, by writing
  \[\OR=\lbr{4\pi k_i:0<k_1<k_2<\cdots,k_i>0},\]
we also prove that the sets
  \[\bigcup_{\substack{\beta\in[4\pi k_i+\delta,4\pi k_i]\\ p\in[1+\delta,2]}}\CR_{p,\beta}, \quad \bigcup_{\substack{\beta\in[4\pi k_i+\delta,4\pi k_i-\delta]\\ p\in[1,2]}}\CR_{p,\beta}\]
are compact for any fixed $\delta>0$, i.e., blow-up can only occur either for $\beta\downarrow4\pi k_i$ or for $p\to1$ and $\beta\to4\pi k_i$ for some $k_i$.

\vskip0.1in
\noindent{\bf Notations.}
\begin{itemize}
\item For sequences $A_n$ and $B_n$, we use $A_n=O(B_n)$ to denote $A_n\le CB_n$ for some constant $C>0$ independent of $n$, and $A_n\sim B_n$ to denote $\f{1}{C}B_n\le A_n\le C B_n$ for some constant $C>1$ independent of $n$, and $A_n=o(B_n)$ or $A_n\ll B_n$ to denote $\displaystyle\lim_{n\to+\iy}A_n/B_n=0$.
\item We use $q$ or $q_n$ to denote points in $\Sigma$, and $x,y,z$ to denote points in $\R^2$. We use $d_{g_0}(\cdot,\cdot)$ to denote distance in $(\Sigma,g_0)$, and use $d(\cdot,\cdot)$ to denote distance in $\R^2$. Without causing any ambiguity, we use the same notation  $B_r(*)$ to denote the ball of radius $r$ and center $*$ in $\Sigma$ or $\R^2$. 
\item In this article, when we get a convergence result, we often means in the sense of up to a subsequence, and we don’t use different notations for sequences and subsequences.  We use $C,\bar C,\wt C,C_1,\cdots$  to denote constants whose values may be different in different places.
\item In the sequel, for any radially symmetric function $f$ around $0\in\R^2$, since no confusion is then possible, we often make an abuse of notation and write $f(r)$ instead of $f(x)$ for $|x|=r$.
\end{itemize}

\vs
\section{Preliminaries}
\subsection{Concentration phenomena}\ 

Let $h$ be given by \eqref{h} and $f$ be a smooth positive function. 
Let $(\la_n)_n$ be a sequence of positive numbers, $(p_n)_n$ be a sequence of numbers in $[1,2]$, $(u_n)_n$ be a sequence of solutions solving \eqref{equ-5}. Let $(\beta_n)_n$ be given by \eqref{betan}. We also assume \eqref{bound}. Then, by integrating \eqref{equ-5} in $\Sigma$ and using $pt^{p-1}e^{t^p}\ge t$ for all $t\geq0$ and $p\in[1,2]$, we obtain
\[\int_{\Sigma}hu_n\rd v_{g_0}=\la_n\int_{\Sigma}p_nhfu_n^{p_n-1}e^{u_n^{p_n}}\rd v_{g_0}\geq \la_n\min_{\Sigma}f\int_{\Sigma}hu_n\rd v_{g_0},\]
and so
\be\lab{boundla} \la_n\le \f{1}{\min_{\Sigma}f}.\ee
Since $\f{2-p_n}{p_n}+\f{2(p_n-1)}{p_n}=1$, H\"older inequality and \eqref{betan}-\eqref{bound} give that
\be\lab{tem-1}\la_n\int_\Sigma hfu_n^{2(p_n-1)}\sbr{e^{u_n^{p_n}}-1}\rd v_{g_0}=O(\beta_n)=O(1). \ee
By \eqref{boundla}-\eqref{tem-1}, $p_n\in[1,2]$ and the fact that $h\in L^1(\Sigma,g_0)$, we have
\be\lab{tem-2}\begin{aligned}
	\la_n\int_\Sigma hu_n^{2(p_n-1)}\rd v_{g_0}&=\la_n\int_{\{u_n\le2\}}hu_n^{2(p_n-1)}\rd v_{g_0}+\la_n\int_{\{u_n>2\}}hu_n^{2(p_n-1)}\rd v_{g_0}\\
	&\le O(1)+\frac{\la_n}{e^2-1}\int_\Sigma hu_n^{2(p_n-1)}(e^{u_n^{p_n}}-1)\rd v_{g_0}=O(1),
\end{aligned}\ee
which implies
\[ \la_n\int_\Sigma hu_n^{2(p_n-1)}e^{u_n^{p_n}}\rd v_{g_0}=\la_n\int_\Sigma hu_n^{2(p_n-1)}(e^{u_n^{p_n}}-1)\rd v_{g_0}+\la_n\int_\Sigma hu_n^{2(p_n-1)}\rd v_{g_0}=O(1).\]
Then as a consequence, for any $p\in[0,2(p_n-1)]$,
\be\lab{bound-1}\begin{aligned} \la_n\int_\Sigma hu_n^pe^{u_n^{p_n}}\rd v_{g_0}
&=\la_n\int_{\{u_n\leq 1\}} hu_n^pe^{u_n^{p_n}}\rd v_{g_0}+\la_n\int_{\{u_n> 1\}} hu_n^pe^{u_n^{p_n}}\rd v_{g_0}\\
&\leq O(1)+ \la_n\int_\Sigma hu_n^{2(p_n-1)}e^{u_n^{p_n}}\rd v_{g_0}=O(1).\end{aligned}\ee

As a first step, observe that we can directly get the following rough, subcritical but global bounds on the $u_n$.
\begin{lemma}\lab{bound-0}
There exists a $C>0$ and a small $\ep_0>0$ such that
	$$\int_\Sigma he^{\ep_0u_n^{1/3}}\rd v_{g_0}\le C$$
for all $n$. In particular, we have that for any $0<s<+\iy$, $hu_n^s$ and $u_n^s$ are both bounded in $L^1(\Sigma, g_0)$, and that for any $s'\in(1,1/|\al_\DR|)$, $(hu_n)^{s'}$ is bounded in $L^{1}(\Sigma,g_0)$. 
\end{lemma}
\begin{proof}
Thanks to the singular Moser-Trudinger inequality \eqref{SMT}, we can repeat the proof of \cite[Lemma 4.1]{MT-blowup-7} to obtain
  $$\int_\Sigma he^{\ep_0u_n^{1/3}}\rd v_{g_0}\le C,$$
from which we see that $hu_n^s$ is bounded in $L^1(\Sigma, g_0)$ for any $0<s<+\iy$. Since $h>0$ is bounded away from zero due to $\al_\DR\in(-1,0)$, we get that $u_n^s$ is bounded in $L^1(\Sigma, g_0)$ for any $0<s<+\iy$. Note that \[h\in L^{s}(\Sigma, g_0)\quad\text{ for any } \quad s\in [1, 1/|\al_{\DR}|).\]
For any $s'\in(1,1/|\al_\DR|)$, we may take a $t>1$ satisfying $2+2\al_\DR\f{s't-1}{t-1}>0$, so that $h^{\f{s't-1}{t-1}}\in L^1(\Sigma,g_0)$, and then, by H\"older inequality, we get that
$$\begin{aligned}
  \sbr{\int_{\Sigma}(hu_n)^{s'}\rd v_{g_0}}^{1/s'}
  &=O\sbr{ \sbr{\int_{\Sigma}h^{\f{s't-1}{t-1}}\rd v_{g_0} }^{\sbr{1-\f{1}{t}}\f{1}{s'}} 
    \sbr{\int_{\Omega}hu_n^{s't}\rd x }^{\f{1}{s't}}  }=O(1).
\end{aligned}$$
This finishes the proof. 
\end{proof}

Now, we also assume that $u_n$ blows up, namely we assume that \eqref{blow-up} holds. By using \eqref{bound-1}, we can get a pointwise exhaustion of concentration points, and then get a rough blow-up picture gathered in Proposition \ref{bubble} below,  the detailed proof of which will be postponed in Section \ref{sec3}.
As mentioned in the introduction, the strategy in \cite[Section 2]{MT-blowup-2} seems to be inefficient due to the appearance of singularities, and instead, we introduce a widely used selection process suitable for $u_n$. Then, we locate the concentration points, which is a union of finite points, and more importantly we get the pointwise estimate and the gradient estimate for $u_n$. 

To give the statement of Proposition \ref{bubble}, we need to introduce some notations that will be used frequently in this paper.
\begin{itemize}
\item
Let $\DR\subset\Sigma$ be the singularity set, and for any fixed point $q_0\in\DR$, we let $q_{n,0}=q_0$ for any $n$, and also let points $q_{n,i}\in\Sigma\setminus\DR$ for all $i\in\{1,\cdots,N\}$ and all $n$. 
\item As in \cite{MT-blowup-7}, for any $i\in\{0,1,\cdots,N\}$, we may choose isothermal coordinates $(B_{\kappa_1}(q_{n,i})$, $\phi_{n,i},U_{n,i})$ around $q_{n,i}$, such that $\phi_{n,i}$ is a diffeomorphism from $B_{\kappa_1}(q_{n,i})\subset\Sigma$ to $U_{n,i}\subset\R^2$, where $\kappa_1>0$ is some appropriate given positive constant and $B_{\kappa_1}(q_{n,i})$ is the ball of radius $\kappa_1$ and center $q_{n,i}$ for the metric $g_0$, such that $\phi_{n,i}(q_{n,i})=0$, $B_{2\kappa}(0)\subset U_{n,i}$, and such that $(\phi_{n,i})_*g_0=e^{2\vp_{n,i}}|\rd x|^2$, where $|\rd x|^2$ denotes the standard metric of $\mathbb R^2$. 
We may also assume that $(\vp_{n,i})_n$  satisfies
\be\lab{vpn} \forall~n,\quad \vp_{n,i}(0)=0 \quad\text{and}\quad \lim_{n\to+\iy}\vp_{n,i}=\vp_i~\text{in}~C_{\loc}^2(B_{2\kappa}(0)). \ee
\item At last, we set
\be\lab{ufh} u_{n,i}=u_n\circ\phi_{n,i}^{-1},\quad f_{n,i}=f\circ\phi_{n,i}^{-1} \quad\text{and}\quad h_{n,i}=e^{2\vp_{n,i}}\cdot\sbr{h\circ\phi_{n,i}^{-1}} \ee
in $B_{2\kappa}(0)$.  
\end{itemize}

\begin{proposition}\lab{bubble}
Up to a subsequence, there exist a non-negative integer $N$, and $N$ sequences of points $(q_{n,i})_n$ in $\Sigma\setminus\DR$ for $i\in\{1,\cdots,N\}$ (if $N\ge1$), such that the following things hold:
\begin{itemize}[fullwidth,itemindent=1em] 
\item[(1)] for all $i\in\{1,\cdots,N\}$, there hold $\nabla u_n(q_{n,i})=0$, and by setting $\ga_{n,i}=u_n(q_{n,i})$,
\be\lab{Sigman} \Sigma_n:=\lbr{q_{n,1},\cdots,q_{n,N}}, \ee
\be\lab{mun-1} \mu_{n,i}:=\sbr{\f{8}{\la_np_n^2h_{n,i}(0)f_{n,i}(0)\ga_{n,i}^{2(p_n-1)}e^{\ga_{n,i}^{p_n}}}}^{\f{1}{2}}, \ee
we have $\ga_{n,i}\to+\iy$, $\mu_{n,i}\to0$,
\be\lab{dist-1} \f{d_{g_0}(q_{n,i},\DR\cup\Sigma_n\setminus\{q_{n,i}\})}{\mu_{n,i}}\to+\iy, \ee
and 
\be\lab{conv-1} \f{p_n}{2}\ga_{n,i}^{p_n-1}(\ga_{n,i}-u_{n,i}(\mu_{n,i}\cdot))\to \ln(1+|\cdot|^2)~\text{in}~C_{\loc}^1(\R^2), \ee
as $n\to+\iy$.
\item[(2)] for any $q_0\in\DR$, let $q_{n,0}=q_0$ and $\ti h_{n,0}=|\cdot|^{-2\al}h_{n,0}$ in $B_{2\kappa}(0)$ for all $n$, then there is an alternative that either the following holds or not: let $\ga_{n,0}=u_n(q_{n,0})$,
\be\lab{mun-2} \mu_{n,0}:=\sbr{\f{8(1+\al_{q_0})^2}{\la_np_n^2\ti h_{n,0}(0)f_{n,0}(0)\ga_{n,0}^{2(p_n-1)}e^{\ga_{n,0}^{p_n}}}}^{\f{1}{2(1+\al_{q_0})}}, \ee
we have $\ga_{n,0}\to+\iy$, $\mu_{n,0}\to0$,
\be\lab{dist-2} \f{d_{g_0}(q_{n,0},\DR\cup\Sigma_n\setminus\{q_{n,0}\})}{\mu_{n,0}}\to+\iy, \ee
and 
\be\lab{conv-2} \f{p_n}{2}\ga_{n,0}^{p_n-1}(\ga_{n,0}-u_{n,0}(\mu_{n,0}\cdot))\to \ln(1+|\cdot|^{2(1+\al_{q_0})})~\text{in}~C_{\loc}^{0,\delta}(\R^2)\cap C_\loc^1(\R^2\setminus\{0\}) \ee
for any $\delta\in(0,\min\{1,2(1+\al_{q_0})\})$, as $n\to+\iy$. Moreover, we define the set $\DR'$ of singular concentration points as 
\be\lab{DR'} \DR'=\lbr{q_0\in\DR:~\text{ the above alternative \eqref{mun-2}-\eqref{conv-2} holds for $q_0$}}, \ee
then there must be $N+|\DR'|>0$. 
\item[(3)] there exists a constant $C>0$ such that
\be\lab{est-1} d_{g_0}(~\cdot~,\DR\cup\Sigma_n)^2\la_nhfu_n^{2(p_n-1)}e^{u_n^{p_n}}\le C\quad\text{in}~\Sigma\setminus\bigcup_{a\in\DR}B_{\mu_n(a)}(a), \ee
and
\be\lab{est-2} d_{g_0}(~\cdot~,\DR\cup\Sigma_n)u_n^{p_n-1}|\nabla_{g_0} u_n|\le C\quad\text{in}~\Sigma\setminus\bigcup_{a\in\DR}B_{\mu_n(a)}(a), \ee
for all $n$, where for each $a\in\DR$, $\mu_n(a)$ is defined by
	$$\mu_n(a):=\begin{cases}\mu_{n,0}, &\quad\text{if the alternative of (2) holds for $a$, i.e., $a\in\DR'$}, \\0, &\quad\text{if the alternative of (2) doesn't hold for $a$, i.e., $a\not\in\DR'$}, \end{cases}$$
namely $\cup_{a\in\DR}B_{\mu_n(a)}(a)=\cup_{a\in\DR'}B_{\mu_n(a)}(a)$.
\item[(4)] we also have that $\displaystyle\lim_{n\to+\iy}q_{n,i}=q_i$ for all $i\in\lbr{1,\cdots,N}$, and that there exists $u_0\in C^2(\Sigma\setminus(\DR\cup\Sigma_\iy))$ such that
\be\lab{conv-0} \lim_{n\to+\iy}u_n=u_0\quad\text{in}~C_\loc^{2}(\Sigma\setminus(\DR\cup\Sigma_\iy)), \ee
where 
\be\lab{Sigma}\Sigma_\iy:=\{q_1,\cdots,q_N\}. \ee
\end{itemize}
\end{proposition}

\begin{remark} (1) Note that there might happen $q_i=q_j$  for $i\neq j$ in $\Sigma_\iy$ and  $\Sigma_\iy\cap \DR\neq\emptyset$. On the other hand, we will prove in Section \ref{sec7} that $\Sigma_\iy\cup \DR'$ is precisely the blow-up set of $(u_n)_n$, namely for any point $q\in\Sigma\setminus(\Sigma_\iy\cup \DR')$, there is $r_q>0$ such that $\max_{B_{r_q}(q)}u_n\leq C$ for all $n$.

(2) In Proposition \ref{bubble}-(2), we put the singular bubble centering at the singularity (i.e. $q_{n,0}\equiv q_0$), which will bring great convenience to the ODE discussion in Section \ref{appe2-1}. Moreover, although we lose the maximality of the singular bubble center $q_0\in\DR$, i.e., $\nabla_{g_0}u_n(q_0)\neq 0$ may happen, we still have a classification result for the linear system due to $\al_{q_0}\in(-1,0)$, see Lemma \ref{appe3-1} (comparing with Lemma \ref{appe3-0}) in Appendix \ref{appe3}. 

(3) The estimate \eqref{est-1} is reasonable. Let $q_0\in\DR$ with $\mu_n(q_0)=0$, i.e., $q_0\notin \DR'$. Noting that $d_{g_0}(~\cdot~,\DR\cup\Sigma_n)\approx d_{g_0}(~\cdot~,q_0)$ near $q_0$, we get from \eqref{hh} that $d_{g_0}(~\cdot~,\DR\cup\Sigma_n)^2h(\cdot)$ is bounded near $q_0$ even though $h$ is infinity at $q_0$. 
\end{remark}

\vs
\subsection{Local compactness at the non-singularity}\lab{sec2-3}\

By Proposition \ref{bubble}, we give the following definition.
\begin{definition}
If $N\ge1$, we say that there is a non-singular bubble at $q_{n,i}$ for each $i\in\{1,\cdots,N\}$. 
If the alternative in conclusion (2) of Proposition \ref{bubble} holds for $q_0\in\DR$, i.e., $q_0\in\DR'$, we say that there is a singular bubble at $q_{n,0}=q_0$. 
\end{definition}
Remark that, Proposition \ref{bubble}-(2) shows that at each singularity $q_0\in\DR$, there is at most one singular bubble concentrating at $q_0$. 

Moreover, we also find that if the concentration point is not a singularity, then all the bubbles concentrating at this point are non-singular bubbles, and the local energy quantization has been actually given by \cite{MT-blowup-7}. More precisely, for any subdomain $\Omega\subset\Sigma$, let
\be\lab{betadomain} \beta_n(\Omega):=\f{\la_n p_n^2}{2}\sbr{\int_\Omega hf\sbr{e^{u_n^{p_n}}-1}\rd v_{g_0}}^{\f{2-p_n}{p_n}}\sbr{\int_\Omega hfu_n^{p_n}e^{u_n^{p_n}}\rd v_{g_0}}^{\f{2(p_n-1)}{p_n}}, \ee
then the following statement was proved in \cite{MT-blowup-7}.

\begin{theorem}\lab{thm0-1}(\cite[Section 4-5]{MT-blowup-7})
Let $q_*\in\Sigma_\iy\setminus\DR$ and $\kappa_*>0$ such that $B_{2\kappa_*}(q_*)\cap(\DR\cup\Sigma_\iy)=\{q_*\}$, then it holds $\displaystyle\lim_{n\to+\iy}\la_n=0$ and 
	$$\lim_{n\to+\iy}\beta_n(B_{\kappa_*}(q_*))\in 4\pi\N_+.$$
Moreover, if $p_n\equiv p\in(1,2]$ for all $n$, then
	$$\beta_n(B_{\kappa_*}(q_*))-\lim_{n\to+\iy}\beta_n(B_{\kappa_*}(q_*))>0$$ 
for large $n$.

\end{theorem}

\vs
\subsection{Local compactness at the singularity}\lab{sec2-4}\

Thanks to Theorem \ref{thm0-1}, we only need to consider the case that the concentration point is a singularity. For this case,  we find that apart from non-singular bubbles, there appears at most one singular bubble (see the conclusion (2) of Proposition \ref{bubble}), so that we have two different types of concentration phenomena. 

Let $N\ge0$ and $q_{n,i}$, $i\in\{1,\cdots,N\}$ (if $N\ge1$) be given by Proposition \ref{bubble}. 
Let the set $\DR'$ of singular concentration points be defined by \eqref{DR'}. For any $q_0\in\DR$, according to the singular bubble at $q_0$ appears or not, we consider:
\begin{itemize}
\vskip0.1in
\item[{\bf Type I}] {\it There are only non-singular bubbles at $q_0$, i.e., $q_0\in\Sigma_\iy\cap\DR\setminus\DR'$.} We assume that $N\ge1$ and there exists a positive integer $1\le N'\le N$ such that
	$$\lim_{n\to+\iy}q_{n,i}=q_0,~\forall~i\in\{1,\cdots,N'\},$$ 
and
	$$\lim_{n\to+\iy}q_{n,i}\neq q_0,~\forall~i\in\{N'+1,\cdots,N\},$$
and the alternative in conclusion (2) of Proposition \ref{bubble} doesn't hold for $q_0$.
\vskip0.1in
\item[{\bf Type II}] {\it There are a singular bubble and some non-singular bubbles at $q_0$, i.e., $q_0\in\DR'$.} We assume that there exists a non-negative integer $0\le N'\le N$ such that
	$$\lim_{n\to+\iy}q_{n,i}=q_0,\quad\forall~i\in\{1,\cdots,N'\}~\text{(if $N'\ge1$)},$$ 
and
	$$\lim_{n\to+\iy}q_{n,i}\neq q_0,\quad\forall~i\in\{N'+1,\cdots,N\},$$
and the alternative in conclusion (2) of Proposition \ref{bubble} does hold for $q_0$.
\end{itemize}
Then we will prove the following new quantization result, which is also the main contribution of this paper.

\begin{theorem}\lab{thm0-2}
Let $q_0\in(\Sigma_\iy\cap\DR\setminus\DR')\cup\DR'$ and $\kappa_0>0$ such that $B_{2\kappa_0}(q_0)\cap(\DR\cup\Sigma_\iy)=\{q_0\}$, then $\displaystyle\lim_{n\to+\iy}\la_n=0$ and 
	$$\lim_{n\to+\iy}\beta_n(B_{\kappa_0}(q_0))=\begin{cases}4\pi N',&\text{for { Type I, i.e., $q_0\in (\Sigma_\infty\cap \DR)\setminus \DR'$}},\\4\pi(1+\al_{q_0})+4\pi N',&\text{for {Type II}, i.e., $q_0\in\DR'$}. \end{cases}$$
Moreover, if $p_n\equiv p\in(1,2]$ for all $n$, then
	$$\beta_n(B_{\kappa_0}(q_0))-\lim_{n\to+\iy}\beta_n(B_{\kappa_0}(q_0))>0$$ 
for large $n$.
\end{theorem}

Proposition \ref{bubble} will be proved in Section \ref{sec3}, and 
the rest sections are mainly devoted to the proof of Theorem \ref{thm0-2}, which is very complicated and quite long. Once Theorem \ref{thm0-2} is proved, we can complete the proofs of our main results Theorems \ref{thm0} and \ref{thm1} in Section \ref{sec7}.

\vs
\section{Concentration phenomena}\lab{sec3}

In this section, we give the proof of Proposition \ref{bubble}.
Let $h$ be given by \eqref{h} and $f$ be a smooth positive function. 
Let $(\la_n)_n$ be a sequence of positive numbers satisfying \eqref{boundla}, $(p_n)_n$ be a sequence of numbers in $[1,2]$, $(u_n)_n$ be a sequence of functions solving \eqref{equ-5}. Let $(\beta_n)_n$ be given by \eqref{betan}. We also assume \eqref{bound}. 

The following notations are used only in this section. 
To avoid ambiguity with the number sequence $p_n\in [1,2]$, we use $\p$, $\p_n$ or $\p_{n,i}$ to denote points in $\Sigma$.
For any fixed point $\p\in\Sigma$, we may choose isothermal coordinates $(B_r(\p),\Psi_\p,\Omega_\p)$ around $\p$, such that $\Psi_\p:B_r(\p)\to\Omega_\p\subset\R^2$ is a diffeomorphism with $\Psi_\p(\p)=0$ and $(\Psi_\p)_*g_0=e^{2\psi_\p}|\rd x|^2$, and such that $\psi_\p(0)=0$ and $B_{r'}(0)\subset\Omega_\p$.  Then, without causing ambiguity, we don't distinguish $u_n,f,h$ with their corresponding functions in $\Omega_\p$, and absorb the transform factor $\psi_\p$ into the function $h$, that is, for $x\in\Omega_\p$, 
\be\lab{ufh-1} u_n(x):=u_n\circ\Psi_\p^{-1}(x), \quad f(x):=f\circ\Psi_\p^{-1}(x), \quad h(x):=e^{2\psi_\p(x)}\cdot(h\circ\Psi_\p^{-1})(x). \ee
Then it follows from \eqref{equ-5} that  $u_n$ satisfies the local equation
\be\lab{equ-4-1} -\Delta u_n+hu_n=\la_np_nhfu_n^{p_n-1}e^{u_n^{p_n}},\quad u_n>0, \quad\text{in}~\Omega_\p.\ee
In addition, for $q_0\in\DR$, we may choose $r$ small such that $B_r(q_0)\cap\DR=\{q_0\}$, then since $h$ has a conical singularity at $q_0$ of order $\al_{q_0}\in(-1,0)$, we may assume that
\be\lab{h0} h(x)=h_{q_0}(x)|x|^{2\al_{q_0}}, \quad \text{for}~x\in\Omega_{q_0}, \ee
where $h_{q_0}$ is a positive $C^1$ function.

\subsection{Selection of bubbles}\lab{sec3-1}\ 

The purpose of this subsection is to locate the concentration points, and the main idea comes from \cite{Selection-4}.
However, since we handle the singular case $\al_a<0$ (not $\al_a>0$ as in \cite{Selection-4}), which leads to that $h$ goes to infinity near each singularity, we need to use different auxiliary functions in different cases, see \eqref{604-1} and \eqref{604-2} for example.

\begin{proposition}\lab{bubble-1}
Assume that $u_n$ satisfies \eqref{blow-up}. Then, up to a subsequence, there exist a positive integer $M$, a sequence of finite points $\wt\Sigma_n:=\{\p_{n,1},\p_{n,2},\cdots,\p_{n,M}\}$ and positive numbers $d_{n,1},d_{n,2},\cdots,d_{n,M}$ such that the following hold.
\begin{itemize}[fullwidth,itemindent=1em]
\item[(1)] $\displaystyle\lim_{n\to+\iy}\p_{n,i}=\p_i^*\in\Sigma$, $\displaystyle\lim_{n\to+\iy}d_{n,i}=0$, and $d_{n,i}\le\f{1}{2}d_{g_0}(\p_{n,i},\wt\Sigma_n\setminus\{\p_{n,i}\})$, for any $i=1,\cdots,M$.
\item[(2)] $u_n(\p_{n,i})=\displaystyle\max_{B_{d_{n,i}}(\p_{n,i})}u_n$ and $\displaystyle\lim_{n\to+\iy}u_n(\p_{n,i})=+\iy$, for any $i=1,\cdots,M$.
\item[(3)] For any $i=1,\cdots,M$, taking the isothermal coordinate system $(B_{r}(\p_i^*),\Psi_i,\Omega_i)$ as above \eqref{ufh-1}, and setting $x_{n,i}=\Psi_i(\p_{n,i})$, we have 
\begin{itemize}[fullwidth,itemindent=2em]
\item[(3-1)] If $\p_i^*\not\in\DR$, or $\p_i^*\in\DR$ and $\la_n|x_{n,i}|^{2(1+\al_{\p_i^*})}u_n(x_{n,i})^{2(p_n-1)}e^{u_n(x_{n,i})^{p_n}}\to+\iy$ as $n\to+\iy$, by defining
	$$\mu_{n,i}^{-2}:=\f{1}{8}\la_np_n^2h(x_{n,i})f(x_{n,i})u_n(x_{n,i})^{2(p_n-1)}e^{u_n(x_{n,i})^{p_n}},$$
	$$\eta_{n,i}:=\f{p_n}{2}u_n(x_{n,i})^{p_n-1}(u_n(x_{n,i}+\mu_{n,i}\cdot)-u_n(x_{n,i})),$$
then $\mu_{n,i}\to0$ and
	$$\eta_{n,i}\to -\ln\sbr{1+|\cdot|^2} \quad\text{in}~C_\loc^1(\R^2),$$
as $n\to+\iy$. Furthermore, $\f{|x_{n,i}|}{\mu_{n,i}}\to+\infty$ if $\p_i^*\in\DR$.
\item[(3-2)] If $\p_i^*\in\DR$ and $\la_n|x_{n,i}|^{2(1+\al_{\p_i^*})}u_n(x_{n,i})^{2(p_n-1)}e^{u_n(x_{n,i})^{p_n}}=O(1)$, by defining
	$$\mu_{n,i}^{-2(1+\al_{\p_i^*})}:=\f{1}{8(1+\al_{\p_i^*})^2}\la_np_n^2h_{\p_i^*}(x_{n,i})f(x_{n,i})u_n(x_{n,i})^{2(p_n-1)}e^{u_n(x_{n,i})^{p_n}},$$
	$$\eta_{n,i}:=\f{p_n}{2}u_n(x_{n,i})^{p_n-1}(u_n(x_{n,i}+\mu_{n,i}\cdot)-u_n(x_{n,i})),$$
where $h_{\p_i^*}(x)=h(x)/|x|^{2\al_{\p_i^*}}$ is given by \eqref{h0}, then $\mu_{n,i}\to0$, $\f{|x_{n,i}|}{\mu_{n,i}}\to0$ and 
	$$\eta_{n,i}\to-\ln\sbr{1+|\cdot|^{2(1+\al_{\p_i^*})}}\quad\text{in}~C^{0,\delta}_{\loc}(\R^2)\cap C^1_\loc(\R^2\setminus\{0\}),$$
for any $\delta\in(0,\min\{1,2(1+\al_{\p_i^*})\})$, as $n\to+\iy$.
\end{itemize}
\item[(4)] For any $i=1,\cdots,M$, $\mu_{n,i}=o(d_{n,i})$.
\item[(5)] For any singularity $q_0\in\DR$, there is at most one point sequence $\p_{n,i_0}\in\wt\Sigma_n$ such that 
	\be\lab{tem-321-1}\p_{i_0}^*=q_0\quad\text{and}\quad\la_n|x_{n,i_0}|^{2(1+\al_{q_0})}u_n(x_{n,i_0})^{2(p_n-1)}e^{u_n(x_{n,i_0})^{p_n}}=O(1).\ee
Further, by collecting all singularities $q_0$ satisfying \eqref{tem-321-1} in a set $\DR'\subset\DR$ (maybe $\DR'=\emptyset$),
we have that 
	$$\lim_{n\to+\iy}\f{|x_{n,l}|}{|x_{n,i_0}|}=+\iy,\quad\text{for all}\; l\neq i_0\;\text{satisfying}\;\p_l^*=q_0,$$
for any $q_0\in \DR'$, and that
\be\lab{tem-325-1} d_{g_0}(\cdot,\wt\Sigma_n\cup\DR)^2\la_nhfu_n^{2(p_n-1)}e^{u_n^{p_n}}=O(1), \quad\text{uniformly in}\;\Sigma\setminus \bigcup_{q_0\in\DR'}B_{\mu_{n,i_0}}(q_0).\ee
\end{itemize}
\end{proposition}

\begin{remark}\lab{326-2}
Note that for any $q_0\in\DR'$, i.e., there is a unique $i_0$ such that \eqref{tem-321-1} holds, we see from $\Psi_{i_0}(q_0)=\Psi_{i_0}(\p_{i_0}^*)=0$ and $x_{n,i_0}=\Psi_{i_0}(\p_{n,i_0})$ that $d_{g_0}(\p_{n,i_0},q_0)=(1+o(1))|x_{n,i_0}|$. Since conclusion (3-2) of Proposition \ref{bubble-1} says $\frac{|x_{n,i_0}|}{\mu_{n,i_0}}\to 0$, 
we have $d_{g_0}(\p_{n,i_0},q_0)=o(\mu_{n,i_0})$, then $d_{g_0}(\cdot,q_0)=(1+o(1)) d_{g_0}(\cdot,\p_{n,i_0})$ uniformly in $\Sigma\setminus B_{\mu_{n,i_0}}(q_0)$, so that \eqref{tem-325-1} is actually
	$$\dist\sbr{\cdot,\wt\Sigma_n\cup(\DR\setminus\DR')}^2\la_nhfu_n^{2(p_n-1)}e^{u_n^{p_n}}=O(1), \;\text{uniformly in}\;\Sigma\setminus \bigcup_{q_0\in\DR'}B_{\mu_{n,i_0}}(q_0).$$
\end{remark}

\begin{proof}[\bf Proof of Proposition \ref{bubble-1}] 
Since all of our arguments are locally near some point, the number of singularities does not affect the proof and only brings complexity in notations, so we may assume that there is only one singularity, i.e., $\DR=\{q_0\}$ and the order of $q_0$ is $\al=\al_{q_0}\in(-1,0)$.

The proof is quite long and we divide it into several steps. 

\vskip0.1in
\noindent{\it Step 1: Construction of the first bubble $\p_{n,1}$ and $d_{n,1}$.}

By \eqref{blow-up}, we may choose $\p_n\in\Sigma$ such that
	$$\ga_n:=u_n(\p_n)=\max_{\Sigma}u_n\to+\iy,\quad\text{as}~n\to+\iy.$$
Up to a subsequence, we assume $\p_n\to \p^*$ as $n\to+\iy$.  We claim that 
\be\lab{514-1} \lim_{n\to+\iy}\la_n\ga_n^{2(p_n-1)}e^{\ga_n^{p_n}}=+\iy. \ee 
Indeed, if $\la_n\ga_n^{2(p_n-1)}e^{\ga_n^{p_n}}=O(1)$, then it follows from Lemma \ref{bound-0} and equation \eqref{equ-5} that $\Delta u_n$ is bounded in $L^{s}(\Sigma,g_0)$ for any $s\in(1,1/|\al|)$. Then, by the standard elliptic theory (see e.g. \cite[Theorem 9.11]{GT}), we get that $u_n$ is bounded in $W^{2,s}(\Sigma,g_0)$ for any $s\in(1,1/|\al|)$, so that by the embeeding theorem (see e.g. \cite[Theorem 7.26]{GT}), we get that $u_n$ is bounded in $L^\iy(\Sigma)$, which is a contradiction with $\displaystyle\lim_{n\to+\iy}\ga_n=+\iy$. This proves \eqref{514-1}.

Take the isothermal coordinates $(B_r(\p^*),\Psi,\Omega)$ around $\p^*$ as above \eqref{ufh-1}. Set $x_n=\Psi(\p_n)$ and 
\be\lab{3-13} \eta_n(\cdot):=\f{p_n}{2}\ga_n^{p_n-1}(u_n(x_n+\mu_n\cdot)-\ga_n)\quad\text{in}~\Omega_n:=\f{\Omega-x_n}{\mu_n}, \ee
where $\mu_n$ is a positive number going to $0$ that will be chosen later. Then $x_n\to 0$ and $\Omega_n\to\mathbb R^2$ as $n\to+\infty$. Now there are two possibilities.

\vskip0.1in
\noindent{\bf Case 1:} $\p^*=q_0$ and $|x_n|^{2+2\al}\la_n\ga_n^{2(p_n-1)}e^{\ga_n^{p_n}}=O(1)$.

In this case, we take $\mu_n$ by
\be\lab{mu-2-2} \mu_n^{-2(1+\al)}:=\f{1}{8(1+\al)^2}\la_np_n^2h_0(x_n)f(x_n)\ga_n^{2(p_n-1)}e^{\ga_n^{p_n}},\ee
where $h_0$ is a positive $C^1$ function given like in \eqref{h0}, i.e.,
\be\lab{h0-1} h(x)=h_0(x)|x|^{2\al}, \quad \text{in}~\Omega. \ee
Then, by \eqref{514-1} and \eqref{mu-2-2}, we get that $\mu_n=o(1)$ and $\f{|x_n|}{\mu_n}=O(1)$. Passing to a subsequence, we may assume $\displaystyle\lim_{n\to+\iy}\f{x_n}{\mu_n}=x_\iy$. Let $$v_n(\cdot):=\frac{u_n(x_n+\mu_n\cdot)}{\ga_n}\quad\text{in}\;\Omega_n=\frac{\Omega-x_n}{\mu_n}.$$
By using \eqref{mu-2-2}-\eqref{h0-1}, we get from \eqref{equ-4-1}  that $v_n$ satisfies
$$\begin{aligned}
	-\Delta v_n &=\mu_n^2\ga_n^{-1}\sbr{\la_np_nhfu_n^{p_n-1}e^{u_n^{p_n}}(x_n+\mu_n\cdot)-hu_n(x_n+\mu_n\cdot)}\\
		&=\f{8(1+\al)^2}{p_n}\abs{\cdot+\f{x_n}{\mu_n}}^{2\al}\f{h_0f(x_n+\mu_n\cdot)}{h_0(x_n)f(x_n)}\f{v_n^{p_n-1}}{\ga_n^{p_n}}e^{u_n^{p_n}-\ga_n^{p_n}}-\mu_n^2hv_n,
\end{aligned}$$
in $\Omega_n$ for all $n$. Noting that $\ga_n=\max_\Sigma u_n\to+\infty$, $\mu_n\to0$ and $\Omega_n\to\R^2$, we get that $|v_n|\le 1$ and $\nm{\Delta v_n}_{L^s(B_R(-x_\iy))}=o(1)$ for any $s\in(1,1/|\al|)$ and $R>1$. It follows from $v_n(0)\equiv1$ that 
\be\lab{3-10} \lim_{n\to+\iy}v_n= 1 \quad\text{in}~C^0_{\loc}(\R^2)\cap C^1_\loc(\R^2\setminus\{-x_\iy\}). \ee
Since $x_n=\Psi(\p_n)\to 0$, it is easy to see that for any $\tilde{R}_n>0$ satisfying $\tilde{R}_n\mu_n\to 0$ (such as $\tilde{R}_n\equiv R>0$) and all large $n$,
\be\lab{3-12} B_{\tilde{R}_n\mu_n/2}(\p_n) \subset\Psi^{-1}\sbr{B_{\tilde{R}_n\mu_n}(x_n)} \subset B_{2\tilde{R}_n\mu_n}(\p_n), \ee
we get from \eqref{3-10}, Lemma \ref{bound-0} and the H\"older inequality that
\be\lab{3-11}\begin{aligned}
	\mu_n\ga_n^s &=(1+o(1))\f{\mu_n}{\pi}\int_{B_1(0)}\ga_n^sv_n^s\rd x\leq C\mu_n\sbr{\int_{B_1(0)}\ga_n^{3s}v_n^{3s}\rd x}^{1/3}\\
&\leq C\mu_n^{1/3}\sbr{\int_{B_{\mu_n}(x_n)}u_n^{3s}\rd x}^{1/3}\leq C\mu_n^{1/3}=o(1),\quad \forall s\geq 1.
\end{aligned}\ee
Recalling $\eta_n$ defined in \eqref{3-13}, and by using \eqref{mu-2-2}-\eqref{h0-1}, we get from \eqref{equ-4-1} that
\be\lab{3-14}\begin{aligned}
	&-\Delta \eta_n \\ 
	&=4(1+\al)^2\abs{\cdot+\f{x_n}{\mu_n}}^{2\al}\f{h_0f}{h_0(x_n)f(x_n)}\f{u_n^{p_n-1}}{\ga_n^{p_n-1}}e^{u_n^{p_n}-\ga_n^{p_n}}(x_n+\mu_n\cdot)\\ &\quad -\f{p_n}{2}\mu_n^2\ga_n^{p_n-1}hu_n(x_n+\mu_n\cdot)\\
	&=4(1+\al)^2\abs{\cdot+\f{x_n}{\mu_n}}^{2\al}\f{h_0f(x_n+\mu_n\cdot)}{h_0(x_n)f(x_n)}\sbr{1+\f{2\eta_n}{p_n\ga_n^{p_n}}}^{p_n-1}e^{\ga_n^{p_n}[(1+\f{2\eta_n}{p_n\ga_n^{p_n}})^{p_n}-1]}\\
	&\quad -\f{p_n}{2}\mu_n^2\ga_n^{p_n-1}hu_n(x_n+\mu_n\cdot),\quad \text{in }\;\Omega_n.
\end{aligned}\ee
For any fixed $R>0$, we let $\eta_n=\eta_n^{(1)}+\eta_n^{(2)}$ with
	$$\begin{cases}\Delta\eta_n^{(1)}=\Delta\eta_n\quad\text{in}~B_R(0),\\
	\eta_n^{(1)}=0\quad\text{on}~\partial B_R(0),\end{cases}  \quad\text{and}\quad
	 \begin{cases}\Delta\eta_n^{(2)}=0\quad\text{in}~B_R(0),\\\eta_n^{(2)}=\eta_n\quad\text{on}~\pa B_R(0).\end{cases}$$
In view of \eqref{3-11} and the first expression of $\Delta\eta_n$ in \eqref{3-14}, we have that $\Delta\eta_n^{(1)}$ is bounded in $L^s(B_R(0))$ for any $s\in(1,1/|\al|)$, so by the elliptic theory (see e.g. \cite[Theorem 9.17]{GT}), we get that $\eta_n^{(1)}$ is bounded in $W^{2,s}(B_R(0))$ for any $s\in(1,1/|\al|)$, which implies that $\eta_n^{(1)}$ is bounded in $L^\iy(B_R(0))$. Since $\eta_n\le 0$, we get that $\eta_n^{(2)}$ is a harmonic function bounded from above in $B_R(0)$. Furthermore, $\eta_n^{(2)}(0)=\eta_n(0)-\eta_n^{(1)}(0)=-\eta_n^{(1)}(0)$ is bounded. Applying the Harnack inequality (see e.g. \cite{book-1}), we conclude that $\eta_n^{(2)}$ is bounded in $L^\iy(B_{R/2}(0))$, so that $\eta_n$ is bounded in $L^\iy_\loc(\R^2)$. Thus, by the elliptic theory, we get from the second equality in \eqref{3-14}  that
	$$\lim_{n\to+\iy}\eta_n=\eta \quad\text{in}~C^{0,\delta}_{\loc}(\R^2)\cap C^1_\loc(\R^2\setminus\{-x_\iy\}),$$
for any $\delta\in(0,\min\{1,2(1+\al)\})$, and $\eta$ satisfies 
	$$-\Delta\eta=4(1+\al)^2|x+x_\iy|^{2\al}e^{2\eta}\quad\text{in}~\R^2,\quad \eta(0)=0=\max_{\R^2}\eta.$$
Moreover, by using \eqref{bound-1} with $p=2(p_n-1)$ and Fatou lemma, we get $\int_{\R^2}|x+x_\iy|^{2\al}e^{2\eta(x)}\rd x<+\infty$. 
Then, noting $\al\in(-1,0)$, the classification result of  Prajapat-Tarantello \cite{clas-1} gives that $x_\iy=0$ and
	$$\eta(x)=-\ln\sbr{1+|x|^{2+2\al}},$$
and that
	$$\int_{\R^2}|x|^{2\al}e^{2\eta(x)}\rd x=\f{\pi}{1+\al}.$$
By $x_\iy=0$, we get that
\be\lab{514-2} |x_n|=o(\mu_n). \ee 
By $\mu_n=o(1)$, we may take a sequence $R_n\to+\iy$ such that \[R_n\mu_n=o(1), \quad\ln R_n=o(\ga_n^{p_n}),\quad\nm{\eta_n-\eta}_{L^\iy(B_{6R_n}(0))}=o(1).\]
Then
\be\lab{520-1}\begin{aligned}
  &|x-x_n|^{2(1+\al)}\la_nu_n^{2(p_n-1)}e^{u_n^{p_n}}\\
  &=\mu_n^{2+2\al}|y|^{2+2\al}\la_n\sbr{\ga_n+\f{2\eta_n(y)}{p_n\ga_n^{p_n-1}}}^{2(p_n-1)}\exp\sbr{\sbr{\ga_n+\f{2\eta_n(y)}{p_n\ga_n^{p_n-1}}}^{p_n}}\\
  &=\f{8(1+\al)^2}{p_n^2h_0(x_n)f(x_n)}|y|^{2+2\al}\sbr{1+\f{2\eta_n(y)}{p_n\ga_n^{p_n}}}^{2(p_n-1)}\exp\sbr{\ga_n^{p_n}\sbr{\sbr{1+\f{2\eta_n(y)}{p_n\ga_n^{p_n}}}^{p_n}-1}}\\
  &=O\sbr{|y|^{2+2\al}e^{2\eta(y)}}
\end{aligned}\ee
uniformly for $x=x_n+\mu_ny\in B_{6R_n\mu_n}(x_n)$ (i.e., $y\in B_{6R_n}(0)$) for large $n$. 
Here, the first equality uses the definition \eqref{3-13} of $\eta_n$; the second one uses the definition \eqref{mu-2-2} of $\mu_n$; the last one uses first $p_n\in[1,2]$, $\al\in(-1,0)$ and the positivity of $h_0$ and $f$, and then the fact that $\eta_n=\eta+o(1)=o(\ga_n^{p_n})$ in $B_{6R_n}(0)$ by the choice of $R_n$. 

Noting that
	$$\eta(x)=-2(1+\al)\ln|x|+O(1),\quad\text{for }\;|x|\geq 2,$$
we get from \eqref{520-1} that
	$$|x-x_n|^{2(1+\al)}\la_nu_n^{2(p_n-1)}e^{u_n^{p_n}}=O(1),$$
uniformly for $x\in B_{6R_n\mu_n}(x_n)$ for large $n$. 
Then, for any $c_0>0$, thanks to $|x|=(1+o(1))|x-x_n|$ for $x\in B_{6R_n\mu_n}(x_n)\setminus B_{c_0\mu_n}(x_n)$ by \eqref{514-2}, and by using \eqref{h0-1}, we get that
	$$|x-x_n|^2\la_nhu_n^{2(p_n-1)}e^{u_n^{p_n}}=|x-x_n|^2|x|^{2\al}\la_nh_0u_n^{2(p_n-1)}e^{u_n^{p_n}}=O(1),$$
uniformly for $x\in B_{6R_n\mu_n}(x_n)\setminus B_{c_0\mu_n}(x_n)$ for large $n$.
It follows from \eqref{3-12} that
	$$d_{g_0}(\cdot,\p_n)^2\la_nhu_n^{2(p_n-1)}e^{u_n^{p_n}}=O(1),$$
uniformly in $B_{3R_n\mu_n}(\p_n)\setminus B_{2c_0\mu_n}(\p_n)$ for large $n$.
Choosing $\p_{n,1}=\p_n$, $\mu_{n,1}=\mu_n$ and $d_{n,1}=R_n\mu_{n,1}/2$, we get the first bubble in Case 1. 
Moreover, since $d_{g_0}(\p_{n,1},q_0)=(1+o(1))|x_n|=o(\mu_{n,1})$ by \eqref{514-2}, we obtain
  $$d_{g_0}(\cdot,\{\p_{n,1},q_0\})^2\la_nhu_n^{2(p_n-1)}e^{u_n^{p_n}}=O(1),$$
uniformly in $B_{6d_{n,1}}(\p_{n,1})\setminus B_{\mu_{n,1}}(\p_{n,1})$ and also uniformly in $B_{6d_{n,1}}(\p_{n,1})\setminus B_{\mu_{n,1}}(q_0)$ for large $n$.

\vskip0.1in
\noindent{\bf Case 2:} $\p^*\neq q_0$, or, $\p^*=q_0$ and $|x_n|^{2+2\al}\la_n\ga_n^{2(p_n-1)}e^{\ga_n^{p_n}}\to+\iy$ as $n\to+\iy$.

In this case, we choose
	\be\lab{mu-2-3}\mu_n^{-2}:=\f{1}{8}\la_np_n^2h(x_n)f(x_n)\ga_n^{2(p_n-1)}e^{\ga_n^{p_n}}.\ee
Again by \eqref{514-1}, we get that $\mu_n=o(1)$.
To proceed, without loss of generality, we assume that $\p^*=q_0$ and $|x_n|^{2+2\al}\la_n\ga_n^{2(p_n-1)}e^{\ga_n^{p_n}}\to+\iy$ as $n\to+\iy$, and the case $\p^*\neq q_0$ is simpler and can be treated similarly. Then by \eqref{h0-1} and \eqref{mu-2-3}, we get that 
\be\lab{514-3} \mu_n=o(|x_n|). \ee 
Recalling $\eta_n$ given by \eqref{3-13}, we get from \eqref{equ-4-1}, \eqref{h0-1} and \eqref{mu-2-3} that
\be\lab{3-14-1}\begin{aligned}
	&-\Delta\eta_n\\
	&=4\abs{\f{x_n}{|x_n|}+\f{\mu_n}{|x_n|}\cdot}^{2\al}\f{h_0f}{h_0(x_n)f(x_n)}\f{u_n^{p_n-1}}{\ga_n^{p_n-1}}e^{u_n^{p_n}-\ga_n^{p_n}}(x_n+\mu_n\cdot)\\ 
		&\quad -\f{p_n}{2}\mu_n^2\ga_n^{p_n-1}hu_n(x_n+\mu_n\cdot)\\
	&=4\abs{\f{x_n}{|x_n|}+\f{\mu_n}{|x_n|}\cdot}^{2\al}\f{h_0f(x_n+\mu_n\cdot)}{h_0(x_n)f(x_n)}\sbr{1+\f{2\eta_n}{p_n\ga_n^{p_n}}}^{p_n-1} e^{\ga_n^{p_n}[(1+\f{2\eta_n}{p_n\ga_n^{p_n}})^{p_n}-1]}\\ 
		&\quad -\f{p_n}{2}\mu_n^2\ga_n^{p_n-1}hu_n(x_n+\mu_n\cdot),\quad\text{in}\;\Omega_n.
\end{aligned}\ee
By the argument between \eqref{h0-1} and \eqref{3-11}, we get that $\displaystyle\lim_{n\to+\iy}\ga_n^{-1}u_n(x_n+\mu_n\cdot)=1$ in $C_\loc^1(\R^2)$ (here $C_\loc^1(\R^2)$ instead of $C_\loc^0(\R^2)$ because $|\f{x_n}{|x_n|}+\f{\mu_n}{|x_n|}x|^{2\al}\to 1$ locally)
and that $\mu_n\ga_n^s=o(1)$ for any $s\geq1$.
Then by the argument below \eqref{3-14}, we get from \eqref{514-3}-\eqref{3-14-1} that
	$$\lim_{n\to+\iy}\eta_n=\eta \quad\text{in}~C_\loc^1(\R^2),$$
and $\eta$ satisfies 
	$$-\Delta\eta=4e^{2\eta}~~\text{in}~\R^2,\quad \eta(0)=0=\max_{\R^2}\eta, \quad\int_{\R^2}e^{2\eta}\rd x<+\iy.$$
The classification result of Chen-Li \cite{clas-2} gives that
\be \eta(x)=-\ln\sbr{1+|x|^2},\quad \int_{\R^2}e^{2\eta}\rd x=\pi. \ee
By using \eqref{514-3}, we may take a sequence $R_n\to+\iy$ such that $R_n\mu_n=o(|x_n|)$, $\ln R_n=o(\ga_n^{p_n})$ and 
\be\lab{514-4} \nm{\eta_n-\eta}_{L^\iy(B_{6R_n}(0))}=o(1).\ee
Thanks to $\la_n\mu_n^2h(x_n)\ga_n^{2(p_n-1)}e^{\ga_n^{p_n}}=O(1)$ by \eqref{mu-2-3}, and by using the argument in \eqref{520-1} with our choice of $R_n$, we get that
	$$|x-x_n|^2\la_nh(x_n)u_n^{2(p_n-1)}e^{u_n^{p_n}}=O(1),$$
uniformly for $x\in B_{6R_n\mu_n}(x_n)$ for large $n$.
Then, noting that $|x|=(1+o(1))|x_n|$ for $x\in B_{6R_n\mu_n}(x_n)$ by $R_n\mu_n=o(|x_n|)$, and by using \eqref{h0-1}, we get that
	$$|x-x_n|^2\la_nhu_n^{2(p_n-1)}e^{u_n^{p_n}}=O(1),$$
uniformly for $x\in B_{6R_n\mu_n}(x_n)$ for large $n$.
It follows from \eqref{3-12} that
	$$d_{g_0}(\cdot,\p_n)^2\la_nhu_n^{2(p_n-1)}e^{u_n^{p_n}}=O(1),$$
uniformly in $B_{3R_n\mu_n}(\p_n)$ for large $n$.
Choosing $\p_{n,1}=\p_n$, $\mu_{n,1}=\mu_n$ and $d_{n,1}=R_n\mu_{n,1}/2$, we get the first bubble in Case 2 and finish the proof of Step 1. Moreover, it is easy to check that
  $$d_{g_0}(\cdot,\{\p_{n,1},q_0\})^2\la_nhu_n^{2(p_n-1)}e^{u_n^{p_n}}=O(1),$$
uniformly in $B_{6d_{n,1}}(\p_{n,1})$ for large $n$, thanks to \be\label{325-40-0} d_{n,1}=R_n\mu_{n,1}/2=o(|x_n|)=o(d_{g_0}(\p_{n,1},q_0)).\ee

\vskip0.1in
\noindent{\it Step 2: Construction of the second bubble $\p_{n,2}$ and $d_{n,2}$.}

If 
	$$d_{g_0}(\cdot,\{\p_{n,1},q_0\})^2\la_nhu_n^{2(p_n-1)}e^{u_n^{p_n}}=O(1),$$
uniformly in $\Sigma\setminus B_{6d_{n,1}}(\p_{n,1})$, then we are done with $M=1$. Otherwise, we have that
	$$\lim_{n\to+\iy}\max_{\Sigma\setminus B_{d_{n,1}}(\p_{n,1}) } d_{g_0}(\cdot,\{\p_{n,1},q_0\})^2\la_nhu_n^{2(p_n-1)}e^{u_n^{p_n}}=+\iy. $$
Let $\bar \p_n\in\Sigma\setminus B_{d_{n,1}}(\p_{n,1})$ be the maximal point such that 
\be\lab{325-1} \lim_{n\to+\iy}d_{g_0}(\bar \p_n,\{\p_{n,1},q_0\})^2\la_nh(\bar \p_n)u_n(\bar \p_n)^{2(p_n-1)}e^{u_n(\bar \p_n)^{p_n}}=+\iy. \ee
By \eqref{325-1} and $\al\in(-1,0)$, we get that $\bar \p_n\neq q_0$ for large $n$. Then, up to a subsequence, we consider two different cases
	$$d_{g_0}(\bar \p_n,\p_{n,1})=O\sbr{d_{g_0}(\bar \p_n,q_0)},\quad\text{or}\quad d_{g_0}(\bar \p_n,q_0)=o\sbr{d_{g_0}(\bar \p_n,\p_{n,1})}.$$

\vskip0.1in
\noindent{\bf Case I:} Firstly, we assume 
\be\lab{3-28-0} d_{g_0}(\bar \p_n,\p_{n,1})=O\sbr{d_{g_0}(\bar \p_n,q_0)}. \ee

We get from \eqref{325-1} that
	$$\lim_{n\to+\iy}d_{g_0}(\bar \p_n,\p_{n,1})^2\la_nh(\bar \p_n)u_n(\bar \p_n)^{2(p_n-1)}e^{u_n(\bar \p_n)^{p_n}}=+\iy.$$
By \eqref{3-28-0} we can take $t_n=\f{1}{L}d_{g_0}(\bar \p_n,\p_{n,1})$ for some large $L\ge2$  such that $q_0\not\in B_{2t_n}(\bar \p_n)$. Define
\be\lab{604-1} \Phi_n:=\sbr{t_n-d_{g_0}(\cdot,\bar \p_n)}^2\la_nhu_n^{2(p_n-1)}e^{u_n^{p_n}}, \quad\text{in}~B_{t_n}(\bar \p_n). \ee
Since $d_{g_0}(\bar \p_n,\p_{n,1})\ge 6d_{n,1}$, we have 
\be\lab{3-28} B_{t_n}(\bar \p_n)\cap B_{3d_{n,1}}(\p_{n,1})=\emptyset, \ee
for large $n$. 
Thanks to $\Phi_n=0$ on $\pa B_{t_n}(\bar \p_n)$, and noting that $h$ has no singularity in $B_{t_n}(\bar \p_n)$ by $q_0\not\in B_{2t_n}(\bar \p_n)$, we may take a point sequence $\ti \p_n\in B_{t_n}(\bar \p_n)$ such that $\displaystyle\Phi_n(\ti \p_n)=\max_{B_{t_n}(\bar \p_n)}\Phi_n$, then 
	$$\Phi_n(\ti \p_n)\ge\Phi_n(\bar \p_n)=\f{1}{L^2}d_{g_0}(\bar \p_n,\p_{n,1})^2\la_nh(\bar \p_n)u_n(\bar \p_n)^{2(p_n-1)}e^{u_n(\bar \p_n)^{p_n}}\to+\iy,$$
as $n\to+\iy$. 
Denote
\be\lab{3-50} s_n:=\f{1}{2}\sbr{t_n-d_{g_0}(\ti \p_n,\bar \p_n)}, \ee
then we get
\be\lab{3-20} \Phi_n(\ti \p_n)=4s_n^2\la_nh(\ti \p_n)u_n(\ti \p_n)^{2(p_n-1)}e^{u_n(\ti \p_n)^{p_n}}\to+\iy, \ee
as $n\to+\iy$. 
By $q_0\not\in B_{2t_n}(\bar \p_n)$ and $\ti \p_n\in B_{t_n}(\bar \p_n)$, we obtain 
\be\lab{325-3} d_{g_0}(\ti \p_n,q_0)\ge t_n\ge2s_n,\ee 
and that
\be\lab{325-2} d_{g_0}(\ti \p_n,q_0)^{2+2\al}\la_nh_0(\ti \p_n)u_n(\ti \p_n)^{2(p_n-1)}e^{u_n(\ti \p_n)^{p_n}}\to+\iy, \ee
as $n\to+\iy$, by using \eqref{3-20} and \eqref{h0-1}. This implies $\displaystyle\lim_{n\to+\iy}u_n(\ti \p_n)=+\iy$. Since \eqref{3-50} implies $B_{s_n}(\ti \p_n)\subset B_{t_n}(\bar \p_n)$, we get that for any $\p\in B_{s_n}(\ti \p_n)$,
$$\begin{aligned}
	\Phi_n(\p)&=\sbr{t_n-d_{g_0}(\p,\bar p_n)}^2\la_nh(\p)u_n(\p)^{2(p_n-1)}e^{u_n(\p)^{p_n}}\\
	&\le\Phi_n(\ti \p_n)=4s_n^2\la_nh(\ti \p_n)u_n(\ti \p_n)^{2(p_n-1)}e^{u_n(\ti \p_n)^{p_n}}.
\end{aligned}$$
Noting from \eqref{3-50} that $t_n-d_{g_0}(\p,\bar \p_n)\ge s_n $ for $\p\in B_{s_n}(\ti \p_n)$, we get that
\be\lab{3-22} h(\p)u_n(\p)^{2(p_n-1)}e^{u_n(\p)^{p_n}}\le 4h(\ti \p_n)u_n(\ti \p_n)^{2(p_n-1)}e^{u_n(\ti \p_n)^{p_n}}, \ee
uniformly for $\p\in B_{s_n}(\ti \p_n)$ for large $n$.

Now, suppose $\displaystyle\lim_{n\to+\iy}\ti \p_n=\ti \p^*$, and take the isothermal coordinates $(B_r(\ti \p^*),\wt\Psi,\wt\Omega)$ around $\ti \p^*$ such that $\wt\Psi(\ti \p^*)=0$. Set $\ti x_n=\wt\Psi(\ti \p_n)\to 0$, $\ti\ga_n=u_n(\ti \p_n)\to+\infty$ and 
\be\lab{3-23}  \ti\eta_n:=\f{p_n}{2}\ti\ga_n^{p_n-1}(u_n(\ti x_n+\ti\mu_n\cdot)-\ti\ga_n)\quad\text{in}~\wt\Omega_n:=\f{\wt\Omega-\ti x_n}{\ti\mu_n}, \ee
where $\ti\mu_n$ is given by
\be\lab{mu-3-4}  \ti\mu_n^{-2}:=\f{1}{8}\la_np_n^2h(\ti x_n)f(\ti x_n)\ti\ga_n^{2(p_n-1)}e^{\ti\ga_n^{p_n}}. \ee
To proceed, we may assume without loss of generality that $\ti \p^*=q_0$, and the case $\ti \p^*\neq q_0$ may be treated similarly. Recalling the local coordinate around $q_0$ given in Step 1,  we have $(B_r(\ti \p^*),\wt\Psi,\wt\Omega)=(B_r(q_0),\Psi,\Omega)$. Then by \eqref{3-20}, \eqref{325-3} and \eqref{mu-3-4}, we get that
\be\lab{3-31} 2s_n\leq d_{g_0}(\ti \p_n,q_0)= |\ti x_n|(1+o(1)),\quad\text{and}\quad \ti\mu_n=o(s_n). \ee
Since $|\ti x_n|=o(1)$, we have $s_n=o(1)$, $\ti\mu_n=o(1)$, $\ti\mu_n=o(|\ti x_n|)$ and $\displaystyle\lim_{n\to+\iy}\wt\Omega_n=\mathbb R^2$.
Since $\Psi^{-1}\sbr{B_{s_n/2}(\ti x_n)}\subset B_{s_n}(\ti \p_n)$ for large $n$, and the first assertion in \eqref{3-31} implies $|x|\le3|\ti x_n|$ for $x\in B_{s_n/2}(\ti x_n)$, it follows from \eqref{3-22} and \eqref{h0-1} that
\be\lab{3-24} u_n^{2(p_n-1)}e^{u_n^{p_n}}\le 4\f{|\ti x_n|^{2\al}h_0(\ti x_n)}{|x|^{2\al}h_0}\ti\ga_n^{2(p_n-1)}e^{\ti\ga_n^{p_n}}=O\sbr{\ti\ga_n^{2(p_n-1)}e^{\ti\ga_n^{p_n}}}, \ee
uniformly for $x\in B_{s_n/2}(\ti x_n)$ for large $n$. From here and the fact that $t^{2(p-1)}e^{t^p}$ is increasing as a function of $t>0$ for any $p\in [1,2]$, we also have
\be\lab{3-25} u_n=O(\ti\ga_n),\quad u_n^{p_n-1}e^{u_n^{p_n}}=O\sbr{\ti\ga_n^{p_n-1}e^{\ti\ga_n^{p_n}} } \ee
uniformly for $x\in B_{s_n/2}(\ti x_n)$ for large $n$. 
Recalling $\ti\eta_n$ given by \eqref{3-23},  we get from \eqref{equ-4-1}, \eqref{h0-1} and \eqref{mu-3-4}  that 
\be\lab{3-14-2}\begin{aligned}
	&-\Delta\ti\eta_n\\
	&=4\abs{\f{\ti x_n}{|\ti x_n|}+\f{\ti \mu_n}{|\ti x_n|}\cdot}^{2\al}\f{h_0f}{h_0(\ti x_n)f(\ti x_n)}\f{u_n^{p_n-1}}{\ti \ga_n^{p_n-1}}e^{u_n^{p_n}-\ti \ga_n^{p_n}}(\ti x_n+\ti \mu_n\cdot)\\ 
		&\quad -\f{p_n}{2}\ti \mu_n^2\ti \ga_n^{p_n-1}hu_n(\ti x_n+\ti \mu_n\cdot)\\
	&=4\abs{\f{\ti x_n}{|\ti x_n|}+\f{\ti \mu_n}{|\ti x_n|}\cdot}^{2\al}\f{h_0f(\ti x_n+\ti \mu_n\cdot)}{h_0(\ti x_n)f(\ti x_n)}\sbr{1+\f{2\ti \eta_n}{p_n\ti \ga_n^{p_n}}}^{p_n-1} e^{\ti\ga_n^{p_n}[(1+\f{2\ti \eta_n}{p_n\ti \ga_n^{p_n}})^{p_n}-1]} \\ 
		&\quad -\f{p_n}{2}\ti \mu_n^2\ti \ga_n^{p_n-1}hu_n(\ti x_n+\ti \mu_n\cdot),\quad\text{in }\;\wt\Omega_n.
\end{aligned}\ee
Remark that comparing to Step 1, the different thing is that $\ti\gamma_n$ is not necessarily the maximum value of $u_n$ in $\widetilde\Omega$, i.e., we do not have $u_n\leq \ti\gamma_n$ in $\widetilde\Omega$. However,
thanks to \eqref{3-24}-\eqref{3-25} and $\frac{s_n}{\ti\mu_n}\to+\infty$ (i.e., $\frac{B_{s_n/2}(\ti x_n)-\ti x_n}{\ti\mu_n}\to\R^2$), we can still get by the argument between \eqref{h0-1} and \eqref{3-11} that $\displaystyle\lim_{n\to+\iy}\ti \ga_n^{-1}u_n(\ti x_n+\ti \mu_n\cdot)=1$ in $C_\loc^1(\R^2)$, and that $\ti\mu_n\ti\ga_n^s=o(1)$ for any $s\geq 1$. Consequently, by the argument below \eqref{3-14}, we eventually get from \eqref{3-14-2} and $\ti\mu_n=o(|\ti x_n|)$ that
	$$\lim_{n\to+\iy}\ti\eta_n=\ti\eta \quad\text{in}~C_\loc^1(\R^2),$$
and $\ti\eta$ satisfies 
	$$-\Delta\ti\eta=4e^{2\ti\eta}~~\text{in}~\R^2,\quad \ti\eta(0)=0, \quad\int_{\R^2}e^{2\ti\eta}\rd x<\iy.$$
Again the different thing is that $\ti\eta(0)=0$ is not necessarily the maximum value of $\ti\eta$. Then the result of Chen-Li \cite{clas-2} gives that
\be\lab{3-27} \ti\eta(x)=\ln\xi-\ln\sbr{1+\xi^2|x-b_0|^2},\ee
for some $\xi>0$ and $b_0\in\R^2$ with $\xi=1+\xi^2|b_0|^2$, and $\int_{\R^2}e^{2\ti\eta}\rd x=\pi$. 

Now we want to replace $\ti x_n$ with another sequence $x_{n,2}$ such that $b_0$ in \eqref{3-27} can be replaced by $0$. 
Since $\ti\mu_n=o(s_n)$, we may take a sequence $\ti R_n\to+\iy$ such that $\ti R_n\ti\mu_n=o(s_n)$ and 
\be\lab{3-26} \nm{\ti\eta_n-\ti\eta}_{L^\iy(B_{8\ti R_n}(0))}=o(1). \ee
Let $y_n\in B_{8\ti R_n(0)}$ such that
\be\lab{def-yn}\ti\eta_n(y_n)=\max_{B_{8\ti R_n}(0)}\ti\eta_n.\ee
Let $x_{n,2}:=\ti x_n+\ti\mu_ny_n$ and
	 \be\label{fc-3}\mu_{n,2}^{-2}:=\f{1}{8}\la_np_n^2h(x_{n,2})f(x_{n,2})u_n(x_{n,2})^{2(p_n-1)}e^{u_n(x_{n,2})^{p_n}}.\ee
By \eqref{3-27}, \eqref{3-26} and \eqref{3-23}, we get that $\displaystyle\lim_{n\to+\iy}y_n=b_0$ and 
\be\lab{3-29} u_n(x_{n,2})=\ti\ga_n+\f{2\ti\eta_n(y_n)}{p_n\ti\ga_n^{p_n-1}}=\ti\ga_n+\f{2(\ln\xi+o(1))}{p_n\ti\ga_n^{p_n-1}}\to+\infty. \ee
It follows from \eqref{3-31} that
\be\lab{3-34} |x_{n,2}-\ti x_n|=O(\ti\mu_n)=o(s_n)=o(|\ti x_n|),\ee
namely$|x_{n,2}|=(1+o(1))|\ti x_n|$ and hence \eqref{3-29} implies
\be\lab{3-30}  \sbr{\f{\mu_{n,2}}{\ti\mu_n}}^2=\f{|\ti x_n|^{2\al}h_0(\ti x_n)f(\ti x_n)\ti\ga_n^{2(p_n-1)}e^{\ti\ga_n^{p_n}}}{|x_{n,2}|^{2\al}h_0(x_{n,2})f(x_{n,2})u_n(x_{n,2})^{2(p_n-1)}e^{u_n(x_{n,2})^{p_n}}}=\xi^{-2}+o(1).\ee
Then by \eqref{3-23}, \eqref{3-27}, \eqref{3-26}, \eqref{3-29} and \eqref{3-30}, it is easy to check that
\be\lab{3-32}\begin{aligned}
	\eta_{n,2}(x):&=\f{p_n}{2}u_n(x_{n,2})^{p_n-1}(u_n(x_{n,2}+\mu_{n,2}x)-u_n(x_{n,2}))\\
		&=\f{u_n(x_{n,2})^{p_n-1}}{\ti\ga_n^{p_n-1}}\ti\eta_n\sbr{y_n+\f{\mu_{n,2}}{\ti\mu_n}x}+\f{p_n}{2}u_n(x_{n,2})^{p_n-1}(\ti\ga_n-u_n(x_{n,2}))\\
&\to \ti\eta\left(b_0+\frac{x}{\xi}\right)-\ln\xi=-\ln\sbr{1+|x|^2}=\eta(x) \quad\text{in}~C_\loc^1(\R^2),
\end{aligned}\ee
as $n\to+\iy$. 
We now take another sequence $R_n\to+\iy$ such that $R_n\le \xi\ti R_n$, $R_n\mu_{n,2}=o(s_n)$, $\ln R_n=o(\ti\ga_n^{p_n})$ and 
\be\lab{3-33} \nm{\eta_{n,2}-\eta}_{L^\iy(B_{6R_n}(0))}=o(1). \ee
and choose $\p_{n,2}=\Psi^{-1}(x_{n,2})$, and $d_{n,2}=R_n\mu_{n,2}/2$.

Now we check that $\{\p_{n,1},\p_{n,2}\}$ satisfies all the conditions of Proposition \ref{bubble-1}. 
Obviously, $d_{n,2}=o(s_n)$ and it follows from \eqref{3-34} that $\lim_{n\to+\iy}\p_{n,2}=q_0$ and $\p_{n,2}\in B_{s_n}(\ti \p_n)\subset B_{t_n}(\bar \p_n)$. Then by \eqref{3-28} and $t_n=\frac{1}{L}d_{g_0}(\p_{n,1},\bar\p_n)\geq 2s_n$, we get that
	$$d_{g_0}(\p_{n,1},\p_{n,2})\ge 2d_{n,1},$$
and
$$d_{g_0}(\p_{n,1},\p_{n,2})\geq d_{g_0}(\p_{n,1},\bar\p_n)-d_{g_0}(\bar\p_n,\p_{n,2})\geq (L-1)t_n\geq 2s_n\geq 2d_{n,2},$$
for large $n$. 
By the definition \eqref{def-yn} of $y_n$ and \eqref{3-23}, we get  that
\be\lab{3-35} u_n(x_{n,2})=\max_{B_{8\ti R_n\ti\mu_{n}}(\ti x_n)}u_n=\max_{B_{7R_n\mu_{n,2}}(\ti x_n)}u_n=\max_{B_{6R_n\mu_{n,2}}(x_{n,2})}u_n\to+\iy.\ee
Since, like \eqref{3-12}, we have
\be\label{3-12-2} B_{6d_{n,2}}(\p_{n,2}) =B_{3R_n\mu_{n,2}}(\p_{n,2}) \subset\Psi^{-1}\sbr{B_{6R_n\mu_{n,2}}(x_{n,2})} \subset B_{12R_n\mu_{n,2}}(\p_{n,2}), \ee
so 
	$$u_n(\p_{n,2})=\max_{B_{6d_{n,2}}(\p_{n,2})}u_n\to+\iy.$$
That is, conclusions (1) and (2) of Proposition \ref{bubble-1} hold for $\{\p_{n,1},\p_{n,2}\}$. Obviously, conclusion (4) holds and \eqref{3-32} gives  conclusion (3). Finally, by the argument in \eqref{520-1} with our choice of $R_n$,  and using $\la_n\mu_{n,2}^2h(x_{n,2})u_n(x_{n,2})^{2(p_n-1)}e^{u_n(x_{n,2})^{p_n}}=O(1)$ by \eqref{fc-3}, we can get as before that
	$$|x-x_{n,2}|^2\la_nh(x_{n,2})u_n^{2(p_n-1)}e^{u_n^{p_n}}=O(1),$$
uniformly for $x\in B_{6R_n\mu_{n,2}}(x_{n,2})$ for large $n$.
Using $|x|=(1+o(1))|x_{n,2}|$ for $x\in B_{6R_n\mu_{n,2}}(x_{n,2})$ by $R_n\mu_{n,2}=o(s_n)=o(|\ti x_n|)=o(|x_{n,2}|)$, we get that 
	$$|x-x_{n,2}|^2\la_nhu_n^{2(p_n-1)}e^{u_n(x)^{p_n}}=O(1),$$
uniformly for $x\in B_{6R_n\mu_{n,2}}(x_{n,2})$ for large $n$.
It follows from \eqref{3-12-2} that
	$$d_{g_0}(\cdot,\p_{n,2})^2\la_nhu_n^{2(p_n-1)}e^{u_n^{p_n}}=O(1),$$
uniformly in $B_{6d_{n,2}}(\p_{n,2})$ for large $n$.
As a result, we get that
	$$d_{g_0}(\cdot,\{\p_{n,1},\p_{n,2},q_0\})^2\la_nhu_n^{2(p_n-1)}e^{u_n^{p_n}}=O(1),$$
uniformly in $B_{6d_{n,2}}(\p_{n,2})$ for large $n$. This finishes the proof for Case I.

\vskip0.1in
\noindent{\bf Case II:} Secondly, we assume that
\be\lab{325-4-0} d_{g_0}(\bar \p_n,q_0)=o\sbr{d_{g_0}(\bar \p_n,\p_{n,1})},\ee
and so
$$d_{g_0}(\bar \p_n,\p_{n,1})=(1+o(1))d_{g_0}(q_0,\p_{n,1}).$$ 

This means that there must be Case 2 of Step 1 to hold, i.e., $(p_{n,1},d_{n,1})$ must be a non-singular bubble. Indeed, if Case 1 of Step 1 holds, then we get from \eqref{514-2}, $d_{n,1}=R_n\mu_{n,1}/2$ and $\bar \p_n\in\Sigma\setminus B_{d_{n,1}}(\p_{n,1})$ that 
$$d_{g_0}(\p_{n,1},q_0)=|x_{n,1}|(1+o(1))=o(\mu_{n,1})=o(d_{n,1})=o(d_{g_0}(\bar \p_n,\p_{n,1})),$$ and so
$$d_{g_0}(\bar \p_n,q_0)=(1+o(1)) d_{g_0}(\bar \p_n,\p_{n,1}),$$ which is a contradiction with our assumption \eqref{325-4-0}. Then by \eqref{325-40-0} in Case 2 of Step 1, 
\be\lab{325-4} d_{n,1}=o(d_{g_0}(\p_{n,1},q_0))=o(d_{g_0}(\bar \p_n,\p_{n,1})). \ee
In the following arguments, for convenience we use the same notations $t_n,\Phi_n,\ti \p_n, s_n$, etc, as Case I of Step 2, but the precise definitions of them may be different from Case 1 of Step 2. 

Observe from \eqref{325-1} and \eqref{325-4-0} that $d_{g_0}(\bar \p_n,\p_{n,1})\ge 6d_{n,1}$ and 
	$$\lim_{n\to+\iy}d_{g_0}(\bar \p_n,q_0)^2\la_nh(\bar \p_n)u_n(\bar \p_n)^{2(p_n-1)}e^{u_n(\bar \p_n)^{p_n}}=+\iy.$$
Denote $t_n:=2d_{g_0}(\bar \p_n,q_0)$, and define
\be\lab{604-2} \Phi_n:=\sbr{t_n-d_{g_0}(\cdot,\bar \p_n)}^{2+2\al}\la_nu_n^{2(p_n-1)}e^{u_n^{p_n}}, \quad\text{in}~B_{t_n}(\bar \p_n). \ee
Remark that the definitions of $t_n$ and $\Phi_n$ are different from those in \eqref{604-1}.
By \eqref{325-4} and $t_n=o(d_{g_0}(\bar \p_n,\p_{n,1}))$ by \eqref{325-4-0}, we get that 
\be\lab{325-5} B_{t_n}(\bar \p_n)\cap B_{3d_{n,1}}(\p_{n,1})=\emptyset. \ee
Thanks to $\Phi_n=0$ on $\partial B_{t_n}(\bar \p_n)$, we may take $\ti \p_n\in B_{t_n}(\bar \p_n)$ such that $\displaystyle\Phi_n(\ti \p_n)=\max_{B_{t_n}(\bar \p_n)}\Phi_n$, and then, by \eqref{h0-1} and \eqref{604-2}, we get that
	$$\Phi_n(\ti \p_n)\ge\Phi_n(\bar \p_n)\ge Cd_{g_0}(\bar \p_n,q_0)^2\la_nh(\bar \p_n)u_n(\bar \p_n)^{2(p_n-1)}e^{u_n(\bar \p_n)^{p_n}}\to+\iy,$$
as $n\to+\iy$. 
Denote
\be\lab{325-9} s_n:=\f{1}{2}\sbr{t_n-d_{g_0}(\ti \p_n,\bar \p_n)}, \ee
then we get 
\be\lab{325-6} \Phi_n(\ti \p_n)=(2s_n)^{2+2\al}\la_nu_n(\ti \p_n)^{2(p_n-1)}e^{u_n(\ti \p_n)^{p_n}}\to+\iy. \ee
This implies $\displaystyle\lim_{n\to+\iy}u_n(\ti \p_n)=+\iy$. Since $B_{s_n}(\ti \p_n)\subset B_{t_n}(\bar \p_n)$, we get  that  for any $\p\in B_{s_n}(\ti \p_n)$,
\be\lab{325-7-0}\begin{aligned}
	\Phi_n(\p)&=\sbr{t_n-d_{g_0}(\p,\bar \p_n)}^{2+2\al}\la_nu_n(\p)^{2(p_n-1)}e^{u_n(\p)^{p_n}}\\
	&\le\Phi_n(\ti \p_n)=(2s_n)^{2+2\al}\la_nu_n(\ti \p_n)^{2(p_n-1)}e^{u_n(\ti \p_n)^{p_n}}.
\end{aligned}\ee
Noting from \eqref{325-9} that $t_n-d_{g_0}(\p,\bar \p_n)\ge s_n$ for $\p\in B_{s_n}(\ti \p_n)$, we get that
\be\lab{325-7} u_n(\p)^{2(p_n-1)}e^{u_n(\p)^{p_n}}\le 2^{2+2\al}u_n(\ti \p_n)^{2(p_n-1)}e^{u_n(\ti \p_n)^{p_n}}, \ee
uniformly for $\p\in B_{s_n}(\ti \p_n)$ for large $n$.

Now, since  \eqref{325-4-0} implies
$$d_{g_0}(\ti\p_n, q_0)\leq t_n+d_{g_0}(\bar\p_n, q_0)=3d_{g_0}(\bar\p_n, q_0)\to 0,$$
i.e., $\lim_{n\to+\iy}\ti \p_n=q_0$, we take the isothermal coordinates $(B_r(q_0),\Psi,\Omega)$ around $q_0$ given in Step 1. Set $$\ti x_n=\Psi(\ti \p_n)\to 0,\quad \ti\ga_n=u_n(\ti \p_n)\to+\infty.$$ It follows from \eqref{325-7} and $B_{s_n/2}(\ti x_n)\subset\Psi (B_{s_n}(\ti\p_n))$ that
\be\lab{3-38} u_n^{2(p_n-1)}e^{u_n^{p_n}}=O\sbr{\ti\ga_n^{2(p_n-1)}e^{\ti\ga_n^{p_n}}}, \ee
and so
\be\lab{3-39} u_n =O(\ti\ga_n), \quad u_n^{p_n-1}e^{u_n^{p_n}}=O\sbr{\ti\ga_n^{p_n-1}e^{\ti\ga_n^{p_n}}},\ee
uniformly in $B_{s_n/2}(\ti x_n)$ for large $n$.  
Define 
\be\lab{325-8} \ti\eta_n:=\f{p_n}{2}\ti\ga_n^{p_n-1}(u_n(\ti x_n+\ti\mu_n\cdot)-\ti\ga_n)\quad\text{in}~\wt\Omega_n:=\f{\wt\Omega-\ti x_n}{\ti\mu_n}, \ee
where $\ti\mu_n$ is a positive number going to $0$ which will be chosen later.  
We have to consider the following three cases.
\vskip0.1in
\noindent{\bf Case II-1:} $|\ti x_n|^{2+2\al}\la_n\ti\ga_n^{2(p_n-1)}e^{\ti\ga_n^{p_n}}=O(1)$.

This case is similar to Case 1 of Step 1, but the different thing here is that $\ti \p_n$ is a local maximum point of $\Phi_n$ but not $u_n$.
Like Case 1 of Step 1,  we define
\be\lab{mu-3-5} \ti\mu_n^{-(2+2\al)}:=\f{1}{8(1+\al)^2}\la_np_n^2h_0(\ti x_n)f(\ti x_n)\ti\ga_n^{2(p_n-1)}e^{\ti\ga_n^{p_n}}, \ee
where $h_0$ is given by \eqref{h0-1}. Then $|\ti x_n|=O(\ti\mu_n)$ and hence we may assume 
\be\lab{3-38-0} \lim_{n\to+\iy}\f{\ti x_n}{\ti\mu_n}=\ti x_\iy. \ee
We get from \eqref{325-6} and \eqref{mu-3-5} that 
\be\lab{snmu-n} \sbr{\f{s_n}{\ti\mu_n}}^{2+2\al}=\f{p_n^2}{2^{5+2\al}(1+\al)^2}h_0(\ti x_n)f_n(\ti x_n)\Phi_n(\ti \p_n)\to+\iy,\ee
so $\ti\mu_n=o(s_n)=o(1)$.  
Recalling $\ti\eta_n$ defined by \eqref{325-8}, we get from \eqref{equ-4-1}, \eqref{h0-1} and \eqref{mu-3-5} that 
\be\lab{3-14-4}\begin{aligned}
	&-\Delta \ti \eta_n \\ 
	&=4(1+\al)^2\abs{\cdot+\f{\ti x_n}{\ti \mu_n}}^{2\al}\f{h_0f}{h_0(\ti x_n)f(\ti x_n)}\f{u_n^{p_n-1}}{\ti \ga_n^{p_n-1}}e^{u_n^{p_n}-\ti \ga_n^{p_n}}(\ti x_n+\ti \mu_n\cdot)\\ &\quad -\f{p_n}{2}\ti \mu_n^2\ti \ga_n^{p_n-1}hu_n(\ti x_n+\ti \mu_n\cdot)\\
	&=4(1+\al)^2\abs{\cdot+\f{\ti x_n}{\ti \mu_n}}^{2\al}\f{h_0f(\ti x_n+\ti \mu_n\cdot)}{h_0(\ti x_n)f(\ti x_n)}\sbr{1+\f{2\ti \eta_n}{p_n\ti \ga_n^{p_n}}}^{p_n-1}e^{\ti \ga_n^{p_n}[(1+\f{2\ti \eta_n}{p_n\ti \ga_n^{p_n}})^{p_n}-1]}\\
	&\quad -\f{p_n}{2}\ti \mu_n^2\ti \ga_n^{p_n-1}hu_n(\ti x_n+\ti \mu_n\cdot),\quad\text{in}\;\wt\Omega_n.
\end{aligned}\ee
By using \eqref{3-38}-\eqref{3-39} and $\frac{s_n}{\ti\mu_n}\to+\infty$ (i.e., $\frac{B_{s_n/2}(\ti x_n)-\ti x_n}{\ti\mu_n}\to\R^2$), we get first by the argument between \eqref{h0-1} and \eqref{3-11} that $\displaystyle\lim_{n\to+\iy}\ti\ga_n^{-1}u_n(\ti x_n+\ti \mu_n\cdot)=1$ in $C_\loc^0(\R^2)$, and that $\ti\mu_n\ti\ga_n^s=o(1)$ for any $s\geq1$, and then, by the argument below \eqref{3-14}, we eventually get from \eqref{3-14-4} and \eqref{3-38-0} that
	$$\lim_{n\to+\iy}\ti\eta_n=\ti\eta \quad\text{in}~C^{0,\delta}_{\loc}(\R^2)\cap C^1_\loc(\R^2\setminus\{-\ti x_\iy\}),$$
for any $\delta\in(0,\min\{1,2(1+\al)\})$, and $\ti\eta$ satisfies 
	$$-\Delta\ti\eta=4(1+\al)^2|x+\ti x_\iy|^{2\al}e^{2\ti\eta}~~\text{in}~\R^2,\quad \ti\eta(0)=0, \quad\int_{\R^2}|x+\ti x_\iy|^{2\al}e^{2\ti\eta}\rd x<\iy.$$
Again the result of  Prajapat-Tarantello \cite{clas-1} gives that
\be\lab{3-41} \ti\eta(x)=\ln\xi-\ln\sbr{1+\xi^2\abs{x+\ti x_\iy}^{2+2\al}},\ee
where $\xi>0$ satisfies $\xi=1+\xi^2\abs{\ti x_\iy}^{2+2\al}$, and that
	$$\int_{\R^2}|x+\ti x_\iy|^{2\al}e^{2\ti\eta(x)}\rd x=\f{\pi}{1+\al}.$$
Now similarly as Case I of Step 2, we want to replace $\ti x_n$ with another sequence $x_{n,2}$ such that $\ti x_\infty$ in \eqref{3-27} can be replaced by $0$. 
Using $\ti \mu_n=o(s_n)$ by \eqref{snmu-n}, we may take a sequence $\ti R_n\to+\iy$ such that $\ti R_n\ti\mu_n=o(s_n)$ and
\be\lab{3-42} \nm{\ti\eta_n-\ti\eta}_{L^\iy(B_{8\ti R_n}(0))}=o(1). \ee
Let $y_n\in B_{8\ti R_n(0)}$ be such that
	\be\lab{def-yn2}\ti\eta_n(y_n)=\max_{B_{8\ti R_n}(0)}\ti\eta_n.\ee
Let $x_{n,2}:=\ti x_n+\ti\mu_ny_n$ and 
	$$\mu_{n,2}^{-(2+2\al)}:=\f{1}{8(1+\al)^2}\la_np_n^2h_0(x_{n,2})f(x_{n,2})u_n(x_{n,2})^{2(p_n-1)}e^{u_n(x_{n,2})^{p_n}}.$$
By \eqref{3-41}, \eqref{3-42}, \eqref{def-yn2} and \eqref{325-8}, we get that $\displaystyle\lim_{n\to+\iy}y_n=-\ti x_\iy$ and 
\be\lab{3-43} u_n(x_{n,2})=\ti\ga_n+\f{2\ti\eta_n(y_n)}{p_n\ti\ga_n^{p_n-1}}=\ti\ga_n+\f{2(\ln\xi+o(1))}{p_n\ti\ga_n^{p_n-1}}\to+\infty. \ee
It follows that $\lim_{n\to+\infty}x_{n,2}=\lim_{n\to+\infty}\ti x_{n}=0$ and
\be\lab{3-44}  \sbr{\f{\mu_{n,2}}{\ti\mu_n}}^{2+2\al}=\f{h_0(\ti x_n)f(\ti x_n)\ti\ga_n^{2(p_n-1)}e^{\ti\ga_n^{p_n}}}{h_0(x_{n,2})f(x_{n,2})u_n(x_{n,2})^{2(p_n-1)}e^{u_n(x_{n,2})^{p_n}}}=\xi^{-2}+o(1),\ee
and hence
\be\lab{3-45} \f{x_{n,2}}{\mu_{n,2}}=\frac{\ti\mu_n}{\mu_{n,2}}\left(\frac{\ti x_n}{\ti \mu_n}+y_n\right)\to 0. \ee
By \eqref{325-8}, \eqref{3-41}, \eqref{3-42}, \eqref{3-43}, \eqref{3-44} and \eqref{3-45}, we get that
\be\lab{3-46}\begin{aligned}
	&\eta_{n,2}(x)\\
:=&\f{p_n}{2}u_n(x_{n,2})^{p_n-1}(u_n(x_{n,2}+\mu_{n,2}x)-u_n(x_{n,2}))\\
		=&\f{u_n(x_{n,2})^{p_n-1}}{\ti\ga_n^{p_n-1}}\ti\eta_n\sbr{y_n+\f{\mu_{n,2}}{\ti\mu_n}x}+\f{p_n}{2}u_n(x_{n,2})^{p_n-1}(\ti\ga_n-u_n(x_{n,2}))\\
		\to &\ti\eta\left(\frac{x}{\xi}-\ti x_\infty\right)-\ln\xi=-\ln\sbr{1+|x|^{2+2\al}}=\eta(x)\;\text{in}~C^{0,\delta}_{\loc}(\R^2)\cap C^1_\loc(\R^2\setminus\{0\}),
\end{aligned}\ee
for any $\delta\in(0,\min\{1,2(1+\al)\})$, as $n\to+\iy$. Now, we take another sequence $R_n\to+\iy$ such that $R_n\le \xi\ti R_n$, $R_n\mu_{n,2}=o(s_n)$, $\ln R_n=o(\ti\ga_n^{p_n})$ and 
\be\lab{3-47} \nm{\eta_{n,2}-\eta}_{L^\iy(B_{6R_n}(0))}=o(1), \ee
and choose $\p_{n,2}=\Psi^{-1}(x_{n,2})$, and $d_{n,2}=R_n\mu_{n,2}/2$.  
Then, by repeating the argument between \eqref{3-33} and \eqref{325-4-0} in Case I of Step 2, we get that $\{\p_{n,1},\p_{n,2}\}$ satisfies all the conditions of Proposition \eqref{bubble-1}. 
This finishes the proof of Case II-1.

\vskip0.1in
\noindent{\bf Case II-2:} $|\ti x_n|^{2+2\al}\la_n\ti\ga_n^{2(p_n-1)}e^{\ti\ga_n^{p_n}}\to+\iy$ as $n\to+\iy$ and $|\ti x_n|=O(s_n)$. 

Like Case I of Step 2, we choose
	$$\ti\mu_n^{-2}:=\f{1}{8}\la_np_n^2h(\ti x_n)f(\ti x_n)u_n(\ti x_n)^{2(p_n-1)}e^{u_n(\ti x_n)^{p_n}}.$$
Then	$\ti \mu_n=o(|\ti x_n|)$. From here and our assumption $|\ti x_n|=O(s_n)$, we obtain $\ti \mu_n=o(|s_n|)$.
Recall $\ti\eta_n$ given by \eqref{325-8}. Thanks to \eqref{3-38}-\eqref{3-39}, we can repeat the argument between \eqref{3-14-2} and \eqref{3-27}, and get that
	$$\lim_{n\to+\iy}\ti\eta_n=\ti\eta \quad\text{in}~C_\loc^1(\R^2),$$
with 
	$$\ti\eta(x)=\ln\xi-\ln\sbr{1+\xi^2|x-b_0|^2},$$
for some $\xi>0$ and $b_0\in\R^2$, and $\int_{\R^2}e^{2\ti\eta}\rd x=\pi$. Then, by the argument of Case I of Step 2 starting from \eqref{3-26}, we may take $x_{n,2}$, $\mu_{n,2}$, $\p_{n,2}$ and $d_{n,2}$ satisfying all conditions in Proposition \ref{bubble-1}. This finishes the proof of Case II-2.

\vskip0.1in
\noindent{\bf Case II-3:} $|\ti x_n|^{2+2\al}\la_n\ti\ga_n^{2(p_n-1)}e^{\ti\ga_n^{p_n}}\to+\iy$ as $n\to+\iy$ and $s_n=o(|\ti x_n|)$. 

In this case, we first get from $|\ti x_n|^{2+2\al}\la_n\ti\ga_n^{2(p_n-1)}e^{\ti\ga_n^{p_n}}\to+\iy$ and $d_{g_0}(\ti\p_n, q_0)=|\ti x_n|(1+o(1))$ that
	$$\lim_{n\to+\iy}d_{g_0}(\ti \p_n,q_0)^2\la_nh(\ti \p_n)u_n(\ti \p_n)^{2(p_n-1)}e^{u_n(\ti \p_n)^{p_n}}=+\iy.$$
Denote $\hat t_n=\f{1}{2}d_{g_0}(\ti \p_n,q_0)$ such that $B_{\hat t_n}(\ti \p_n)\cap B_{\hat t_n}(q_0)=\emptyset$. Now similarly as \eqref{604-1}, we define
	$$\hat\Phi_n:=\sbr{\hat t_n-d_{g_0}(\cdot,\ti \p_n)}^{2}\la_nhu_n^{2(p_n-1)}e^{u_n^{p_n}}, \quad\text{in}~B_{\hat t_n}(\ti \p_n).$$
By using $\ti \p_n, q_0\in B_{t_n}(\bar \p_n)$ and \eqref{325-4-0}-\eqref{325-4} which say that $\hat t_n\leq t_n=o(d_{g_0}(\bar\p_n, \p_{n,1}))$ and $d_{n,1}=o(d_{g_0}(\bar\p_n, \p_{n,1}))$, we get that $B_{\hat t_n}(\ti \p_n)\cap B_{3d_{n,1}}(\p_{n,1})=\emptyset$ for large $n$. 
Noting that $\hat\Phi_n=0$ on $\partial B_{\hat t_n}(\ti \p_n)$, we may take a point sequence $\hat \p_n\in B_{\hat t_n}(\ti \p_n)$ such that $\hat \Phi_n(\hat \p_n)=\max_{B_{\hat t_n}(\ti \p_n)}\hat \Phi_n$, then we get that
	$$\hat \Phi_n(\hat \p_n)\ge\hat \Phi_n(\ti \p_n)=\f{1}{4}d_{g_0}(\ti \p_n,q_0)^2\la_nh(\ti \p_n)u_n(\ti \p_n)^{2(p_n-1)}e^{u_n(\ti \p_n)^{p_n}}\to+\iy,$$
as $n\to+\iy$. 
Denote
	$$\hat s_n:=\f{1}{2}\sbr{\hat t_n-d_{g_0}(\hat \p_n,\ti \p_n)}, $$
then
	$$\hat \Phi_n(\hat \p_n)=4\hat s_n^2\la_nh(\hat \p_n)u_n(\hat \p_n)^{2(p_n-1)}e^{u_n(\hat \p_n)^{p_n}}\to+\iy.$$
Since $B_{\hat t_n}(\ti \p_n)\cap B_{\hat t_n}(q_0)=\emptyset$ and $\hat \p_n\in B_{\hat t_n}(\ti \p_n)$ implies $$d_{g_0}(\hat\p_n, q_0)\geq \hat t_n\geq2\hat s_n,$$ we also have
$$d_{g_0}(\hat\p_n, q_0)^2\la_nh(\hat \p_n)u_n(\hat \p_n)^{2(p_n-1)}e^{u_n(\hat \p_n)^{p_n}}\to+\iy$$
and so $u_n(\hat \p_n)\to+\infty$.
Furthermore, like the proof of \eqref{3-22}, by using $B_{\hat s_n}(\hat \p_n)\subset B_{\hat t_n}(\ti \p_n)$ and that $\hat t_n-d_{g_0}(\p,\ti \p_n)\ge \hat s_n$ for $\p\in B_{\hat s_n}(\hat \p_n)$, we get from the maximality of $\hat\p_n$ that
	$$h(\p)u_n(\p)^{2(p_n-1)}e^{u_n(\p)^{p_n}}\le 4h(\hat \p_n)u_n(\hat \p_n)^{2(p_n-1)}e^{u_n(\hat \p_n)^{p_n}},$$
uniformly for $\p\in B_{\hat s_n}(\hat \p_n)$ for large $n$. 
Now, since $\displaystyle\lim_{n\to+\iy}\hat \p_n=q_0$, we take the isothermal coordinates $(B_r(q_0),\Psi,\Omega)$ around $q_0$ given in Step 1. Set $\hat x_n:=\Psi(\hat \p_n)\to 0$, $\hat\ga_n:=u_n(\hat \p_n)\to+\infty$ and 
	$$\hat\eta_n:=\f{p_n}{2}\hat\ga_n^{p_n-1}(u_n(\hat x_n+\hat\mu_n\cdot)-\hat\ga_n)\quad\text{in}~\hat\Omega_n:=\f{\hat\Omega-\hat x_n}{\hat\mu_n},$$
where $\hat\mu_n$ is given by
	$$\hat\mu_n^{-2}:=\f{1}{8}\la_np_n^2h(\hat x_n)f(\hat x_n)u_n(\hat x_n)^{2(p_n-1)}e^{u_n(\hat x_n)^{p_n}}.$$
Clearly, we have
	$$2\hat s_n\le d_{g_0}(\hat\p_n, q_0)=|\hat x_n|(1+o(1)) \quad\text{and}\quad \hat\mu_n=o(\hat s_n)=o(|\hat x_n|).$$
Then by repeating the argument between \eqref{3-31} and \eqref{325-4-0}, we may take again $x_{n,2}$, $\mu_{n,2}$, $\p_{n,2}$ and $d_{n,2}$ satisfying all conditions in Proposition \ref{bubble-1}. This finishes the proof of Case II-3 and hence the proof of Step 2.

\vskip0.1in
\noindent{\it Step 3: We complete the proof of Proposition \ref{bubble-1}.} 

By the above two steps, we have defined the selection process. Continuously, we consider the function
	$$d_{g_0}(\cdot,\{\p_{n,1},\p_{n,2},q_0\})^2\la_nhu_n^{2(p_n-1)}e^{u_n^{p_n}},\quad\text{in}~\Sigma\setminus\bigcup_{i=1}^2 B_{6d_{n,i}}(\p_{n,i}).$$
If it is uniformly bounded, then we stop and conclude this proposition for $M=2$. Otherwise, using the same arguments, we get $\p_{n,3}$ and $d_{n,3}$. By induction, suppose there are $k$ bubble areas $B_{d_{n,i}}(\p_{n,i})$, $i=1,\cdots,k$, then by conclusion (3) of Proposition \ref{bubble-1}, we get that
	$$\lim_{R\to\iy}\lim_{n\to\iy}\f{p_n}{2}\int_{B_{R\mu_{n,i}}(\p_{n,i})}\la_np_nhfu_n^{2(p_n-1)}e^{u_n^{p_n}}\rd v_{g_0}=4\pi~\text{or}~4\pi(1+\al).$$
That is, each bubble area $B_{d_{n,i}}(\p_{n,i})$ contributes a positive energy of at least $4\pi(1+\al)$, thus by using \eqref{bound-1}, this process must stop after finite steps. 

This completes the proof of Proposition \ref{bubble-1}.
\end{proof}

\subsection{Gradient estimate}\lab{sec3-2}\ 

To obtain Proposition \ref{bubble}, we need to give the gradient estimate.
We take $\wt\Sigma_\iy$ as the limit of $\wt\Sigma_n$ given in Proposition \ref{bubble-1},
\be \wt\Sigma_\iy:=\lbr{p_i^*:~i=1,\cdots,M}. \ee
In view of conclusion (5) of Proposition \ref{bubble-1}, given any compact subset $K\subset\Sigma\setminus(\wt\Sigma_\iy\cup\DR)$, we get that $\la_nhfu_n^{2(p_n-1)}e^{u_n^{p_n}}=O(1)$, so $$\la_nhfu_n^{p}e^{u_n^{p_n}}=O(1)$$
uniformly in $K$ for any $p\in[0,2(p_n-1)]$ and large $n$. Furthermore, since $h$ is bounded in $K$, we see from Lemma \ref{bound-0} that $u_n$ and $hu_n$ are bounded in $L^s(K)$ for any $s>1$. 
Thus, it follows from \eqref{equ-5} that $\Delta_{g_0}u_n$ is bounded in $L^s_\loc(\Sigma\setminus(\wt\Sigma_\iy\cup\DR))$ for any $s>1$, and by the elliptic theory, we get that $u_n$ is bounded in $W^{2,s}_{\loc}(\Sigma\setminus(\wt\Sigma_\iy\cup\DR))$ for any $s>1$, and then
\be\lab{out} \lim_{n\to+\iy}u_n=u_0 \quad\text{in}~C_\loc^1(\Sigma\setminus(\wt\Sigma_\iy\cup\DR)). \ee
where $u_0\geq 0$. 

\begin{definition}
For $\p_{n,i}\in\wt\Sigma_n$, we call $(\p_{n,i},d_{n,i})$ a bubble. More precisely,
\begin{itemize}[fullwidth,itemindent=0em]
\item[(a)] we call $(\p_{n,i},d_{n,i})$ a singular bubble, if the second case (3-2) of the conclusion (3) in Proposition \ref{bubble-1} happens, i.e.
	\[\p_i^*\in\DR\quad\text{and } \;\lim_{n\to+\iy}\f{|x_{n,i}|}{\mu_{n,i}}=0.\]
\item[(b)] we call $(\p_{n,i},d_{n,i})$ a non-singular bubble, if the first case (3-1) of the conclusion (3) in Proposition \ref{bubble-1} happens, i.e.
	\[\p_i^*\not\in\DR, \quad\text{or } \; \p_i^*\in\DR\;\text{and}\; \lim_{n\to+\iy}\f{|x_{n,i}|}{\mu_{n,i}}=+\infty.\]
\end{itemize}
Remark that, Proposition \ref{bubble-1}-(5) shows that at each singularity $q_0\in\DR$, there is at most one singular bubble $(\p_{n,i_0},d_{n,i_0})$ concentrating at $q_0$.
\end{definition}

In this subsection, we establish the following gradient estimate on $u_n$, which can be viewed as an analog of \cite[Proposition 2]{MT-blowup-2} by Druet for conical singular surfaces. The basic idea of the proof also comes from \cite[Proposition 2]{MT-blowup-2}. Of course, since we deal with the singular case, some different techniques are needed.

\begin{proposition}\lab{gradient}
Assume that $u_n$ satisfies \eqref{blow-up}.  Let $M\in\N_+$ and $\wt\Sigma_n$ be given by Proposition \ref{bubble-1}, and let $\DR'$ be  given by conclusion (5) of Proposition \ref{bubble-1}, such that \eqref{tem-321-1} holds for any $q_0\in\DR'$. Then
\be\lab{326-1} d_{g_0}(\cdot,\wt\Sigma_n\cup\DR)u_n^{p_n-1}|\nabla_{g_0} u_n|=O(1),\ee
or equivalently,
 \be\label{326-5} d_{g_0}\sbr{\cdot,\wt\Sigma_n\cup(\DR\setminus\DR')}u_n^{p_n-1}|\nabla_{g_0} u_n|=O(1),\ee
uniformly in $\Sigma\setminus \cup_{q_0\in\DR'}B_{\mu_{n,i_0}}(q_0)$ for large $n$.
\end{proposition}

\begin{proof} First, by the same argument as Remark \ref{326-2}, we know that  \eqref{326-1} is equivalent to \eqref{326-5}.

As in the proof of Proposition \ref{bubble-1}, since the discussion is locally near some point, the number of singularities does not affect the proof and only brings complexity in notations, so we may assume that $\DR=\{q_0\}$ and the order of $q_0$ is $\al=\al_{q_0}\in(-1,0)$.

We divide the proof into two cases.

\vskip0.1in
\noindent{\bf Case 1:} There is a singular bubble concentrating at $q_0$, say $(\p_{n,1},d_{n,1})$, i.e., $\DR'=\DR=\{q_0\}$.

Suppose by contradiction that \eqref{326-1} or equivalently \eqref{326-5} does not hold,  then there exists a point sequence $q_n\in \Sigma\setminus B_{\mu_{n,1}}(q_0)\subset \Sigma\setminus B_{4\mu_{n,1}/5}(\p_{n,1})$ such that
\be\lab{66} d_{g_0}(q_n,\wt\Sigma_n)u_n(q_n)^{p_n-1}|\nabla_{g_0} u_n(q_n)|=\max_{\Sigma\setminus B_{4\mu_{n,1}/5}(\p_{n,1})}d_{g_0}(\cdot,\wt\Sigma_n)u_n^{p_n-1}|\nabla_{g_0} u_n|, \ee
and that
\be\lab{60} \lim_{n\to+\iy}d_{g_0}(q_n,\wt\Sigma_n)u_n(q_n)^{p_n-1}|\nabla_{g_0}u_n(q_n)|=+\iy. \ee
Up to a subsequence, we have $\displaystyle\lim_{n\to+\iy}q_n=q^*$. By \eqref{out} and \eqref{60}, we get that $q^*\in\wt\Sigma_\iy$. 
If $q^*\neq q_0$, this proposition has been proved in \cite[Proposition 4.1]{MT-blowup-3}. Therefore, we only need to consider the case $q^*=q_0$, namely we may assume that $q^*=q_0$, and that $\p_1^*=\p_2^*=\cdots=\p_l^*=q_0$ for some $1\le l\le M$, and that $\p_{j}^*\neq q_0$ for $j=l+1,\cdots,M$. 

Take the isothermal coordinates $(B_r(q_0),\Psi,\Omega)$ around $q_0$ as above \eqref{ufh-1}. 
Set $y_n=\Psi(q_n)$ and $x_{n,i}=\Psi(\p_{n,i})$ for $i=1,\cdots,l$. 
Define 
\be s_n:=\min_{i=1,\cdots,l}|y_n-x_{n,i}| \quad\text{and}\quad y_{n,i}:=\f{x_{n,i}-y_n}{s_n},~i=1,\cdots,l. \ee
Clearly, $|y_{n,i}|\geq 1$ for all $i$ and $n$. We claim:
\be\lab{61-0}\f{\mu_{n,1}}{s_n}=o(|y_{n,1}|),\ee
where $\mu_{n,i}$, $i=1,\cdots,l$, are given by Proposition \ref{bubble-1}.
Indeed, if $|y_{n,1}|=\frac{|x_{n,1}-y_n|}{s_n}\to+\infty$ as $n\to+\iy$, then $s_n=|y_n-x_{n,i}|$ for some $2\leq i\leq l$, and we see from $\mu_{n,1}\ll d_{n,1}\leq |x_{n,1}-x_{n,i}|$ that
	\[\f{\mu_{n,1}}{s_n}\ll \frac{|x_{n,1}-x_{n,i}|}{s_n}\leq \frac{|x_{n,1}-y_{n}|}{s_n}+\frac{|y_{n}-x_{n,i}|}{s_n}=|y_{n,1}|+1,\]
so \eqref{61-0} holds. If $1\leq |y_{n,1}|=\frac{|x_{n,1}-y_n|}{s_n}=O(1)$, then to prove \eqref{61-0} is equivalent to prove
\be\lab{61} \lim_{n\to+\iy}\f{s_n}{\mu_{n,1}}=+\iy. \ee
Assume by contradiction that $s_n=|y_n-x_{n,i}|=O(\mu_{n,1})$ for some $i\in\{1,\cdots,l\}$. If $i\ge2$, then we have
	$$\f{Cs_n}{\mu_{n,1}}\ge\f{|y_n-x_{n,1}|}{\mu_{n,1}}\ge\f{|x_{n,i}-x_{n,1}|-|y_n-x_{n,i}|}{\mu_{n,1}}\ge\f{d_{n,1}-O(\mu_{n,1})}{\mu_{n,1}}\to\iy, $$
which is impossible, so that $i=1$, i.e., $s_n=|y_n-x_{n,1}|$. From here and $q_n\in \Sigma\setminus B_{4\mu_{n,1}/5}(\p_{n,1})$, we obtain $C^{-1}\leq \f{s_n}{\mu_{n,1}}=\frac{|y_n-x_{n,1}|}{\mu_{n,1}}\leq C$.  By conclusion (3-2) of Proposition \ref{bubble-1}, we know that $\eta_{n,1}=\f{p_n}{2}\ga_{n,1}^{p_n-1}\big(u_n(x_{n,1}+\mu_{n,1}\cdot)-\ga_{n,1}\big)$ converges to $-\ln(1+|\cdot|^{2(1+\al)})$ in $C_\loc^1(\R^2\setminus\{0\})$, so that,
	$$u_n(y_n)=\ga_{n,1}+\frac{2\eta_{n,1}(\frac{y_n-x_{n,1}}{\mu_{n,1}})}{p_n\ga_{n,1}^{p_n-1}}=\ga_{n,1}+O\sbr{\f{1}{\ga_{n,1}^{p_n-1}}},\; \nabla u_n(y_n)=O\sbr{\f{1}{\ga_{n,1}^{p_n-1}\mu_{n,1}}}.$$
It follows from $d_{g_0}(q_n,\wt\Sigma_n)=|q_n-\p_{n,1}|=(1+o(1))|y_n-x_{n,1}|=(1+o(1))s_n$ that
	$$d_{g_0}(q_n,\wt\Sigma_n)u_n(q_n)^{p_n-1}|\nabla u_n(q_n)|_{g_0}=O\sbr{s_nu_n(y_n)^{p_n-1}\abs{e^{-2\psi_0(y_n)}\nabla u_n(y_n)}}=O(1),$$
which is a contradiction with \eqref{60}. This proves \eqref{61} and so \eqref{61-0} holds.

We define
	$$v_n(x):=u_n(y_n+s_nx),\quad\text{for}~x\in\Omega_n:=\f{\Omega-y_n}{s_n}.$$
Since $|y_n|=o(1)$ and $s_n=o(1)$, we have $\lim_{n\to+\iy}\Omega_n=\R^2$. Denote
	$$S_n:=\lbr{y_{n,i}:~i=1,\cdots,l},\quad S:=\lim_{n\to+\iy}S_n.$$
We have
	$$d(0,S_n)=\min_{i=1,\cdots,l}\f{|x_{n,i}-y_n|}{s_n}=1,$$
so $d(0,S)=1$. By \eqref{equ-4-1}, we get that
\be\lab{68-0} -\Delta v_n=s_n^2\la_np_nhf(y_n+s_n\cdot)v_n^{p_n-1}e^{v_n^{p_n}}-s_n^2h(y_n+s_n\cdot)v_n\quad\text{in }\;\Omega_n. \ee
By Remark \ref{326-2}, we obtain that
\be\lab{68} s_n^2\la_nh(y_n+s_n\cdot)v_n^{2(p_n-1)}e^{v_n^{p_n}}=O\sbr{\f{1}{d(\cdot,S_n)^2}}, \ee
uniformly in  $\Omega_n\setminus B_{2\f{\mu_{n,1}}{s_n}}(y_{n,1})$ for large $n$.
For any fixed $R\geq 4$, let 
	$$A_R:=B_R(0)\setminus \bigcup_{y\in S}B_{\f{1}{R}}(y).$$
Since $A_R\subset \Omega_n\setminus B_{2\f{\mu_{n,1}}{s_n}}(y_{n,1})$ for $n$ large by \eqref{61-0}, so $s_n^2\la_nh(y_n+s_n\cdot)v_n^{2(p_n-1)}e^{v_n^{p_n}}=O(1)$ in $A_R$ and for $n$ large. Thus, for any $s>1$, we get by using $p_n\in[1,2]$ and \eqref{boundla} that
$$\begin{aligned} 
	&~~\int_{A_R}\sbr{s_n^2\la_nh(y_n+s_nx)v_n(x)^{p_n-1}e^{v_n(x)^{p_n}}}^s\rd x\\
	&\le \sbr{\int_{A_R\cap\{v_n\ge \ga\}}+\int_{A_R\cap\{v_n\le \ga\}}}\sbr{s_n^2\la_nh(y_n+s_nx)v_n(x)^{p_n-1}e^{v_n(x)^{p_n}}}^s\rd x\\
	&\le \f{C_{R,s}}{\ga^{s(p_n-1)}}+C_{\ga}s_n^{2s}\int_{A_R}h(y_n+s_nx)^{s}\rd x,
\end{aligned}$$
and then, by letting $n\to+\iy$ and $\ga\to+\iy$, and noting from $\frac{|x_{n,1}|}{s_n}=o(\frac{\mu_{n,1}}{s_n})=o(|y_{n,1}|)$ and $\alpha\in (-1,0)$ that
	\be\lab{hAR} h(y_n+s_nx)=O\sbr{|y_n+s_nx|^{2\al}}=O\sbr{s_n^{2\al}\abs{x-y_{n,1}+\f{x_{n,1}}{s_n}}^{2\al}}=O\sbr{s_n^{2\al}},\ee
uniformly for $x\in A_R$, we get 
\be\lab{62} \nm{s_n^2\la_nh(y_n+s_n\cdot)v_n^{p_n-1}e^{v_n^{p_n}}}_{L^s(A_R)}=o(1). \ee
For the linear term $s_n^2hv_n$, by \eqref{hAR} and Lemma \ref{bound-0}, we get that for any $s>1$,
\be\lab{63}\begin{aligned} 
	\int_{A_R}\sbr{s_n^2h(y_n+s_nx)v_n}^s \rd x&=O\sbr{s_n^{2s+2\al(s-1)}\int_{A_R}h(y_n+s_nx)v_n^s\rd x} \\
		&=O\sbr{s_n^{(2+2\al)(s-1)}\int_{\Sigma}hu_n^s\rd v_{g_0}}=o(1).
\end{aligned}\ee
Then, by \eqref{62} and \eqref{63}, we get from \eqref{68-0} that $$\nm{\Delta v_n}_{L^s(A_R)}=o(1) \quad\text{for any}\; s>1.$$ 

Now we claim that $v_n(0)\to+\iy$ as $n\to+\iy$. 
If not, we assume $v_n(0)=O(1)$. For any $R\geq 4$, take $w_n$ to be the harmonic part of $v_n$ on $A_R$ with $w_n=v_n$ on $\pa A_R$, i.e.,
$$ \begin{cases}\Delta w_n=0\quad\text{in}~A_R,\\w_n=v_n\quad\text{on}~\pa A_R,\end{cases} \quad\text{and}\quad
	 \begin{cases}\Delta (v_n-w_n)=\Delta v_n\quad\text{in}~A_R,\\
	v_n-w_n=0\quad\text{on}~\partial A_R.\end{cases}$$ 
Then $\Delta (v_n-w_n)$ is bounded in $L^s(A_R)$ for any $s>1$, so that the elliptic theory gives that $v_n-w_n$ is bounded in $W^{2,s}(A_R)$ and hence in $C^{1}(A_R)$. 
Since $v_n>0$ and $v_n(0)=O(1)$, we know that $w_n$ is a harmonic function in $A_R$, which is bounded from below and satisfies $w_n(0)=O(1)$. Applying the Harncak inequality, we obtain that $w_n$ is bounded in $L^\iy(A_{R/2})$, and so it is bounded in $C^{1}(A_{R/2})$. Therefore, $v_n$ is bounded in $C^1(A_{R/2})$.
Noting from $d(0,S)=1$ that $0\in A_{R/2}$, we have $v_n(0)^{p_n-1}|\nabla v_n(0)|=O(1)$. While by \eqref{60}, we get
\be\lab{64} v_n(0)^{p_n-1}|\nabla v_n(0)|=s_n u_n(y_n)^{p_n-1}|\nabla u_n(y_n)|\to+\iy,\ee
as $n\to+\iy$. This is a contradiction. This proves $v_n(0)\to+\iy$ as $n\to+\iy$.

Replacing $v_n$ by $\f{v_n}{v_n(0)}-1$ in the above arguments, we obtain
\be\lab{65} \lim_{n\to+\iy}\f{v_n}{v_n(0)}=1,\quad\text{in}~C_\loc^1(\R^2\setminus S). \ee
We now set
	$$\ti v_n:=\f{v_n-v_n(0)}{|\nabla v_n(0)|},\quad\text{in}~\Omega_n.$$
Given any compact subset $K\subset \Omega_n\setminus (S_n\cup B_{\f{\mu_{n,1}}{s_n}}(y_{n,1}))$, since 
$$d_{g_0}(\Psi^{-1}(y_n+s_nx), \widetilde\Sigma_n)=(1+o(1))d(y_n+s_nx, \{x_{n,i}\}_{i=1}^l)=(1+o(1))s_n d(x, S_n)$$
uniformly for $x\in K$,
it follows from \eqref{66} and $d(0, S_n)=1$ that
	$$v_n(x)^{p_n-1}|\nabla v_n(x)|\le (1+o(1))\f{v_n(0)^{p_n-1}|\nabla v_n(0)|}{d(x,S_n)}, \quad\text{ uniformly for $x\in K$}.$$
This together with \eqref{65} gives
	$$|\nabla \ti v_n(x)|\le\f{1+o(1)}{d(x,S_n)}=\f{1+o(1)}{d(x,S)}, \quad\text{ uniformly for $x\in K$}.$$
From here and $\ti v_n(0)=0$, we get that $\ti v_n$ is bounded in $C_\loc^1(\R^2\setminus S)$. 
In view of \eqref{65}, we get from \eqref{68-0} that
	$$\begin{aligned}-\Delta\ti v_n=\f{1+o(1)}{v_n(0)^{p_n-1}|\nabla v_n(0)|}\sbr{s_n^2\la_nhf(y_n+s_n\cdot)v_n^{2(p_n-1)}e^{v_n^{p_n}}-s_n^2h(y_n+s_n\cdot)v_n^{p_n}}, \end{aligned}$$
in $A_R$. Then, by \eqref{68}, \eqref{63} and \eqref{64}, we obtain $\nm{\Delta \ti v_n}_{L^s(A_R)}=o(1)$ for any $s>1$ and $R\geq 4$. By the elliptic theory, we get that
\be\lab{69} \lim_{n\to+\iy}\ti v_n=\ti v, \quad\text{in}~C_\loc^1(\R^2\setminus S), \ee
where $\ti v$ satisfies
	$$-\Delta \ti v=0~~\text{in}~\R^2\setminus S,\quad \ti v(0)=0,\quad |\nabla\ti v(0)|=1, $$
and
\be\lab{70} |\nabla\ti v(x)|\le\f{1}{d(x,S)},\quad\forall x\in\R^2\setminus S. \ee

We claim that $\ti v$ is smooth in $\R^2$. Let $y_0\in S$ and $r_0=\f{1}{2}d(y_0,S\setminus\{y_0\})$. 
Let $\zeta_n$ be such that
  $$-\Delta\zeta_n=-s_n^2h(y_n+s_n\cdot)v_n~~\text{in}~B_{r_0}(y_0),\quad \zeta_n=0~~\text{on}~\pa B_{r_0}(y_0).$$
Then, for any $s\in(1,1/|\al|)$, we may take a $t>1$ such that $2+2\al\f{st-1}{t-1}>0$, and by H\"older inequality and \eqref{h0-1}, and by Lemma \ref{bound-0}, we get that
$$\begin{aligned}
  \nm{\Delta\zeta_n}_{L^s(B_{r_0}(y_0))}&=\sbr{\int_{B_{r_0}(y_0)}(s_n^{2}h(y_n+s_n\cdot)v_n)^s\rd x}^{1/s}\\
&=s_n^{2-\f{2}{s}}\sbr{\int_{B_{s_nr_0}(s_ny_0+y_n)}h^su_n^s\rd x}^{1/s}\\
  &\leq C s_n^{2-\f{2}{s}}\sbr{\int_{\Omega}\abs{x}^{2\al\f{st-1}{t-1}}\rd x }^{\sbr{1-\f{1}{t}}\f{1}{s}} 
    \sbr{\int_{\Omega}hu_n^{st}\rd x }^{\f{1}{st}}  \\
  &=o(1).
\end{aligned}$$
By the elliptic theory, we get that $\nm{\zeta_n}_{L^\iy(B_{r_0}(y_0))}=o(1)$. 
Then, observing that $-\Delta(v_n-\zeta_n)\ge0$ in $B_{r_0}(y_0)$ and by applying the maximum principle, and by using \eqref{65}, we get that
\be\lab{515-1} v_n\ge\inf_{\pa B_{r_0}(y_0)} v_n+o(1)\ge(1+o(1))v_n(0), \ee
uniformly in $B_{r_0}(y_0)$ for large $n$. 
Therefore, for any $0<r<r_0$, by \eqref{515-1}, and by using \eqref{bound-1} and Lemma \ref{bound-0}, we get from \eqref{68-0} that 
	$$\begin{aligned}&\left|\int_{B_r(y_0)}v_n(0)^{p_n-1}\Delta v_n\rd x\right|\\
\leq &C\int_{B_r(y_0)}s_n^2\la_np_nhf(y_n+s_n\cdot)v_n^{2(p_n-1)}e^{v_n^{p_n}}+s_n^2h(y_n+s_n\cdot)v_n^{p_n}\rd x\\\leq & C\int_{\Sigma}\la_np_nhfu_n^{2(p_n-1)}e^{u_n^{p_n}}+hu_n^{p_n}\rd x=O(1),\end{aligned}$$
and then, by integrating by parts, we get that
	$$\int_{\pa B_r(y_0)}v_n(0)^{p_n-1}\pa_\nu v_n\rd\sigma=O(1).$$
Independently, \eqref{69} leads to
	$$\int_{\pa B_r(y_0)}v_n(0)^{p_n-1}\pa_\nu v_n\rd\sigma=v_n(0)^{p_n-1}|\nabla v_n(0)|\sbr{\int_{\pa B_r(y_0)}\pa_\nu\ti v\rd\sigma+o(1)}.$$
This together with \eqref{64} gives
	$$\int_{\pa B_r(y_0)}\pa_\nu\ti v\rd\sigma=0,\quad\forall~0<r<r_0,$$
which implies
	$$\f{\rd}{\rd r}\sbr{\f{1}{2\pi r}\int_{\pa B_r(y_0)}\ti v\rd\sigma}=\f{1}{2\pi r}\int_{\pa B_r(y_0)}\pa_\nu\ti v\rd\sigma=0.$$
Hence there exists some constant $c(y_0)$ depending on $y_0$ such that
	$$\f{1}{2\pi r}\int_{\pa B_r(y_0)}\ti v\rd\sigma=c(y_0),\quad\forall~0<r<r_0.$$
It then follows from \eqref{70} that $|\ti v(x)-c(y_0)|=|\ti v(x)-\f{1}{2\pi r}\int_{\pa B_r(y_0)}\ti v\rd\sigma|\le\pi$ for any $x\in \pa B_r(y_0)$. This indicates that $\ti v$ is bounded near $y_0$, so we conclude that $\ti v$ is smooth in $\R^2$. 

By the mean value property, we have $\int_{\pa B_R(0)}\ti v\rd \sigma=\ti v(0)=0$ for any $R>0$, then we see from \eqref{70} that $\ti v$ is bounded in $L^\iy(\R^2)$. Applying the Liouville theorem, we have $\ti v(x)\equiv \ti v(0)=0$, which contradicts with the fact that $|\nabla\ti v(0)|=1$. This finishes the proof of Case 1.

\vskip0.1in
\noindent{\bf Case 2:} All the bubbles are non-singular, i.e., $\DR'=\emptyset$. 

Taking the same notations in Case 1, apart from the non-singular bubble center $x_{n,i}=\Psi(\p_{n,i})$ $=o(1)$, $i=1,\cdots,l$, we denote the singularity $x_{n,0}=\Psi(q_0)=0$, and again we define
	$$s_n:=\min_{i=0,\cdots,l}|y_n-x_{n,i}| \quad\text{and}\quad y_{n,i}:=\f{x_{n,i}-y_n}{s_n},~i=0,\cdots,l,$$
$$S_n:=\{y_{n,i} : i=0,1,\cdots,l\}.$$
Since conclusion (5) of Proposition \ref{bubble-1} says 
$$d_{g_0}(\cdot,\wt\Sigma_n\cup\{q_0\})^2\la_nhfu_n^{2(p_n-1)}e^{u_n^{p_n}}=O(1), \quad\text{uniformly in}\;\Sigma,$$
it follows that, instead of \eqref{68}, there holds
	$$s_n^2\la_nh(y_n+s_n\cdot)v_n^{2(p_n-1)}e^{v_n^{p_n}}\le O\sbr{\f{1}{d(\cdot,S_n)^2}}, $$
uniformly in $\Omega_n\setminus S_n$ for large $n$, so that the same argument as Case 1 leads to a contradiction, and we finish the proof of Case 2.
\end{proof}

Now we are ready to give the proof of Proposition \ref{bubble}. 

\begin{proof}[\bf Proof of Proposition \ref{bubble}]
If a bubble $(\p_{n,i},d_{n,i})$ satisfies conclusion (3-2) of Proposition \ref{bubble-1}, then we may transform the bubble center $\p_{n,i}$ to $\p_i^*=\displaystyle\lim_{n\to+\iy}\p_{n,i}\in\DR'$, and by using $d_{g_0}(\p_{n,i},q_0)=(1+o(1))|x_{n,i}|=o(\mu_{n,i})$, we get the alternative in Proposition \ref{bubble}-(2). Therefore, by using Proposition \ref{bubble-1}, Proposition \ref{gradient}, Remark \ref{326-2} and \eqref{out}, we easily get Proposition \ref{bubble}, except one thing that we lose the maximality of the singular bubble center, but this has no influence on our later proofs (see Remark \ref{515-2}-(2)).
\end{proof}

\vs
\section{Blow-up analysis of single bubbles}\lab{appe2}

In this section, we need to prove some refined estimates for the solution sequence near each bubbles, see Propositions \ref{regular} and \ref{singular} below, and in the proofs we need to use some results from Appendices \ref{appe1} and \ref{appe3}. These refined estimates will be applied in the proof of the quantization result 
Theorem \ref{thm0-2} in later sections. Again we need to discuss the non-singular bubble and the singular bubble separately.

\subsection{Non-singular bubble}\lab{appe2-2}\ 

In this subsection, we deal with the non-singular bubble case. Let us give our general setting, which is a modification of those  in \cite[Section 3]{MT-blowup-7}.
Let $\al\in\sbr{-1,0}$ and $(p_n)_n$ be a sequence of numbers in $[1,2]$, and let $(\mu_n)_n,(\bar r_n)_n$ be sequences of positive numbers. 
Let $(u_n)_n$ be a sequence of smooth functions in $B_{\bar r_n}(0)\subset\R^2$. 
We assume that
\be\lab{4-50} \nabla u_n(0)=0\ee
for all $n$, and
\be\lab{4-51} \ga_n:=u_n(0)\to+\iy \ee
as $n\to+\iy$. Let $(x_n)_n$ be a bounded sequence of points in $\R^2\setminus\{0\}$, and $(\la_n)_n$ be given by
\be\lab{4-52} \mu_n^{2}\la_np_n^2|x_n|^{2\al}h_n(0)f_n(0)\ga_n^{2(p_n-1)}e^{\ga_n^{p_n}}=8, \ee
where $f_n,h_n$ are given positive functions satisfying
\be\lab{4-53} \f{1}{C}\le\nm{f_n}_{L^\iy(B_{\bar r_n}(0))}\le C,\quad \nm{\nabla f_n}_{L^\iy(B_{\bar r_n}(0))}\le C,\ee
\be\lab{4-53-0} \f{1}{C}\le\nm{h_n}_{L^\iy(B_{\bar r_n}(0))}\le C,\quad\nm{\nabla h_n}_{L^\iy(B_{\bar r_n}(0))}\le C, \ee
for some given constant $C\ge1$. 
Let $t_n$ be given by
\be\lab{4-54} t_n(x):=\ln\sbr{1+\f{|x|^{2}}{\mu_n^{2}}}. \ee
Let $\eta\in(0,1)$ be fixed. We assume that 
\be\lab{4-55} \lim_{n\to+\iy}\f{\mu_n}{\bar r_n}=0, \ee
\be\lab{4-56} t_n(\bar r_n)\le \eta\f{p_n\ga_n^{p_n}}{2}, \ee
\be\lab{4-58} \bar r_n\le\f{1}{2}|x_n|=O(1), \ee
\be\lab{4-57} \int_{B_{\bar r_n}(0)}|x-x_n|^{2\al}e^{\ep_0 u_n^{1/3}}\rd x\le \bar C, \ee
for all large $n$, where $\ep_0$ is given in Lemma \ref{bound-0}, and that
\be\lab{4-59} \f{p_n}{2}\ga_n^{p_n-1}(\ga_n-u_n(\mu_nx))\to\ln(1+|x|^{2})\quad\text{in}~C^1_{\loc}(\R^2)\ee
as $n\to+\iy$. By Lemma \ref{bound-0} and Proposition \ref{bubble}, the last two assumptions are natural ones.

We assume that $u_n$ solves
\be\lab{4-61} -\Delta u_n+|x-x_n|^{2\al}h_nu_n=|x-x_n|^{2\al}\la_np_nh_nf_nu_n^{p_n-1}e^{u_n^{p_n}},\; u_n>0,\;\text{in}~B_{\bar r_n}(0), \ee
and the following key gradient estimate holds: there exists $C_G>0$ such that
\be\lab{4-62} |x|u_n(x)^{p_n-1}|\nabla u_n(x)|\le C_G\quad\text{for all}~x\in B_{\bar r_n}(0) \ee
for all $n$. 
Let $(v_n)_n$ be a sequence of functions solving
\be\lab{4-60}\begin{cases}
	-\Delta v_n+|x_n|^{2\al}h_n(0)v_n=|x_n|^{2\al}\la_np_nh_n(0)f_n(0)v_n^{p_n-1}e^{v_n^{p_n}},\\
	v_n(0)=\ga_n,\\
	v_n~\text{is radially symmetric and positive in }~B_{\bar r_n}(0),
\end{cases}\ee
and let $w_n$ be given by
\be\lab{4-63} u_n=v_n+w_n. \ee
Then by \eqref{4-58}, we know that $w_n$ is smooth in $B_{\bar r_n}(0)$. We have the following refined estimates.

\begin{proposition}\lab{regular}
We have that 
\be\lab{regular-0} \ln\ga_n=o\sbr{\ln\f{1}{\bar r_n|x_n|^\al}}, \ee
\be\lab{regular-1} \abs{w_n(x)}\le\f{C_0|x|}{\ga_n^{p_n-1}\bar r_n}~\text{for all}~x\in B_{\bar r_n}(0), \ee
\be\lab{regular-2} \nm{\nabla w_n}_{L^\iy(B_{\bar r_n}(0))}\le\f{C_0}{\ga_n^{p_n-1}\bar r_n}, \ee
for all large $n$, where $C_0$ is any fixed constant bigger than $4+(C_G/(1-\eta))$, for $C_G$ as in \eqref{4-62} and $\eta$ as in \eqref{4-56}. 

Furthermore, up to a subsequence, there exists a harmonic function $\psi_0$ in $B_1(0)$ such that
\be\lab{regular-3} \lim_{n\to+\iy}\ga_n^{p_n-1}w_n(\bar r_n\cdot)=\psi_0\quad\text{in}~C_{\loc}^1(B_1(0)\setminus\{0\}), \ee
\be\lab{regular-4} \nabla\psi_0(0)=\lim_{n\to+\iy}\f{2\al\bar r_n}{p_n|x_n|}\f{x_n}{|x_n|}. \ee
\end{proposition}

Proposition \ref{regular} is a generalization of \cite[Proposition 3.1]{MT-blowup-7}. Remark that there is a significant difference between Proposition \ref{regular} and \cite[Proposition 3.1]{MT-blowup-7}, namely there holds $\nabla\psi_0(0)=0$ in \cite[Proposition 3.1]{MT-blowup-7}, but here we may have $\nabla\psi_0(0)\neq0$ due to the new term $|\cdot|^{2\alpha}$ with $\al<0$. 

\begin{proof}[\bf Proof of Proposition \ref{regular}] We develop further the ideas of \cite[Proposition 3.1]{MT-blowup-7} to prove this proposition. Remark that due to the new term $|\cdot|^{2\alpha}$ with $\al<0$, we need to overcome some new difficulties. For example, we need to solve a new limiting equation \eqref{4-71}.

Let $r_n$ be given by
\be\lab{4-67} r_n:=\sup\lbr{ r\in[0,\bar r_n]~s.t.~\begin{cases}\ga_n^{p_n-1}r\nm{\nabla w_n}_{L^\iy(B_r(0))}\le C_0,\\ r^2|x_n|^{2\al}\exp\sbr{\f{\ep_0}{2}(1-\eta)^{1/3}\ga_n^{1/3}}\le\f{4\bar C}{\pi}, \end{cases} } \ee
for all $n$, with $\ep_0,\bar C$ as in \eqref{4-57}. Then, by choosing $C_0>\f{C_G}{1-\eta}+4$, we can prove as in \cite[Proposition 3.1]{MT-blowup-7} that $r_n=\bar r_n$ for large $n$ (see also the argument between \eqref{4-18} and \eqref{4-23}). This gives \eqref{regular-2}. As a byproduct, we have 
\be\lab{4-68} \bar r_n^2|x_n|^{2\al}\exp\sbr{\f{\ep_0}{2}(1-\eta)^{1/3}\ga_n^{1/3}}=O(1), \ee
which together with $\gamma_n\to+\infty$ gives \eqref{regular-0}. Note also from \eqref{4-58} and \eqref{4-68} that
$$\bar r_n=o(1).$$
By \eqref{4-50}, \eqref{4-51} and \eqref{4-60}, we find
\be\lab{4-69} w_n(0)=0\quad\text{and}\quad\nabla w_n(0)=0. \ee
This together with \eqref{regular-2} gives \eqref{regular-1}.

Now we give proofs of \eqref{regular-3} and \eqref{regular-4}. By applying Proposition \ref{single-2} and \eqref{16-66} to $v_n$ in $B_{\bar r_n}(0)$, we obtain that
\be\lab{4-64} \ga_n\ge v_n\ge (1-\eta+o(1))\ga_n, \ee
and that
\be\lab{4-65} \la_np_n|x_n|^{2\al}h_n(0)f_n(0)v_n^{p_n-1}e^{v_n^{p_n}}=\f{8e^{-2t_n}}{\mu_n^{2}\ga_n^{p_n-1}p_n}\sbr{1+O\sbr{\f{e^{\ti\eta t_n}}{\ga_n^{p_n}}}}, \ee
uniformly in $B_{\bar r_n}(0)$ and for all large $n$, where $\ti\eta$ is any fixed constant in $(\eta,1)$. 
Using \eqref{4-64} and $|w_n|=O\sbr{\ga_n^{1-p_n}}$ by \eqref{regular-1}, we have $u_n^{p_n}=v_n^{p_n}+p_nv_n^{p_n-1}w_n(1+o(1))$ and $u_n^{p_n-1}=v_n^{p_n-1}\sbr{1+O(|w_n|/\ga_n)}$, then
	$$u_n^{p_n-1}e^{u_n^{p_n}}=v_n^{p_n-1}e^{v_n^{p_n}}\mbr{1+p_nv_n^{p_n-1}w_n\sbr{1+O(\ga_n^{p_n-1}|w_n|)+O(\ga_n^{-p_n})}}.$$
Observe from \eqref{4-58} that
	$$|\cdot-x_n|^{2\al}=|x_n|^{2\al}\sbr{1-2\al\f{\abr{~\cdot~,x_n}}{|x_n|^2}+O\sbr{\f{|\cdot|^2}{|x_n|^2}} }, $$
where $\abr{\cdot,\cdot}$ denotes the standard scalar product in $\R^2$, and that $h_nf_n=h_n(0)f_n(0)(1+O(|\cdot|))$ by \eqref{4-53} and \eqref{4-53-0}, and that $w_n=O\sbr{\ga_n^{1-p_n}|\cdot|/\bar r_n}$, and using \eqref{4-61} and \eqref{4-60}, we may write at last
\be\lab{4-66}\begin{aligned} 
	&-\Delta w_{n}\\
	&=O(|x_n|^{2\al}w_n)+O\sbr{|x_n|^{2\al}\f{|\cdot|}{|x_n|}\ga_n}+\la_np_n|x_n|^{2\al}h_n(0)f_n(0)v_n^{p_n-1}e^{v_n^{p_n}}\\
	&\quad\times\mbr{ p_nv_n^{p_n-1}w_n\sbr{1+O\sbr{\f{|\cdot|}{\bar r_n}+\f{1}{\ga_n^{p_n}}}}-2\al\f{\abr{~\cdot~,x_n}}{|x_n|^2}+O\sbr{|\cdot|+\f{|\cdot|^2}{|x_n|\bar r_n}} }, 
\end{aligned}\ee
uniformly in $B_{\bar r_n}(0)$ and for large $n$. Setting now $$\ti w_n:=\ga_n^{p_n-1}\f{\bar r_n}{\mu_n}w_n(\mu_n\cdot),$$ and given any $R\gg1$, and using $v_n=\ga_n+O(\ga_n^{1-p_n})$ in $B_{R\mu_n}(0)\subset B_{\bar r_n}(0)$ by Proposition \ref{single-2}, we get from \eqref{4-54}, \eqref{4-65} and \eqref{4-66} that
\be\lab{4-66-0}\begin{aligned} 
	&-\Delta \ti w_{n}\\
&=\mu\bar{r}_n\gamma_n^{p_n-1}(-\Delta w_n)(\mu_n\cdot)\\
	&=O(\mu_n^2|x_n|^{2\al}\ti w_n)+O\sbr{\mu_n^2\ga_n^{p_n}|x_n|^{2\al}\f{\bar r_n}{|x_n|}}+\f{8e^{-2T_0}}{p_n}\sbr{1+O(\ga_n^{-p_n})}\\
	&\quad\times\mbr{ p_n\ti w_n\sbr{1+O\sbr{\f{\mu_n}{\bar r_n}+\f{1}{\ga_n^{p_n}}}}-2\al\bar r_n\f{\abr{~\cdot~,x_n}}{|x_n|^2}+O\sbr{\bar r_n+\f{\mu_n}{|x_n|}} }
\end{aligned} \ee
uniformly in $B_R(0)$ for all $n$, where $T_0=\ln(1+|\cdot|^2)$. Then, observing that $|\ti w_n|\le C_0|\cdot|$ in $B_R(0)$ by \eqref{regular-1}, and using \eqref{4-55}, \eqref{4-58} and \eqref{4-68}, we get $\mu_n^2\gamma_n^{p_n}|x_n|^{2\al}=o(1)$ and so
	$$-\Delta \ti w_{n}=8e^{-2T_0}\sbr{\ti w_n-\f{2\al\bar r_n}{p_n|x_n|}\abr{~\cdot~,\f{x_n}{|x_n|}}}+o(1)$$
uniformly in $B_R(0)$ for all $n$. Then, by the elliptic theory and \eqref{4-69}, we get that, up to a subsequence,
\be\lab{4-70} \lim_{n\to+\iy}\ti w_n=w_0\quad\text{in}~C_{\loc}^1(\R^2), \ee
where $w_0$ satisfies
\be\lab{4-71}\begin{cases}
	-\Delta w_0=8e^{-2T_0}\sbr{w_0-\abr{~\cdot~,\vec X}}\quad\text{in}~\R^2,\\
	|w_0|\le C_0|\cdot|\quad\text{in}~\R^2,\\
w_0(0)=0, \nabla w_0(0)=0,
\end{cases}\ee
with 
\be\lab{4-72} \vec X:=\lim_{n\to+\iy}\f{2\al\bar r_n}{p_n|x_n|}\f{x_n}{|x_n|} \in\R^2. \ee
Since $\f{|\cdot|^2}{1+|\cdot|^2}\langle\cdot~,\vec X\rangle$ is a special solution of the equation in \eqref{4-71}, we get by using Lemma \ref{appe3-0} that
\be\lab{4-73} w_0=\f{|\cdot|^2}{1+|\cdot|^2}\abr{~\cdot~,\vec X}. \ee

Let $\psi_n$ be given by
\be\lab{4-74} -\Delta \psi_n=0~~\text{in}~B_{\bar r_n}(0),\quad \psi_n=w_n~~\text{on} ~\pa B_{\bar r_n}(0), \ee
By the hamonic property of $\psi_n(\bar r_n\cdot)$ and the fact that 
	$$\nm{\psi_n(\bar r_n\cdot)}_{L^\iy(\pa B_1(0))}=\nm{w_n}_{L^\iy(\pa B_{\bar r_n}(0))}=O(\ga_n^{1-p_n}),$$
we obtain, through the elliptic theory, that
\be\lab{4-75} \lim_{n\to+\iy}\ga_n^{p_n-1}\psi_n(\bar r_n\cdot)=\psi_0\quad\text{in}~C_\loc^1(B_1(0)), \ee
where $\psi_0$ is harmonic in $B_1(0)$, and we conclude from $\frac{\mu_n}{\bar r_n}\to0$ that
\be\lab{4-76} \lim_{n\to+\iy}\nabla\sbr{\ga_n^{p_n-1}\f{\bar r_n}{\mu_n}\psi_n(\mu_n\cdot)}=\nabla\psi_0(0)\quad\text{in}~C^0_\loc(\R^2). \ee

Let $G_n$ be the Green function of $-\Delta$ in $B_{\bar r_n}(0)$ with the Dirichlet boundary condition (for an explicit formula for $G_n$, see for instance Han-Lin \cite[Proposition 1.22]{book-1}). Then (see also for instance \cite[Appendix B]{MT-blowup-6}), there exists $C>0$ independent of $n$ such that
	$$\abs{\nabla_y G_n(x,y)}\le\f{C}{|x-y|},$$
for any $x,y\in B_{\bar r_n}(0)$, $x\neq y$.
Let $y_n$ be any sequence such that $y_n\in B_{\bar r_n}(0)$. By the Green representation formula, and using \eqref{4-64}, \eqref{4-65}, \eqref{4-66}, \eqref{4-74} and \eqref{regular-1}, we get that
\be\lab{4-77} \begin{aligned}
	\abs{\nabla(w_n-\psi_n)(y_n)}=&\left|\int_{B_{\bar r_n}(0)}\nabla_y G(x,y_n)\Delta w_n(x)dx\right|\\=&O\sbr{\sbr{\f{|x_n|^{2\al}}{\ga_n^{p_n-1}\bar r_n}+|x_n|^{2\al-1}\ga_n}\int_{B_{\bar r_n}(0)}\f{|x|}{|x-y_n|} \rd x}\\
	&+O\sbr{ \f{1}{\ga_n^{p_n-1}\bar r_n} \int_{B_{\bar r_n}(0)}\f{|x|e^{(-2+\ti\eta)t_n}}{\mu_n^{2}|x-y_n|}\rd x }.
\end{aligned}\ee
By the change of variable $x=\bar r_ny$, we first deduce
	$$\int_{B_{\bar r_n}(0)}\f{|x|}{|x-y_n|}\rd x=O(\bar r_n^{2}).$$
If we have $|y_n|=O(\mu_n)$,  again by the change of variable $x=\mu_ny$, we get that
	$$\int_{B_{\bar r_n}(0)}\f{|x|e^{(-2+\ti\eta)t_n}}{\mu_n^{2}|x-y_n|} \rd x=O(1);$$
otherwise, up to a subsequence, we have $\mu_n=o(|y_n|)$ and 
$$\begin{aligned}
	\int_{B_{\bar r_n}(0)}\f{|x|e^{(-2+\ti\eta)t_n}}{\mu_n^{2}|x-y_n|} \rd x
	&=\int_{B_{\f{\bar r_n}{|y_n|}}(0)}\f{1}{\ti\mu_n^{2}}\f{|y|}{|y-\ti y_n|}\f{\rd y}{\sbr{1+\f{|y|^{2}}{\ti\mu_n^{2}}}^{2-\ti\eta}}=O(\ti\mu_n^{1-\ti\eta}),
\end{aligned}$$
by the change of variable $x=|y_n|y$, where $\ti y_n=y_n/|y_n|$ and $\ti\mu_n=\mu_n/|y_n|\to 0$, and the last equality is obtained by estimating the integral on $B_{\ti\mu_n}(0)$ and $B_{\f{\bar r_n}{|y_n|}}(0)\setminus B_{\ti\mu_n}(0)$ separately.
Plugging these estimates in \eqref{4-77}, we get in any case
\be\lab{4-78}\begin{aligned} 
	&\abs{\nabla(w_n-\psi_n)(y_n)}\\
	&=O\sbr{\bar r_n^{2} \sbr{\f{|x_n|^{2\al}}{\ga_n^{p_n-1}\bar r_n}+|x_n|^{2\al-1}\ga_n} }+O\sbr{\f{1}{\ga_n^{p_n-1}\bar r_n}\sbr{ 1+\f{|y_n|}{\mu_n} }^{-(1-\ti\eta)}}\\
	&=\f{1}{\ga_n^{p_n-1}\bar r_n}\sbr{ O\sbr{\sbr{ 1+\f{|y_n|}{\mu_n} }^{-(1-\ti\eta)}}+o(1)  },
\end{aligned}\ee
where we use \eqref{4-58} and \eqref{4-68} to obtain the second equality.

Now let $R\gg1$ and $y_n\in\pa B_{R\mu_n}(0)$ for large $n$. We get from \eqref{4-70}, \eqref{4-76} and \eqref{4-78} that
	$$\abs{\nabla w_0\sbr{\lim_{n\to+\iy}\f{y_n}{\mu_n}}-\nabla\psi_0(0)}=O\sbr{\sbr{ 1+\f{|y_n|}{\mu_n} }^{-(1-\ti\eta)}}+o(1).$$
By letting $R\to+\iy$, we get from \eqref{4-73} that
\be\lab{4-79} \nabla\psi_0(0)=\vec X=\lim_{n\to+\iy}\f{2\al\bar r_n}{p_n|x_n|}\f{x_n}{|x_n|}. \ee
Fix $\ep>0$ small, by \eqref{4-78} and \eqref{4-55}, we get that
\be\lab{4-80} \abs{\nabla\sbr{\ga_n^{p_n-1}(w_n-\psi_n)(\bar r_n\cdot)}}=O\sbr{\sbr{ 1+\f{\bar r_n|\cdot|}{\mu_n} }^{-(1-\ti\eta)}}+o(1)=o(1) \ee
uniformly in $B_1(0)\setminus B_\ep(0)$. By using the fact that $(w_n-\psi_n)(\bar r_n\cdot)=0$ on $\pa B_1(0)$ for all $n$, we obtain from \eqref{4-80} that
	$$\lim_{n\to+\iy}\ga_n^{p_n-1}(w_n-\psi_n)(\bar r_n\cdot)=0\quad\text{in}~C^1_\loc(B_1(0)\setminus\{0\}).$$
This together with \eqref{4-75} and \eqref{4-79} gives \eqref{regular-3} and \eqref{regular-4}. The proof is complete.
\end{proof}

\vs
\subsection{Singular bubble}\lab{appe2-1}\ 

In this subsection, we deal with the singular bubble case, which is a completely new case comparing to  \cite{MT-blowup-7}.
Here for brevity, we will use some of the same symbols as in Section \ref{appe2-2}, which may have different definitions from Section \ref{appe2-2}. Let us give our general setting.
Let $\al\in\sbr{-1,0}$ be fixed and $(p_n)_n$ be a sequence of numbers in $[1,2]$, and let $(\mu_n)_n,(\bar r_n)_n$ be sequences of positive numbers. 
Let $(u_n)_n$ be a sequence of continuous functions in $B_{\bar r_n}(0)$. 
We assume that
\be\lab{4-1} \ga_n:=u_n(0)\to+\iy \ee
as $n\to+\iy$. Let $(\la_n)_n$ be given by
\be\lab{4-2} \mu_n^{2(1+\al)}\la_np_n^2h_n(0)f_n(0)\ga_n^{2(p_n-1)}e^{\ga_n^{p_n}}=8(1+\al)^2, \ee
where $f_n,h_n$ are given positive functions satisfying
\be\lab{4-3} \f{1}{C}\le\nm{f_n}_{L^\iy(B_{\bar r_n}(0))}\le C,\quad \nm{\nabla f_n}_{L^\iy(B_{\bar r_n}(0))}\le C\ee
\be\lab{4-3-0} \f{1}{C}\le\nm{h_n}_{L^\iy(B_{\bar r_n}(0))}\le C,\quad\nm{\nabla h_n}_{L^\iy(B_{\bar r_n}(0))}\le C, \ee
for some given constant $C\ge1$. 
Let $t_n$ be given by
\be\lab{4-4} t_n(x):=\ln\sbr{1+\f{|x|^{2(1+\al)}}{\mu_n^{2(1+\al)}}}. \ee
Let $\eta\in(0,1)$ be fixed. We assume that 
\be\lab{4-5} \lim_{n\to+\iy}\f{\mu_n}{\bar r_n}=0, \ee
\be\lab{4-6} t_n(\bar r_n)\le \eta\f{p_n\ga_n^{p_n}}{2}, \ee
\be\lab{4-7} \int_{B_{\bar r_n}(0)}|\cdot|^{2\al}e^{\ep_0 u_n^{1/3}}\rd x\le \bar C, \ee
for all large $n$, where $\ep_0$ is given in Lemma \ref{bound-0}, and that
\be\lab{4-8} \f{p_n}{2}\ga_n^{p_n-1}(\ga_n-u_n(\mu_n\cdot))\to\ln(1+|\cdot|^{2(1+\al)})\quad\text{in}~C^{0,\delta}_{\loc}(\R^2)\cap C^1_\loc(\R^2\setminus\{0\}),\ee
as $n\to+\iy$, for any $\delta\in\sbr{0,\min\{1,2(1+\al)\}}$. By Lemma \ref{bound-0} and Proposition \ref{bubble}, the last two assumptions are natural ones.

We assume that $u_n$ solves
\be\lab{4-10} -\Delta u_n+|\cdot|^{2\al}h_nu_n=|\cdot|^{2\al}\la_np_nh_nf_nu_n^{p_n-1}e^{u_n^{p_n}},\quad u_n>0,\quad\text{in}~B_{\bar r_n}(0), \ee
and the following key gradient estimate holds: there exists $C_G>0$ such that
\be\lab{4-11} |x|u_n(x)^{p_n-1}|\nabla u_n(x)|\le C_G\quad\text{for all}~x\in B_{\bar r_n}(0)\setminus B_{\mu_n}(0) \ee
for all $n$. 
Let $(v_n)_n$ be a sequence of functions solving
\be\lab{4-9}\begin{cases}
	-\Delta v_n+|\cdot|^{2\al}h_n(0)v_n=|\cdot|^{2\al}\la_np_nh_n(0)f_n(0)v_n^{p_n-1}e^{v_n^{p_n}},\\
	v_n(0)=\ga_n,\\
	v_n~\text{is radially symmetric and positive in }~B_{\bar r_n}(0),
\end{cases}\ee
for all $n$. Letting $w_n$ be given by
\be\lab{4-12} u_n=v_n+w_n, \ee
then the following proposition holds:
\begin{proposition}\lab{singular}
We have that 
\be\lab{singular-0} \ln\ga_n=o\sbr{\ln\f{1}{\bar r_n}}, \ee
\be\lab{singular-1} \abs{w_n(x)}\le\f{C_0|x|}{\ga_n^{p_n-1}\bar r_n}+O\sbr{\f{\mu_n^{\delta_0}}{\ga_n^{p_n-1}\bar r_n^{\delta_0}}}~\text{for all}~x\in B_{\bar r_n}(0), \ee
for some small $\delta_0>0$, and that
\be\lab{singular-2} \nm{\nabla w_n}_{L^\iy(B_{\bar r_n}(0)\setminus B_{\mu_n}(0))}\le\f{C_0}{\ga_n^{p_n-1}\bar r_n}, \ee
for all large $n$, where $C_0$ is any fixed constant greater that $4+(C_G/(1-\eta))$, for $C_G$ as in \eqref{4-11} and $\eta$ as in \eqref{4-6}. 
\end{proposition}
\begin{remark}\lab{515-2}
We give some remarks. 
\begin{itemize}[fullwidth,itemindent=1em] 
\item[(1)] In general, we can not expect a similar convergence result like \eqref{regular-3} and \eqref{regular-4}. Indeed, if $\al\in\sbr{-\f{1}{2},0}$, then by the elliptic theory, we have that the convergence in \eqref{4-8} also holds in $C^1_\loc(\R^2)$, so that we may have similar results like \eqref{regular-3} and \eqref{regular-4} by the same argument. However, if $\al\in\left(-1,-\f{1}{2}\right]$, then the converge in \eqref{4-8} should only hold in the  H\"older sense $C^{0,\delta}_\loc(\R^2)$, and we have no information of $\nabla u_n$ near 0, and worse there may be
	 $$\nm{\nabla w_n}_{L^\iy(B_{\bar r_n}(0)\setminus\{0\})}=+\iy,$$
which leads to that the argument in the proof of Proposition \ref{regular} fails to apply. To overcome this difficulty for the latter case, we use the H\"older seminorm instead of the gradient norm, and we succeed to prove a weaker control on $w_n$ (comparing \eqref{singular-1} with \eqref{regular-1}), which is enough to our discussion in Section \ref{sec6}, and particularly, it permits us to use the {\it local Pohozaev identity} technique in the proof of Lemma \ref{lemma6-1}.
\item[(2)] We point out that the exact form of the last term in \eqref{singular-1} is importmant for our proof in Section \ref{sec6}. By easy computations, one can get a weaker estimate
  $$\ga_n^{p_n-1}\abs{w_n}\le\f{C_0|\cdot|}{\bar r_n}+o\sbr{1}.$$
However, this is not enough for our discussion in Section \ref{sec6}, because we need to compare the rate of the right term with $\f{1}{\ga_n^{p_n}}$, see for example the argument before \eqref{s6-31}, and the equality \eqref{s6-31-1}. This is the motivation of the following complicated argument by using the H\"older seminorm.
\item[(3)] Comparing with Proposition \ref{regular}, we need no the assumption $\nabla u_n(0)=0$ (see \eqref{4-50}), due to the classification result for the singular Liouville equation in Appendix \ref{appe3}. 
\end{itemize}
\end{remark}

\begin{proof}[\bf Proof of Proposition \ref{singular}]
By \eqref{4-5}, \eqref{4-6} and \eqref{4-8} , we may take a sequence $\ti r_n\le \bar r_n$ such that $\mu_n=o(\ti r_n)$, 
\be\lab{4-13} t_n(\ti r_n)=\ln\sbr{1+\f{|\ti r_n|^{2(1+\al)}}{\mu_n^{2(1+\al)}}}=o(\ga_n^{p_n}), \ee 
and that
\be\lab{4-14} \nm{\f{p_n}{2}\ga_n^{p_n-1}(\ga_n-u_n(\mu_n\cdot))-\ln\sbr{1+|\cdot|^{2(1+\al)}} }_{C^{0,\delta}(B_{\ti r_n/\mu_n}(0))}=o(1), \ee
for any $\delta\in\sbr{0,\min\{1,2(1+\al)\}}$, and that
\be\lab{4-14-1} \nm{\f{p_n}{2}\ga_n^{p_n-1}(\ga_n-u_n(\mu_n\cdot))-\ln\sbr{1+|\cdot|^{2(1+\al)}} }_{C^1(B_{\ti r_n/\mu_n}(0)\setminus B_1(0))}=o(1). \ee
By using \eqref{4-14} with \eqref{4-13}, we get that $$u_n=\ga_n(1+o(1))\quad \text{uniformly in $B_{\ti r_n}(0)$},$$
so we get from \eqref{4-7} that 
\be\lab{4-15-0} \ti r_n^{2+2\al}\exp\sbr{\f{\ep_0}{2}\ga_n^{1/3}}\le\f{\bar C}{\pi},\ee
for all large $n$. This implies $\ga_n^{2p_n}\ti r_n^{2(1+\al)}=o(1)$, so we can apply Propositon \ref{single-1} and \eqref{7-1}  to $v_n$ in $B_{\ti r_n}(0)$, to get that
\be\lab{4-15} \nm{\f{p_n}{2}\ga_n^{p_n-1}(\ga_n-v_n(\mu_n\cdot))-\ln\sbr{1+|\cdot|^{2(1+\al)}} }_{C^{0,\delta}(B_{\ti r_n/\mu_n}(0))}=o(1), \ee
for any $\delta\in\sbr{0,\min\{1,2(1+\al)\}}$, and that
\be\lab{4-15-1} \nm{\f{p_n}{2}\ga_n^{p_n-1}(\ga_n-v_n(\mu_n\cdot))-\ln\sbr{1+|\cdot|^{2(1+\al)}} }_{C^1(B_{\ti r_n/\mu_n}(0)\setminus B_1(0))}=o(1). \ee
By using \eqref{4-14} and \eqref{4-15}, we immediately get that
\be\lab{4-17} |w_n|=o(\ga_n^{1-p_n}),\quad\text{uniformly in $B_{\ti r_n}(0)$ for large $n$,}\ee
while \eqref{4-14-1} and \eqref{4-15-1} implies that
\be\lab{4-17--3} \frac{p_n}{2}\gamma_n^{p-1}\mu_n\|\nabla w_n\|_{L^{\infty}(B_{\ti r_n}(0)\setminus B_{\mu_n}(0))}=o(1),\quad\text{for large $n$.}\ee

Let $r_n$ be given by
\be\lab{4-18} r_n:=\sup\lbr{ r\in[\mu_n,\bar r_n]~s.t.~\begin{cases}\ga_n^{p_n-1}r\nm{\nabla w_n}_{L^\iy(B_r(0)\setminus B_{\mu_n}(0))}\le C_0,\\ r^{2+2\al}\exp\sbr{\f{\ep_0}{2}(1-\eta)^{1/3}\ga_n^{1/3}}\le\f{4\bar C}{\pi}, \end{cases} }, \ee
for all $n$, with $\ep_0,\bar C$ as in \eqref{4-57}, and  $C_0>\f{C_G}{1-\eta}+4$. We want to prove $r_n=\bar r_n$ for large $n$.

By $\mu_n=o(\ti r_n)$, \eqref{4-15-0} and \eqref{4-17--3}, we see that $r_n\ge R\mu_n$ for any given $R>1$, and hence $\mu_n=o(r_n)$. Then we may use Propositon \ref{single-1} to $v_n$ in $B_{r_n}(0)$ to get that $v_n'(s)=-\frac{2}{p_n}\gamma_n^{1-p_n}t_n'(s)(1+o(1))$ and so
\be\lab{4-19} \sup_{s\in[\mu_n,r_n]}\f{p_n}{2}\ga_n^{p_n-1}s\abs{v_n'(s)}\le 2(1+\al)+o(1) \ee
for large $n$. 
Using \eqref{4-17}, we get from the first condition in \eqref{4-18} that $|w_n|=O(\ga_n^{1-p_n})$ so that 
\be\lab{4-20-0} u_n=v_n+O(\ga_n^{1-p_n}) \ee 
uniformly in $B_{r_n}(0)$ for large $n$. Independently, we get from Proposition \ref{single-1} and \eqref{7-1-1-1} that
\be\lab{4-20} \ga_n\ge v_n\ge(1-\eta+o(1))\ga_n\quad\text{in}~[0,r_n]. \ee
Then, writing $\abs{\nabla w_n}\le\abs{\nabla u_n}+\abs{\nabla v_n}$, using first \eqref{4-11} and \eqref{4-19}, and then \eqref{4-20-0}-\eqref{4-20} together with $p_n\in[1,2]$, we get that
\be\lab{4-21} \nm{\nabla w_n}_{L^\iy(\pa B_{r_n}(0))}\le\f{1+o(1)}{\ga_n^{p_n-1}r_n}\sbr{\f{C_G}{(1-\eta)^{p_n-1}}+4}<\f{C_0}{\ga_n^{p_n-1}r_n} \ee
for large $n$. Independently, by using \eqref{4-20-0}-\eqref{4-20}, we get from \eqref{4-7} that 
\be\lab{4-22} r_n^{2+2\al}\exp\sbr{\f{\ep_0}{2}(1-\eta)^{1/3}\ga_n^{1/3}}\le\f{\bar C}{\pi}.\ee
Thus, by the continuity of $\nabla w_n$ with \eqref{4-21}, and by using \eqref{4-22}, we have that both inequalities in \eqref{4-18} is strict for $r=r_n$ and for large $n$, so that we conclude that $r_n=\bar r_n$ for large $n$. This gives \eqref{singular-0} and \eqref{singular-2}.

Now we turn to the proof of \eqref{singular-1}. Let $\delta_0\in\sbr{0,\min\{1,2(1+\al)\}}$ be fixed, and let $\hat r_n$ be given by
\be\lab{4-23} \hat r_n:=\sup\lbr{ r\in[0,\bar r_n]~s.t.~\ga_n^{p_n-1}r^{\delta_0}\mbr{w_n}_{C^{0,\delta_0}(B_r(0))}\le \hat C_0 }, \ee
where $\hat C_0\ge 4(C_0+2)$ and $\mbr{~\cdot~}_{C^{0,{\delta_0}}(B_r(0))}$ denotes the H\"older seminorm
	$$\mbr{u}_{C^{0,{\delta_0}}(B_r(0))}:=\sup_{x,y\in B_r(0), x\neq y}\f{|u(x)-u(y)|}{|x-y|^{{\delta_0}}}. $$ 
By \eqref{4-14} and \eqref{4-15}, we get that $\ga_n^{p_n-1}\mu_n^{\delta_0}\mbr{w_n}_{C^{0,\delta_0}(B_{R\mu_n}(0))}=o(1)$ for any given $R>1$ and large $n$, so that we get $\mu_n=o(\hat r_n)$.
We aim to show that for proper choice of $\delta_0$ and for large $n$, 
\be\lab{4-23-1} \hat r_n\ge\f12\bar r_n. \ee 

Assume by contradiction that $2\hat r_n<\bar r_n$. We consider  a decomposition $w_n=w_{n,1}+w_{n,2}+w_{n,3}$ of $w_n$ in $B_{2\hat r_n}(0)$, where $w_{n,i}$ are given by
\be\lab{4-24} -\Delta w_{n,1}=0~~\text{in}~B_{2\hat r_n}(0),\quad w_{n,1}=w_n~~\text{on} ~\pa B_{2\hat r_n}(0), \ee
\be\lab{4-25} -\Delta w_{n,2}=|\cdot|^{2\al}\sbr{h_n(0)v_n-h_nu_n}~~\text{in}~B_{2\hat r_n}(0),\quad w_{n,2}=0~~\text{on} ~\pa B_{2\hat r_n}(0), \ee
and 
\be\lab{4-26} -\Delta w_{n,3}=-\Delta(w_n-w_{n,2})~~\text{in}~B_{2\hat r_n}(0),\quad w_{n,3}=0~~\text{on} ~\pa B_{2\hat r_n}(0). \ee
To estimate $w_{n,1}$, we get first from \eqref{4-17} with \eqref{singular-2} that 
	$$\nm{w_{n,1}(\hat r_n\cdot)}_{L^\iy(\pa B_2(0))}=\nm{w_n}_{L^\iy(\pa B_{2\hat r_n}(0))}\le \f{C_0+1}{\ga_n^{p_n-1}}$$
for large $n$, and then, by using the gradient estimate for harmonic functions, and using the maximum principle, we get that
	$$\nm{\nabla w_{n,1}(\hat r_n\cdot)}_{L^\iy(B_1(0))}\le 2\nm{w_{n,1}(\hat r_n\cdot)}_{L^\iy(B_{2}(0))}\le \f{2(C_0+1)}{\ga_n^{p_n-1}},$$
so that we have
\be\lab{4-27} \ga_n^{p_n-1}\hat r_n\nm{\nabla w_{n,1}}_{L^\iy(B_{\hat r_n}(0))}\le2(C_0+1), \ee
for large $n$. By the definition of the H\"older seminorm, we clearly get from \eqref{4-27} that
\be\lab{4-28} \ga_n^{p_n-1}\hat r_n^{\delta_0}\mbr{w_{n,1}}_{C^{0,\delta_0}(B_{\hat r_n}(0))}\leq 2\ga_n^{p_n-1}\hat r_n\nm{\nabla w_{n,1}}_{L^\iy(B_{\hat r_n}(0))}\le4(C_0+1), \ee
for large $n$.
To estimate $w_{n,2}$, we know that both \eqref{4-20-0} and \eqref{4-20} hold on $B_{2\hat r_n}(0)$ since $r_n=\bar r_n>2\hat r_n$, so we see from \eqref{4-3-0} that $h_n(0)v_n-h_nu_n=O(\gamma_n)$ in $B_{2\hat r_n}(0)$. From here, \eqref{4-25} and \eqref{4-22}, we get that
$$\begin{aligned}
	\nm{\Delta w_{n,2}(\hat r_n\cdot)}_{L^s(B_2(0))}&=\hat r_n^{2-\f{2}{s}}\nm{\Delta w_{n,2}}_{L^s(B_{2\hat r_n}(0))}
=O\sbr{ \ga_n\hat r_n^{2-\f{2}{s}}\nm{|\cdot|^{2\al}}_{L^s(B_{2\hat r_n}(0))} }\\&=O\left(\gamma_n\hat r_n^{2+2\al}\right)=o\sbr{\f{1}{\ga_n^{p_n-1}}}
\end{aligned}$$
for any $s\in(1,1/|\al|)$, so by the elliptic theory, we get that 
\be\lab{4-36} \ga_n^{p_n-1}\hat r_n^{\delta_0}\mbr{w_{n,2}}_{C^{0,\delta_0}(B_{2\hat r_n}(0))}=\ga_n^{p_n-1}\mbr{w_{n,2}(\hat r_n\cdot)}_{C^{0,\delta_0}(B_2(0))}=o(1),\ee
for large $n$. 
For $w_{n,3}$, we claim that 
\be\lab{4-30} \ga_n^{p_n-1}\hat r_n^{\delta_0}\mbr{w_{n,3}}_{C^{0,\delta_0}(B_{2\hat r_n}(0))}=o(1), \ee
for large $n$. We assume by contradiction with \eqref{4-30} that there exists a $c_0>0$ such that
\be\lab{4-31} \tau_n:=\mbr{w_{n,3}}_{C^{0,\delta_0}(B_{2\hat r_n}(0))}\ge \f{c_0}{\ga_n^{p_n-1}\hat r_n^{\delta_0}}. \ee
Let 
\be\lab{4-32} \ti w_n:=\f{1}{\tau_n\mu_n^{\delta_0}}w_n(\mu_n\cdot) \quad\text{and}\quad  \ti w_{n,i}:=\f{1}{\tau_n\mu_n^{\delta_0}}w_{n,i}(\mu_n\cdot),~i=1,2,3. \ee
At first, by using \eqref{4-23} with $w_n(0)=0$ to estimate $w_n$ in $B_{\mu_n}(0)$, and using \eqref{singular-2} to estimate $w_n$ in $B_{2\hat r_n}(0)\setminus B_{\mu_n}(0)$, we get that
\be\lab{4-29} \abs{w_n(x)}=O\sbr{\f{|x|^{\delta_0}}{\ga_n^{p_n-1}\hat r_n^{\delta_0}}}\quad\text{uniformly in $B_{2\hat r_n}(0)$ for large $n$. }  \ee
From here and the assumption \eqref{4-31}, we have 
$$\ti w_n(x)=O(|x|^{\delta_0})\quad\text{ uniformly in $B_{2\hat r_n/\mu_n}(0)$ for large $n$.} $$
Then by repeating the argument involving \eqref{4-66-0}, and using Proposition \ref{single-1}, we get that
	$$-\Delta\ti w_n=8(1+\al)^2|\cdot|^{2\al}e^{-2T_\al}(\ti w_n+o(1))+o(1)$$
uniformly locally in $\R^2$ for large $n$, where $T_\al:=\ln(1+|\cdot|^{2(1+\al)})$. By the elliptic theory, and by using $\ti w_n(0)=0$ and Lemma \ref{appe3-1}, we get that
\be\lab{4-33} \lim_{n\to+\iy}\ti w_n=0\quad\text{in}~C^{0,\delta_0}_\loc(\R^2). \ee
Then, we have
\be\lab{4-34} \mbr{w_{n}}_{C^{0,\delta_0}(B_{R\mu_n}(0))}=\tau_n\mbr{\ti w_{n}}_{C^{0,\delta_0}(B_{R}(0))}=o(\tau_n),\ee
for any given $R\gg 1$ and for large $n$. Independently, we get from \eqref{4-27} that
\be\lab{4-35} \mbr{w_{n,1}}_{C^{0,\delta_0}(B_{R\mu_n}(0))}=O\sbr{\f{1}{\ga_n^{p_n-1}\hat r_n}\mu_n^{1-\delta_0}}=o(\tau_n), \ee
for any given $R\gg 1$ and for large $n$, where in the last estimate we use $\mu_n=o(\hat r_n)$ and the assumption \eqref{4-31}. Similarly, we see from \eqref{4-36} and \eqref{4-31} that
\be\label{4-36--2}\mbr{w_{n,2}}_{C^{0,\delta_0}(B_{2\hat r_n}(0))}=o\sbr{\f{1}{\ga_n^{p_n-1}\hat r_n^{\delta_0}}}=o(\tau_n), \ee
Plugging \eqref{4-34}-\eqref{4-36--2} into the decomposition $w_n=w_{n,1}+w_{n,2}+w_{n,3}$, we get that
\be\lab{4-37} \mbr{w_{n,3}}_{C^{0,\delta_0}(B_{R\mu_n}(0))}=o(\tau_n),\ee
for large $n$ and for any $R\gg1$. By using the fact that
	$$\begin{aligned} \tau_n&=\mbr{w_{n,3}}_{C^{0,\delta_0}(B_{2\hat r_n}(0))}\\
	&\le\max\lbr{\mbr{w_{n,3}}_{C^{0,\delta_0}(B_{R\mu_n}(0))},\mbr{w_{n,3}}_{C^{0,\delta_0}(B_{2\hat r_n}(0)\setminus B_{R\mu_n}(0))} }\sup_{t\in(0,+\iy)}\f{1+t^{\delta_0}}{(1+t)^{\delta_0}},\end{aligned}$$
for large $n$ and for any $R\gg1$, we get from \eqref{4-37} that there exist two point sequences $(x_n)_n$ and $(y_n)_n$ such that $x_n,y_n\in B_{2\hat r_n}(0)$, $\mu_n=o(|x_n|)$, $\mu_n=o(|y_n|)$ and 
\be\lab{4-38} \tau_n=O\sbr{\f{\abs{w_{n,3}(x_n)-w_{n,3}(y_n)}}{|x_n-y_n|^{\delta_0}}} \ee
for large $n$. 
Let $G_n$ be the Green function of $-\Delta$ in $B_{2\hat r_n}(0)$ with the Dirichlet boundary condition (for an explicit formula for $G_n$, see for instance Han-Lin \cite[Proposition 1.22]{book-1}). 
By repeating the argument involving \eqref{4-66-0}, and using \eqref{4-26}, \eqref{4-29}, \eqref{4-31} and Proposition \ref{single-1}, we get that
	$$\abs{\Delta w_{n,3}}=O\sbr{\f{\tau_n}{\mu_n^{2+2\al}}|\cdot|^{\delta_0+2\al}e^{(-2+\ti\eta)t_n}},$$
uniformly in $B_{2\hat r_n}(0)$ for large $n$, for some $\ti\eta\in(0,1)$. 
Then, by the Green representation formula, we have that
\be\lab{4-39}\begin{aligned} 
	&\f{\abs{w_{n,3}(x_n)-w_{n,3}(y_n)}}{|x_n-y_n|^{\delta_0}}\\
	&=O\sbr{ \int_{B_{2\hat r_n}(0)}\f{\abs{G_{n}(x_n,z)-G_{n}(y_n,z)}}{|x_n-y_n|^{\delta_0}}\f{\tau_n}{\mu_n^{2+2\al}}|z|^{\delta_0+2\al}e^{(-2+\ti\eta)t_n(z)}\rd z } \\
	&= O\sbr{ \tau_n\int_{B_{\rho_n}(0)}\f{\abs{\wt G_{n}(\ti x_n,z)-\wt G_{n}(\ti y_n,z)}}{|\ti x_n-\ti y_n|^{\delta_0}}\f{|z|^{\delta_0+2\al}}{(1+|z|^{2(1+\al)})^{(2-\ti\eta)}}\rd z },
\end{aligned}\ee
by using the change of variable $z\mapsto\mu_nz$, where $\rho_n=\f{2\hat r_n}{\mu_n}$, $\ti x_n=\f{x_n}{\mu_n}$, $\ti y_n=\f{y_n}{\mu_n}$, and $\wt G_n$ is the Green function of $-\Delta$ in $B_{\rho_n}(0)$ with the Dirichlet boundary condition
	$$\wt G_n(x,z)=\f{1}{2\pi}\ln\f{1}{|x-z|}+\f{1}{2\pi}\ln\sbr{\f{|x|}{\rho_n}\abs{z-\f{\rho_n^2}{|x|^2}x}}.$$
We now estimate the last term in \eqref{4-39}. Clearly, we have that $\ti x_n,\ti y_n\in B_{\rho_n}(0)$ and $$\lim_{n\to+\iy}|\ti x_n|=+\iy,\quad\lim_{n\to+\iy}|\ti y_n|=+\iy.$$ On the one hand, if we have $|\ti x_n|\sim |\ti y_n|$, then using the fact that (see for instance \cite[Appendix B]{MT-blowup-6})
	$$\abs{\nabla \wt G_n(x,z)}=O\sbr{\f{1}{|x-z|}},$$
for any $x,z\in B_{\rho_n}(0)$, and using the mean value theorem, we get that 
\be\lab{4-40}\begin{aligned}
&\int_{B_{\rho_n}(0)}\f{\abs{\wt G_{n}(\ti x_n,z)-\wt G_{n}(\ti y_n,z)}}{|\ti x_n-\ti y_n|^{\delta_0}}\f{|z|^{\delta_0+2\al}}{(1+|z|^{2(1+\al)})^{(2-\ti\eta)}}\rd z\\
&=O\sbr{ |\ti x_n-\ti y_n|^{1-\delta_0}\int_{B_{\rho_n}(0)}\f{1}{|z-\xi_n|}\f{|z|^{\delta_0+2\al}}{(1+|z|^{2(1+\al)})^{(2-\ti\eta)}}\rd z }\;(\text{for some}~\xi_n\sim |\ti x_n|)\\
&=O\sbr{|\xi_n|^{-(1+\al)(1-\ti\eta)}}=o(1),
\end{aligned}\ee
for large $n$, where the second equality is obtained by choosing $\delta_0\in(0,(1+\al)(1-\ti\eta))$ and computing the integral  in $B_{|\xi_n|/2}(\xi_n)$, $B_{2|\xi_n|}(\xi_n)\setminus B_{|\xi_n|/2}(\xi_n)$ and $B_{\rho_n}(0)\setminus B_{2|\xi_n|}(\xi_n)$ separately. On the other hand, if we have $|\ti x_n|=o(|\ti y_n|)$ (or $|\ti y_n|=o(|\ti x_n|)$), then by using the fact that 
$$\begin{aligned}
	&\abs{\wt G_n(\ti x_n,\cdot)-\wt G_n(\ti y_n,\cdot)}\\
	&=\begin{cases}O\sbr{\abs{\ln|\cdot-\ti x_n|}+\abs{\ln|\cdot-\ti y_n|}}+o(1),\quad&\text{if}~|y_n|=o(\rho_n),\\ 
	O\sbr{\abs{\ln|\cdot-\ti x_n|}+\abs{\ln|\cdot-\ti y_n|}+\abs{\ln|\cdot-\f{\rho_n^2}{|\ti y_n|^2}\ti y_n|}}+O(\ln\rho_n),\quad&\text{if}~|y_n|\sim\rho_n, \end{cases}
\end{aligned}$$
uniformly in $B_{\rho_n}(0)$, and by choosing $\delta_0\in(0,(1+\al)(1-\ti\eta))$, we get that
\be\lab{4-41} \int_{B_{\rho_n}(0)}\f{\abs{\wt G_{n}(\ti x_n,z)-\wt G_{n}(\ti y_n,z)}}{|\ti x_n-\ti y_n|^{\delta_0}}\f{|z|^{\delta_0+2\al}}{(1+|z|^{2(1+\al)})^{(2-\ti\eta)}}\rd z=O\sbr{\f{\ln|\ti y_n|}{|\ti y_n|^{\delta_0}}}=o(1), \ee
for large $n$. Plugging \eqref{4-40} and \eqref{4-41} in \eqref{4-39}, we get in any case that
	$$\f{\abs{w_{n,3}(x_n)-w_{n,3}(y_n)}}{|x_n-y_n|^{\delta_0}}=o(\tau_n),$$
which leads to a contradiction with \eqref{4-38}. This proves \eqref{4-30}. 

Therefore, plugging \eqref{4-28}, \eqref{4-36} and \eqref{4-30} in the decomposition $w_n=w_{n,1}+w_{n,2}+w_{n,3}$, we get that
	$$\ga_n^{p_n-1}\hat r_n^{\delta_0}\mbr{w_{n}}_{C^{0,\delta_0}(B_{\hat r_n}(0))}\le4(C_0+1)+o(1)<\hat C_0,$$
for large $n$, using our assumption on $\hat C_0$. This means the inequality in \eqref{4-23} for $r=\hat r_n$ is strict, and then conclude the proof of \eqref{4-23-1}. 

Finally, by using $w_n(0)=0$ and \eqref{4-23}-\eqref{4-23-1} to estimate $w_n$ in $B_{\mu_n}(0)$, and using \eqref{singular-2} to estimate $w_n$ in $B_{\bar r_n}(0)\setminus B_{\mu_n}(0)$, we obtain the desired estimate \eqref{singular-1}. This completes the proof.
\end{proof}

\vs
\section{Compactness for Type I}\lab{sec5}
In this section, we prove the compactness result at a singularity for {\bf Type I} defined in Section \ref{sec2-3}; see Theorem \ref{type1}. Precisely, recalling Proposition \ref{bubble}, we let $q_0\in\Sigma_\iy\cap\DR$ be fixed, and we assume that {\bf Type I} happens at $q_0$ and so $q_0\notin \DR'$, that is, there exists a positive integer $1\le N'\le N$ such that
	$$\lim_{n\to+\iy}q_{n,i}=q_0,~\forall~i\in\{1,\cdots,N'\},$$ 
and
	$$\lim_{n\to+\iy}q_{n,i}\neq q_0,~\forall~i\in\{N'+1,\cdots,N\},$$
where $q_{n,i}$ and $N$ are given by Proposition \ref{bubble}, and such that 
\be\lab{5-0} \text{the alternative in conclusion (2) of Proposition \ref{bubble} doesn't hold for $q_0$.}\ee
The purpose of this section is to prove the following result.
\begin{theorem}\lab{type1}
Let $q_0\in\Sigma_\iy\cap\DR$ and $\kappa_0>0$ such that $B_{2\kappa_0}(q_0)\cap(\DR\cup\Sigma_\iy)=\{q_0\}$. Assume that {\bf Type I} happens at $q_0$, i.e., $q_0\notin \DR'$. Then $\displaystyle\lim_{n\to+\iy}\la_n=0$ and 
\be\lab{conclusion-1} \lim_{n\to+\iy}\beta_n(B_{\kappa_0}(q_0))=4\pi N', \ee
where $\beta_n(\cdot)$ is given by \eqref{betadomain} and $N'$ is given as above.

Moreover, if $p_n\equiv p\in(1,2]$ for all $n$, then
\be\lab{conclusion-2} \beta_n(B_{\kappa_0}(q_0))>4\pi N' \ee
for large $n$.
\end{theorem}

The proof of Theorem \ref{type1} is quite long, and let us give the preliminary setting.
Let $q_{n,0}\equiv q_0$ for all $n$, and let the isothermal coordinates $(B_{\kappa_0}(q_{n,i}),\phi_{n,i},U_{n,i})$ for $i\in\{0,1,\cdots,N'\}$ be given in the paragraph before Proposition \ref{bubble}. 
Thanks to \eqref{equ-5}, \eqref{ufh} and  $\Delta_{g_0}=e^{-2\vp_{n,i}}\Delta$, we get 
\be\lab{equ-6} -\Delta u_{n,i}+h_{n,i}u_{n,i}=\la_np_nh_{n,i}f_{n,i}u_{n,i}^{p_n-1}e^{u_{n,i}^{p_n}},\quad u_{n,i}>0, \quad\text{in}~B_{2\kappa}(0),\ee
for all $i\in\{0,1,\cdots,N'\}$ and $n$, where $\Delta$ is the Laplace operator in $\R^2$. By conclusion (1) of Proposition \ref{bubble}, we have $$\nabla u_{n,i}(0)=0,\quad\forall i\in\{1,\cdots,N'\}.$$ Denote
\be\lab{xn} x_{n,i}:=\phi_{n,i}(q_0),~~\forall~i\in\{0,1,\cdots,N'\}. \ee
Then $x_{n,0}=0$. 
Recalling the assumption \eqref{hhh} on $h$, we get that
\be\lab{hi} h_{n,i}=\ti h_{n,i}~|\cdot-x_{n,i}|^{2\al}\quad\text{in}~B_{2\kappa}(0),\ee
where $\ti h_{n,i}$ is positive $C^1$ function for all $i\in\{0,\cdots,N'\}$ and $\al=\alpha_{q_0}$.

For all $i\in\{1,\cdots,N'\}$, we set
\be\lab{5-1} r_{n,i}:=\f{1}{3}d_{g_0}\big(q_{n,i},(\QR_n\cup\{q_0\})\setminus\{q_{n,i}\}\big), \ee
for all $n$, where $\QR_n$ is given by
	$$\QR_n:=\{q_{n,1},\cdots,q_{n,N'}\}.$$  
By \eqref{vpn}, we get that 
$$r_{n,i}\leq \frac13d_{g_0}(q_{n,i},q_0)=\frac13(1+o(1))|\phi_{n,i}(q_{n,i})-\phi_{n,i}(q_0)|=\frac13(1+o(1))|x_{n,i}|$$
and so\be\lab{5-2} r_{n,i}\le\f{1}{2}|x_{n,i}|,\quad\forall~i\in\{1,\cdots,N'\}, \ee
for large $n$. Let $\ga_{n,i}:=u_n(q_{n,i})$ and $\mu_{n,i}$ be given by \eqref{mun-1}, i.e.,
$$\mu_{n,i}^2\lambda_np_n^2|x_{n,i}|^{2\alpha}\ti h_{n,i}(0)f_{n,i}(0)\gamma_{n,i}^{2(p_n-1)}e^{\gamma_{n,i}^{p_n}}=8.$$
 Then by \eqref{dist-1}, we get that $\ga_{n,i}\to+\infty$, $\mu_{n,i}\to 0$, 
\be\lab{5-3} \lim_{n\to+\iy}\f{\mu_{n,i}}{r_{n,i}}=0,\quad\lim_{n\to+\iy}\f{\mu_{n,i}}{|x_{n,i}|}=0,\quad\forall~i\in\{1,\cdots,N'\}, \ee
and by \eqref{conv-1}, we get that
\be\lab{5-4} \lim_{n\to+\iy}\f{p_n}{2}\ga_{n,i}^{p_n-1}(\ga_{n,i}-u_{n,i}(\mu_{n,i}\cdot))=T_0:=\ln(1+|\cdot|^2)~\text{in}~\CR_{\loc}^1(\R^2),\ee
for every $i\in\{1,\cdots,N'\}$. By conclusion (3) of Proposition \ref{bubble}, we get that there exist $C_1,C_2>0$ such that
\be\lab{est-3} d_{g_0}(~\cdot~,\QR_n\cup\{q_0\})^2\la_nhfu_n^{2(p_n-1)}e^{u_n^{p_n}}\le C_1 \quad\text{in}~B_{\kappa_0}(q_0), \ee
\be\lab{est-4} d_{g_0}(~\cdot~,\QR_n\cup\{q_0\})u_n^{p_n-1}|\nabla_{g_0} u_n|\le C_2 \quad\text{in}~B_{\kappa_0}(q_0), \ee
for all $n$. 
We set $$t_{n,i}:=\ln\sbr{1+|\cdot|^2/\mu_{n,i}^2}\quad\text{ in }\; \R^2,$$
and
\be\lab{5-5} v_{n,i}:=\BR_{n}, \ee
where $\BR_n$ is as in Appendix \ref{appe1-2} for $\ga_{n}=\ga_{n,i}$, $\nu_n=h_{n,i}(0)$, $\tau_n=h_{n,i}(0)f_{n,i}(0)$, $\mu_n=\mu_{n,i}$ and $x_{n}=x_{n,i}$, for all $i\in\{1,\cdots,N'\}$ and all $n$.

To use the estimates of Section \ref{appe2-2} and Appendix \ref{appe1-2}, we need some preliminary observations. Let $l\in\{1,\cdots,N'\}$ be given. Given a parameter $\eta\in(0,1)$ that is going to take several values in the proofs below, we let $r_{n,l}^{(\eta)}$ be given by
\be\lab{5-6} t_{n,l}(r_{n,l}^{(\eta)})=\ln\sbr{1+\frac{|r_{n,l}^{(\eta)}|^2}{\mu_{n,i}^2}}=\eta\f{p_n\ga_{n,l}^{p_n}}{2},\ee
and, for $r_{n,l}$ as in \eqref{5-1}, we set
\be\lab{5-7} \bar r_{n,l}^{(\eta)}:=\min\lbr{r_{n,l},r_{n,l}^{(\eta)}}\leq \frac12|x_{n,l}|. \ee 
By \eqref{5-3} and \eqref{5-6}, we obtain
\be\label{5-6--3}\lim_{n\to+\infty}\frac{\mu_{n,i}}{\bar r_{n,l}^{(\eta)}}=0,\quad  t_{n,l}(\bar r_{n,l}^{(\eta)})\leq\eta\f{p_n\ga_{n,l}^{p_n}}{2}.\ee
By collecting the above preliminary information, we can check that Proposition \ref{regular} applies with $\bar r_n=\bar r_{n,l}^{(\eta)}$, $f_n=f_{n,l}$, $h_n=\ti h_{n,l}$, $x_n=x_{n,l}$, $u_{n}=u_{n,l}$, $\ga_n=\ga_{n,l}$ and $v_n=v_{n,l}$. As a result, we get from Proposition \ref{regular} that
\be\lab{5-10} \abs{u_{n,l}-v_{n,l}}=O\sbr{\f{|\cdot|}{\ga_{n,l}^{p_n-1}\bar r_{n,l}^{(\eta)} }},\ee 
\be\lab{5-11} \abs{\nabla(u_{n,l}-v_{n,l})}=O\sbr{\f{1}{\ga_{n,l}^{p_n-1}\bar r_{n,l}^{(\eta)} }},\ee 
uniformly in $B_{\bar r_{n,l}^{(\eta)}}(0)$ for $n$ large, and that 
\be\lab{5-8} \ln\ga_{n,l}=o\sbr{\ln\f{1}{\bar r_{n,l}^{(\eta)}|x_{n,l}|^\al}},\ee 
which gives that $\ga_{n,l}^s(\bar r_{n,l}^{(\eta)})^2|x_{n,l}|^{2\al}=O(1)$ for any $s>1$, so that Proposition \ref{single-2} also applies, and we get
\be\lab{5-9}\begin{aligned} 
	\ga_{n,l}\ge v_{n,l}&=\ga_{n,l}\sbr{1-\f{2t_{n,l}(1+O(\ga_{n,l}^{-p_n}))}{p_n\ga_{n,l}^{p_n}}}\\
		&\ge(1-\eta)\ga_{n,l}+O(\ga_{n,l}^{1-p_n}), 
\end{aligned} \ee
uniformly in $[0,\bar r_{n,l}^{(\eta)}]$ for $n$ large, using \eqref{5-6--3}.

\vs
\subsection{A first estimate}\lab{sec5-1}\ 

At first, inspired by \cite[Section 4]{MT-blowup-7},  we prove the following result:
\begin{lemma}\lab{lemma5-1}
For all $i\in\{1,\cdots,N'\}$, we have that
\be\lab{5-12} \liminf_{n\to+\iy}\f{2t_{n,i}(r_{n,i})}{p_n\ga_{n,i}^{p_n}}\ge1. \ee 
\end{lemma}

\begin{remark}\lab{614-1}
The basic idea of proving \eqref{5-12} comes from \cite[Step 4.1]{MT-blowup-7}, but the singular potential $h_{n,i}$ bring new difficulties. To prove \eqref{5-12}, we assume that it does not hold for some $i$, and then lead to a contradiction. If $r_{n,i}=o(|x_{n,i}|)$, the proof is very similar to that of \cite[Step 4.1]{MT-blowup-7}, since the bubble is so far from the singularity $q_0$ (recall \eqref{xn} that $x_{n,i}=\phi_{n,i}(q_0)$); otherwise, we have $r_{n,i}\sim|x_{n,i}|$ by \eqref{5-2}. In the latter case, there seems a ``hole'' near $x_{n,i}$, since all the estimate are given outside the singularity. Here, by using the assumption \eqref{5-0} with the pointwise estimate \eqref{est-3}, we may complete the ``hole'', see \eqref{5-45} for details. 

To get a contradiction, on the one hand, we get by careful analysis that there is a vector $\vec{X}$, starting from some vertex point of a convex hull containing $0$ and $x_{n,i}$, directing to the inner domain of this convex hull; on the other hand, we get from Proposition \ref{regular}-\eqref{regular-4} that the same vector $\vec X$ directs to an opposite direction due to $\al<0$ although $\nabla\psi_0(0)\neq0$, which leads to the desired contradiction. This is the first place where our method highly and nontrivially relies on $\al<0$. 
\end{remark}

\begin{proof}[\bf Proof of Lemma \ref{lemma5-1}]
Up to renumbering, we may assume that 
	$$r_{n,1}\le r_{n,2}\le\cdots\le r_{n,N'}.$$
We apply the induction method on $i\in\{1,\cdots,N'\}$. In particular, we assume that \eqref{5-12} holds true at steps $1,\cdots,i-1$ if $i\ge2$. Note that for $i=1$ we assume nothing. By contradiction, assume in addition that \eqref{5-12} does not hold at step $i$. Thus, by \eqref{5-6} and \eqref{5-7}, up to a subsequence, we may choose and fix $\eta\in(0,1)$ sufficiently close to $1$ such that $t_{n,i}(r_{n,l}^{(\eta)})>t_{n,i}(r_{n,i})$  and so
\be\lab{5-13} \bar r_{n,i}^{(\eta)}=r_{n,i} \quad\text{for large $n$}.\ee
Set $$J_i=\lbr{j\in\{1,\cdots,N'\}:~d_{g_0}(q_{n,j},q_{n,i})=O(r_{n,i})}.$$ Obviously, we get from \eqref{5-1} that 
\be\lab{5-14} r_{n,l}=O(r_{n,i})\quad\forall~l\in J_i. \ee
Since $r_{n,i}=o(1)$ by \eqref{5-8}, we get from \eqref{vpn} that (see \cite[(4.42)]{MT-blowup-7})
\be\lab{5-15} \lim_{n\to+\iy}\mbr{(\phi_{n,l})_*g_0}(r_{n,i}\cdot)=|\rd x|^2~\text{in}~C_\loc^2(\R^2), \ee
and from \eqref{5-14} that
\be\lab{5-15-1} |x_{n,l}|=\begin{cases}(1+o(1))|x_{n,i}|,&\quad\text{if}~r_{n,i}=o(|x_{n,i}|),\\
	O(|x_{n,i}|),&\quad\text{otherwise},\end{cases}\ee
for all $l\in J_i$.
We may assume that
	$$\lim_{n\to+\iy}\f{\phi_{n,i}(q_{n,l})}{r_{n,i}}=\hat x_l\in\R^2$$
for all $l\in J_i$ and we also assume
	$$\lim_{n\to+\iy}\f{\phi_{n,i}(q_{0})}{r_{n,i}}=\lim_{n\to+\iy}\f{x_{n,i}}{r_{n,i}}=\hat x_0\in\R^2,\quad\text{if}~|x_{n,i}|=O(r_{n,i}).$$
By \eqref{5-2}, we get $\abs{\hat x_0}\ge 2$. Let $\SR_i=\lbr{\hat x_l:~l\in J_i}$. We see that if $r_{n,i}=o(|x_{n,i}|)$, i.e., $\hat x_0$ does not exist, then by \eqref{5-1}, we have that $\SR_i$ contains at least two distinct points. Then, we set
\be\lab{5-15-12} \wt\SR_i:=\begin{cases}\SR_i\cup\{\hat x_0\}&\text{if }\;|x_{n,i}|=O(r_{n,i}),\\
\SR_i&\text{if }\;r_{n,i}=o(|x_{n,i}|),\end{cases}\ee
and we may choose and fix $\tau\in(0,1)$ small enough such that
	$$3\tau<\displaystyle\min\lbr{|x-y|:~x,y\in\wt\SR_i,x\neq y},\quad\text{and }\;\wt\SR_i\subset B_{\f{1}{3\tau}}(0).$$
We can check  that there exists $C>0$ such that any point in
	$$\Omega_{n,i}:=\sbr{B_{\f{r_{n,i}}{\tau}}(0)\setminus B_{\tau r_{n,i}}(x_{n,i})} \setminus \bigcup_{j\in J_i}B_{\tau r_{n,i}}(\phi_{n,i}(q_{n,l}))$$
may be joined to $\pa B_{\tau r_{n,i}}(0)$ by  a $C^1$ path in $\Omega_{n,i}$ of length at most $Cr_{n,i}$, for all $n$. Therefore, by \eqref{5-13} with \eqref{5-10} and \eqref{5-9} for $l=i$, we may first estimate $u_{n,i}$ on $\pa B_{\tau r_{n,i}}(0)$, and then get from \eqref{est-4} and \eqref{5-15} that
\be\lab{5-16} u_{n,i}=\bar u_{n,i}(\tau r_{n,i})+O(\ga_{n,i}^{1-p_n})\ge(1-\eta)\ga_{n,i}+O(1), \ee
uniformly in $\Omega_{n,i}$ and for all large $n$, where $\bar u_{n,i}$ is the average function given by
\be\lab{average} \bar u_{n,i}(r):=\f{1}{2\pi r}\int_{\pa B_r(0)}u_{n,i}\rd\sigma, \ee
for all $r\in(0,2\kappa_0)$. 
Independently, we get from \eqref{16}, \eqref{5-10} and \eqref{5-13} that
\be\lab{5-17} \bar u_{n,i}(\tau r_{n,i})=-\sbr{\f{2}{p_n}-1}\ga_{n,i}+\f{2}{p_n\ga_{n,i}^{p_n-1}}\sbr{\ln\f{1}{\la_n\ga_{n,i}^{2(p_n-1)}r_{n,i}^2|x_{n,i}|^{2\al}}+O(1)}. \ee

\begin{itemize}[fullwidth,itemindent=0em]
\vskip0.1in
\item We prove that
\be\lab{5-19} \nm{w_{n,l}}_{L^\iy(B_{r_{n,i}/(2\tau)}(0))}=o(\ga_{n,i}^{1-p_n}). \ee
for all $l\in J_i$, where $w_{n,l}$, associated to the linear term of \eqref{equ-6}, is given by
\be\lab{5-18}\begin{cases}
	-\Delta w_{n,l}=-h_{n,l}u_{n,l}\quad\text{in}~B_{\f{r_{n,i}}{2\tau}}(0),\\
	w_{n,l}=0\quad\text{on}~\pa B_{\f{r_{n,i}}{2\tau}}(0).\\
\end{cases}\ee
Let $s\in(1,1/|\al|)$, we get that
	$$\nm{\Delta w_{n,l}(r_{n,i}\cdot)}_{L^s(B_{1/(2\tau)}(0))}=r_{n,i}^{2(1-1/s)}\sbr{\int_{B_{r_{n,i}/(2\tau)}(0)}h_{n,l}^su_{n,l}^s\rd x}^{1/s}.$$
If $r_{n,i}=o(|x_{n,i}|)$, by \eqref{5-15-1}, \eqref{hi}, \eqref{5-8} and Lemma \ref{bound-0}, we get that
	$$\nm{\Delta w_{n,l}(r_{n,i}\cdot)}_{L^s(B_{1/(2\tau)}(0))}=O\sbr{ \sbr{r_{n,i}^2|x_{n,i}|^{2\al}}^{1-\f{1}{s}} }=o(\ga_{n,i}^{1-p_n}),$$
so, by the elliptic theory, we get \eqref{5-19}. 
If $|x_{n,i}|=O(r_{n,i})$, we may take a $r>1$ such that $2+2\al\f{sr-1}{r-1}>0$, then by H\"older inequality and \eqref{hi}, and by Lemma \ref{bound-0} we get
$$\begin{aligned}
	&\sbr{\int_{B_{r_{n,i}/(2\tau)}(0)}h_{n,l}^su_{n,l}^s\rd x}^{1/s}\\
	&=O\sbr{ \sbr{\int_{B_{r_{n,i}/(2\tau)}(0)}\abs{x-x_{n,l}}^{2\al\f{sr-1}{r-1}}\rd x }^{\sbr{1-\f{1}{r}}\f{1}{s}} 
		\sbr{\int_{B_{r_{n,i}/(2\tau)}(0)}h_{n,l}u_{n,l}^{sr}\rd x }^{\f{1}{sr}}  }\\
	&=O\sbr{\sbr{r_{n,i}+|x_{n,i}|}^{-2+\f{2}{s}+(2+2\al)(1-\f{1}{sr})}},
\end{aligned}$$
where in the last equality we use \eqref{5-15-1} and $-2+\f{2}{s}+(2+2\al)(1-\f{1}{sr})>0$. This together with \eqref{5-8} gives that
	$$\nm{\Delta w_{n,l}(r_{n,i}\cdot)}_{L^s(B_{1/(2\tau)}(0))}=O\sbr{(r_{n,i}+|x_{n,i}|)^{(2+2\al)(1-\f{1}{sr})}}=o(\ga_{n,i}^{1-p_n}),$$
so, by the elliptic theory, we finish the proof of \eqref{5-19}.

\vskip0.1in
\item We prove that
\be\lab{5-20} \lim_{n\to+\iy}\f{\ga_{n,i}}{\ga_{n,j}}=0\quad\text{for any}~j\in J_i~\text{and}~j<i. \ee
Let $j\in J_i$ such that $j<i$. Then, by the induction assumption, we know from \eqref{5-12} at $j$ and from \eqref{5-6} and \eqref{5-7} that, given $\eta_2\in(\eta,1)$,
\be\lab{5-21} \bar r_{n,j}^{(\eta_2)}=r_{n,j}^{(\eta_2)}, \ee
for large $n$. By \eqref{5-9}, by \eqref{5-10} for $l=j$ and $\eta=\eta_2$, and by the definition \eqref{5-6}  of $r_{n,j}^{(\eta_2)}$,  we have that
\be\lab{5-24} \bar u_{n,j}(r_{n,j}^{(\eta_2)})\le(1-\eta_2)\ga_{n,j}(1+o(1)).\ee
Let $w_{n,j}$ be given by \eqref{5-18} with $l=j$. By observing that $-\Delta(u_{n,j}-w_{n,j})\ge0$, the maximum principle yields that
\be\lab{5-22} u_{n,j}\ge\inf_{\pa B_{r_{n,i}/(2\tau)}(0)}u_{n,j}+o(\ga_{n,i}^{1-p_n})\ge(1-\eta)\ga_{n,i}+O(1), \ee
uniformly in $B_{r_{n,i}/(2\tau)}(0)$, where in the last inequality we use \eqref{5-16} and the fact that
\be\lab{5-23} \phi_{n,i}\circ\phi_{n,l}^{-1}\sbr{\pa B_{r_{n,i}/(2\tau)}(0)}\subset \Omega_{n,i}, \ee 
by \eqref{5-15} and choosing $\tau>0$ small enough from the begining. Then, by $r_{n,j}\le r_{n,i}$ due to $j<i$, we get from \eqref{5-24} and \eqref{5-22} that
	$$(1-\eta)\ga_{n,i}\le (1-\eta_2)\ga_{n,j}(1+o(1)),$$
for any $\eta_2\in(\eta,1)$, so that letting $\eta_2\to1$, we get \eqref{5-20}.

\vskip0.1in
\item We prove that
\be\lab{5-25} \ga_{n,j}=O(\ga_{n,i})\quad\text{for any}~j\in J_i. \ee
By contradiction, if \eqref{5-25} does not hold, we choose $j\in J_i$ such that
\be\lab{5-26} \lim_{n\to+\iy}\f{\ga_{n,i}}{\ga_{n,j}}=0. \ee 
In particular, we have $j\neq i$. If $j>i$, we may write that
\be\lab{5-27}\begin{aligned}
	t_{n,j}(r_{n,j})&=\ln\f{r_{n,j}^2}{\mu_{n,j}^2}+o(1)\\
		&=\ln\f{r_{n,j}^2}{r_{n,i}^2}+t_{n,i}(r_{n,i})+\ln\f{\mu_{n,i}^2}{\mu_{n,j}^2}+o(1)\\
		&=O(1)+\eta\f{p_n}{2}\ga_{n,i}^{p_n}+\ga_{n,j}^{p_n}-\ga_{n,i}^{p_n}+O(\ln\ga_{n,j}+\ln\ga_{n,i})+2\al\ln\f{|x_{n,j}|}{|x_{n,i}|}\\
		&\ge \f{p_n}{2}\ga_{n,j}(1+o(1)).
\end{aligned}\ee
The first two equalities use \eqref{5-3}; the third one uses first \eqref{5-14} and the fact $r_{n,j}\ge r_{n,i}$ due to the assumption $j>i$, then the definition \eqref{5-6} for $\eta$, and at last the definition \eqref{mun-1} of $\mu_{n,i}$ with \eqref{hi}; the last equality uses \eqref{5-15-1}, \eqref{5-26} and $\al<0$. 
Thus, given any $\eta_2\in(0,1)$, we get that \eqref{5-21} holds also if $j>i$. As a first consequence, for all given $\eta_2'\in(0,\eta_2]$, we get that
\be\lab{5-28} \lim_{n\to+\iy}\f{r_{n,j}^{(\eta_2')}}{r_{n,i}}=0, \ee 
using \eqref{5-14} and the fact by \eqref{5-6} that
\be\lab{5-29} \ln\f{r_{n,j}^{(\eta_2')}}{r_{n,j}^{(\eta_2)}}=-\f{p_n\ga_{n,j}^{p_n}}{4}(\eta_2-\eta_2')+o(1). \ee
We get from \eqref{16}, \eqref{5-10} and \eqref{5-21} that
\be\lab{5-30} \bar u_{n,j}(r_{n,j}^{(\eta_2')})=-\sbr{\f{2}{p_n}-1}\ga_{n,j}+\f{2}{p_n\ga_{n,j}^{p_n-1}}\sbr{\ln\f{1}{\la_n\ga_{n,j}^{2(p_n-1)}(r_{n,j}^{(\eta_2')})^2|x_{n,j}|^{2\al}}+O(1)}. \ee

To get the desired contradiction with \eqref{5-26}, fixing $\eta_2\in(0,1)$, we prove now the following estimate
\be\lab{5-31} \bar u_{n,j}(r_{n,j}^{(\eta_2)})\ge \bar u_{n,i}(\tau r_{n,i})+\f{4}{p_n\ga_{n,j}^{p_n-1}}\ln\f{r_{n,i}}{r_{n,j}^{(\eta_2)}}+O(\ga_{n,i}^{1-p_n}). \ee
Let $\psi_n$ be given by 
\be\lab{5-32}\begin{cases}
	-\Delta \psi_n=0\quad\text{in}~B_{\f{r_{n,i}}{2\tau}}(0),\\
	\psi_n=u_{n,j}\quad\text{on}~\pa B_{\f{r_{n,i}}{2\tau}}(0).\\
\end{cases}\ee
By \eqref{5-16} and \eqref{5-23}, we get first that $u_{n,i}=\bar u_{n,i}(\tau r_{n,i})+O(\ga_{n,i}^{1-p_n})$ uniformly on $\pa B_{r_{n,i}/(2\tau)}(0)$, and then, by the harmonicity of $\psi_n$, we get that
\be\lab{5-33} \psi_n=\bar u_{n,i}(\tau r_{n,i})+O(\ga_{n,i}^{1-p_n})\ee 
uniformly in $B_{r_{n,i}/(2\tau)}(0)$ for large $n$. 
Let $(z_n)_n$ be any sequence of points in $\pa B_{r_{n,j}^{(\eta_2)}}(0)$. 
Let $G_n$ be the Green function of $-\Delta$ in $B_{r_{n,j}^{(\eta_2)}/(2\tau)}(0)$ with the Dirichlet boundary condition. We know that $G_n>0$. Let $\eta_1\in(0,\eta_2)$ be fixed. Let $w_{n,j}$ be given by \eqref{5-18}. 
By applying the Green representation formula to $u_{n,j}-w_{n,j}-\psi_n$ with \eqref{equ-6}, and by using \eqref{5-19} and \eqref{5-33}, we get that
\be\lab{5-34} u_{n,j}(z_n)\ge\bar u_{n,i}(\tau r_{n,i})+O(\ga_{n,i}^{1-p_n})+\la_np_n\int_{B_{r_{n,j}^{(\eta_1)}}(0)}G_n(z_n,y)h_{n,j}f_{n,j}u_{n,j}^{p_n-1}e^{u_{n,j}^{p_n}}\rd y.\ee
By \cite[Appendix B]{MT-blowup-6}, there exists $C>0$ such that
	$$\abs{G_n(z,y)-\f{1}{2\pi}\ln\f{r_{n,i}}{|z-y|}}\le C,$$
for any $y\in B_{r_{n,i}/(2\tau)}(0)$, $z\in B_{5r_{n,i}/(12\tau)}(0)$, $y\neq z$ and all $n$. Then, since $|z_n|=r_{n,j}^{(\eta_2)}$, by using \eqref{5-29}, we get that 
\be\lab{5-35} G_n(z_n,\cdot)=\f{1}{2\pi}\ln\f{r_{n,i}}{r_{n,j}^{(\eta_2)}}+O(1)\ee 
uniformly in $B_{r_{n,j}^{(\eta_1)}}(0)$. Now, by computing as in the argument involving \eqref{4-66} and using Proposition \ref{single-2}, we get that
\be\lab{5-36} \la_np_nh_{n,j}f_{n,j}u_{n,j}^{p_n-1}e^{u_{n,j}^{p_n}}=\f{8e^{-2t_{n,j}}}{\mu_{n,j}^2\ga_{n,j}^{p_n-1}p_n}\sbr{1+O\sbr{e^{\ti\eta t_{n,j}} \sbr{\f{|\cdot|}{r_{n,j}^{(\eta_2)}}+|\cdot|+\f{1}{\ga_{n,j}^{p_n}} } } } \ee
uniformly in $B_{r_{n,j}^{(\eta_1)}}(0)$, where $\ti\eta\in(\eta_2,1)$.  
We claim that
\be\lab{5-37} \ln\f{r_{n,i}}{r_{n,j}^{(\eta_2)}}=O(\ga_{n,j}^{p_n}). \ee 
Indeed, resuming the arguments in \eqref{5-27}, we have that
\be\lab{5-38}\begin{aligned} 
	0<2\ln\f{r_{n,i}}{r_{n,j}^{(\eta_2)}}&\le \ln\f{r_{n,i}^2}{\mu_{n,i}^2}+\ln\f{\mu_{n,i}^2}{\mu_{n,j}^2}	=(1+o(1))\ga_{n,j}^{p_n}+2\al\ln\f{|x_{n,j}|}{|x_{n,i}|}.
\end{aligned}\ee 
If $r_{n,i}=o(|x_{n,i}|)$, by \eqref{5-15-1} and \eqref{5-38}, we easily get \eqref{5-37}. If $|x_{n,i}|=O(r_{n,i})$, then we have
\be\lab{5-39} \begin{aligned}
	2\al\ln\f{|x_{n,j}|}{|x_{n,i}|}&\le \al\ln\f{r_{n,j}^2}{r_{n,i}^2}+O(1)\\
	&=\al\sbr{t_{n,j}(r_{n,j})-t_{n,i}(r_{n,i})+\ln\f{\mu_{n,j}^2}{\mu_{n,i}^2}}+O(1)\\
	&\le \al\sbr{\f{p_n}{2}-1+o(1)}\ga_{n,j}^{p_n}-2\al^2\ln\f{|x_{n,j}|}{|x_{n,i}|}+O(1),
\end{aligned}\ee
where the first inequality uses $\al<0$, $|x_{n,i}|=O(r_{n,i})$ and \eqref{5-2}; the second equality uses \eqref{5-3}; the last inequality first uses \eqref{5-27}, then \eqref{5-6} with \eqref{5-13},  and at last the assumption \eqref{5-26}. Then, since $\al\in(-1,0)$, we get from \eqref{5-39} that $2\al\ln\f{|x_{n,j}|}{|x_{n,i}|}=O(\ga_{n,j}^{p_n})$. This together with \eqref{5-38} finishes the proof of \eqref{5-37}.
By \eqref{5-35}, \eqref{5-36} and \eqref{5-37}, we get that
$$\begin{aligned}
	&\la_np_n\int_{B_{r_{n,j}^{(\eta_1)}}(0)}G_n(z_n,y)h_{n,j}f_{n,j}u_{n,j}^{p_n-1}e^{u_{n,j}^{p_n}}\rd y\\
	&=\sbr{\f{1}{2\pi}\ln\f{r_{n,i}}{r_{n,j}^{(\eta_2)}}+O(1) }\f{8\pi}{\ga_{n,j}^{p_n-1}p_n}\mbr{ 1+O\sbr{\f{\mu_{n,j}^2}{(r_{n,j}^{(\eta_1)})^2}}+O\sbr{\f{r_{n,j}^{(\eta_1)}}{r_{n,j}^{(\eta_2)}}+\f{1}{\ga_{n,j}^{p_n}}} }\\
	&=\f{4}{p_n\ga_{n,j}^{p_n-1}}\ln\f{r_{n,i}}{r_{n,j}^{(\eta_2)}}+O(\ga_{n,j}^{1-p_n}),
\end{aligned}$$
using \eqref{5-29}. By plugging this last estimate in \eqref{5-34}, we conclude the proof of \eqref{5-31}.

We now plug \eqref{5-17} and \eqref{5-30} in \eqref{5-31}, and by using \eqref{boundla}, \eqref{5-15-1} and \eqref{5-26}, we get
	$$\ln\f{1}{r_{n,i}|x_{n,i}|^{\al}}\le O\sbr{\ln\ga_{n,i}},$$
which is an obvious contradiction with \eqref{5-8} and \eqref{5-13}. This finishes the proof of \eqref{5-25}.

\vskip0.1in
\item We prove that there exists a constant $C\ge1$ such that
\be\lab{5-40} \f{1}{C}\le \f{r_{n,j}}{r_{n,i}}\le C, \quad \f{1}{C}\le\f{\ga_{n,j}}{\ga_{n,i}}\le C,\quad\forall~j\in J_i, \ee
for large $n$, and there exists $\eta_3\in(\eta,1)$ such that
\be\lab{5-41} \bar r_{n,j}^{(\eta_3)}=r_{n,j} \ee 
for all $j\in J_i$ and large $n$.
By \eqref{5-20} and \eqref{5-25}, we see that
\be\lab{5-42} r_{n,j}\ge r_{n,i},\quad \ga_{n,j}=O(\ga_{n,i}),\quad\forall~j\in J_i. \ee
This together with \eqref{5-14} gives the first assertion of \eqref{5-40}.
We now assume by contradiction with \eqref{5-41} that there exists $j\in J_i$ such that
	$$\f{2t_{n,j}(r_{n,j})}{p_n\ga_{n,j}^{p_n}}\ge1+o(1).$$
As a remark, we must have $j\neq i$ by \eqref{5-13}. Then, for all given $\eta_2\in(0,1)$, \eqref{5-21} holds and the argument between \eqref{5-21} and \eqref{5-25} gives \eqref{5-20}, which does not occur by \eqref{5-42} and proves \eqref{5-41}. For $j\in J_i$, by \eqref{vpn} and the first assertion of \eqref{5-40}, and by choosing $\tau>0$ small enough in the begining, we first get that
	$$\phi_{n,i}\circ\phi_{n,j}^{-1}\sbr{\pa B_{r_{n,j}/2}(0)}\subset \Omega_{n,i},$$
then, by \eqref{5-16}, we get
	$$\bar u_{n,j}(r_{n,j}/2)\ge(1-\eta)\ga_{n,i}+O(1),$$
so that we eventually have $\ga_{n,i}=O(\ga_{n,j})$, using the fact that $\bar u_{n,j}(r_{n,j}/2)\le 2\ga_{n,j}$ by \eqref{5-9}, \eqref{5-10} and \eqref{5-41}. This last estimate proves the second assertion of \eqref{5-40}, by recalling \eqref{5-42}.

\vskip0.1in
\item We prove that if $|x_{n,i}|=O(r_{n,i})$, then there holds
\be\lab{5-45} u_{n,i}=\bar u_{n,i}(\tau r_{n,i})+O(\ga_{n,i}^{1-p_n})\ge(1-\eta)\ga_{n,i}+O(1) \ee 
uniformly in $B_{\tau r_{n,i}}(x_{n,i})$ for large $n$. Assuming $|x_{n,i}|=O(r_{n,i})$, by using \eqref{5-2}, \eqref{5-15-1} and the first assertion of \eqref{5-40}, we get the existence of a constant $C\ge1$ such that
\be\lab{5-46} \f{1}{C}\le \f{|x_{n,j}|}{r_{n,i}}\le C, \quad\forall~j\in J_i  \ee
for large $n$. This means that $\SR_i\cap\{\hat x_0\}=\emptyset$, where $\SR_i$ and $\hat x_0$ are given below \eqref{5-15-1}. 
Then, by \eqref{5-16}, \eqref{5-15} and the choice of $\tau$, we get that 
\be\lab{5-47} u_{n,i}=\bar u_{n,i}(\tau r_{n,i})+O(\ga_{n,i}^{1-p_n})\ge(1-\eta)\ga_{n,i}+O(1) \ee 
on $\pa B_{\tau r_{n,i}}(x_{n,i})$.  
Let $\ga_n,x_n,\mu_n$ be given by
\be\lab{5-50} \ga_n:=u_{n,i}(x_n):=\max_{B_{\tau r_{n,i}}(x_{n,i})}u_{n,i}\ge(1-\eta)\ga_{n,i}+O(1), \ee
and
\be\lab{5-51} \mu_{n}:=\sbr{\f{8(1+\al)^2}{\la_np_n^2\ti h_{n,i}(x_n)f_{n,i}(x_n)\ga_{n}^{2(p_n-1)}e^{\ga_{n}^{p_n}}}}^{\f{1}{2(1+\al)}},\ee
where $\ti h_{n,i}$ is given by \eqref{hi}. We get from \eqref{est-3} that
\be\lab{5-48} \abs{~\cdot-x_{n,i}}^{2+2\al}\la_nu_{n,i}^{2(p_n-1)}e^{u_{n,i}^{p_n}}=O(1)\ee 
uniformly in $B_{2\tau r_{n,i}}(x_{n,i})$ for large $n$. This gives that $|x_n-x_{n,i}|=O(\mu_n)$, and hence gives that $r_{n,i}=O(\mu_n)$. Indeed, if $r_{n,i}\gg\mu_n$, then $x_n$ is a good candidate to be another concentration point for $u_n$, and a standard process leads to a singular bubble at $q_0$, which is a contradiction with the assumption \eqref{5-0}. Then, we get from \eqref{equ-6}, \eqref{5-50} and \eqref{5-51} that
\be\lab{5-55} \abs{\Delta\sbr{\f{u_{n,i}(x_{n,i}+r_{n,i}\cdot)}{\ga_n}}}=O\sbr{r_{n,i}^{2(1+\al)}|\cdot|^{2\al}}+O\sbr{\f{r_{n,i}^{2(1+\al)}}{\ga_n^{p_n}\mu_n^{-2(1+\al)}}|\cdot|^{2\al}}=o(|\cdot|^{2\al})\ee
uniformly in $B_{2\tau}(0)$, so that, by the elliptic theory, we obtain that
\be\lab{5-52} \lim_{n\to+\iy}\f{u_{n,i}(x_{n,i}+r_{n,i}\cdot)}{\ga_n}=1\quad\text{in}~C^0(B_{2\tau}(0)). \ee
By using Lemma \ref{bound-0} and \eqref{5-52}, we get that
\be\lab{5-53} r_{n,i}^{2+2\al}\exp(\ep_0\ga_n^{1/3}/2)=O\sbr{\int_{B_{2\tau r_{n,i}}(x_{n,i})}h_{n,i}e^{\ep_0u_{n,i}^{1/3}}\rd x}=O(1), \ee
for large $n$. 
We set $$\hat u_n:=\ga_n^{p_n-1}u_{n,i}(x_{n,i}+r_{n,i}\cdot)\quad\text{ in }\;B_{2\tau}(0),$$ and by using \eqref{5-55} and \eqref{5-53}, we get that
\be\lab{5-54} \abs{\Delta\hat u_n}=O(|\cdot|^{2\al}) \ee
uniformly in $B_{2\tau}(0)$ for large $n$. Let $\hat\psi_n$ be given by 
	$$-\Delta\hat\psi_n=-\Delta\hat u_n~~\text{in}~B_{2\tau}(0),\quad \hat\psi_n=0~~\text{on}~\pa B_{2\tau}(0).$$
Then, by \eqref{5-54} and the elliptic theory, we get that $|\hat\psi_n|=O(1)$ uniformly in $B_{2\tau}(0)$. 
Since $\hat u_n-\hat\psi_n$ is harmonic in $B_{2\tau}(0)$ and $\hat u_n$ has bounded oscillations on $\pa B_{2\tau}(0)$ by \eqref{5-47}, we have
\be\lab{5-54-1} \hat u_n=\ga_{n}^{p_n-1}\bar u_{n,i}(\tau r_{n,i})+O\sbr{1+\f{\ga_{n}^{p_n-1}}{\ga_{n,i}^{p_n-1}}}\ee
uniformly in $B_{2\tau}(0)$ for large $n$. This, returning back to $u_{n,i}$, clearly gives \eqref{5-45}.

\vskip0.1in
\item We are now in position to conclude the proof of \eqref{5-12}. 
Setting 
	$$\ti u_n:=\ga_{n,i}^{p_n-1}\sbr{u_{n,i}(r_{n,i}\cdot)-\bar u_{n,i}(r_{n,i})},$$
with an argument similar to the proof of \eqref{5-16}, one deduces from \eqref{est-4} and \eqref{5-15} that $\ti u_n$ is uniformly locally bounded in $\R^2\setminus \wt\SR_i$ for all $n$, where $\wt\SR_i$ is given by \eqref{5-15-12}. Then, using \eqref{5-9} and \eqref{5-10} with \eqref{5-13}, we get from \eqref{equ-6} and \eqref{hi} that
$$\begin{aligned}
	-\Delta\ti u_n &=O\sbr{\ga_{n,i}^{p_n}r_{n,i}^2\abs{x_{n,i}-r_{n,i}~\cdot~}^{2\al}}\\
	&\quad+O\sbr{r_{n,i}^2\la_n\ga_{n,i}^{2(p_n-1)}e^{\ga_{n,i}^{p_n}}\abs{x_{n,i}-r_{n,i}~\cdot~}^{2\al} }\\
	&=o(1)
\end{aligned}$$
unifromly locally in $\R^2\setminus\wt\SR_i$ for large $n$, where in the last estimate we use \eqref{5-8} and \eqref{5-3}, and also use $\abs{x_{n,i}-r_{n,i}~\cdot~}^{2\al}=O(|x_{n,i}|^{2\al})$ unifromly locally in $\R^2\setminus\wt\SR_i$ by our definition of $\wt\SR_i$.
Hence, there exists a harmonic function $\ti u_0$ such that $\ti u_n\to\ti u_0$ in $C^1_\loc(\R^2\setminus\wt\SR_i)$ as $n\to+\iy$. Now, observe that \eqref{est-4} and $r_{n,i}\le\f{1}{2}|x_{n,i}|$ imply the existence of $C>0$ such that
	$$|\nabla\ti u_0|\le C\sum_{x\in\wt\SR_i}\f{1}{|\cdot-x|}\quad\text{in}~\R^2\setminus\wt\SR_i.$$
Then, by harmonic function's theory, there exists real numbers $\iota_x$ and $\Lambda$ such that
\be\lab{5-60} \ti u_0=\Lambda+\sum_{x\in\wt\SR_i}\iota_x\ln\f{1}{|\cdot-x|}\quad\text{in}~\R^2\setminus\wt\SR_i.\ee
If $\hat x_0\in\wt\SR_i$, then by the definition of $\wt\SR_i$, we have $|x_{n,i}|=O(r_{n,i})$, and hence by \eqref{5-45} and \eqref{5-16}, we get $\ti u_n=O(1)$ near $\hat x_0$, which together with \eqref{5-60} gives $\iota_{\hat x_0}=0$.  
For any $j\in J_i$, by noting \eqref{5-41}, we  get first from Proposition \ref{single-2} that
	$$\ga_{n,j}^{p_n-1}r_{n,j}\nabla v_{n,j}(r_{n,j}\cdot)=\f{4}{p_n}\nabla\sbr{\ln\f{1}{|\cdot|}}+o(1)~\text{in}~C^0_\loc(B_1(0)\setminus\{0\}),$$
for large $n$, and then, we get from Proposition \ref{regular} that
\be\lab{5-62}\ga_{n,j}^{p_n-1}r_{n,j}\nabla u_{n,j}(r_{n,j}\cdot)=\nabla\psi_j+\f{4}{p_n}\nabla\sbr{\ln\f{1}{|\cdot|}}+o(1)~\text{in}~C^0_\loc(B_1(0)\setminus\{0\}),  \ee
for large $n$, where $\psi_j$ is a harmonic function in $B_1(0)$ satisfying
\be\lab{5-63} \nabla\psi_j(0)=\lim_{n\to+\iy}\f{2\al r_{n,j}}{p_n|x_{n,j}|}\f{x_{n,j}}{|x_{n,j}|}. \ee
Then, by \eqref{vpn}, \eqref{5-15}, \eqref{5-40} and \eqref{5-62}, we conclude that
\be\lab{5-64} \nabla\ti u_n=\f{\ga_{n,i}^{p_n-1}r_{n,i}}{\ga_{n,j}^{p_n-1}r_{n,j}}\sbr{\nabla\sbr{\psi_j+\f{4}{p_n}\ln\f{1}{|\cdot|}}}\circ\sbr{\f{\phi_{n,j}\circ\phi_{n,i}^{-1}(r_{n,i}\cdot)}{r_{n,j}}} +o(1)\ee
in $C^0_\loc(B_{\delta_j}(0)\setminus\{0\})$ for some small $\delta_j>0$. Thus, by \eqref{5-60}, \eqref{5-64} and \eqref{5-63}, we decuce that
\be\lab{5-66} \iota_{\hat x_j}=\lim_{n\to+\iy}\f{4\ga_{n,i}^{p_n-1}r_{n,i}}{p_n\ga_{n,j}^{p_n-1}r_{n,j}}>0 \ee
and 
\be\lab{5-65} \nabla\sbr{\ti u_0-\iota_{\hat x_j}\ln\f{1}{|\cdot-\hat x_j|} }(\hat x_j)=\f{\al\iota_{\hat x_j}}{2}\lim_{n\to+\iy}\f{r_{n,j}}{|x_{n,j}|}\f{x_{n,j}}{|x_{n,j}|}=:\vec X_j. \ee
Moreover, for all $j\in J_i$, we have that
\be\lab{5-67} \vec X_j=\begin{cases} 0,&\quad\text{if $\hat x_0$ does not exist},\\ c_j(\hat x_j-\hat x_0),&\quad\text{if $\hat x_0$ does exist}. \end{cases}\ee 
where $c_j>0$. Indeed, if $\hat x_0$ does not exist, then we have $r_{n,i}=o(|x_{n,i}|)$, and then by \eqref{5-15-1} and \eqref{5-40}, we get that $r_{n,j}=o(|x_{n,j}|)$, so that $\vec X_j=0$; if $\hat x_0$ exists, then we have $|x_{n,i}|=O(r_{n,i})$, and by \eqref{5-40} and \eqref{5-46}, we get that $\displaystyle\lim_{n\to+\iy}\f{r_{n,j}}{|x_{n,j}|}>0$, so that by \eqref{xn} and \eqref{5-15}, we get
	 $$\vec X_j=\f{\al\iota_{\hat x_j}}{2}\f{\hat x_0-\hat x_j}{|\hat x_0-\hat x_j|}\lim_{n\to+\iy}\f{r_{n,j}}{|x_{n,j}|},$$
which gives \eqref{5-67} due to $\al<0$.
Recalling that $\wt\SR_i$ has at least two distinct points, we may pick $\hat x_{j_0}$ as an extreme point of the convex hull of $\wt\SR_i$, such that $\hat x_{j_0}\neq\hat x_0$ if $\hat x_0$ exists. Then, by plugging \eqref{5-60} in \eqref{5-65}, and by using $\iota_{\hat x_0}=0$, we get that
	$$\vec X_{j_0}=\sum_{j\in J_i\setminus\{j_0\}}\iota_j\f{\hat x_j-\hat x_{j_0}}{|\hat x_j-\hat x_{j_0}|^2},$$
which leads to a contradiction with \eqref{5-67}, and hence concludes the proof of \eqref{5-12}.
\end{itemize}

This completes the proof of Lemma \ref{lemma5-1}.
\end{proof}

Once Lemma \ref{lemma5-1} is proven, we have a first estimate as follows:
\begin{lemma}\lab{lemma5-2}
There exists $C>0$ such that for all $i\in\{1,\cdots,N'\}$,
\be\lab{5-70} 0<\bar u_{n,i}(r)\le-\sbr{\f{2}{p_n}-1}\ga_{n,i}+\f{2}{p_n\ga_{n,i}^{p_n-1}}\ln\f{C}{\la_n\ga_{n,i}^{2(p_n-1)}|x_{n,i}|^{2\al}r^2}+O(r^{3/2}), \ee 
uniformly in $(0,\kappa]$ and for all $n$, where $\bar u_{n,i}$ is the average function given by \eqref{average}.
\end{lemma}
\begin{proof}
Let $\eta_1<\eta_2$ be two given numbers in $(0,1)$. Then, by \eqref{5-6}, \eqref{5-7} and \eqref{5-12}, we get
	$$\bar r_{n,i}^{(\eta_1)}=r_{n,i}^{(\eta_1)},\quad \bar r_{n,i}^{(\eta_2)}=r_{n,i}^{(\eta_2)}$$
for large $n$. Then \eqref{5-70} holds uniformly in $(0,r_{n,i}^{(\eta_2)}]$ by using \eqref{16} and \eqref{5-10}. We get also from \eqref{5-9}, \eqref{5-6} and \eqref{5-10} that
\be\lab{5-71} \bar u_{n,i}(r_{n,i}^{(\eta_1)})=v_{n,i}(r_{n,i}^{(\eta_1)})+O(\ga_{n,i}^{1-p_n})=(1-\eta_1)\ga_{n,i}+O(\ga_{n,i}^{1-p_n}), \ee
and from \eqref{5-11} that
\be\lab{5-72} \nm{\nabla(u_{n,i}-v_{n,i})}_{L^\iy\sbr{\pa B_{r_{n,i}^{(\eta_1)}}(0)}}=O\sbr{\f{1}{\ga_{n,i}^{p_n-1}r_{n,i}^{(\eta_2)}}}. \ee
We write by \eqref{equ-6} that
\be\lab{5-73} \int_{B_r(0)}-\Delta u_{n,i}\rd x\ge \int_{B_{r_{n,i}^{(\eta_1)}}(0)}-\Delta u_{n,i}\rd x+O\sbr{\int_{B_r(0)\setminus B_{r_{n,i}^{(\eta_1)}}(0)}h_{n,i}u_{n,i}\rd x },\ee
uniformly in $r\in(r_{n,i}^{(\eta_1)},\kappa)$. Since $\al\in(-1,0)$, we may take $s=\f{2}{3}\sbr{\f{1}{|\al|}-1}$ and $t=\sbr{\f{1}{4}-\f{1}{s}}^{-1}$, then by H\"older inequality, and using \eqref{hi} and Lemma \ref{bound-0}, we get 
$$\begin{aligned}
	\int_{B_r(0)\setminus B_{r_{n,i}^{(\eta_1)}}(0)}h_{n,i}u_{n,i}\rd x
	&\le \sbr{\int_{B_r(0)}1\rd x}^{\f{3}{4}} \sbr{\int_{B_r(0)}h_{n,i}^{(1-1/t)s}\rd x}^{\f{1}{s}} \sbr{\int_{B_r(0)}h_{n,i}u_{n,i}^t\rd x}^{\f{1}{t}}\\
	&=O(r^{3/2}). 
\end{aligned}$$
Then, we get from \eqref{5-73} that
\be\lab{5-74} \bar u_{n,i}'(r)\le \f{r_{n,i}^{(\eta_1)}}{r}\bar u_{n,i}'(r_{n,i}^{(\eta_1)})+O(r^{1/2}) \ee
uniformly in $r\in(r_{n,i}^{(\eta_1)},\kappa)$. 
Using the fact by \eqref{5-6} that
	$$ \ln\f{r_{n,i}^{(\eta_1)}}{r_{n,i}^{(\eta_2)}}=-\f{p_n\ga_{n,i}^{p_n}}{4}(\eta_2-\eta_1)+o(1),$$
and using \eqref{5-72} and Proposition \ref{single-2}, we now conclude that
	$$r_{n,i}^{(\eta_1)}\bar u_{n,i}'(r_{n,i}^{(\eta_1)})=-\f{4}{p_n\ga_{n,i}^{p_n-1}}+O(\ga_{n,i}^{1-2p_n}).$$
Then, integrating \eqref{5-74} in $[r_{n,i}^{(\eta_1)},r]$ and using the fundamental theorem of calculus, we get that
\be\lab{5-75} \bar u_{n,i}'(r)-\bar u_{n,i}'(r_{n,i}^{(\eta_1)})\le -\f{4}{p_n\ga_{n,i}^{p_n-1}}\ln\f{r}{r_{n,i}^{(\eta_1)}}\sbr{1+O(\ga_{n,i}^{-p_n})}+O(r^{3/2}), \ee 
uniformly in $r\in[r_{n,i}^{(\eta_1)},\kappa]$. Observing by \eqref{5-3}, \eqref{5-6} and \eqref{mun-1} that
$$\begin{aligned}
	2\ln\f{r_{n,i}^{(\eta_1)}}{r} &=2\ln\f{r_{n,i}^{(\eta_1)}}{\mu_{n,i}}+\ln\f{\mu_{n,i}^2}{r^2}
	\le -\sbr{1-\eta_1\f{p_n}{2}}\ga_{n,i}^{p_n}+\ln\f{C}{\la_n\ga_{n,i}^{2(p_n-1)}|x_{n,i}|^{2\al}r^2},
\end{aligned}$$
and by using \eqref{5-75} and \eqref{5-71}, we conclude the proof of \eqref{5-70}.
\end{proof}

\begin{corollary}\lab{lemma5-3}
There holds 
\be\lab{5-80} \lim_{n\to+\iy}\la_n=0. \ee
\end{corollary}
\begin{proof}
For any $i\in\{1,\cdots,N'\}$, by evaluating \eqref{5-70} at $r=\kappa\ga_{n,i}^{2(1-p_n)/3}\le\kappa$, we get that
\be\lab{5-79} \sbr{1-\f{p_n}{2}}\ga_{n,i}^{p_n}+\ln\sbr{\ga_{n,i}^{\f{4}{3}(p_n-1)}|x_{n,i}|^{2\al}}\le\ln\f{1}{\la_n}+O(1). \ee
This clearly proves \eqref{5-80} since $\al\in(-1,0)$.
\end{proof}
\begin{remark}
In the proof of \eqref{5-80} we use $\al<0$. This is the second place where our method highly and nontrivially relies on $\al<0$. 
\end{remark}

\vs
\subsection{Quantization for subcritical exponent \texorpdfstring{$p_n\not\to 2$}{}}\lab{sec5-2}\ 

Up to a subsequence, we may assume 
\be\lab{5-81}\lim_{n\to+\iy}p_n=p_0\ee
for some $p_0\in[1,2]$. In this subsection, we prove that
\begin{lemma}\lab{lemma5-4}
Let $p_0\in[1,2)$, then there holds
	$$\lim_{n\to+\iy}\beta_n(B_{\kappa_0}(q_0))=4\pi N',$$ 
where $N'$ is given above \eqref{5-0} and $\beta_n(\cdot)$ is given by \eqref{betadomain}.
\end{lemma}
\begin{proof}
Up to a renumbering, we fix $i\in\{1,\cdots,N'\}$ such that 
\be\lab{5-82} \ga_{n,i}=\max\{\ga_{n,j}:~j=1,\cdots,N'\}, \ee
for all $n$. Given any $\eta\in(0,1)$, setting $r_{n,l}^{(\eta)}$ as in \eqref{5-6}, we know from \eqref{5-12} that $\bar r_{n,l}^{(\eta)}=r_{n,l}^{(\eta)}$ for all $l$ and $n$. 
By the argument to get \eqref{4-66}, and using \eqref{5-9}, \eqref{5-10} and Proposition \ref{single-2}, we get that
\be\lab{5-83} \f{p_n}{2}\ga_{n,l}^{p_n-2}\int_{B_{r_{n,l}^{(\eta)}}(0)} \la_np_nh_{n,l}f_{n,l}u_{n,l}^{p_n}e^{u_n^{p_n}}\rd x=4\pi+o(1)\ee
and
\be\lab{5-84} \f{p_n}{2}\ga_{n,l}^{2(p_n-1)}\int_{B_{r_{n,l}^{(\eta)}}(0)} \la_np_nh_{n,l}f_{n,l}e^{u_n^{p_n}}\rd x=4\pi+o(1)\ee
for all $l$ and large $n$.
Since $u_n\gg1$ uniformly in $B_{r_{n,l}^{(\eta)}}(0)$ by \eqref{5-9} and \eqref{5-10}, we get from \eqref{5-84} that
\be\lab{5-84-1} \f{p_n}{2}\ga_{n,l}^{2(p_n-1)}\int_{B_{r_{n,l}^{(\eta)}}(0)} \la_np_nh_{n,l}f_{n,l}\rd x=o(1) \ee
for all $l$ and large $n$.
We get also from \eqref{5-9}, \eqref{5-6} and \eqref{5-10} that
\be\lab{5-85}\begin{aligned} 
	u_{n,l}(r_{n,l}^{(\eta)})&=(1-\eta)\ga_{n,l}+O(\ga_{n,l}^{1-p_n})\\
	&=-\sbr{\f{2}{p_n}-1}\ga_{n,l}+\f{2}{p_n\ga_{n,l}^{p_n-1}}\ln \f{1}{\la_n\ga_{n,l}^{2(p_n-1)}|x_{n,l}|^{2\al}(r_{n,l}^{(\eta)})^{2}} +O(\ga_{n,l}^{1-p_n})
\end{aligned}\ee
uniformly on $\pa B_{r_{n,l}^{(\eta)}}(0)$ for large $n$ and all $l$, where in the second equality we use \eqref{16}. 
By comparing the two terms of \eqref{5-85}, and by using \eqref{5-2}, we first get
\be\lab{5-85-1} \ln\f{1}{\la_n}\le\sbr{1-\eta\f{p_0}{2}+o(1)}\ga_{n,l}^{p_n}, \ee
so that by \eqref{5-79} we have
	$$\sbr{1-\f{p_0}{2}+o(1)}\ga_{n,l}\le\sbr{1-\eta\f{p_0}{2}+o(1)}\ga_{n,i},$$
for all $l$ and $n$, which gives, by \eqref{5-82} and $p_0<2$, and by letting $\eta\to1$, that
\be\lab{5-86} \ga_{n,l}=(1+o(1))\ga_{n,i}\ee
for all $l$.

We prove that
\be\lab{5-90} u_n\le (1-\eta')\ga_{n,i} \quad\text{in}~\Omega_n:=B_{\kappa_0}(q_0)\setminus\bigcup_{l=1}^{N'}\phi_{n,l}^{-1}\sbr{B_{r_{n,l}^{(\eta)}}(0)} \ee 
for large $n$  and any $\eta'\in(0,\eta)$. 
By contradiction, we assume that there holds
\be\lab{5-91} u_n(q_n):=\max_{\Omega_n}u_n>(1-\eta')\ga_{n,i}, \ee
for some $q_n\in\Omega_n$. By \eqref{conv-0}, we have $\lim_{n\to+\iy}q_n=q_0$. Let the isothermal coordinate $(B_{\kappa_0}(q_{n,0}),\phi_{n,0},U_{n,0})$ be given as above Proposition \ref{bubble}. Setting $x_n=\phi_{n,0}(q_n)$ and 
$$\ti\Omega_n=B_{\kappa}(0)\setminus\cup_{l=1}^{N'}\phi_{n,0}\circ\phi_{n,l}^{-1}\sbr{B_{r_{n,l}^{(\eta)}}(0)},$$ we get from \eqref{5-91} that
\be\lab{5-92} u_{n,0}(x_n):=\max_{\wt\Omega_n}u_{n,0}>(1-\eta')\ga_{n,i} .\ee
By \eqref{vpn}, \eqref{5-92} and \eqref{5-85}, we get that
\be\lab{5-93} \abs{x_n-\phi_{n,0}(q_{n,l})}\ge\f{1}{2}r_{n,l}^{(\eta)}\gg r_{n,l}^{(\eta')}\ee 
for large $n$ and all $l$. 
We set 
\be\lab{5-94} d_n=d_{g_0}(q_n,\QR_n)\ee 
with $\QR_n$ be given below \eqref{5-1}. 
We first assume $d_{g_0}(q_n,q_0)=o(d_n)$. Let $I_0$ be given by
\be\lab{5-95} I_0=\lbr{j\in\{1,\cdots,N'\}:~d_{g_0}(q_n,q_{n,j})=O(d_n)} \ee
By \eqref{5-93} and \eqref{vpn}, we get that
\be\lab{5-96} r_{n,l}^{(\eta)}=O(d_n)\ee 
for all $l\in I_0$. 
By the assumption $d_{g_0}(q_n,q_0)=o(d_n)$, we have $|x_n|=o(d_n)$, and we may assume that
	$$\lim_{n\to+\iy}\f{\phi_{n,0}(q_{n,l})}{d_n}=\check x_l\in\R^2$$
for all $l\in I_0$. 
Let $\SR_0=\lbr{\check x_l:~l\in I_0}$. Then, we may choose and fix $\tau_0\in(0,1)$ small enough such that
	$$3\tau_0<\displaystyle\min\lbr{|x-y|:~x,y\in\SR_0\cup\{0\},x\neq y},$$
and $R_0\gg1$ such that $\SR_0\subset B_{\f{1}{3}R_0}(0)$. We can check  that there exists $C>0$ such that any point in
	$$\Omega_{n,0}:=\sbr{B_{R_0d_n}(x_n)\setminus B_{\tau_0d_n}(0)} \setminus \bigcup_{l\in I_0}B_{\tau_0d_n}(\phi_{n,0}(q_{n,l}))$$
may be joined to $\pa B_{\tau_0d_n}(\phi_{n,0}(q_{n,l}))$ by  a $C^1$ path in $\Omega_{n,0}$ of length at most $Cd_n$, for some $l\in I_0$ and all $n$. 
Therefore, by applying \eqref{5-70} with $r=\tau_0d_n$, and by \eqref{5-96} and \eqref{5-85}, we may first get
	$$\bar u_{n,l}(\tau_0d_n)\le(1-\eta)\ga_{n,l}+O(\ga_{n,l}^{1-p_n})$$
for all $l\in I_0$, and then get from \eqref{est-4} and \eqref{5-86} that
\be\lab{5-97} u_{n,0}=\bar u_{n,l}(\tau_0d_n)+O\sbr{\f{1}{1+\bar u_{n,l}(\tau_0d_n)^{p_n-1}}}\le(1-\eta+o(1))\ga_{n,i}, \ee
uniformly on $\pa B_{\tau_0d_n}(0)$ for all $n$.
Then, by the argument between \eqref{5-47} and \eqref{5-54-1}, and using \eqref{5-97}, we get that 
	$$u_{n,0}\le(1-\eta+o(1))\ga_{n,i}$$
uniformly in $B_{\tau_0d_n}(0)$ for all $n$. This together with \eqref{5-92} and $|x_n|=o(d_n)$ gives 
	$$(1-\eta')\ga_{n,i}<(1-\eta+o(1))\ga_{n,i},$$
for large $n$, which is a contradiction with $\eta'<\eta$.  
As a result, we must have $d_n=O(d_{g_0}(q_n,q_0))$. Let $\ti d_n=\min\{|x_n|,d_n\}$, and $\ti I_0$ be given by
	$$\ti I_0=\lbr{j\in\{1,\cdots,N'\}:~d_{g_0}(q_n,q_{n,j})=O(\ti d_n)}.$$
We must have $\ti I_0\neq\emptyset$ since $d_n=O(d_{g_0}(q_n,q_0))$. Then, by a similar argument to get \eqref{5-97}, we have that
	$$u_{n,0}(x_n)\le(1-\eta+o(1))\ga_{n,i}$$
for all $l\in \ti I_0$ and all $n$, which together with \eqref{5-92} again leads to a contradiction, so that we finish the proof of \eqref{5-90}.

By \eqref{5-90} and \eqref{5-79}, we get that 
\be\lab{5-99} \begin{aligned} 
	&\int_{\Omega_n}\la_nhf\sbr{1+u_n^{p_n}}e^{u_n^{p_n}}\rd v_{g_0}\\
	&\quad=O\sbr{|x_{n,i}|^{-2\al}\exp\sbr{(1-\eta')^{p_n}\ga_{n,i}^{p_n}-\sbr{1-\f{p_0}{2}}\ga_{n,i}^{p_n}+o(\ga_{n,i}^{p_n})} },
\end{aligned}\ee
for all $n$. Noting that $p_0<2$, by choosing $\eta'\in(0,\eta)$ sufficiently close to 1 from the begining, we may plug \eqref{5-99}, \eqref{5-83}, \eqref{5-84}  and \eqref{5-84-1} in \eqref{betadomain} to conclude the proof of this lemma, using also \eqref{5-86}.
\end{proof}

\vs
\subsection{Quantization for critical exponent \texorpdfstring{$p_n\to2$}{}}\lab{sec5-3}\ 

In this subsection, we assume that $p_0=2$ in \eqref{5-81}. Recall the notations $r_{n,l}$, $r_{n,l}^{(\eta)}$ and $\bar r_{n,l}^{(\eta)}$ given by \eqref{5-1}, \eqref{5-6} and \eqref{5-7}.

For all $\eta\in(0,1)$, by \eqref{5-12} and \eqref{5-6}, we have that
\be\lab{5-100} r_{n,l}^{(\eta)}=o(r_{n,l}),\quad \bar r_{n,l}^{(\eta)}=r_{n,l}^{(\eta)}\ee 
for all $n$ and $l\in\{1,\cdots,N'\}$.
For all $\eta'\in(0,\eta)$, as a consequence of \eqref{5-11}, we get that
\be\lab{5-100-1} \abs{\nabla(u_{n,l}-v_{n,l})}=o\sbr{\f{1}{\ga_{n,l}^{p_n-1}\bar r_{n,l}^{(\eta')} }}.\ee
Then, since we also have 
	$$\abs{v_{n,l}-\bar v_{n,l}(r)}\le\f{2+o(1)}{\ga_{n,l}^{p_n-1}}\ln\f{2r}{|\cdot|},$$
using the estimate in Proposition \ref{appe1-2}, we eventually get that
\be\lab{5-101} \abs{u_{n,l}-\bar u_{n,l}(r)}\le \f{2+o(1)}{\ga_{n,l}^{p_n-1}}\ln\f{2r}{|\cdot|},\ee 
uniformly in $B_{r}(0)\setminus\{0\}$ for all large $n$ and $r\in\big(0,r_{n,l}^{(\eta')}\big]$. 

Now, let $\eta_0\in(0,1)$ be fixed and will be chosen later. For any $j\in\{1,\cdots,N'\}$, we define 
\be\lab{5-102} 
\nu_{n,j}:=\sup\lbr{ r_0\in\left(r_{n,j}^{(\eta_0)},\kappa\right]:~
\left\{\begin{aligned} &\abs{u_{n,j}-\bar u_{n,j}(r)}\le 5\pi C_2\bar u_{n,j}(r)^{1-p_n}\\
&\quad\quad\quad+10\sum_{l\in I_{n,j}(r)}\ga_{n,l}^{1-p_n}\ln\f{6r}{\abs{\cdot-\phi_{n,j}(q_{n,l})}} \\
&\text{in}~B_r(0)\setminus\cup_{l\in I_{n,j}(r)}B_{r_{n,l}^{(\eta_0)}}\sbr{\phi_{n,j}(q_{n,l})}\\
&\text{for all}~r\in(0,r_0]
\end{aligned} \right. }  
\ee
for large $n$, where $C_2$ is as in \eqref{est-4} and $I_{n,j}(r)$ is given by
\be\lab{5-103} I_{n,j}(r):=\lbr{l\in\{1,\cdots,N'\}:~\phi_{n,j}(q_{n,l})\in B_{\f{3r}{2}}(0)}. \ee
As a first remark, it follows from \eqref{5-101} that, for all given $\eta_2\in[\eta_0,1)$, we have
\be\lab{5-104} \nu_{n,j}\ge r_{n,j}^{(\eta_2)} \ee 
for large $n$ and any $j$.
Our main goal now is to show that

\begin{lemma}\lab{lemma5-5}
By  choosing $\eta_0$ such that
	$$\eta_0\in\sbr{1-\f{1}{201(1+N')^2},1}$$
we have that 
\be\lab{5-105} \bar u_{n,j}(\nu_{n,j})=O(1), \ee 
for all $j\in\{1,\cdots,N'\}$.
\end{lemma}
\begin{proof}
The proof is quite similar to that of \cite[(4.79)]{MT-blowup-7}, except that we need to complete the ``hole'' near the singularity like in the proof of Proposition \ref{lemma5-1}. Here, we write it out for reader's convenience.

We assume by contradiction with \eqref{5-105} that there is an $i\in\{1,\cdots,N'\}$ such that
\be\lab{5-106} \nu_{n,i}=\min\lbr{\nu_{n,j}:~\lim_{n\to+\iy}\bar u_{n,j}(\nu_{n,j})=+\iy }. \ee
Since $\lim_{n\to+\iy}\bar u_{n,i}(\nu_{n,i})=+\iy$, by \eqref{conv-0}, we have that $\nu_{n,i}=o(1)$. 
Let $J_i$ be defined by
\be\lab{5-107} J_i:=\lbr{l\in\{1,\cdots,N'\}:~d_{g_0}(q_{n,l},q_{n,i})=O(\nu_{n,i})}. \ee
Then, using \eqref{vpn} and $\nu_{n,i}=o(1)$, we get that
\be\lab{5-108} \lim_{n\to+\iy}\sbr{(\phi_{n,l})_*g_0)}(\nu_{n,i}\cdot)=|\rd x|^2~\text{in}~C_\loc^2(\R^2)\ee
for all $l\in J_i$. 
We may assume that
	$$\lim_{n\to+\iy}\f{\phi_{n,i}(q_{n,l})}{\nu_{n,i}}=\ti x_l\in\R^2$$
for all $l\in J_i$, and we also assume that
	$$\lim_{n\to+\iy}\f{\phi_{n,i}(q_{0})}{\nu_{n,i}}=\lim_{n\to+\iy}\f{x_{n,i}}{\nu_{n,i}}=\ti x_0\in\R^2,\quad\text{if}~|x_{n,i}|=O(\nu_{n,i}).$$
Let $\SR_i=\lbr{\ti x_l:~l\in J_i}$ and $$\wt\SR_i=\begin{cases}\SR_i\cup\{\ti x_0\} &\text{if}\; |x_{n,i}|=O(\nu_{n,i})\\ \SR_i &\text{if}\; \nu_{n,i}=o(|x_{n,i}|).\end{cases}$$ We may choose and fix $\tau\in(0,1)$ small and $R\ge1$ large, such that
	$$3\tau<\begin{cases} 1,&\quad\text{if}~\wt\SR_i=\{0\},\\ \min\lbr{|x-y|:~x,y\in\wt\SR_i,x\neq y}, &\quad\text{otherwise}, \end{cases}$$
and $\wt\SR_i\subset B_{\f{1}{3}R}(0)$.
We set
	$$D_{n}:=\sbr{B_{R\nu_{n,i}}(0)\setminus B_{\tau\nu_{n,i}/3}(x_{n,i})} \setminus \bigcup_{l\in J_i}B_{\tau\nu_{n,i}/3}(\phi_{n,i}(q_{n,l})),$$
for all $n$.

Let $w_{n}$ be given by
\be\lab{5-109}\begin{cases}
	-\Delta w_{n}=-h_{n,i}u_{n,i}\quad\text{in}~B_{R\nu_{n,i}}(0),\\
	w_{n}=0\quad\text{on}~\pa B_{R\nu_{n,i}}(0).\\
\end{cases}\ee
By the arguments below \eqref{5-18}, we get that
\be\lab{5-110} 
	\nm{w_{n}}_{L^\iy(B_{R\nu_{n,i}}(0))}=\begin{cases} O\sbr{ \sbr{\nu_{n,i}^2|x_{n,i}|^{2\al}}^{1-\f{1}{s}} }, &\quad\text{if}~\nu_{n,i}=o(|x_{n,i}|),\\O\sbr{\nu_{n,i}^{(2+2\al)(1-\f{1}{st})}}, &\quad\text{if}~|x_{n,i}|=O(\nu_{n,i}), \end{cases}
\ee
for some $s,t>1$. This gives $\nm{w_{n}}_{L^\iy(B_{R\nu_{n,i}}(0))}=o(1)$ in any case. Then, observing that $-\Delta(u_{n,i}-w_{n})\ge0$ in $B_{R\nu_{n,i}}(0)$ so that $\ov{u_{n,i}-w_{n}}$ is radially decreasing in $[0,R\nu_{n,i}]$, we get that $\bar u_{n,i}(\tau\nu_{n,i})\ge\bar u_{n,i}(\nu_{n,i})+o(1)$, which together with \eqref{5-106} leads to
\be\lab{5-111} \Ga_n:=\bar u_{n,i}(\tau\nu_{n,i})\to+\iy \ee
as $n\to+\iy$. By the argument to get \eqref{5-16}, and by using \eqref{est-4} and \eqref{5-108}, we get that
\be\lab{5-112} u_{n,i}=\Ga_n+O(\Ga_n^{1-p_n})\ee 
uniformly in $D_n$ for large $n$. Further, by applying the maximum principle to $u_{n,i}-w_{n}$, and using the fact $\nm{w_{n}}_{L^\iy(B_{R\nu_{n,i}}(0))}=o(1)$, we also get 
\be\lab{5-113} u_{n,i}\ge\Ga_n+o(1)\ee
uniformly in $B_{R\nu_{n,i}}(0)$ for large $n$.

We prove now that
\be\lab{5-114} \Ga_n=o(\ga_{n,j}),\quad\text{for all}~j\in J_i. \ee
For all given $\eta_2\in(\eta_0,1)$, we get from \eqref{5-6}, \eqref{5-100}, \eqref{5-9} and \eqref{5-10} that
\be\lab{5-115} \bar u_{n,j}\sbr{r_{n,j}^{(\eta_2)}}=(1-\eta_2+o(1))\ga_{n,j}\ee 
for large $n$ and $j\in J_i$. Using now  \eqref{5-100}, \eqref{5-1} and the definition of $J_i$, we obtain that
\be\lab{5-116} r_{n,j}^{(\eta_2)}=o(\nu_{n,i}),\quad\text{for all}~j\in J_i. \ee
Then, by \eqref{vpn}, we get
	$$\phi_{n,i}\circ\phi_{n,j}^{-1}\sbr{\pa B_{r_{n,j}^{(\eta_2)}}(0)}\subset B_{R\nu_{n,i}}(0),$$
for large $n$, so that by \eqref{5-113}, we eventually get that
\be\lab{5-117} \bar u_{n,j}\sbr{r_{n,j}^{(\eta_2)}}\ge\Ga_n+o(1)\ee
for all $j\in J_i$. Pluging \eqref{5-115} in \eqref{5-117}, and by letting $\eta_2\to1$ in \eqref{5-115}, we conclude the proof of \eqref{5-114}.

Now we can improve the estimate in \eqref{5-110}. In fact, by \eqref{5-113}, \eqref{hi} and  Lemma \ref{bound-0}, we get that
\be\lab{5-118-0} \nu_{n,i}^2\sbr{\nu_{n,i}+|x_{n,i}|}^{2\al}\exp\sbr{\f{\ep_0}{2}\Ga_n^{1/3}}=O\sbr{\int_{B_{R\nu_{n,i}}(0)} h_{n,i}e^{\ep_0u_{n,i}^{1/3}}\rd x}=O(1). \ee
This together with \eqref{5-110} clearly gives 
\be\lab{5-118} \nm{w_{n}}_{L^\iy(B_{R\nu_{n,i}}(0))}=o(\Ga_n^{1-p_n}), \ee
where $w_{n}$ is given by \eqref{5-109}. Let $\psi_n$ be given by
\be\lab{5-119}\begin{cases}
	-\Delta \psi_n=0\quad\text{in}~B_{R\nu_{n,i}}(0),\\
	\psi_n=u_{n,i}\quad\text{on}~\pa B_{R\nu_{n,i}}(0).\\
\end{cases}\ee
By keeping track of the constant $C_2$ of \eqref{est-4} and choosing $R\gg1$ large enough, using \eqref{5-108} and \eqref{5-111}, we may get a slightly more precise version of \eqref{5-112} on $\pa B_{R\nu_{n,i}}(0)$, namely we have that
\be\lab{5-120} \sup_{B_{R\nu_{n,i}}(0)}\abs{\psi_n-\bar u_{n,i}(R\nu_{n,i})}\le  \sup_{\pa B_{R\nu_{n,i}}(0)}\abs{u_{n,i}-\bar u_{n,i}(R\nu_{n,i})}\le\f{2\pi C_2}{\Ga_n^{p_n-1}} \ee 
for large $n$. 
Let $(z_n)_n$ be any sequence of points such that
\be\lab{5-120-1} z_n\in B_{R\nu_{n,i}}(0)\setminus\cup_{l\in J_i}B_{r_{n,l}^{(\eta_0)}}\sbr{\phi_{n,i}(q_{n,l})}.\ee
Let $G_n$ be the Green function of $-\Delta$ in $B_{R\nu_{n,i}}(0)$ with the Dirichlet boundary condition. 
Then, we have that
	$$0<G_n(z_n,\cdot)\le\f{1}{2\pi}\ln\f{2R\nu_{n,i}}{|\cdot-z_n|}$$
uniformly in $B_{R\nu_{n,i}}(0)\setminus\{z_n\}$ for large $n$. Thus, by the Green representation formula and \eqref{equ-6}, we get that
\be\lab{5-121} 0\le\sbr{u_{n,i}-w_{n}-\psi_n}(z_n)\le \IR_n(B_{R\nu_{n,i}}(0)),\ee
where $w_{n}$ is given by \eqref{5-109}, $\psi_n$ is given by \eqref{5-119} and $\IR_n(\Omega)$ is given by
\be\lab{5-121-1} \IR_n(\Omega):=
\f{\la_np_n}{2\pi}\int_{\Omega}\ln\f{2R\nu_{n,i}}{|y-z_n|} h_{n,i}f_{n,i}u_{n,i}^{p_n-1}e^{u_{n,i}^{p_n}}(y)\rd y. \ee 

We now estimate the last term of \eqref{5-121}. 
We have that
\be\lab{5-122} \nu_{n,j}\ge\tau\nu_{n,i},\quad\text{for all}~j\in J_i. \ee 
Indeed, for $j\in J_i$, if $\lim_{n\to+\iy}\bar u_{n,j}(\nu_{n,j})=+\iy$, then \eqref{5-106} gives \eqref{5-122}; otherwise, we have $\bar u_{n,j}(\nu_{n,j})=O(1)$, and by the choice of $R$ and \eqref{5-113}, we get that
	$$u_{n,j}\ge\Ga_n+o(1)$$
uniformly in $B_{R\nu_{n,i}/3}(0)$ for large $n$, so that we also get \eqref{5-122} by choosing $\tau$ small enough in the begining. Using \eqref{5-11} and Proposition \eqref{appe1-2} with \eqref{5-100}, we have that, for all $l\in \{1,\cdots,N'\}$,
	$$\abs{\nabla u_{n,l}}=O\sbr{\f{1}{\ga_{n,l}^{p_n-1}r_{n,l}^{(\eta_0)}}}\quad\text{uniformly in}~B_{3r_{n,l}^{(\eta_0)}}(0)\setminus B_{r_{n,l}^{(\eta_0)}/3}(0),$$
so that, for all $j\in J_i$, we get as a byproduct of \eqref{5-102} and \eqref{5-122} that
	$$\abs{u_{n,j}-\bar u_{n,j}(\tau\nu_{n,i})}\le 5\pi C_2\bar u_{n,j}(\tau\nu_{n,i})^{1-p_n}+10\sum_{l\in I_{n,j}(\tau\nu_{n,i})}\ga_{n,l}^{1-p_n}\ln\f{6\tau\nu_{n,i}}{\abs{\cdot-\phi_{n,j}(q_{n,l})}}$$
uniformly in $B_{\tau\nu_{n,i}}(0)\setminus\cup_{l\in I_{n,j}(\tau\nu_{n,i})}B_{r_{n,l}^{(\eta_0)}}\sbr{\phi_{n,j}(q_{n,l})}$,  and then we eventually obtain from \eqref{5-112} and the definition of $\tau$ that
\be\lab{5-123} \abs{u_{n,i}-\Ga_n}=6\pi C_2\Ga_n^{1-p_n}+10\sum_{l\in I_{n,j}(\tau\nu_{n,i})}\ga_{n,l}^{1-p_n}\ln\f{6\tau\nu_{n,i}}{\abs{\cdot-\phi_{n,i}(q_{n,l})}} \ee
uniformly in $D_{n,j}$ given by
	$$D_{n,j}:=B_{\tau\nu_{n,i}/2}(\phi_{n,i}(q_{n,j}))\setminus\bigcup_{l\in I_{n,j}(\tau\nu_{n,i})}B_{r_{n,l}^{(\eta_0)}/2}\sbr{\phi_{n,i}(q_{n,l})}.$$
Independently, for any $l\in J_i$, using that $|z_n-\phi_{n,i}(q_{n,l})|\ge r_{n,l}^{(\eta_0)}$ since $I_{n,j}(\tau\nu_{n,i})\subset J_i$, we have
\be\lab{5-124} \ln\f{2R\nu_{n,i}}{|\cdot-z_n|}=\ln\f{2R\nu_{n,i}}{|\phi_{n,i}(q_{n,l})-z_n|}+O(1) \ee
uniformly in $B_{r_{n,l}^{(\eta_0)}}(\phi_{n,i}(q_{n,l})$ for large $n$.
By \eqref{5-100-1} for some given $\eta'\in(\eta_0,1)$ and since $u_{n,l}(0)=v_{n,l}(0)$, we get $u_{n,l}-v_{n,l}=o(\ga_{n,l}^{1-p_n})$, so we eventually get that
\be\lab{5-125} \la_np_nh_{n,l}f_{n,l}u_{n,l}^{p_n-1}e^{u_{n,l}^{p_n}}=\f{8e^{-2t_{n,l}}(1+o(e^{\ti\eta t_{n,l}}))}{\mu_{n,l}^{2}\ga_{n,l}^{p_n-1}p_n} \ee 
uniformly in $B_{r_{n,l}^{(\eta_0)}}(0)$ for large $n$ and $l\in J_i$, still by applying the argument involving \eqref{4-66} and using Proposition \ref{single-2}. Then, using \eqref{5-124} and \eqref{5-125}, we get that
\be\lab{5-126} \IR_n\sbr{B_{r_{n,l}^{(\eta_0)}/2}(\phi_{n,i}(q_{n,l}))}=\f{4+o(1)}{p_n\ga_{n,l}^{p_n-1}}\ln\f{2R\nu_{n,i}}{|\phi_{n,i}(q_{n,l})-z_n|}+O\sbr{\ga_{n,l}^{1-p_n}}  \ee
for all $l\in J_i$. 
Using the basic inequalities $|(1+t)^p-1|\le 4(|t|+|t|^p)$ for all $t>-1$, and $\sbr{\sum_{t=1}^{N'}a_t}^p\le (N')^2\sum_{t=1}^{N'}a_t^p$ for all $a_t\ge0$ and all $p\in[1,2]$, we get first from \eqref{5-123} that
\be\lab{5-127}\begin{aligned} 
	u_{n,i}^{p_n} \le \Ga_n^{p_n}&+25\pi C_2 +40\sum_{l\in I_{n,j}(\tau\nu_{n,i})}\sbr{\f{\Ga_n}{\ga_{n,l}}}^{p_n-1}\ln\f{6\tau\nu_{n,i}}{|\cdot-\phi_{n,i}(q_{n,l})|} \\
	&+400(1+N')^2\sum_{l\in I_{n,j}(\tau\nu_{n,i})}\sbr{\ga_{n,l}^{1-p_n}\ln\f{6\tau\nu_{n,i}}{|\cdot-\phi_{n,i}(q_{n,l})|}}^{p_n}
\end{aligned}\ee
uniformly in $D_{n,j}$ for large $n$ and $j\in J_i$. Independently, we first get from \eqref{5-12}, \eqref{5-2} and \eqref{mun-1} with \eqref{hi} that
	$$\sbr{\f{p_n}{2}+o(1)}\ga_{n,l}\le t_{n,l}(|x_{n,l}|)=\ln\sbr{\la_n|x_{n,l}|^{2(1+\al)}}+(1+o(1))\ga_{n,l}^{p_n}+O(1),$$
so that, by \eqref{5-80} and $p_0=2$ in \eqref{5-81}, we get $\ln|x_{n,l}|\ge o(\ga_{n,l}^{p_n})$, and then we eventually get from \eqref{5-6} and \eqref{5-100} that
\be\lab{5-128}\begin{aligned} 
	2\ln\f{1}{r_{n,l}^{(\eta_0)}}&=-t_{n,l}(r_{n,l}^{(\eta_0)})+O(1)+(1+o(1))\ga_{n,l}^{p_n}+2\al\ln|x_{n,l}|\\
		&\le(1-\eta_0+o(1))\ga_{n,l}^{p_n},
\end{aligned}\ee
for large $n$ and all $l$. 
Then, we may get from \eqref{5-116} and \eqref{5-128} that
\be\lab{5-129}\begin{aligned}  
	\sbr{\ga_{n,l}^{1-p_n}\ln\f{6\tau\nu_{n,i}}{|\cdot-\phi_{n,i}(q_{n,l})|}}^{p_n}
	&\le\sbr{\ga_{n,l}^{-p_n}\ln\f{12\tau\nu_{n,i}}{r_{n,l}^{(\eta_0)}}}^{p_n-1}\ln\f{6\tau\nu_{n,i}}{|\cdot-\phi_{n,i}(q_{n,l})|}\\
	&\le \f{1}{2}(1-\eta_0+o(1))\ln\f{6\tau\nu_{n,i}}{|\cdot-\phi_{n,i}(q_{n,l})|},
\end{aligned}\ee
uniformly in $D_{n,j}$ for large $n$ and $l\in I_{n,j}(\tau\nu_{n,i})$. We compute and then get from \eqref{5-127}, \eqref{5-129}, \eqref{5-114} and \eqref{hi} that
\be\lab{5-130}\begin{aligned} 
	\IR_n(D_{n,j})&=O\Bigg( \la_n\Ga_n^{p_n-1}e^{\Ga_n^{p_n}}\int_{B_{\tau\nu_{n,i}/2}\sbr{\phi_{n,i}(q_{n,j})}} \ln\f{2R\nu_{n,j}}{|y-z_n|}|y-x_{n,i}|^{2\al}\\
	&\quad\quad\quad \prod_{l\in I_{n,j}(\tau\nu_{n,i})}\sbr{\f{6\tau\nu_{n,i}}{|\cdot-\phi_{n,i}(q_{n,l})|}}^{201(1+N')^2(1-\eta_0)} \rd x \Bigg)\\
	&=O\sbr{\la_n\Ga_n^{p_n-1}e^{\Ga_n^{p_n}}\nu_{n,i}^2\sbr{\nu_{n,i}+|x_{n,i}|}^{2\al}},
\end{aligned}\ee
for large $n$ and $j\in J_i$, using that $201(1+N')^2(1-\eta_0)<1$ to get the last estimate.
At last, it readily follows from \eqref{5-112} that
\be\lab{5-131} \IR_n(D_n)=O\sbr{\la_n\Ga_n^{p_n-1}e^{\Ga_n^{p_n}}\nu_{n,i}^2\sbr{\nu_{n,i}+|x_{n,i}|}^{2\al}}, \ee
for large $n$.
We are now in position to deduce that
\be\lab{5-132}\begin{aligned} 
	\IR_n\sbr{B_{R\nu_{n,i}}(\phi_{n,i}(q_{n,l}))}\le&\sum_{l\in J_i}\f{4+o(1)}{p_n\ga_{n,l}^{p_n-1}}\ln\f{6\tau\nu_{n,i}}{|\phi_{n,i}(q_{n,l})-z_n|}  +o\sbr{\Ga_n^{1-p_n}},
\end{aligned}\ee
for large $n$. We discuss in three cases. 

Firstly, if $\nu_{n,i}=o(|x_{n,i}|)$, then we have $B_{R\nu_{n,i}}(0)\cap B_{\tau\nu_{n,i}/3}(x_{n,i})=\emptyset$ for large $n$, so that, recalling that $I_{n,j}(\tau\nu_{n,i})\subset J_i$ for any $j\in J_i$, we have that
\be\lab{5-134} B_{R\nu_{n,i}}(0)\subset D_n\bigcup\sbr{\bigcup_{j\in J_i}D_{n,j}}\bigcup\sbr{\bigcup_{l\in J_i}B_{r_{n,l}^{(\eta_0)}}(\phi_{n,i}(q_{n,l}))}, \ee
for large $n$. Then, we get from \eqref{5-126}, \eqref{5-130} and \eqref{5-131} and $\nu_{n,i}=o(|x_{n,i}|)$ that
\be\lab{5-133}\begin{aligned} 
	\IR_n\sbr{B_{R\nu_{n,i}}(\phi_{n,i}(q_{n,l}))}\le&\sum_{l\in J_i}\f{4+o(1)}{p_n\ga_{n,l}^{p_n-1}}\sbr{\ln\f{6\tau\nu_{n,i}}{|\phi_{n,i}(q_{n,l})-z_n|}+O(1)} \\
	& +O\sbr{\la_n\Ga_n^{p_n-1}e^{\Ga_n^{p_n}}\nu_{n,i}^2|x_{n,i}|^{2\al}},
\end{aligned}\ee
for large $n$. By the estimate \eqref{5-118-0}, we get that $\nu_{n,i}^{3/2}=o(\Ga_n^{1-p_n})$, and then, evaluating \eqref{5-70} at $\tau\nu_{n,i}$ and by \eqref{5-111}, we get that
	$$\Ga_n\le\f{2}{p_n\ga_{n,i}^{p_n-1}}\sbr{\ln\f{1}{\la_n\ga_{n,i}^{2(p_n-1)}|x_{n,i}|^{2\al}\nu_{n,i}^2}+O(1)}+o(\Ga_n^{1-p_n}),$$
then with \eqref{5-114} that
	$$e^{\Ga_n^{p_n}}\le \exp\sbr{\f{2\Ga_n^{p_n-1}}{p_n\ga_{n,i}^{p_n-1}}\ln\f{1}{\la_n\ga_{n,i}^{2(p_n-1)}|x_{n,i}|^{2\al}\nu_{n,i}^2}+o(1)},$$
that
	$$\la_n\ga_{n,i}^{2(p_n-1)}|x_{n,i}|^{2\al}\le\exp\sbr{-(1+o(1))\Ga_n\ga_{n,i}^{p_n-1}},$$
and eventually that
\be\lab{5-135} \la_n\Ga_n^{p_n-1}e^{\Ga_n^{p_n}}\nu_{n,i}^2|x_{n,i}|^{2\al}= o(\Ga_n^{1-p_n}),\ee
for large $n$. Plugging \eqref{5-135} and \eqref{5-114} in \eqref{5-133} gives \eqref{5-132} in the first case. 

Secondly, if $|x_{n,i}|=O(\nu_{n,i})$ and $d_{g_0}(q_0,\lbr{q_{n,l}:~l\in J_i})=o(\nu_{n,i})$,  then $B_{\tau\nu_{n,i}/3}(x_{n,i})\subset B_{\tau\nu_{n,i}/2}(\phi_{n,i}(q_{n,j_0}))$ for some $j_0\in J_i$, so that we still have \eqref{5-134}. Then, the same argument between \eqref{5-133} and \eqref{5-135} gives \eqref{5-132} in the second case. 

Finally, we assume that $|x_{n,i}|=O(\nu_{n,i})$ but $\nu_{n,i}=O(d_{g_0}(q_0,\lbr{q_{n,l}:~l\in J_i}))$ for large $n$.  This means that $B_{\tau\nu_{n,i}/3}(x_{n,i})\subset B_{R\nu_{n,i}}(0)$ and $B_{\tau\nu_{n,i}/3}(x_{n,i})\cap B_{\tau\nu_{n,i}/3}(\phi_{n,i}(q_{n,l}))=\emptyset$ for all $j\in J_i$ and large $n$. Noting that, by \eqref{5-112} and by the choice of $\tau$, we have that $u_{n,i}=\Ga_n+O(\Ga_n^{1-p_n})$ on $\pa B_{\tau\nu_{n,i}/2}(x_{n,i})$. Then, by the argument between \eqref{5-47} and \eqref{5-54-1}, and using \eqref{5-97} and the assumption \ref{5-0}, we get that 
	$$u_{n,i}=\Ga_n+O(\Ga_n^{1-p_n})$$
uniformly in $B_{\tau\nu_{n,i}/2}(x_{n,i})$ for large $n$. This gives that
\be\lab{5-136-0} \IR_n\sbr{B_{\tau\nu_{n,i}/2}(x_{n,i})}=O\sbr{\la_n\Ga_n^{p_n-1}e^{\Ga_n^{p_n}}\nu_{n,i}^2\sbr{\nu_{n,i}+|x_{n,i}|}^{2\al}}, \ee
so that the same argument between \eqref{5-133} and \eqref{5-135} gives again \eqref{5-132} in the final case.

We are now in position to conclude the proof of \eqref{5-105}. Plugging \eqref{5-118}, \eqref{5-120} and \eqref{5-132} in \eqref{5-121}, and by the choice \eqref{5-120-1} of $z_n$, we get that
\be\lab{5-136}\begin{aligned} 
	\abs{u_{n,i}-\bar u_{n,i}(R\nu_{n,i})}&\le(2\pi C_2+o(1))\Ga_n^{1-p_n}+\sum_{l\in J_i}\f{4+o(1)}{p_n\ga_{n,l}^{p_n-1}}\ln\f{6\tau\nu_{n,i}}{|\cdot-\phi_{n,i}(q_{n,l})|} ,
\end{aligned}\ee
in $B_{R\nu_{n,i}}(0)\setminus\cup_{l\in J_i}B_{r_{n,l}^{(\eta_0)}}\sbr{\phi_{n,i}(q_{n,l})}$. Then, using \eqref{5-114}, we get from \eqref{5-136} and \eqref{5-101} that
\be\lab{5-137} \abs{\bar u_{n,i}(\nu_{n,i})-\bar u_{n,i}(R\nu_{n,i})}\le(2\pi C_2+o(1))\Ga_n^{1-p_n}, \ee
for large $n$. Then, $p_0=2$, \eqref{5-136} and \eqref{5-137} together imply that
\be\lab{5-138} \abs{u_{n,i}-\bar u_{n,i}(\nu_{n,i})}\le\f{9}{4}\pi C_2\Ga_n^{1-p_n}+\sum_{l\in I_{n,i}(\nu_{n,i})}\f{2+o(1)}{\ga_{n,l}^{p_n-1}}\ln\f{6\tau\nu_{n,i}}{|\cdot-\phi_{n,i}(q_{n,l})|}, \ee
in $B_{\nu_{n,i}}(0)\setminus\cup_{l\in I_{n,i}(\nu_{n,i})}B_{r_{n,l}^{(\eta_0)}}\sbr{\phi_{n,i}(q_{n,l})}$ for large $n$. But by \eqref{5-104}, the assumption \eqref{5-106} with \eqref{conv-0}, the inequality in \eqref{5-102} for $j=i$ and $r=\nu_{n,i}$ should be an equality somewhere on $\pa B_{\nu_{n,i}}(0)$, which leads to a contradiction to \eqref{5-138} and concludes the proof of \eqref{5-105}. 
\end{proof}

Now, we prove the quantization result.
\begin{lemma}\lab{lemma5-6}
Let $p_0=2$ in \eqref{5-81}, then there holds
	$$\lim_{n\to+\iy}\beta_n(B_{\kappa_0}(q_0))=4\pi N',$$ 
where $N'$ is given above \eqref{5-0} and $\beta_n(\cdot)$ is given by \eqref{betadomain}.
\end{lemma}
\begin{proof}
Picking a sequence $(\ti\Ga_n)_n$ such that $\lim_{n\to+\iy}\ti\Ga_n=+\iy$ and $\ti\Ga_n=o(\ga_{n,j})$, and setting
	$$\ti\nu_{n,j}:=\sup\lbr{r\in[0,\kappa]:~\bar u_{n,j}\ge\ti\Ga_n~\text{in}~[0,r]},$$
we get from \eqref{5-105} that (letting $\eta_0$ be close to 1)
\be\lab{5-140} \ti\nu_{n,j}\le\nu_{n,j} \ee
for all $j\in\{1,\cdots,N'\}$ and large $n$.
By \eqref{conv-0}, $\ti\nu_{n,j}=o(1)$, and hence $\bar u_{n,j}(\ti\nu_{n,j})=\ti\Ga_n$ for all $j$. 
By the argument below \eqref{5-106}, we may fix a large $R\gg1$ such that 
	$$d\sbr{\pa B_{R\ti\nu_{n,j}}(0),\{x_{n,j}\}\cup\lbr{\phi_{n,j}(q_{n,l}):~l=1,\cdots,N'}}\ge\f{\ti\nu_{n,j}}{R},$$
then, estimating $\bar u_{n,j}(\ti\nu_{n,j}/R)$ by applying \eqref{5-102} with $r=\ti\nu_{n,j}$, and using \eqref{est-4}, we get that
\be\lab{5-141} u_n=\ti\Ga_n+o(1),\ee
uniformly in $\phi_{n,j}^{-1}\sbr{\pa B_{R\ti\nu_{n,j}}(0)}$ for large $n$ and all $j$. Arguing as below \eqref{5-90}, we get from \eqref{5-141} and the assumption \eqref{5-0} that
\be\lab{5-142} u_n\le 2\ti\Ga_n \quad\text{unifomly in}~B_{\kappa_0}(q_0)\setminus\bigcup_{j=1}^{N'}\phi_{n,j}^{-1}\sbr{B_{R\ti\nu_{n,j}}(0)}\ee
for large $n$. 
We claim that
\be\lab{5-143} (2-p_n)\ln\sbr{1+\la_n\ga_{n,j}^{2(p_n-1)}}=o(1) \ee
for all $j$. Indeed, if $(2-p_n)\ga_{n,j}^{p_n}\le 4(p_n-1)\ln\ga_{n,j}$, then by \eqref{5-80} we have  that 
	$$(2-p_n)\ln\sbr{1+\la_n\ga_{n,j}^{2(p_n-1)}}\le \f{16(\ln\ga_{n,j})^2}{\ga_{n,j}^{p_n}}=o(1);$$
otherwise, we have $(2-p_n)\ga_{n,j}^{p_n}\ge 4(p_n-1)\ln\ga_{n,j}$, and then by \eqref{5-79} we get that $\la_n\ga_{n,j}^{2(p_n-1)}=O(1)$, which concludes the proof of \eqref{5-143}.

Then, by using \eqref{5-80} and \eqref{5-143}, we may choose and fix $\ti\Ga_n$ growing slowly to $+\iy$, such that
\be\lab{5-144}\begin{aligned} 
	\la_n\ti\Ga_n^{p_n}&e^{(2\ti\Ga_n)^{p_n}}=o(\ga_{n,j}^{2-p_n}) \quad\text{and}\\
	&(2-p_n)\ln\sbr{1+\la_n\ga_{n,j}^{2(p_n-1)}e^{(2\ti\ga_n)^{p_n}} }=o(1) 
\end{aligned}\ee
for large $n$ and all $j$. 

We may now compute and use either \eqref{5-142} in $B_{\kappa_0}(q_0)\setminus\bigcup_{j=1}^{N'}\phi_{n,j}^{-1}\sbr{B_{R\ti\nu_{n,j}}(0)}$, or the controls given by the inequality in \eqref{5-2} for $r=\ti\nu_{n,j}$ thanks to \eqref{5-140}, allowing to estimate the nonlinearity as in \eqref{5-127}-\eqref{5-130} and \eqref{5-136-0}, and to estimate the $\ti\nu_{n,j}$ as in \eqref{5-118-0}. This leads to the following estimates:
\be\lab{5-145}\begin{aligned}
	&\f{\la_np_n^2}{2}\int_{B_{\kappa_0}(q_0)\setminus\bigcup_{j=1}^{N'}\phi_{n,j}^{-1}\sbr{B_{r_{n,j}^{(\eta_0)}}(0)}}hfu_n^{p_n}e^{u_n^{p_n}}\rd v_{g_0}=o(\ga_{n,j}^{2-p_n}),\\
	&\f{\la_np_n^2}{2}\int_{B_{\kappa_0}(q_0)\setminus\bigcup_{j=1}^{N'}\phi_{n,j}^{-1}\sbr{B_{r_{n,j}^{(\eta_0)}}(0)}} hf\sbr{e^{u_n^{p_n}}-1}\rd v_{g_0}=O\sbr{\la_ne^{(2\ti\Ga_n)^{p_n}}},
\end{aligned}\ee
while, computing as in \eqref{5-126}, we get that
\be\lab{5-146}\begin{aligned}
	&\f{\la_np_n^2}{2}\ga_{n,j}^{p_n-2}\int_{\phi_{n,j}^{-1}\sbr{B_{r_{n,j}^{(\eta_0)}}(0)}}hfu_n^{p_n}e^{u_n^{p_n}}\rd v_{g_0}=4\pi+o(1),\\
	&\f{\la_np_n^2}{2}\ga_{n,j}^{2(p_n-1)}\int_{\phi_{n,j}^{-1}\sbr{B_{r_{n,j}^{(\eta_0)}}(0)}} hf\sbr{e^{u_n^{p_n}}-1}\rd v_{g_0}=4\pi+o(1),
\end{aligned}\ee
for large $n$. Thus, by plugging \eqref{5-125} and \eqref{5-126} in \eqref{betadomain}, and by using the condition \eqref{5-144} on $\ti\Ga_n$, we get that 
\be\lab{5-147}\begin{aligned} 
	&\beta_n(B_{\kappa_0}(q_0))\\
	&=\sbr{\sum_{j=1}^{N'}\f{4\pi+o(1)}{\ga_{n,j}^{2(p_n-1)}}+\la_ne^{(2\ti\Ga_n)^{p_n}}}^{\f{2-p_n}{p_n}} \sbr{\sum_{j=1}^{N'}(4\pi+o(1))\ga_{n,j}^{p_n-2}}^{\f{2(p_n-1)}{p_n}}\\
	&=(4\pi+o(1))\sbr{1+\sum_{j\neq j_0}\sbr{\f{\ga_{n,j_0}}{\ga_{n,j}}}^{2(p_n-1)} }^{\f{2-p_n}{p_n}} \sbr{ 1+\sum_{j\neq j_0}\sbr{\f{\ga_{n,j}}{\ga_{n,j_0}}}^{2-p_n} }^{\f{2(p_n-1)}{p_n}}\\
	&=(4\pi+o(1))\sbr{1+\sum_{j\neq j_0}\sbr{\f{\ga_{n,j}}{\ga_{n,j_0}}}^{2-p_n} }^{\f{2(p_n-1)}{p_n}}
\end{aligned}\ee
where $j_0\in\{1,\cdots,N'\}$ is chosen such that $\ga_{n,j_0}=\min\lbr{\ga_{n,j}:~j=1,\cdots,N'}$ for large $n$. 
Noting that \eqref{5-85-1} still holds for $l=j_0$ and $\eta=\eta_0$, and by \eqref{5-79}, we get that
	$$(2-p_n)\ga_{n,j}^{p_n}\le 2(1-\eta_0+o(1))\ga_{n,j_0}^{p_n},$$
so that $\sbr{\ga_{n,j}/\ga_{n,j_0}}^{2-p_n}=1+o(1)$, for all $j$. Plugging this estimate in the last line of \eqref{5-147}, we conclude the proof of Lemma \ref{lemma5-6}.
\end{proof}

\vs
\subsection{Compactness for critical level \texorpdfstring{$4\pi N'$}{}}\lab{sec5-4}\ 

In this subsection, we assume that 
\be\lab{5-150} p_n\equiv p\in(1,2] \ee 
for all $n$. 
By Lemma \ref{lemma5-4} and \ref{lemma5-6}, we see that 
  $$\lim_{n\to+\iy}\beta_n(B_{\kappa_0}(q_0))=4\pi N',$$ 
where $N'$ is given above \eqref{5-0} and $\beta_n(\cdot)$ is given by \eqref{betadomain}.
Here, we prove a more accurate estimates of the energy.

\begin{lemma}\lab{lemma5-7}
Let \eqref{5-150} hold, then we have that
\be\lab{5-151} \beta_n(B_{\kappa_0}(q_0))\ge 4\pi N'+\f{4(p-1)}{p^2}(4\pi+o(1))\sum_{i=1}^{N'}\f{1}{\ga_{n,i}^{2p}} \ee 
for large $n$. 
\end{lemma}
\begin{proof}
Let now $\bar r_{n,i}$ be given by
\be\lab{5-152} t_{n,i}(\bar r_{n,i})=\sqrt{\ga_{n,i}} \ee
for any $i\in\{1,\cdots,N'\}$ and all $n$. Choosing $\eta=\f12$ in \eqref{5-6} and \eqref{5-7}, we get from Lemma \ref{lemma5-1} that $\bar r_{n,i}^{(1/2)}=r_{n,l}^{(1/2)}$ for any $i\in\{1,\cdots,N'\}$ and large $n$.  
Moreover,  since $|x_{n,i}|=o(1)$ according to \eqref{xn} and \eqref{vpn}, we get from \eqref{5-2}  that
\be\lab{5-153} r_{n,i}^{(1/2)}=O(r_{n,i})=O(|x_{n,i}|)=o(1) \ee
for large $n$ and all $i\in\{1,\cdots,N'\}$. By \eqref{5-6} and \eqref{5-152}, we first deduce that
\be\lab{5-153-1} 2\ln\f{\bar r_{n,i}}{r_{n,i}^{(1/2)}}=t_{n,i}(\bar r_{n,i})-t_{n,i}(r_{n,i}^{(1/2)})+o(1)\le-3\ga_{n,i},\ee
and then, we find from \eqref{5-153} that
\be\lab{5-154} \bar r_{n,i}=O(|x_{n,i}|e^{-\ga_{n,i}})=o(e^{-\ga_{n,i}})  \ee
for large $n$ and all $i\in\{1,\cdots,N'\}$.
An easy consequence of \eqref{vpn}, \eqref{5-1} and \eqref{5-153} is that any two of the domains $\phi_{n,i}^{-1}(B_{\bar r_{n,i}}(0))$, $i=1,\cdots, N'$, are disjoint for large $n$. Then, by \eqref{betadomain}, we may write that
\be\lab{5-155} \begin{aligned}
  \beta_n(B_{\kappa_0}(q_0))
  &\ge \sbr{\sum_{i=1}^{N'}\f{\la_np^2}{2}\int_{B_{\bar r_{n,i}}(0)}h_{n,i}f_{n,i}(e^{u_{n,i}^p}-1)\rd v_{g_0}}^{\f{2-p}{p}} \\
    &\qquad\qquad \times\sbr{\sum_{i=1}^{N'}\f{\la_np^2}{2}\int_{B_{\bar r_{n,i}}(0)}h_{n,i}f_{n,i}u_{n,i}^pe^{u_{n,i}^p}\rd v_{g_0}}^{\f{2(p-1)}{p}}\\
  &\ge \sum_{i=1}^{N'} \left[ \sbr{\f{\la_np^2}{2}\int_{B_{\bar r_{n,i}}(0)}h_{n,i}f_{n,i}(e^{u_{n,i}^p}-1)\rd v_{g_0}}^{\f{2-p}{p}} \right. \\
    &\qquad\qquad \left. \times\sbr{\f{\la_np^2}{2}\int_{B_{\bar r_{n,i}}(0)}h_{n,i}f_{n,i}u_{n,i}^pe^{u_{n,i}^p}\rd v_{g_0}}^{\f{2(p-1)}{p}}\right]
\end{aligned}\ee
for large $n$, where in the second inequaltiy we use H\"older inequality for vectors in $\R^{N'}$. 
By using \eqref{5-153-1} and \eqref{5-10} with $\eta=\f12$, we get that
  $$\abs{u_{n,i}-v_{n,i}}=O\sbr{\f{\bar r_{n,i}}{\ga_{n,i}^{p-1}r_{n,i}^{(1/2)}}}=o(e^{-\ga_{n,i}})$$
uniformly in $B_{\bar r_{n,i}}(0)$ for large $n$ and all $i$. 
Then, using Proposition \ref{single-2} to get that $v_{n,i}=(1+o(1))\ga_{n,i}$, we obtain that
\be\lab{5-156}     u_{n,i}^{p}=v_{n,i}^p\sbr{1+O\sbr{\f{e^{-\ga_{n,i}}}{\ga_{n,i}}} } \quad\text{and}\quad  e^{u_{n,i}^{p}}=e^{v_{n,i}^p}\sbr{1+o\sbr{\f{1}{\ga_{n,i}^{2p}}} } \ee
uniformly in $B_{\bar r_{n,i}}(0)$ for large $n$ and all $i$. 
By noting that $\bar r_{n,i}=o(|x_{n,i}|)$ by \eqref{5-153} and \eqref{5-153-1}, and by using \eqref{5-154}, we get first from \eqref{hi} that
  $$h_{n,i}f_{n,i}=\ti h_{n,i}(0)f_{n,i}(0)|x_{n,i}|^{2\al}(1+o(e^{-\ga_{n,i}}))$$
uniformly in $B_{\bar r_{n,i}}(0)$, and then, by using \eqref{5-156}, we get that
\be\lab{5-157} \begin{aligned}  
  &\f{\la_np^2}{2}\int_{B_{\bar r_{n,i}}(0)}h_{n,i}f_{n,i}(e^{u_{n,i}^p}-1)\rd v_{g_0}\\
  &=\f{\la_np^2}{2}\int_{B_{\bar r_{n,i}}(0)}\ti \tau_{n,i}|x_{n,i}|^{2\al}e^{v_{n,i}^p}\rd v_{g_0}\sbr{1+o(\ga_{n,i}^{-2p})}+O\sbr{\la_n\bar r_{n,i}^2|x_{n,i}|^{2\al}}\\
  &=\f{\la_np^2}{2}\int_{B_{\bar r_{n,i}}(0)}\ti \tau_{n,i}|x_{n,i}|^{2\al}e^{v_{n,i}^p}\rd v_{g_0}\sbr{1+o(\ga_{n,i}^{-2p})}
\end{aligned}\ee
with $\ti \tau_{n,i}=\ti h_{n,i}(0)f_{n,i}(0)$ for large $n$ and all $i$, where in the second equality we use \eqref{boundla}, the first equality in \eqref{5-154} and the fact that
  $$\f{\la_np^2}{2}\ga_{n,i}^{2(p-1)}\int_{B_{\bar r_{n,i}}(0)}\ti \tau_{n,i}|x_{n,i}|^{2\al}e^{v_{n,i}^p}\rd v_{g_0}=O(1)$$
by the argument to get \eqref{5-146}. Similar arguments give that
\be\lab{5-158} \begin{aligned}  
  &\f{\la_np^2}{2}\int_{B_{\bar r_{n,i}}(0)}h_{n,i}f_{n,i}u_{n,i}^pe^{u_{n,i}^p}\rd v_{g_0}\\
  &=\f{\la_np^2}{2}\int_{B_{\bar r_{n,i}}(0)}\ti \tau_{n,i}|x_{n,i}|^{2\al}v_{n,i}^pe^{v_{n,i}^p}\rd v_{g_0}\sbr{1+o(\ga_{n,i}^{-2p})}
\end{aligned}\ee
for large $n$ and all $i$.
By plugging \eqref{5-157} and \eqref{5-158} in \eqref{5-155}, and by applying Proposition \ref{single-3} to $v_{n,i}$, we conclude the proof of \eqref{5-151}.
\end{proof}

\vs

We are ready to prove Theorem \ref{type1}.
\begin{proof}[\bf Proof of Theorem \ref{type1}]
By Corollary \ref{lemma5-3}, we get $\lim_{n\to+\iy}\la_n=0$. The conclusion \eqref{conclusion-1} is given by Lemma \ref{lemma5-4} and \ref{lemma5-6}, and conclusion \eqref{conclusion-2} is given by Lemma \ref{lemma5-7}. 
The proof is complete.
\end{proof}

\vs
\section{Compactness for Type II}\lab{sec6}
In this section, we prove the compactness result at a singularity for {\bf Type II} defined in Section \ref{sec2-3}; see Theorem \ref{type2}. Precisely, let $q_0\in\Sigma_\iy\cap\DR$ be fixed, and we assume that {\bf Type II} happens at $q_0$ and so $q_0\in\DR'$, that is, there exists a non-negative integer $0\le N'\le N$, such that
	$$\lim_{n\to+\iy}q_{n,i}=q_0,~\forall~i\in\{1,\cdots,N'\}\;(\text{if}~N'\ge1),$$ 
and
	$$\lim_{n\to+\iy}q_{n,i}\neq q_0,~\forall~i\in\{N'+1,\cdots,N\},$$
where $q_{n,i}$ and $N$ are given by Proposition \ref{bubble}, and such that 
\be\lab{6-0} \text{the alternative in conclusion (2) of Proposition \ref{bubble} does hold.}\ee

\begin{theorem}\lab{type2}
Let $q_0\in\Sigma_\iy\cap\DR$ be a singularity with order $\al\in(-1,0)$, and $\kappa_0>0$ such that $B_{2\kappa_0}(q_0)\cap(\DR\cup\Sigma_\iy)=\{q_0\}$. Assume that {\bf Type II} happens at $q_0$, i.e., $q_0\in\DR'$. Then $\displaystyle\lim_{n\to+\iy}\la_n=0$ and 
\be\lab{conclusion-3} \lim_{n\to+\iy}\beta_n(B_{\kappa_0}(q_0))=4\pi(1+\al)+4\pi N', \ee
where $\beta_n(\cdot)$ is given by \eqref{betadomain} and $N'$ is given as above.

Moreover, if $p_n\equiv p\in(1,2]$ for all $n$, then
\be\lab{conclusion-4} \beta_n(B_{\kappa_0}(q_0))>4\pi(1+\al)+4\pi N' \ee
for large $n$.
\end{theorem}

In this section, for brevity, we use some of the same symbols as in Section \ref{sec5}, which may have different definitions.
Let $q_{n,0}\equiv q_0$ for all $n$, and for $i\in\{0,\cdots,N'\}$, let the isothermal coordinates $(B_{\kappa_0}(q_{n,i}),\phi_{n,i},U_{n,i})$ be given in the paragraph before Proposition \ref{bubble}. 
Thanks to \eqref{equ-5}, \eqref{ufh} and  $\Delta_{g_0}=e^{-2\vp_{n,i}}\Delta$, we get 
\be\lab{equ-7} -\Delta u_{n,i}+h_{n,i}u_{n,i}=\la_np_nh_{n,i}f_{n,i}u_{n,i}^{p_n-1}e^{u_{n,i}^{p_n}} \quad\text{in}~B_{2\kappa}(0),\ee
for all $i\in\{0,\cdots,N'\}$ and $n$, where $\Delta$ is the Laplace operator in $\R^2$. Denote
\be\lab{6-xn} x_{n,i}:=\phi_{n,i}(q_0),~~\forall~i\in\{0,\cdots,N'\}. \ee
Then $x_{n,0}=0$. 
Recalling the assumption \eqref{hhh} on $h$, we get that
\be\lab{6-hi} h_{n,i}=\ti h_{n,i}~|\cdot-x_{n,i}|^{2\al}\quad\text{in}~B_{2\kappa}(0),\ee
where $\ti h_{n,i}$ is positive $C^1$ function for all $i\in\{0,\cdots,N'\}$.

For all $i\in\{0,\cdots,N'\}$, we set
\be\lab{6-1} \begin{aligned}
  r_{n,0}&:=\begin{cases}\kappa_0,\quad&\text{if}~N'=0,\\
    \f{1}{3}d_{g_0}(q_{n,0},(\QR_n\cup\{q_{n,0}\})\setminus\{q_{n,0}\}),\quad&\text{if}~N'\ge1, \end{cases} \qquad\text{and} \\
  r_{n,i}&:=\f{1}{3}d_{g_0}(q_{n,i},(\QR_n\cup\{q_{n,0}\})\setminus\{q_{n,i}\}),\quad\text{for}~i\in\{1,\cdots,N'\}(\text{if}~N'\ge1),
\end{aligned}\ee
for all $n$, where $\QR_n$ is given by
	$$\QR_n:=\begin{cases} \emptyset,\quad&\text{if}~N'=0,\\ \{q_{n,1},\cdots,q_{n,N'}\},\quad&\text{if}~N'\ge1. \end{cases}$$  
By \eqref{vpn}, the same argument as \eqref{5-2} implies
\be\lab{6-2} r_{n,i}\le\f{1}{2}|x_{n,i}|,\quad\forall~i\in\{1,\cdots,N'\}\;(\text{if}~N'\ge1), \ee
for large $n$. Let $\ga_{n,i}:=u_n(q_{n,i})$, and let $\mu_{n,0}$ be given by \eqref{mun-2} and $\mu_{n,i}$ for $i\ge1$ be given by \eqref{mun-1}. Then by \eqref{dist-1} and \eqref{dist-2}, we get that $\gamma_{n,i}\to+\infty$, $\mu_{n,i}\to 0$,
\be\lab{6-3}\begin{aligned} 
	\lim_{n\to+\iy}\f{\mu_{n,i}}{r_{n,i}}&=0,~~\forall~i\in\{0,\cdots,N'\},\quad\text{and}\\
	\lim_{n\to+\iy}\f{\mu_{n,i}}{|x_{n,i}|}&=0,~~\forall~i\in\{1,\cdots,N'\}\;(\text{if}~N'\ge1). 
\end{aligned}\ee 
Furthermore, we get from \eqref{conv-2} that
\be\lab{6-4-0} \lim_{n\to+\iy}\f{p_n}{2}\ga_{n,0}^{p_n-1}(\ga_{n,0}-u_{n,0}(\mu_{n,0}\cdot))=T_\al:=\ln(1+|\cdot|^{2(1+\al)}),\ee
in $C_{\loc}^{0,\delta}(\R^2)\cap C^1_\loc(\R^2\setminus\{0\})$ for any $\delta\in(0,\min\{1,2(1+\al)\})$, 
and from \eqref{conv-1} that
\be\lab{6-4} \lim_{n\to+\iy}\f{p_n}{2}\ga_{n,i}^{p_n-1}(\ga_{n,i}-u_{n,i}(\mu_{n,i}\cdot))=T_0:=\ln(1+|\cdot|^2)~\text{in}~C_{\loc}^1(\R^2),\ee
for every $i\in\{1,\cdots,N'\}$ (if $N'\ge1$).
By conclusion (3) of Proposition \ref{bubble}, we get that there exist $C_1,C_2>0$ such that
\be\lab{est-5} d_{g_0}(~\cdot~,\QR_n\cup\{q_{n,0}\})^2\la_nhfu_n^{2(p_n-1)}e^{u_n^{p_n}}\le C_1 \quad\text{in}~B_{\kappa_0}(q_0)\setminus B_{\mu_{n,0}(q_0)}, \ee
\be\lab{est-6} d_{g_0}(~\cdot~,\QR_n\cup\{q_{n,0}\})u_n^{p_n-1}|\nabla_{g_0} u_n|\le C_2 \quad\text{in}~B_{\kappa_0}(q_0)\setminus B_{\mu_{n,0}(q_0)}, \ee
for all $n$. 
We set 
	$$t_{n,i}:=\begin{cases} \ln\sbr{1+\f{|\cdot|^{2(1+\al)}}{\mu_{n,i}^{2(1+\al)}}},\quad&\text{if}~i=0,\\ \ln\sbr{1+\f{|\cdot|^{2}}{\mu_{n,i}^{2}}},\quad&\text{if}~i\ge1. \end{cases}$$
We set also
\be\lab{6-5} v_{n,0}:=\BR_{n}, \ee
where $\BR_n$ is as in Appendix \ref{appe1-1} for $\tau_{n,0}=\ti h_{n,0}(0)f_{n,0}(0)$, $\nu_{n,0}=\ti h_{n,0}(0)$, $\ga_{n}=\ga_{n,0}$ and $\mu_n=\mu_{n,0}$ for all $n$; we set also
\be\lab{6-5-1} v_{n,i}:=\BR_n, \quad\text{for}~i\in\{1,\cdots,N'\} \;(\text{if}~N'\ge1), \ee
where $\BR_n$ is as in Appendix \ref{appe1-2} for $\tau_{n,i}=h_{n,i}(0)f_{n,i}(0)$, $\nu_{n,i}=h_{n,i}(0)$, $\ga_{n}=\ga_{n,i}$, $\mu_n=\mu_{n,i}$ and $x_{n}=x_{n,i}$ for all $n$.

To use the estimates of Section \ref{appe2} and Appendix \ref{appe1}, we need some preliminary observations. Let $l\in\{0,\cdots,N'\}$ be given. Given a parameter $\eta\in(0,1)$ that is going to take several values in the proofs below, we let $r_{n,l}^{(\eta)}$ be given by
\be\lab{6-6} t_{n,l}(r_{n,l}^{(\eta)})=\eta\f{p_n\ga_{n,l}^{p_n}}{2},\ee
and, for $r_{n,l}$ as in \eqref{6-1}, we set
\be\lab{6-7} \bar r_{n,l}^{(\eta)}:=\min\lbr{r_{n,l},r_{n,l}^{(\eta)}}. \ee 
By collecting the above preliminary information, we can check that Proposition \ref{singular} applies with $\bar r_n=\bar r_{n,0}^{(\eta)}$, $f_n=f_{n,0}$, $h_n=\ti h_{n,0}$, $u_{n}=u_{n,0}$, $\ga_n=\ga_{n,0}$ and $v_n=v_{n,0}$. As a result, we get from Proposition \ref{singular} that
\be\lab{6-10-0}  \abs{u_{n,0}-v_{n,0}}=O\sbr{\f{|\cdot|}{\ga_{n,0}^{p_n-1}\bar r_{n,0}^{(\eta)} }}+O\sbr{\f{\mu_{n,0}^{\delta_0}}{\ga_{n,0}^{p_n-1}(\bar r_{n,0}^{(\eta)})^{\delta_0} }}\ee
uniformly in $B_{\bar r_{n,0}^{(\eta)}}(0)$, and
\be\lab{6-11-0} \abs{\nabla(u_{n,0}-v_{n,0})}=O\sbr{\f{1}{\ga_{n,0}^{p_n-1}\bar r_{n,0}^{(\eta)} }},\ee 
uniformly in $B_{\bar r_{n,0}^{(\eta)}}(0)\setminus B_{\mu_{n,0}}(0)$ for large $n$. 
In particular, we also get from \eqref{singular-0} that 
\be\lab{6-8-0} \ln\ga_{n,0}=o\sbr{\ln\f{1}{\bar r_{n,0}^{(\eta)}} },\ee 
which gives that $\ga_{n,0}^s(\bar r_{n,l}^{(\eta)})^{2+2\al}=O(1)$ for any $s>1$, so that Proposition \ref{single-1} also applies, and we get
\be\lab{6-9-0}\begin{aligned} 
	\ga_{n,0}\ge v_{n,0}&=\ga_{n,l}\sbr{1-\f{2t_{n,0}(1+O(\ga_{n,0}^{-p_n}))}{p_n\ga_{n,0}^{p_n}}}\\
		&\ge(1-\eta)\ga_{n,0}+O(\ga_{n,0}^{1-p_n}), 
\end{aligned} \ee
uniformly in $[0,\bar r_{n,l}^{(\eta)}]$ and for $n$ large, using \eqref{6-6}.
 
Meanwhile, for $l\ge1$, we can check that Proposition \ref{regular} applies with $\bar r_n=\bar r_{n,l}^{(\eta)}$, $f_n=f_{n,l}$, $h_n=h_{n,l}$, $x_n=x_{n,l}$, $u_{n}=u_{n,l}$, $\ga_n=\ga_{n,l}$ and $v_n=v_{n,l}$. 
As a result, we get from Proposition \ref{regular} that
\be\lab{6-10}  \abs{u_{n,l}-v_{n,l}}=O\sbr{\f{|\cdot|}{\ga_{n,l}^{p_n-1}\bar r_{n,l}^{(\eta)} }}, \ee 
\be\lab{6-11} \abs{\nabla(u_{n,l}-v_{n,l})}=O\sbr{\f{1}{\ga_{n,l}^{p_n-1}\bar r_{n,l}^{(\eta)} }},\ee 
uniformly in $B_{\bar r_{n,l}^{(\eta)}}(0)$ for $n$ large, and that 
\be\lab{6-8} \ln\ga_{n,l}=o\sbr{\ln\f{1}{\bar r_{n,l}^{(\eta)}|x_{n,l}|^\al}},\ee 
which gives that $\ga_{n,l}^s(\bar r_{n,l}^{(\eta)})^2|x_{n,l}|^{2\al}=O(1)$ for any $s>1$, so that Proposition \ref{single-2} also applies, and we get
\be\lab{6-9}\begin{aligned} 
	\ga_{n,l}\ge v_{n,l}&=\ga_{n,l}\sbr{1-\f{2t_{n,l}(1+O(\ga_{n,l}^{-p_n}))}{p_n\ga_{n,l}^{p_n}}}\\
		&\ge(1-\eta)\ga_{n,l}+O(\ga_{n,l}^{1-p_n}), 
\end{aligned} \ee
uniformly in $[0,\bar r_{n,l}^{(\eta)}]$ for $n$ large, using \eqref{6-6}.

\vs
\subsection{A first estimate}\lab{sec6-1}\ 

At first, similarly as Section \ref{sec5-1}, we need to prove the following result:
\begin{lemma}\lab{lemma6-1}
For all $i\in\{0,\cdots,N'\}$, we have that
\be\lab{6-12} \liminf_{n\to+\iy}\f{2t_{n,i}(r_{n,i})}{p_n\ga_{n,i}^{p_n}}\ge1. \ee 
\end{lemma}
\begin{remark}\lab{611-1}
We compare the proof of Lemma \ref{lemma6-1} with that of Lemma \ref{lemma5-1}. Due to the assumption \eqref{6-0}, we have a singular bubble at the singularity $q_0$, instead of a ``hole'' in Lemma \ref{lemma5-1}. However, since we have no similar convergence results like \eqref{regular-3} and \eqref{regular-4} for the singular bubble (see Remark \ref{515-2}), the strategy from the last paragraph in the proof of Lemma \ref{lemma5-1} fails, and here we use the so-called {\it local Pohozaev identity} technique. To get the desired contradiction, we consider the local Pohozaev identity over a small ball centered at the singulartiy $q_0$, such that all bubbles contained in this ball are comparable (see \eqref{6-40}), and then, by computing each term carefully, we get an energy identity 
  $$\sbr{(1+\al)+\sum_{j\in J_0\setminus\{0\}}\theta_j }^2=(1+\al)\sbr{(1+\al)+\sum_{j\in J_0\setminus\{0\}}\theta_j^2},\quad \theta_j\in(0,+\iy),$$
which is impossible due to $\al\in(-1,0)$. This is the final place where our method highly and nontrivially relies on $\al<0$.

It is worth pointing out that  we can also use the local Pohozaev identity technique to give a different proof of Lemma \ref{lemma5-1}, and by this way, we would obtain the energy identity: either for $r_{n,i}=o(|x_{n,i}|)$,
  $$\sbr{1+\sum_{j\in J_i\setminus\{i\}}\theta_j }^2=\sbr{1+\sum_{j\in J_i\setminus\{i\}}\theta_j^2},\quad \theta_j\in(0,+\iy);$$
or for $r_{n,i}\sim|x_{n,i}|$,
  $$\sbr{1+\sum_{j\in J_i\setminus\{i\}}\theta_j }^2=(1+\al)\sbr{1+\sum_{j\in J_i\setminus\{i\}}\theta_j^2},\quad \theta_j\in(0,+\iy).$$
The second identity gives a contradiction again due to $\al\in(-1,0)$. Furthermore, the second identity also gives a contradiction if we have $1+\al<|J_i|$, and we think that this might be useful for positive $\al$, for example when $0<\beta<4\pi(1+\min_{a\in\DR}\al_a)$ with $\al_a>0$ for each $a\in\DR$.

\end{remark}
\begin{proof}[\bf Proof of Lemma \ref{lemma6-1}]
Up to renumbering, when $N'\ge1$, we may assume that
	$$r_{n,1}\le r_{n,2}\le \cdots\le r_{n,N'},$$
and assume that there exists a non-negative integer $N''\le N'$ such that $r_{n,N''}=o(r_{n,0})$ (if $N''\ge1$) and that $r_{n,0}=O(r_{n,N''+1})$ (if $N''\le N'-1$), for all large $n$.
We divide into three parts to prove \eqref{6-12}.

\vskip0.1in
\noindent{\it Part 1.} We first prove that \eqref{6-12} holds for all $i\in\{1,\cdots,N''\}$ (if $N''\ge1$). 

Observing the fact that for any $i\in\{1,\cdots,N''\}$, there hold that
	$$\big\{ j\in\{0,\cdots,N'\}:~d_{g_0}(q_{n,j},q_{n,i})=O(r_{n,i}) \big\}\subset\{1,\cdots,N''\} $$
and that $r_{n,i}=o(|x_{n,i}|)$ by the definition \eqref{6-1} of $r_{n,0}$, and by using the same arguments in the proof of Lemma \ref{lemma5-1}, and using \eqref{equ-7}, \eqref{6-4}, \eqref{est-6}, \eqref{6-6}, \eqref{6-7} and \eqref{6-10}-\eqref{6-9},  and using Proposition \ref{regular}, we can prove that \eqref{6-12} holds for all $i\in\{1,\cdots,N''\}$ (if $N''\ge1$). 

\vskip0.1in
\noindent{\it Part 2.} We next prove that
\be\lab{6-13} \liminf_{n\to+\iy}\f{2t_{n,0}(r_{n,0})}{p_n\ga_{n,0}^{p_n}}\ge1. \ee

By contradiction, we assume that \eqref{6-13} does not hold. Thus, by \eqref{6-6} and \eqref{6-7}, we may choose and fix $\eta\in(0,1)$ sufficiently close to 1 such that
\be\lab{6-14} \bar r_{n,0}^{(\eta)}=r_{n,0} \ee
for large $n$. Set $J_0=\lbr{ j\in\{0,\cdots,N'\}:~d_{g_0}(q_{n,j},q_{n,0})=O(r_{n,0}) }$.  Obviously, we get from \eqref{6-1} that
\be\lab{6-15} r_{n,l}=O(r_{n,0})\quad\forall~l\in J_0, \ee
for large $n$.
We also find from \eqref{6-8-0} and from \eqref{6-14} that $r_{n,0}=o(1)$, so we get from \eqref{vpn} that
\be\lab{6-16} \lim_{n\to+\iy}\mbr{(\phi_{n,l})_*g_0}(r_{n,0}\cdot)=|\rd x|^2~\text{in}~C_\loc^2(\R^2) \ee
for all $l\in J_0$. Up to a subsequence, we may assume that
	$$\lim_{n\to+\iy}\f{\phi_{n,0}(q_{n,l})}{r_{n,0}}=\hat x_l\in\R^2$$
for all $l\in J_0$, and we have that $\SR_0:=\lbr{\hat x_l,~l\in J_0}$ contains at least two distinct points and that $|J_0|\ge2$, by \eqref{6-1}, since $r_{n,0}=o(1)$. We may now choose and fix $\tau\in(0,1)$ small such that
	$$3\tau<\min\big\{|x-y|:~x,y\in\SR_0,x\neq y\big\},$$
and such that $\SR_0\subset B_{\f{1}{3\tau}}(0)$. 
We can check  that there exists $C>0$ such that any point in
	$$\Omega_{n,0}:=B_{\f{r_{n,0}}{\tau}}(0) \setminus \bigcup_{l\in J_0}B_{\tau r_{n,0}}(\phi_{n,0}(q_{n,l}))$$
may be joined to $\pa B_{\tau r_{n,0}}(0)$ by  a $C^1$ path in $\Omega_{n,0}$ of length at most $Cr_{n,0}$, for large $n$. 
Therefore, by \eqref{6-14} with \eqref{6-10-0} and \eqref{6-9-0} for $l=0$, and using $\mu_{n,0}=o(r_{n,0})$ by \eqref{6-3}, we may first estimate $u_{n,0}$ on $\pa B_{\tau r_{n,0}}(0)$, and then get from \eqref{est-6} and \eqref{6-16} that
\be\lab{6-17} u_{n,0}=\bar u_{n,0}(\tau r_{n,0})+O(\ga_{n,0}^{1-p_n})\ge(1-\eta)\ga_{n,0}+O(1), \ee
uniformly in $\Omega_{n,0}$ and for all $n$, where $\bar u_{n,0}$ is the average function given by \eqref{average}. 
Independently, we get from \eqref{7-2}, \eqref{6-10-0} and \eqref{6-14} that
\be\lab{6-18} \bar u_{n,0}(\tau r_{n,0})=-\sbr{\f{2}{p_n}-1}\ga_{n,0}+\f{2}{p_n\ga_{n,0}^{p_n-1}}\sbr{\ln\f{1}{\la_n\ga_{n,0}^{2(p_n-1)}r_{n,0}^{2+2\al}}+O(1)}. \ee

\begin{itemize}[fullwidth,itemindent=0em]
\vskip0.1in
\item By the argument below \eqref{5-18}, and using \eqref{6-8-0}, we have that
\be\lab{6-19} \nm{w_{n,l}}_{L^\iy(B_{r_{n,0}/(2\tau)}(0))}=o(\ga_{n,0}^{1-p_n}). \ee
for all $l\in J_0$, where $w_{n,l}$ is given by
\be\lab{6-19-1}\begin{cases}
	-\Delta w_{n,l}=-h_{n,l}u_{n,l}\quad\text{in}~B_{\f{r_{n,0}}{2\tau}}(0),\\
	w_{n,l}=0\quad\text{on}~\pa B_{\f{r_{n,0}}{2\tau}}(0).\\
\end{cases}\ee

\vskip0.1in
\item We prove that
\be\lab{6-20} \lim_{n\to+\iy}\f{\ga_{n,0}}{\ga_{n,j}}=0\quad\text{for all}~j\in J_0~\text{and}~1\le j\le N''. \ee
Let $j\in J_0$ such that $1\le j\le N''$. By Part 1, we have that \eqref{6-12} holds for $i=j$, and we get from \eqref{6-7} that, given $\eta_2\in(\eta,1)$,
\be\lab{6-21} \bar r_{n,j}^{(\eta_2)}=r_{n,j}^{(\eta_2)}, \ee
for large $n$.
By \eqref{6-9}, by \eqref{6-10} for $l=j$ and $\eta=\eta_2$, and by the definition \eqref{6-6}  of $r_{n,j}^{(\eta_2)}$,  we have that
\be\lab{6-24} \bar u_{n,j}(r_{n,j}^{(\eta_2)})\le(1-\eta_2)\ga_{n,j}(1+o(1)).\ee 
Let $w_{n,j}$ be given by \eqref{6-19-1} with $l=j$. By observing that $-\Delta(u_{n,j}-w_{n,j})\ge0$, the maximum principle yields that
\be\lab{6-22} u_{n,j}\ge\inf_{\pa B_{r_{n,0}/(2\tau)}(0)}u_{n,j}+o(\ga_{n,0}^{1-p_n})\ge(1-\eta)\ga_{n,i}+O(1), \ee
uniformly in $B_{r_{n,0}/(2\tau)}(0)$, where in the last inequality we use \eqref{6-17} and the fact that
\be\lab{6-23} \phi_{n,0}\circ\phi_{n,j}^{-1}\sbr{\pa B_{r_{n,0}/(2\tau)}(0)}\subset \Omega_{n,0}, \ee 
by \eqref{6-16} and choosing $\tau>0$ small enough from the begining. Then, by $r_{n,j}=o(r_{n,0})$ due to $1\le j\le N''$, we get from \eqref{6-24} and \eqref{6-22} that
	$$(1-\eta)\ga_{n,0}\le (1-\eta_2)\ga_{n,j}(1+o(1)),$$
for any $\eta_2\in(\eta,1)$, so that letting $\eta_2\to1$, we get \eqref{6-20}.

\vskip0.1in
\item We prove that
\be\lab{6-25} \ga_{n,j}=O(\ga_{n,0})\quad\text{for any}~j\in J_0. \ee
By contradiction, if \eqref{6-25} does not hold, we choose $j\in J_0$ such that
\be\lab{6-26} \lim_{n\to+\iy}\f{\ga_{n,0}}{\ga_{n,j}}=0. \ee 
In particular, we have $j\neq 0$. Using \eqref{vpn} and \eqref{6-15}, and by the definition of $r_{n,0}$, we get first that 
\be\lab{6-27-0} \f{1}{C}r_{n,0}\le|x_{n,l}|=(1+o(1))d_{g_0}(q_{n,l},q_{n,0})\le Cr_{n,0}, \quad\forall~l\in J_0\setminus\{0\}, \ee
for some $C>0$ and large $n$, and then, by using \eqref{6-26}, we get similarly to \eqref{5-27} that if $N''<j\le N'$,
\be\lab{6-27}\begin{aligned}
	t_{n,j}(r_{n,j})&=\f{p_n}{2}\ga_{n,j}(1+o(1)).
\end{aligned}\ee
Thus, given any $\eta_2\in(0,1)$, we get that \eqref{6-21} holds also if $N''<j\le N'$. As a first consequence, for all given $\eta_2'\in(0,\eta_2]$, we get that
\be\lab{6-28} \lim_{n\to+\iy}\f{r_{n,j}^{(\eta_2')}}{r_{n,0}}=0, \ee 
using \eqref{6-15} and the fact by \eqref{6-6} that
\be\lab{6-29} \ln\f{r_{n,j}^{(\eta_2')}}{r_{n,j}^{(\eta_2)}}=-\f{p_n\ga_{n,j}^{p_n}}{4}(\eta_2-\eta_2')+o(1). \ee
We get from \eqref{16}, \eqref{6-10} and \eqref{6-21} that
\be\lab{6-30} \bar u_{n,j}(r_{n,j}^{(\eta_2')})=-\sbr{\f{2}{p_n}-1}\ga_{n,j}+\f{2}{p_n\ga_{n,j}^{p_n-1}}\sbr{\ln\f{1}{\la_n\ga_{n,j}^{2(p_n-1)}(r_{n,j}^{(\eta_2')})^2|x_{n,j}|^{2\al}}+O(1)}. \ee
Repeating the argument in the paragraph below \eqref{5-31}, and using \eqref{6-27-0} to get a similar estimate as \eqref{5-37} with $i=0$, we get the following estimate
\be\lab{6-31} \bar u_{n,j}(r_{n,j}^{(\eta_2)})\ge \bar u_{n,0}(\tau r_{n,0})+\f{4}{p_n\ga_{n,j}^{p_n-1}}\ln\f{r_{n,0}}{r_{n,j}^{(\eta_2)}}+O(\ga_{n,0}^{1-p_n}). \ee
We now plug \eqref{6-18} and \eqref{6-30} in \eqref{6-31}, and by using \eqref{boundla}, \eqref{6-27-0} and \eqref{6-26}, we get
	$$\ln\f{1}{r_{n,0}}\le O\sbr{\ln\ga_{n,0}},$$
which is an obvious contradiction with \eqref{6-8-0} and \eqref{6-14}. This finish the proof of \eqref{6-25}.

\vskip0.1in
\item We prove that
\be\lab{6-40} \f{1}{C}\le \f{r_{n,j}}{r_{n,0}}\le C, \quad \f{1}{C}\le\f{\ga_{n,j}}{\ga_{n,0}}\le C,\quad\forall~j\in J_0, \ee
for large $n$, and there exists $\eta_3\in(\eta,1)$ such that
\be\lab{6-41} \bar r_{n,j}^{(\eta_3)}=r_{n,j} \ee 
for all $j\in J_0$ and large $n$.
By \eqref{6-20} and \eqref{6-25}, we see that
\be\lab{6-42} r_{n,0}=O(r_{n,j}),\quad \ga_{n,j}=O(\ga_{n,0}),\quad\forall~j\in J_0. \ee
This together with \eqref{6-15} gives the first assertion of \eqref{6-40}.
We now assume by contradiction with \eqref{6-41} that there exists $j\in J_0$ such that
	$$\f{2t_{n,j}(r_{n,j})}{p_n\ga_{n,j}^{p_n}}\ge1+o(1).$$
As a remark, we must have $j\neq 0$ by \eqref{6-13}. Then, for all given $\eta_2\in(0,1)$, \eqref{6-21} holds and the argument between \eqref{6-21} and \eqref{6-25} gives $\ga_{n,0}=o(\ga_{n,j})$, which does not occur by \eqref{6-42} and proves \eqref{6-41}. For $j\in J_0$, by \eqref{vpn} and the first assertion of \eqref{6-40}, and by choosing $\tau>0$ small enough in the begining, we first get that
	$$\phi_{n,0}\circ\phi_{n,j}^{-1}\sbr{\pa B_{r_{n,j}/2}(0)}\subset \Omega_{n,0},$$
then, by \eqref{6-17}, we get
\be\lab{6-43} u_{n,j}\ge(1-\eta)\ga_{n,0}+O(1), \ee
uniformly on $\pa B_{r_{n,j}/2}(0)$ for large $n$, so that we eventually have $\ga_{n,0}=O(\ga_{n,j})$, using the fact that $\bar u_{n,j}(r_{n,j}/2)\le 2\ga_{n,j}$ by \eqref{6-9}, \eqref{6-10} and \eqref{6-41}. This last estimate proves the second assertion of \eqref{6-40}, by recalling \eqref{6-42}.

\vskip0.1in
\item We are now in position to conclude the proof of \eqref{6-13}. Since for the singular bubble we have no convergence results like \eqref{regular-3} and \eqref{regular-4}, the proof is totally different from that of the last paragraph in the proof of Lemma \ref{lemma5-1}, and here we use the local Pohozaev identity technique.

Let $R\gg1$ be fixed. By multiplying equation \eqref{equ-7} for $i=0$ with $\abr{x,\nabla u_{n,0}}$ and integrating on $B_{Rr_{n,0}}(0)$, computing by integration by parts, we get the local Pohozaev identity,
\be\lab{6-45}\begin{aligned}
	&Rr_{n,0}\int_{\pa B_{Rr_{n,0}}(0)}\f{1}{2}|\nabla u_{n,0}|^2-|\pa_\nu u_{n,0}|^2~\rd\sigma_x\\
	&=Rr_{n,0}\int_{\pa B_{Rr_{n,0}}(0)}\la_nh_{n,0}f_{n,0}e^{u_{n,0}^{p_n}}-\f{1}{2}h_{n,0}u_{n,0}^2~\rd\sigma_x\\
	&\quad -(1+\al)\int_{B_{Rr_{n,0}}(0)}2\la_nh_{n,0}f_{n,0}e^{u_{n,0}^{p_n}}-h_{n,0}u_{n,0}^2~\rd x\\
	&\quad -\int_{B_{Rr_{n,0}}(0)}\la_n\abr{x,\nabla(\ti h_{n,0}f_{n,0})}|x|^{2\al}e^{u_{n,0}^{p_n}}-\f{1}{2}\abr{x,\nabla \ti h_{n,0}}|x|^{2\al}u_{n,0}^2~\rd x,
\end{aligned}\ee
where $\ti h_{n,0}$ is given by \eqref{6-hi}, and $\abr{\cdot,\cdot}$ denotes the inner product in $\R^2$. 
Note that, since $\al\in(-1,0)$ and so $\nabla u_{n,0}$ is not continuous near the origin, nevertheless the validity of \eqref{6-45} can be still ensured by the available integrability for $\nabla u_{n,0}$, see for example \cite[(4.24)]{BT}. 
For any $j\in J_0$, let $\ti w_{n,j}$ be given by \eqref{6-19-1} with $l=j$ and radius $\f{r_{n,0}}{2\tau}$ replaced by $\f{r_{n,j}}{2}$, we still have $|\ti w_{n,j}|=o(\ga_{n,0}^{1-p_n})$ uniformly in $B_{r_{n,j}/2}(0)$ by \eqref{6-40}, and then by applying maximum principle to $u_{n,j}-w_{n,j}$, we get from \eqref{6-43} that
\be\lab{6-46-0} u_{n,j}\ge \inf_{\pa B_{r_{n,j}/2}(0)}u_{n,j}+o(\ga_{n,0}^{1-p_n})\ge (1-\eta+o(1))\ga_{n,0} \ee
uniformly in $B_{r_{n,j}/2}(0)$ for large $n$, so that by choosing $\tau$ small in the begining, we get from \eqref{6-16} and \eqref{6-40} that
\be\lab{6-46} u_{n,0}\ge\f{1-\eta}{2}\ga_{n,0} \ee 
uniformly in $B_{\tau r_{n,0}}\sbr{\phi_{n,0}(q_{n,l})}$ for large $n$.
By the argument to get \eqref{6-17}, we see that \eqref{6-17} still holds in $B_{Rr_{n,0}}(0)\setminus B_{r_{n,0}/\tau}(0)$ for fixed $R\gg1$, so by \eqref{6-46}, we have that
\be\lab{6-47} u_{n,0}\ge\f{1-\eta}{2}\ga_{n,0} \ee
uniformly in $B_{Rr_{n,0}}(0)$ for large $n$. 
By using \eqref{6-47} and $r_{n,0}=o(1)$, and recalling \eqref{bound-1} and Lemma \ref{bound-0}, we get that
\be\lab{6-48}\begin{aligned} 
	&\abs{\ga_{n,0}^{2(p_n-1)}\int_{B_{Rr_{n,0}}(0)}h_{n,0}u_{n,0}^2\rd x}=o(1),  \\
	&\abs{\ga_{n,0}^{2(p_n-1)}\int_{B_{Rr_{n,0}}(0)}\abr{x,\nabla \ti h_{n,0}}|x|^{2\al}u_{n,0}^2\rd x}=o(1), \quad\text{and}\\
	&\abs{\ga_{n,0}^{2(p_n-1)}\int_{B_{Rr_{n,0}}(0)}\la_n\abr{x,\nabla(\ti h_{n,0}f_{n,0})}|x|^{2\al}e^{u_{n,0}^{p_n}}\rd x}=o(1).
\end{aligned}\ee
Now, by computing as in the argument involving \eqref{4-66} and using Proposition \ref{single-1} or \ref{single-2}, and by using \eqref{6-10-0}-\eqref{6-9} with \eqref{6-41}, we get that
\be\lab{6-49} \la_np_nh_{n,0}f_{n,0}e^{u_{n,0}^{p_n}}=\f{8(1+\al)^2e^{-2t_{n,0}}|\cdot|^{2\al}}{\mu_{n,0}^{2(1+\al)}\ga_{n,0}^{2(p_n-1)}p_n}\sbr{1+o\sbr{e^{\ti\eta t_{n,0}}}},  \ee
uniformly in $B_{\tau r_{n,0}}(0)$, and that
\be\lab{6-50} \la_np_nh_{n,0}f_{n,0}e^{u_{n,0}^{p_n}}=\f{8e^{-2t_{n,j}}}{\mu_{n,j}^{2}\ga_{n,j}^{2(p_n-1)}p_n}\sbr{1+o\sbr{e^{\ti\eta t_{n,j}}}},  \ee
uniformly in $B_{\tau r_{n,0}}(\phi_{n,0}(q_{n,j}))$ for large $n$ and all $j\in J_0\setminus\{0\}$, where $\ti\eta$ is any constant in $(\eta,1)$.  
By using the first equality in \eqref{6-17} (holds in $B_{Rr_{n,0}}(0)\setminus \cup_{l\in J_0}B_{\tau r_{n,0}}(\phi_{n,0}(q_{n,l}))$), we get from \eqref{6-49} that
\be\lab{6-51} \la_np_nh_{n,0}f_{n,0}e^{u_{n,0}^{p_n}}=O\sbr{ \f{r_{n,0}^{2\al}}{\mu_{n,0}^{2(1+\al)}\ga_{n,0}^{2(p_n-1)}}\sbr{\f{r_{n,0}}{\mu_{n,0}}}^{-2(1+\al)(2-\ti\eta)} }    \ee
uniformly on $B_{Rr_{n,0}}(0)\setminus \cup_{l\in J_0}B_{\tau r_{n,0}}(\phi_{n,0}(q_{n,l}))$ for large $n$. Then, we get that
\be\lab{6-52}\begin{aligned} 
	&\abs{r_{n,0}\ga_{n,0}^{2(p_n-1)}\int_{\pa B_{Rr_{n,0}}(0)}\la_nh_{n,0}f_{n,0}e^{u_{n,0}^{p_n}}\rd\sigma_x}=o(1),\\
	&\abs{r_{n,0}\ga_{n,0}^{2(p_n-1)}\int_{\pa B_{Rr_{n,0}}(0)}h_{n,0}u_{n,0}^2~\rd\sigma_x}=o(1),
\end{aligned} \ee 
by using $\mu_{n,0}=o(r_{n,0})$ and $\ti\eta\in(\eta,1)$. 
Plugging \eqref{6-48} and \eqref{6-52} in \eqref{6-45}, we get that
\be\lab{6-53}\begin{aligned} 
	&\lim_{n\to+\iy}Rr_{n,0}\sbr{\f{p_n}{2}\ga_{n,0}^{p_n-1}}^2\int_{\pa B_{Rr_{n,0}}(0)}\f{1}{2}|\nabla u_{n,0}|^2-|\pa_\nu u_{n,0}|^2~\rd\sigma_x\\
	&=-(1+\al)\lim_{n\to+\iy}\f{p_n}{2}\ga_{n,0}^{2(p_n-1)}\int_{B_{Rr_{n,0}}(0)}\la_np_nh_{n,0}f_{n,0}e^{u_{n,0}^{p_n}}\rd x
\end{aligned}\ee
for any fixed $R\gg1$.
By using \eqref{6-49} and \eqref{6-50}, we get that
\be\lab{6-54}\begin{aligned} 
	&\f{p_n}{2}\ga_{n,0}^{2(p_n-1)}\int_{B_{\tau r_{n,0}}(0)}\la_np_nh_{n,0}f_{n,0}e^{u_{n,0}^{p_n}}\rd x=4\pi(1+\al)+o(1),~~\text{and}\\
	&\f{p_n}{2}\ga_{n,0}^{2(p_n-1)}\int_{B_{\tau r_{n,0}}(\phi_{n,0}(q_{n,j}))}\la_np_nh_{n,0}f_{n,0}e^{u_{n,0}^{p_n}}\rd x=4\pi\theta_j^2+o(1)
\end{aligned} \ee
where $\theta_j$ is given by
\be\lab{6-55} \theta_j:=\lim_{n\to+\iy}\f{\ga_{n,0}^{p_n-1}}{\ga_{n,j}^{p_n-1}}\in(0,+\iy), \ee
thanks to the second assertion of \eqref{6-40} and $p_n\in[1,2]$, for all $j\in J_0\setminus\{0\}$ and large $n$.
Then, we get from \eqref{6-54} that
\be\lab{6-56}  \lim_{n\to+\iy}\f{p_n}{2}\ga_{n,0}^{2(p_n-1)}\int_{B_{Rr_{n,0}}(0)}\la_np_nh_{n,0}f_{n,0}e^{u_{n,0}^{p_n}}\rd x=4\pi\sbr{(1+\al)+\sum_{j\in J_0\setminus\{0\}}\theta_j^2},   \ee
by using \eqref{6-51} to estimating the integral in $B_{Rr_{n,0}}(0)\setminus \cup_{l\in J_0}B_{\tau r_{n,0}}(\phi_{n,0}(q_{n,l}))$, for any fixed $R\gg1$.

Now, we compute the left hand side of \eqref{6-53}. 
Let $\ti\psi_n$ be defined by
	$$-\Delta\ti\psi_n=0~\text{in}~B_{\kappa}(0),\quad \ti\psi_n=u_{n,0}~\text{on}~\pa B_\kappa(0).$$
Then, by \eqref{conv-0}, we clearly have that
\be\lab{6-57} \nm{\ti\psi_n}_{C^1_\loc(B_\kappa(0))}=O(1). \ee
Let $G$ be the Green's function of $\Delta$ in $B_{\kappa}(0)$ with zero Dirichlet boundary condition. Then, we have that
\be\lab{6-58} \nabla_x G(x,y)=-\f{1}{2\pi}\f{x-y}{|x-y|^2}+O(1), \ee
uniformly for $x\in \pa B_{Rr_{n,0}}(0)$ and $y\in B_\kappa(0)$.
By applying the Green's representation formula to $u_{n,0}-\ti\psi_n$, and using \eqref{6-57} and \eqref{6-58}, we get that
\be\lab{6-59}\begin{aligned} 
	&\f{p_n}{2}\ga_{n,0}^{p_n-1}\nabla u_{n,0}(x)\\
	&=-\f{p_n}{4\pi}\ga_{n,0}^{p_n-1}\int_{B_\kappa(0)}\f{x-y}{|x-y|^2}\la_np_nh_{n,0}f_{n,0}u_{n,0}^{p_n-1}e^{u_{n,0}^{p_n}}(y)\rd y \\
	&\quad +\f{p_n}{4\pi}\ga_{n,0}^{p_n-1}\int_{B_\kappa(0)}\f{x-y}{|x-y|^2}h_{n,0}u_{n,0}(y)\rd y+O\sbr{\ga_{n,0}^{p_n-1}},
\end{aligned}\ee
uniformly for $x\in \pa B_{Rr_{n,0}}(0)$.
We split $B_\kappa(0)$ into three subdomains
	$$B_\kappa(0)=B_{\f{1}{2}Rr_{n,0}}(0)\cup \sbr{B_{\f{3}{2}Rr_{n,0}}(0)\setminus B_{\f{1}{2}Rr_{n,0}}(0)}\cup \sbr{B_\kappa(0)\setminus B_{\f{3}{2}Rr_{n,0}}(0)}.$$
Taking a sequence $L_n$ such that $\lim_{n\to+\iy}L_n=+\iy$ and $L_n\mu_{n,0}=o(r_{n,0})$ since $\mu_{n,0}=o(r_{n,0})$, we have that
	$$\f{x-y}{|x-y|^2}=\f{x}{|x|^2}+o\sbr{\f{1}{r_{n,0}}}$$
uniformly for $x\in  \pa B_{Rr_{n,0}}(0)$ and $y\in B_{L_n\mu_{n,0}}(0)$, and then, by the argument involving \eqref{6-54},
we get that 
$$\begin{aligned} 
	&\f{p_n}{4\pi}\ga_{n,0}^{p_n-1}\int_{B_{\tau r_{n,0}}(0)}\f{x-y}{|x-y|^2}\la_np_nh_{n,0}f_{n,0}u_{n,0}^{p_n-1}e^{u_{n,0}^{p_n}}(y)\rd y\\
	&\quad =2(1+\al)\f{x}{|x|^2}+o\sbr{\f{1}{r_{n,0}}}, 
\end{aligned}$$  
uniformly for $x\in \pa B_{Rr_{n,0}}(0)$ for large $n$, by computing the integral seperately in $B_{L_n\mu_{n,0}}(0)$ and $B_{\tau r_{n,0}}(0)\setminus B_{L_n\mu_{n,0}}(0)$. 
By similar argument, we get also that
$$\begin{aligned} 
	&\f{p_n}{4\pi}\ga_{n,0}^{p_n-1}\int_{B_{\tau r_{n,0}}(\phi_{n,0}(q_{n,j}))}\f{x-y}{|x-y|^2}\la_np_nh_{n,0}f_{n,0}u_{n,0}^{p_n-1}e^{u_{n,0}^{p_n}}(y)\rd y\\
	&\quad =2\theta_j\f{x-\phi_{n,0}(q_{n,j})}{|x-\phi_{n,0}(q_{n,j})|^2}+o\sbr{\f{1}{r_{n,0}}},
\end{aligned}$$
uniformly for $x\in \pa B_{Rr_{n,0}}(0)$ and for any $j\in J_0\setminus\{0\}$, where $\theta_j$ is given by \eqref{6-55}. Then, by the argument involving \eqref{6-56}, we eventually get that
\be\lab{6-60} \begin{aligned} 
	&\f{p_n}{4\pi}\ga_{n,0}^{p_n-1}\int_{B_{\f{1}{2}Rr_{n,0}}(0)}\f{x-y}{|x-y|^2}\la_np_nh_{n,0}f_{n,0}u_{n,0}^{p_n-1}e^{u_{n,0}^{p_n}}(y)\rd y\\
	&\quad =2(1+\al)\f{x}{|x|^2}+\sum_{j\in J_0\setminus\{0\}}2\theta_j\f{x-\phi_{n,0}(q_{n,j})}{|x-\phi_{n,0}(q_{n,j})|^2}+o\sbr{\f{1}{r_{n,0}}},
\end{aligned}\ee
uniformly for $x\in \pa B_{Rr_{n,0}}(0)$ as $n\to+\iy$, and for any fixed $R\gg1$.
Independently, by Lemma \ref{bound-0} and \eqref{6-47}, we get that
\be\lab{6-61} 
	\f{p_n}{4\pi}\ga_{n,0}^{p_n-1}\int_{B_{\f{1}{2}Rr_{n,0}}(0)}\f{x-y}{|x-y|^2}h_{n,0}u_{n,0}(y)\rd y=o\sbr{\f{1}{r_{n,0}}}, 
\ee
uniformly for $x\in \pa B_{Rr_{n,0}}(0)$. 
Noting that 
	$$B_{\f{3}{2}Rr_{n,0}}(0)\setminus B_{\f{1}{2}Rr_{n,0}}(0)\subset \bigcup_{k=1}^K B_{\f{3}{4}Rr_{n,0}}(z_k),$$
for some points $z_k\in\pa B_{Rr_{n,0}}(0)$ and a universal constant $K$, we get from \eqref{6-51} and \eqref{6-47} that
\be\lab{6-62}\begin{aligned} 
	&\f{p_n}{4\pi}\ga_{n,0}^{p_n-1}\int_{B_{\f{3}{2}Rr_{n,0}}(0)\setminus B_{\f{1}{2}Rr_{n,0}}(0)}\f{x-y}{|x-y|^2}\la_np_nh_{n,0}f_{n,0}u_{n,0}^{p_n-1}e^{u_{n,0}^{p_n}}(y)\rd y\\
	&=O\sbr{ \f{r_{n,0}^{2\al}}{\mu_{n,0}^{2(1+\al)}}\sbr{\f{r_{n,0}}{\mu_{n,0}}}^{-2(1+\al)(2-\ti\eta)} \sum_{k=1}^K\int_{B_{\f{3}{4}Rr_{n,0}}(z_k)}\f{1}{|x-y|}\rd y }=o\sbr{\f{1}{r_{n,0}}},
\end{aligned}\ee 
and that
\be\lab{6-63}
	\f{p_n}{4\pi}\ga_{n,0}^{p_n-1}\int_{B_{\f{3}{2}Rr_{n,0}}(0)\setminus B_{\f{1}{2}Rr_{n,0}}(0)}\f{x-y}{|x-y|^2}h_{n,0}u_{n,0}(y)\rd y=o\sbr{\f{1}{r_{n,0}}},
\ee
uniformly for $x\in \pa B_{Rr_{n,0}}(0)$.  
Finally, we prove that
\be\lab{6-64} \f{p_n}{4\pi}\ga_{n,0}^{p_n-1}\int_{B_\kappa(0)\setminus B_{\f{3}{2}Rr_{n,0}}(0)}\f{x-y}{|x-y|^2}\la_np_nh_{n,0}f_{n,0}u_{n,0}^{p_n-1}e^{u_{n,0}^{p_n}}(y)\rd y=o\sbr{\f{1}{r_{n,0}}}, \ee
and that
\be\lab{6-64-1} \f{p_n}{4\pi}\ga_{n,0}^{p_n-1}\int_{B_\kappa(0)\setminus B_{\f{3}{2}Rr_{n,0}}(0)}\f{x-y}{|x-y|^2} h_{n,0}u_{n,0}(y)\rd y=o\sbr{\f{1}{r_{n,0}}}, \ee
uniformly for $x\in \pa B_{Rr_{n,0}}(0)$. 
Let 
	$$\XR_n:=\lbr{q_{n,0}, q_{n,1},\cdots,q_{n,N'}}.$$
Noting that $q_{n,0}=q_0$ for all $n$, we define $\XR_n^{(0)}=\{q_{0}\}$ and $r_{n}^{(0)}=r_{n,0}$, and by induction, for $k=1,2,\cdots$,
$$\begin{aligned}
	\XR_n^{(k)}&:=\lbr{q_{n,l}\in\XR_n:~d_{g_0}(q_0,q_{n,l})=O(r_{n,0}^{(k-1)}) }\\
	r_{n}^{(k)}&:=\begin{cases} \f{1}{3}d_{g_0}\sbr{q_0,\XR_n\setminus\XR_n^{(k)}}, \quad&\text{if}~\XR_n\setminus\XR_n^{(k)}\neq\emptyset,\\\kappa_0,\quad&\text{if}~\XR_n\setminus\XR_n^{(k)}\neq\emptyset.\end{cases}
\end{aligned}$$
Thanks to $|J_0|\ge2$, we have $N'\ge1$ and that the above induction terminates at a maximal index $k_{\max}\ge1$ such that $r_n^{(k_{\max})}=\kappa_0$. Moreover, we have that $r_n^{(k)}=o(r_n^{(k+1)})$, and that there exists a constant $R_0>1$ such that $d_{g_0}(q_0,\XR_n^{(k+1)})\le\f{1}{2}R_0r_n^{(k)}$ for any $k\in\{0,\cdots,k_{\max}-1\}$. Then, for any $k\in\{0,\cdots,k_{\max}-1\}$, and any sequences $t_n,R_n$ satisfying $R_0r_n^{(k)}\le t_n\le R_nt_n\le r_n^{(k+1)}$, by taking a curve connecting any two points in $B_{R_nt_n}(0)\setminus B_{t_n}(0)$, such that the length is at most $(\pi+1)R_nt_n$, we get from \eqref{est-6} and \eqref{vpn} that
\be\lab{6-65} \abs{u_{n,0}(y_1)^{p_n}-u_{n,0}(y_2)^{p_n}}\le CR_n,\quad\text{for any}~y_1,y_2\in B_{R_nt_n}(0)\setminus B_{t_n}(0), \ee
for a universal constant $C>0$ and all large $n$. By \eqref{6-47}, and noting that $u_{n,0}=O(1)$ on $\pa B_\kappa(0)$ by \eqref{conv-0}, for any constant $\ep_0\in(0,(1-\eta)/2)$, up to a subsequence, there exists a $k_0\in\{0,\cdots,k_{\max}-1\}$ such that
\be\lab{6-66} \bar u_{n,0}(R_0r_{n}^{(k_0)})\ge\ep_0\ga_{n,0} \quad\text{and}\quad \bar u_{n,0}(r_{n}^{(k_0+1)})=o(\ga_{n,0}), \ee
for large $n$. 
Let $T_n$ be given by
	$$T_n:=\sup\lbr{ t\in\big[R_0r_n^{(k_0)},r_n^{(k_0+1)}\big):~\bar u_{n,0}\ge \f{1}{2}\ep_0\ga_{n,0}~\text{in}~[R_0r_{n}^{(k_0)},t] }.$$
Then, we have $T_n=o(1)$ and that
\be\lab{6-67} \f{\ga_{n,0}^{p_n-1}}{T_n}=o\sbr{\f{1}{r_{n}^{(k_0)}}},\ee
for large $n$. Indeed, if $T_n=O(r_{n}^{(k_0)}\ga_{n,0}^{p_n-1})$, then by applying \eqref{6-65} with $t_n=R_0r_n^{(k_0)}$ and $R_n=T_n/t_n$, and using the first assertion in \eqref{6-66}, we get that
	$$\bar u_{n,0}(T_n)\ge\ep_0\ga_{n,0}+O(1),$$
for large $n$, which is a contradiction with the definition of $T_n$, so that \eqref{6-67} holds.
Thus, by using \eqref{bound-1} and \eqref{6-67}, we get that
	$$\f{p_n}{4\pi}\ga_{n,0}^{p_n-1}\int_{B_\kappa(0)\setminus B_{T_n}(0)}\f{x-y}{|x-y|^2}\la_np_nh_{n,0}f_{n,0}u_{n,0}^{p_n-1}e^{u_{n,0}^{p_n}}(y)\rd y=O\sbr{\f{\ga_{n,0}^{p_n-1}}{T_n}}=o\sbr{\f{1}{r_{n,0}}},$$
and that
	$$\f{p_n}{4\pi}\ga_{n,0}^{p_n-1}\int_{B_\kappa(0)\setminus B_{T_n}(0)}\f{x-y}{|x-y|^2}h_{n,0}u_{n,0}(y)\rd y=o\sbr{\f{1}{r_{n,0}}},$$
uniformly for $x\in \pa B_{Rr_{n,0}}(0)$. 
By the argument involving \eqref{6-46-0}, and by using \eqref{6-65}, we get that
	$$u_{n,0}\ge \inf_{\pa B_{T_n}(0)}u_{n,0}+O(1)\ge \f{1}{3}\ep_0\ga_{n,0}, $$
uniformly in $B_{T_n}(0)$ for large $n$, so that we get  that
	$$\f{p_n}{4\pi}\ga_{n,0}^{p_n-1}\int_{B_{T_n}(0)\setminus B_{\f{3}{2}Rr_{n,0}}(0)}\f{x-y}{|x-y|^2}\la_np_nh_{n,0}f_{n,0}u_{n,0}^{p_n-1}e^{u_{n,0}^{p_n}}(y)\rd y=o\sbr{\f{1}{r_{n,0}}},$$
where we estimate the integral seperately in $B_{T_n}(0)\setminus B_{Lr_{n,0}}(0)$ by using \eqref{bound-1}, and in $B_{Lr_{n,0}}(0)\setminus B_{\f{3}{2}Rr_{n,0}}(0)$ by using \eqref{6-51}, and then letting $L\to+\iy$, 
and from Lemma \ref{bound-0} that
	$$\f{p_n}{4\pi}\ga_{n,0}^{p_n-1}\int_{B_{T_n}(0)\setminus B_{\f{3}{2}Rr_{n,0}}(0)}\f{x-y}{|x-y|^2}h_{n,0}u_{n,0}(y)\rd y=o\sbr{\f{1}{r_{n,0}}},$$
uniformly for $x\in \pa B_{Rr_{n,0}}(0)$, and for large $n$. Combining these estimates, we conclude \eqref{6-64} and \eqref{6-64-1}.
Plugging \eqref{6-60}-\eqref{6-64-1} in \eqref{6-59}, and using \eqref{6-8-0} with \eqref{6-14}, we get that
\be\lab{6-68}\begin{aligned} 
	\f{p_n}{2}\ga_{n,0}^{p_n-1}\nabla u_{n,0}(x)=2(1+\al)\f{x}{|x|^2}+\sum_{j\in J_0\setminus\{0\}}2\theta_j\f{x-\phi_{n,0}(q_{n,j})}{|x-\phi_{n,0}(q_{n,j})|^2}+o\sbr{\f{1}{r_{n,0}}},
\end{aligned}\ee
uniformly for $x\in \pa B_{Rr_{n,0}}(0)$, and for large $n$. 
Therefore, by observing that
	$$\f{x-\phi_{n,0}(q_{n,j})}{|x-\phi_{n,0}(q_{n,j})|^2}=\f{x}{|x|^2}+O\sbr{\f{\abs{\phi_{n,0}(q_{n,j})}}{|x||x-\phi_{n,0}(q_{n,j})|}}=\f{x}{|x|^2}\sbr{1+O\sbr{\f{1}{R}}}$$
uniformly for $x\in \pa B_{Rr_{n,0}}(0)$, and for all $j\in J_0\setminus\{0\}$ and any fixed $R\gg1$, we get that
\be\lab{6-69}\begin{aligned}
	&\lim_{n\to+\iy}Rr_{n,0}\sbr{\f{p_n}{2}\ga_{n,0}^{p_n-1}}^2\int_{\pa B_{Rr_{n,0}}(0)}\f{1}{2}|\nabla u_{n,0}|^2-|\pa_\nu u_{n,0}|^2~\rd\sigma_x\\
	&=-4\pi\mbr{(1+\al)+\sbr{1+O\sbr{\f{1}{R}}}\sum_{j\in J_0\setminus\{0\}}\theta_j }^2.
\end{aligned}\ee

Plugging \eqref{6-56} and \eqref{6-69} in \eqref{6-53}, and letting $R\to+\iy$, we deduce that
	$$\sbr{(1+\al)+\sum_{j\in J_0\setminus\{0\}}\theta_j }^2=(1+\al)\sbr{(1+\al)+\sum_{j\in J_0\setminus\{0\}}\theta_j^2},$$
which leads to
	$$\sbr{\sum_{j\in J_0\setminus\{0\}}\theta_j }^2=(1+\al)\sbr{\sum_{j\in J_0\setminus\{0\}}\theta_j^2-2\sum_{j\in J_0\setminus\{0\}}\theta_j }.$$
This is a contradiction with $|J_0|\ge2$ due to $\al\in(-1,0)$. Then we conclude the proof of \eqref{6-13}.
\end{itemize}

\vskip0.1in
\noindent{\it Part 3.} We finally prove that \eqref{6-12} holds for all $i\in\{N''+1,\cdots,N'\}$ (if $N''\le N'-1$). 

We apply the induction method on $i\in\{N''+1,\cdots,N'\}$. In particular, we assume that \eqref{5-12} holds true at steps $N''+1,\cdots,i-1$ if $i\ge N''+2$. Note that for $i=N''+1$ we assume nothing. By contradiction, assume in addition that \eqref{6-12} does not hold at step $i$. Thus, by \eqref{6-6} and \eqref{6-7}, we may choose and fix $\eta\in(0,1)$ sufficiently close to $1$ such that
\be\lab{s6-13} \bar r_{n,i}^{(\eta)}=r_{n,i} \ee
for large $n$. Set $J_i=\lbr{j\in\{0,\cdots,N'\}:~d_{g_0}(q_{n,j},q_{n,i})=O(r_{n,i})}$. Obviously, we get from \eqref{6-1} that 
\be\lab{s6-14} r_{n,l}=O(r_{n,i})\quad\forall~l\in J_i, \ee
for large $n$. Since $r_{n,i}=o(1)$ by \eqref{6-8} and \eqref{s6-13}, we get from \eqref{vpn} that 
\be\lab{s6-15} \lim_{n\to+\iy}\mbr{(\phi_{n,l})_*g_0}(r_{n,i}\cdot)=|\rd x|^2~\text{in}~C_\loc^2(\R^2). \ee
for all $l\in J_i$.
We may assume that
	$$\lim_{n\to+\iy}\f{\phi_{n,i}(q_{n,l})}{r_{n,i}}=\check x_l\in\R^2$$
for all $l\in J_i$, and we  have that $\SR_i:=\lbr{\check x_l:~l\in J_i}$ contains at least two distinct points. 
Then, we may now choose and fix $\tau\in(0,1)$ small enough such that
	$$3\tau<\displaystyle\min\lbr{|x-y|:~x,y\in\SR_i,x\neq y},$$
and such that $\SR_i\subset B_{\f{1}{3\tau}}(0)$. We can check  that there exists $C>0$ such that any point in
	$$\Omega_{n,i}:=B_{\f{r_{n,i}}{\tau}}(0) \setminus \bigcup_{j\in J_i}B_{\tau r_{n,i}}(\phi_{n,i}(q_{n,l}))$$
may be joined to $\pa B_{\tau r_{n,i}}(0)$ by  a $C^1$ path in $\Omega_{n,i}$ of length at most $Cr_{n,i}$, for all $n$. Therefore, by \eqref{s6-13} with \eqref{6-10} and \eqref{6-9} for $l=i$, we may first estimate $u_{n,i}$ on $\pa B_{\tau r_{n,i}}(0)$, and then get from \eqref{est-6} and \eqref{s6-15} that
\be\lab{s6-16} u_{n,i}=\bar u_{n,i}(\tau r_{n,i})+O(\ga_{n,i}^{1-p_n})\ge(1-\eta)\ga_{n,i}+O(1), \ee
uniformly in $\Omega_{n,i}$ and for all $n$, where $\bar u_{n,i}$ is the average function given by \eqref{average}. 
Independently, we get from \eqref{16}, \eqref{6-10} and \eqref{s6-13} that
\be\lab{s6-17} \bar u_{n,i}(\tau r_{n,i})=-\sbr{\f{2}{p_n}-1}\ga_{n,i}+\f{2}{p_n\ga_{n,i}^{p_n-1}}\sbr{\ln\f{1}{\la_n\ga_{n,i}^{2(p_n-1)}r_{n,i}^2|x_{n,i}|^{2\al}}+O(1)}. \ee

We now prove that 
\be\lab{s6-18} 0\not\in J_i. \ee
Assume by contradiction to \eqref{s6-18} that $0\in J_i$. Then, by \eqref{vpn} and \eqref{6-2}, we get that 
\be\lab{s6-18-1} 2r_{n,i}\le|x_{n,i}|=O(d_{g_0}(q_{n,0},q_{n,i}))=O(r_{n,i}) \ee
for large $n$. 
Given $\eta_2\in(\eta,1)$, we get from \eqref{6-13} that,
\be\lab{s6-21} \bar r_{n,0}^{(\eta_2)}=r_{n,0}^{(\eta_2)}, \ee
for large $n$. Then, since \eqref{6-13} holds, following the line between \eqref{5-21} and \eqref{5-25}, we get that
\be\lab{s6-20} \lim_{n\to+\iy}\f{\ga_{n,0}}{\ga_{n,i}}=0. \ee
For any $\eta_2'\in(0,\eta_2]$, we get from \eqref{7-2}, \eqref{6-10-0} and \eqref{s6-21} that
\be\lab{s6-30} \bar u_{n,0}(r_{n,0}^{(\eta_2')})=-\sbr{\f{2}{p_n}-1}\ga_{n,0}+\f{2}{p_n\ga_{n,0}^{p_n-1}}\sbr{\ln\f{1}{\la_n\ga_{n,0}^{2(p_n-1)}(r_{n,0}^{(\eta_2')})^{2+2\al}}+O(1)}. \ee
By repeating the argument in the paragraph below \eqref{5-31}, and by replacing \eqref{5-36} with the fact that
	$$\la_np_nh_{n,0}f_{n,0}u_{n,0}^{p_n-1}e^{u_{n,0}^{p_n}}=\f{8(1+\al)^2e^{-2t_{n,0}}}{\mu_{n,0}^{2+2\al}\ga_{n,0}^{p_n-1}p_n}\sbr{1+O\sbr{e^{\ti\eta t_{n,0}} \sbr{\f{|\cdot|}{r_{n,0}^{(\eta_2)}}+|\cdot|+\f{1}{\ga_{n,0}^{p_n}} } } }$$
uniformly in $B_{r_{n,0}^{(\eta_1)}}(0)$ for any $\eta_1\in(0,\eta_2)$ and any $\ti\eta\in(\eta_2,1)$, which is obtained by computing as in the argument involving \eqref{4-66} (to estimate the last term in \eqref{6-10-0} we use \eqref{s6-21} and the fact that $\mu_{n,0}/r_{n,0}^{(\eta_2)}=o(\ga_{n,0}^{-p_n})$ thanks to \eqref{6-6}), and by using \eqref{s6-18-1} to get a similar estimate as \eqref{5-37} with $j=0$, we get the following estimate
\be\lab{s6-31} \bar u_{n,0}(r_{n,0}^{(\eta_2)})\ge \bar u_{n,i}(\tau r_{n,i})+\f{4(1+\al)}{p_n\ga_{n,0}^{p_n-1}}\ln\f{r_{n,i}}{r_{n,0}^{(\eta_2)}}+O(\ga_{n,i}^{1-p_n}). \ee
We now plug \eqref{s6-17} and \eqref{s6-30} in \eqref{s6-31}, and by using \eqref{boundla}, \eqref{s6-18-1} and \eqref{s6-20}, we get 
	$$\ln\f{1}{r_{n,i}|x_{n,i}|^\al}\le O\sbr{\ln\ga_{n,i}},$$
which is an obvious contradiction with \eqref{6-8} and \eqref{s6-13}. This finish the proof of \eqref{s6-18}.

Once \eqref{s6-18} is proven, we get from \eqref{vpn} and the definition of $J_i$ that $r_{n,i}=o(|x_{n,i}|)$, and then, by using the fact that 
	$$\big\{ j\in\{0,\cdots,N'\}:~d_{g_0}(q_{n,j},q_{n,i})=O(r_{n,i}) \big\}\subset\{1,\cdots,N'\} $$
and by using the same arguments in the proof of Lemma \ref{lemma5-1}, and using \eqref{equ-7}, \eqref{6-4}, \eqref{est-6}, \eqref{6-6}, \eqref{6-7} and \eqref{6-10}-\eqref{6-9}, and using Proposition \ref{regular}, we may lead to a contradiction exactly like in the final line in the proof of Lemma \ref{lemma5-1}. Thus, we conclude the proof of \eqref{6-12} for all $i\in\{N''+1,\cdots,N'\}$, and finish the proof of Lemma \ref{lemma6-1}.
\end{proof}

Then, exactly as in the proof of Lemma \ref{lemma5-2}, by using Lemma \ref{lemma6-1}, we get a first estimate as follows:
\begin{lemma}\lab{lemma6-2}
There exists $C>0$ such that 
\be\lab{s6-70-0} 0<\bar u_{n,0}(r)\le-\sbr{\f{2}{p_n}-1}\ga_{n,0}+\f{2}{p_n\ga_{n,0}^{p_n-1}}\ln\f{C}{\la_n\ga_{n,0}^{2(p_n-1)}r^{2(1+\al)}}+O(r^{3(1+\al)/2}), \ee 
and that for all $i\in\{1,\cdots,N'\}$ (if $N'\ge1$),
\be\lab{s6-70} 0<\bar u_{n,i}(r)\le-\sbr{\f{2}{p_n}-1}\ga_{n,i}+\f{2}{p_n\ga_{n,i}^{p_n-1}}\ln\f{C}{\la_n\ga_{n,i}^{2(p_n-1)}|x_{n,i}|^{2\al}r^2}+O(r^{3/2}), \ee 
uniformly in $(0,\kappa]$ and for all $n$, where $\bar u_{n,i}$ is the average function given by \eqref{average}.
\end{lemma}

\begin{corollary}\lab{lemma6-3}
There holds 
\be\lab{6-80} \lim_{n\to+\iy}\la_n=0. \ee
\end{corollary}
\begin{proof}
By evaluating \eqref{s6-70-0} at $r=\kappa\ga_{n,0}^{\f{2(1-p_n)}{3(1+\al)}}\le\kappa$, we get that
\be\lab{6-79} \sbr{1-\f{p_n}{2}}\ga_{n,0}^{p_n}+\ln\ga_{n,0}^{\f{4}{3}(p_n-1)}\le\ln\f{1}{\la_n}+O(1). \ee
This clearly prove \eqref{6-80}.
\end{proof}

\vs
\subsection{Quantization for subcritical exponent \texorpdfstring{$p_n\not\to2$}{}}\lab{sec6-2}\ 

Up to a subsequence, we may assume 
\be\lab{6-81}\lim_{n\to+\iy}p_n=p_0\ee
for some $p_0\in[1,2]$. In this subsection, we prove that
\begin{lemma}\lab{lemma6-4}
Let $p_0\in[1,2)$, then there holds
	$$\lim_{n\to+\iy}\beta_n(B_{\kappa_0}(q_0))=4\pi(1+\al)+4\pi N',$$ 
where $N'$ is given above \eqref{6-0} and $\beta_n(\cdot)$ is given by \eqref{betadomain}.
\end{lemma}
\begin{proof}
Up to a renumbering, we fix $i\in\{0,\cdots,N'\}$ such that 
\be\lab{6-82} \ga_{n,i}=\max\{\ga_{n,j}:~j=1,\cdots,N'\}, \ee
for all $n$. Given any $\eta\in(0,1)$, setting $r_{n,l}^{(\eta)}$ as in \eqref{6-6}, we know from \eqref{6-12} that $\bar r_{n,l}^{(\eta)}=r_{n,l}^{(\eta)}$ for all $l\in\{0,\cdots,N'\}$ and $n$. 
By the argument to get \eqref{4-66}, and using \eqref{6-9-0}-\eqref{6-8} and Propositions \ref{single-2}, \ref{single-1}, we get that
\be\lab{6-83} \f{p_n}{2}\ga_{n,l}^{p_n-2}\int_{B_{r_{n,l}^{(\eta)}}(0)} \la_np_nh_{n,l}f_{n,l}u_{n,l}^{p_n}e^{u_n^{p_n}}\rd x=\begin{cases}4\pi(1+\al)+o(1),&\quad\text{if}~l=0,\\ 4\pi+o(1),&\quad\text{if}~l\ge1,\end{cases} \ee
and
\be\lab{6-84} \f{p_n}{2}\ga_{n,l}^{2(p_n-1)}\int_{B_{r_{n,l}^{(\eta)}}(0)} \la_np_nh_{n,l}f_{n,l}e^{u_n^{p_n}}\rd x=\begin{cases}4\pi(1+\al)+o(1),&\quad\text{if}~l=0,\\ 4\pi+o(1),&\quad\text{if}~l\ge1,\end{cases}\ee
for all $l$.
By the argument between \eqref{5-85} and \eqref{5-86}, and using Lemma \ref{lemma6-2}, we get that
\be\lab{6-86} \ga_{n,l}=(1+o(1))\ga_{n,i}\ee
for all $l$.
Moreover, by using \eqref{est-6} and \eqref{6-86}, the argument in the paragraph involving \eqref{5-90} gives that
\be\lab{6-90} u_n\le (1-\eta')\ga_{n,i} \quad\text{in}~\Omega_n:=B_{\kappa_0}(q_0)\setminus\bigcup_{l=0}^{N'}\phi_{n,l}^{-1}\sbr{B_{r_{n,l}^{(\eta)}}(0)} \ee 
for large $n$ and any $\eta'\in(0,\eta)$.

By using \eqref{6-90}, and by using \eqref{6-79} (if $i=0$) and using \eqref{5-79} (if $i\ge1$) to estimate $\la_n$, we get that 
\be\lab{6-99} \begin{aligned} 
	&\int_{\Omega_n}\la_nhf\sbr{1+u_n^{p_n}}e^{u_n^{p_n}}\rd v_{g_0}\\
	&\quad=\begin{cases}O\sbr{\exp\sbr{(1-\eta')^{p_n}\ga_{n,0}^{p_n}-\sbr{1-\f{p_0}{2}}\ga_{n,0}^{p_n}+o(\ga_{n,0}^{p_n})} }, &\quad\text{if}~i=0,\\  O\sbr{|x_{n,i}|^{-2\al}\exp\sbr{(1-\eta')^{p_n}\ga_{n,i}^{p_n}-\sbr{1-\f{p_0}{2}}\ga_{n,i}^{p_n}+o(\ga_{n,i}^{p_n})} }, &\quad\text{if}~i\ge1,\end{cases}
\end{aligned}\ee
for all $n$. Noting that $p_0<2$, by choosing $\eta'\in(0,\eta)$ sufficiently close to 1 from the begining, we may plug \eqref{6-99}, \eqref{6-83} and \eqref{6-84} in \eqref{betadomain} to conclude the proof of this lemma, using also \eqref{6-86}.
\end{proof}

\vs
\subsection{Quantization for critical exponent \texorpdfstring{$p_n\to2$}{}}\lab{sec6-3}\ 

In this subsection, we assume that $p_0=2$ in \eqref{6-81}. Recall the notations $r_{n,l}$, $r_{n,l}^{(\eta)}$ and $\bar r_{n,l}^{(\eta)}$ given by \eqref{6-1}, \eqref{6-6} and \eqref{6-7}.

For all $\eta\in(0,1)$, by \eqref{6-12} and \eqref{6-6}, we have that
\be\lab{6-100} r_{n,l}^{(\eta)}=o(r_{n,l}),\quad \bar r_{n,l}^{(\eta)}=r_{n,l}^{(\eta)}\ee 
for all $n$ and $l\in\{0,\cdots,N'\}$.
For all $\eta'\in(0,\eta)$, as a consequence of \eqref{6-11-0} and \eqref{6-11}, we get that
\be\lab{6-100-1} \abs{\nabla(u_{n,l}-v_{n,l})}=o\sbr{\f{1}{\ga_{n,l}^{p_n-1}\bar r_{n,l}^{(\eta')} }}\ee
uniformly in $B_{r_{n,l}^{(\eta')}}(0)\setminus B_{\mu_{n,0}}(0)$ (if $l\ge1$) or $B_{r_{n,l}^{(\eta')}}(0)$ (if $l\ge1$). 
Then, by using the estimate in Propositions \ref{appe1-2} and \ref{appe1-1}, we get that
\be\lab{6-101} \abs{u_{n,l}-\bar u_{n,l}(r)}\le \f{2+o(1)}{\ga_{n,l}^{p_n-1}}\ln\f{2r}{|\cdot|},\ee 
uniformly in $B_{r}(0)\setminus\{0\}$ for all large $n$ and $r\in\big(0,r_{n,l}^{(\eta')}\big]$. 

Now, let $\eta_0\in(0,1)$ be fixed and will be chosen later. For any $j\in\{0,\cdots,N'\}$, we define 
\be\lab{6-102} 
\nu_{n,j}:=\sup\lbr{ r_0\in\left(r_{n,j}^{(\eta_0)},\kappa\right]:~
\left\{\begin{aligned} &\abs{u_{n,j}-\bar u_{n,j}(r)}\le 5\pi C_2\bar u_{n,j}(r)^{1-p_n}\\
&\quad\quad\quad+10\sum_{l\in I_{n,j}(r)}\ga_{n,l}^{1-p_n}\ln\f{6r}{\abs{\cdot-\phi_{n,j}(q_{n,l})}} \\
&\text{in}~B_r(0)\setminus\cup_{l\in I_{n,j}(r)}B_{r_{n,l}^{(\eta_0)}}\sbr{\phi_{n,j}(q_{n,l})}\\
&\text{for all}~r\in(0,r_0]
\end{aligned} \right. }  
\ee
for large $n$, where $C_2$ is as in \eqref{est-6} and $I_{n,j}(r)$ is given by
\be\lab{6-103} I_{n,j}(r):=\lbr{l\in\{0,\cdots,N'\}:~\phi_{n,j}(q_{n,l})\in B_{\f{3r}{2}}(0)}. \ee
As a first remark, it follows from \eqref{6-101} that, for all given $\eta_2\in[\eta_0,1)$, we have
\be\lab{6-104} \nu_{n,j}\ge r_{n,j}^{(\eta_2)} \ee 
for large $n$ and any $j$.
Moreover, we have that

\begin{lemma}\lab{lemma6-5}
By choosing $\eta_0$ such that
	$$\eta_0\in\sbr{1-\f{1}{201(1+N')^2},1}$$
we have that 
\be\lab{6-105} \bar u_{n,j}(\nu_{n,j})=O(1), \ee 
for all $j\in\{0,\cdots,N'\}$.
\end{lemma}
\begin{proof}
The proof is quite similar to that of Lemma \ref{lemma5-5}, so we omit it.
\end{proof}

Then, we have the quantization result.
\begin{lemma}\lab{lemma6-6}
Let $p_0=2$ in \eqref{6-81}, then there holds
	$$\lim_{n\to+\iy}\beta_n(B_{\kappa_0}(q_0))=4\pi(1+\al)+4\pi N',$$ 
where $N'$ is given above \eqref{6-0} and $\beta_n(\cdot)$ is given by \eqref{betadomain}.
\end{lemma}
\begin{proof}
Picking a sequence $(\ti\Ga_n)_n$ such that $\lim_{n\to+\iy}\ti\Ga_n=+\iy$ and $\ti\Ga_n=o(\ga_{n,j})$, and setting
	$$\ti\nu_{n,j}:=\sup\lbr{r\in[0,\kappa]:~\bar u_{n,j}\ge\ti\Ga_n~\text{in}~[0,r]},$$
we get from \eqref{6-105} that (letting $\eta_0$ be close to 1)
\be\lab{6-140} \ti\nu_{n,j}\le\nu_{n,j} \ee
for all $j\in\{0,\cdots,N'\}$ and large $n$.
By \eqref{conv-0}, we have $\ti\nu_{n,j}=o(1)$, and hence $\bar u_{n,j}(\ti\nu_{n,j})=\ti\Ga_n$ for all $j$. 
By the argument below \eqref{5-106}, we may fix a large $R\gg1$ such that 
	$$d\sbr{\pa B_{R\ti\nu_{n,j}}(0),\lbr{\phi_{n,j}(q_{n,l}):~l=0,\cdots,N'}}\ge\f{\ti\nu_{n,j}}{R},$$
then, estimating $\bar u_{n,j}(\ti\nu_{n,j}/R)$ by applying \eqref{6-102} with $r=\ti\nu_{n,j}$, and using \eqref{est-6}, we get that
\be\lab{6-141} u_n=\ti\Ga_n+o(1),\ee
uniformly in $\phi_{n,j}^{-1}\sbr{\pa B_{R\ti\nu_{n,j}}(0)}$ for large $n$ and all $j$. Arguing as below \eqref{5-90}, we get from \eqref{6-141} that
\be\lab{6-142} u_n\le 2\ti\Ga_n \quad\text{unifomly in}~B_{\kappa_0}(q_0)\setminus\bigcup_{j=0}^{N'}\phi_{n,j}^{-1}\sbr{B_{R\ti\nu_{n,j}}(0)}\ee
for large $n$. 
Like in \eqref{5-143}, we get from \eqref{6-80} that
\be\lab{6-143} (2-p_n)\ln\sbr{1+\la_n\ga_{n,j}^{2(p_n-1)}}=o(1) \ee
for all $j$. 
Then, by using \eqref{6-80} and \eqref{6-143}, we may choose and fix $\ti\Ga_n$ growing slowly to $+\iy$, such that
\be\lab{6-144}\begin{aligned} 
	\la_n\ti\Ga_n^{p_n}&e^{(2\ti\Ga_n)^{p_n}}=o(\ga_{n,j}^{2-p_n}) \quad\text{and}\\
	&(2-p_n)\ln\sbr{1+\la_n\ga_{n,j}^{2(p_n-1)}e^{(2\ti\ga_n)^{p_n}} }=o(1) 
\end{aligned}\ee
for large $n$ and all $j$. 
Following the line below \eqref{5-144}, and using the assumption \eqref{6-0}, we could conclude the proof of Lemma \ref{lemma6-6}.
\end{proof}

\vs
\subsection{Compactness for critical level \texorpdfstring{$4\pi(1+\al)+4\pi N'$}{}}\lab{sec6-4}\ 

In this subsection, we assume that 
\be\lab{6-150} p_n\equiv p\in(1,2] \ee 
for all $n$. 
By Lemma \ref{lemma6-4} and \ref{lemma6-6}, we see that 
  $$\lim_{n\to+\iy}\beta_n(B_{\kappa_0}(q_0))=4\pi(1+\al)+4\pi N',$$ 
where $N'$ is given above \eqref{6-0} and $\beta_n(\cdot)$ is given by \eqref{betadomain}.
Here, we prove a more accurate estimates of the energy.

\begin{lemma}\lab{lemma6-7}
Let \eqref{6-150} hold, then we have that
\be\lab{6-151} \beta_n(B_{\kappa_0}(q_0))\ge 4\pi(1+\al)+4\pi N'+\f{4(p-1)}{p^2}(4\pi+o(1))\sum_{i=0}^{N'}\f{1+\al\delta_{i0}}{\ga_{n,i}^{2p}} \ee 
for large $n$. 
\end{lemma}
\begin{proof}
Let now $\bar r_{n,i}$ be given by
\be\lab{6-152} t_{n,i}(\bar r_{n,i})=\sqrt{\ga_{n,i}} \ee
for any $i\in\{0,\cdots,N'\}$ and all $n$. Choosing $\eta=\f12$ in \eqref{6-6} and \eqref{6-7}, we get from Lemma \ref{lemma6-1} that $\bar r_{n,i}^{(1/2)}=r_{n,l}^{(1/2)}$ for any $i\in\{0,\cdots,N'\}$ and large $n$.  
Moreover, since $r_{n,0}=O(1)$ by \eqref{6-1}, we get that
\be\lab{6-153-0} r_{n,0}^{(1/2)}=o(r_{n,0})=o(1) \ee
for large $n$, and similarly to \eqref{5-153}, we get also that
\be\lab{6-153} r_{n,i}^{(1/2)}=O(r_{n,i})=O(|x_{n,i}|)=o(1) \ee
for large $n$ and all $i\in\{1,\cdots,N'\}$ (if $N'\ge1$). 
An easy consequence of \eqref{vpn}, \eqref{6-1} and \eqref{6-153-0}-\eqref{6-153} is that the domains $\phi_{n,i}^{-1}(B_{\bar r_{n,i}}(0))$, $i\in\{0,\cdots,N'\}$, are two by two disjoint for large $n$. Then, by \eqref{betadomain}, and using H\"older inequality, we may write as in \eqref{5-155} that
\be\lab{6-155} \begin{aligned}
  \beta_n(B_{\kappa_0}(q_0))
  &\ge \sum_{i=0}^{N'} \left[ \sbr{\f{\la_np^2}{2}\int_{B_{\bar r_{n,i}}(0)}h_{n,i}f_{n,i}(e^{u_{n,i}^p}-1)\rd v_{g_0}}^{\f{2-p}{p}} \right. \\
    &\qquad\qquad \left. \times\sbr{\f{\la_np^2}{2}\int_{B_{\bar r_{n,i}}(0)}h_{n,i}f_{n,i}u_{n,i}^pe^{u_{n,i}^p}\rd v_{g_0}}^{\f{2(p-1)}{p}}\right]
\end{aligned}\ee
for large $n$.
By the argument involving \eqref{5-157}, we get that
\be\lab{6-157} \begin{aligned}  
  \f{\la_np^2}{2}&\int_{B_{\bar r_{n,i}}(0)}h_{n,i}f_{n,i}(e^{u_{n,i}^p}-1)\rd v_{g_0}\\
  &=\f{\la_np^2}{2}\int_{B_{\bar r_{n,i}}(0)}\ti \tau_{n,i}|x_{n,i}|^{2\al}e^{v_{n,i}^p}\rd v_{g_0}\sbr{1+o(\ga_{n,i}^{-2p})} \quad\text{and}\\
  \f{\la_np^2}{2}&\int_{B_{\bar r_{n,i}}(0)}h_{n,i}f_{n,i}u_{n,i}^pe^{u_{n,i}^p}\rd v_{g_0}\\
  &=\f{\la_np^2}{2}\int_{B_{\bar r_{n,i}}(0)}\ti \tau_{n,i}|x_{n,i}|^{2\al}v_{n,i}^pe^{v_{n,i}^p}\rd v_{g_0}\sbr{1+o(\ga_{n,i}^{-2p})}
\end{aligned}\ee
with $\ti \tau_{n,i}=\ti h_{n,i}(0)f_{n,i}(0)$, for all large $n$ and any $i\in\{1,\cdots,N'\}$ (if $N'\ge1$).
For $i=0$, by \eqref{6-6} and \eqref{6-152}, we first deduce that
\be\lab{6-153-1} 2(1+\al)\ln\f{\bar r_{n,0}}{r_{n,0}^{(1/2)}}=t_{n,0}(\bar r_{n,0})-t_{n,0}(r_{n,0}^{(1/2)})+o(1)\le-3(1+\al)\ga_{n,0},\ee
and then, we find from \eqref{6-153-0} that
\be\lab{6-154} \bar r_{n,0}=O(r_{n,0}^{(1/2)}e^{-\ga_{n,0}})=o(e^{-\ga_{n,0}})  \ee
for large $n$.
By using \eqref{6-153-1} and \eqref{6-10-0} with $\eta=\f12$, we get that
\be\lab{s6-31-1} \abs{u_{n,0}-v_{n,0}}=O\sbr{\f{\bar r_{n,0}}{\ga_{n,0}^{p-1}r_{n,0}^{(1/2)}}}+O\sbr{\f{\mu_{n,0}^{\delta_0}}{\ga_{n,0}^{p-1}(r_{n,0}^{(1/2)})^{\delta_0}}}=o(e^{-\ga_{n,0}}) \ee
uniformly in $B_{\bar r_{n,0}}(0)$ for large $n$, where in the second equality we use also that
  $$2(1+\al)\ln\f{\mu_{n,0}}{r_{n,0}^{(1/2)}}=-t_{n,0}(r_{n,0}^{(1/2)})+o(1)\le-3(1+\al)\delta_0^{-1}\ga_{n,0}$$
for large $n$.  
Then, using Proposition \ref{single-1} to get that $v_{n,0}=(1+o(1))\ga_{n,0}$, we obtain that
\be\lab{6-156}     u_{n,0}^{p}=v_{n,0}^p\sbr{1+O\sbr{\f{e^{-\ga_{n,0}}}{\ga_{n,0}}} } \quad\text{and}\quad  e^{u_{n,0}^{p}}=e^{v_{n,0}^p}\sbr{1+o\sbr{\f{1}{\ga_{n,0}^{2p}}} } \ee
uniformly in $B_{\bar r_{n,0}}(0)$ for large $n$. 
By using \eqref{6-154}, we get first from \eqref{6-hi} that
  $$h_{n,0}f_{n,0}=\ti h_{n,0}(0)f_{n,0}(0)|\cdot|^{2\al}(1+o(e^{-\ga_{n,i}}))$$
uniformly in $B_{\bar r_{n,0}}(0)$, and then, by using \eqref{6-156}, we get that
\be\lab{6-157-1} \begin{aligned}  
  \f{\la_np^2}{2}&\int_{B_{\bar r_{n,0}}(0)}h_{n,0}f_{n,0}(e^{u_{n,0}^p}-1)\rd v_{g_0}\\
  &=\f{\la_np^2}{2}\int_{B_{\bar r_{n,0}}(0)}\ti \tau_{n,0}|\cdot|^{2\al}e^{v_{n,0}^p}\rd v_{g_0}\sbr{1+o(\ga_{n,0}^{-2p})}+O\sbr{\la_n\bar r_{n,0}^{2+2\al}}\\
  &=\f{\la_np^2}{2}\int_{B_{\bar r_{n,0}}(0)}\ti \tau_{n,0}|\cdot|^{2\al}e^{v_{n,0}^p}\rd v_{g_0}\sbr{1+o(\ga_{n,0}^{-2p})}
\end{aligned}\ee
with $\ti \tau_{n,0}=\ti h_{n,0}(0)f_{n,0}(0)$ for large $n$, where in the second equality we use \eqref{boundla}, \eqref{6-154} and the fact that
  $$\f{\la_np^2}{2}\ga_{n,0}^{2(p-1)}\int_{B_{\bar r_{n,0}}(0)}\ti \tau_{n,0}|x_{n,i}|^{2\al}e^{v_{n,0}^p}\rd v_{g_0}=O(1)$$
by the argument to get \eqref{5-146}. 
Similar arguments give that
\be\lab{6-158-1} \begin{aligned}  
  \f{\la_np^2}{2}&\int_{B_{\bar r_{n,0}}(0)}h_{n,0}f_{n,0}u_{n,0}^pe^{u_{n,0}^p}\rd v_{g_0}\\
  &=\f{\la_np^2}{2}\int_{B_{\bar r_{n,0}}(0)}\ti \tau_{n,0}|\cdot|^{2\al}v_{n,0}^pe^{v_{n,0}^p}\rd v_{g_0}\sbr{1+o(\ga_{n,0}^{-2p})}
\end{aligned}\ee
for large $n$.
By plugging \eqref{6-157} and \eqref{6-157-1}-\eqref{6-158-1} in \eqref{6-155}, and by applying Proposition \ref{single-4} to $v_{n,0}$ and applying Proposition \ref{single-3} to $v_{n,i}$ of $i\ge1$, we conclude the proof of \eqref{6-151}.
\end{proof}

\vs
Now we are ready to prove Theorem \ref{type2}.
\begin{proof}[\bf Proof of Theorem \ref{type2}]
By Corollary \ref{lemma6-3}, we get $\lim_{n\to+\iy}\la_n=0$. Conclusion \eqref{conclusion-3} is given by Lemma \ref{lemma6-4} and \ref{lemma6-6}, and conclusion \eqref{conclusion-4} is given by Lemma \ref{lemma6-7}. 
The proof is complete.
\end{proof}

\vs
\section{Complete proofs}\lab{sec7}

We now complete the proofs of our main results.

\begin{proof}[\bf Proof of Theorem \ref{thm0-2}]
Theorem \ref{thm0-2} follows directly from Theorem \ref{type1} and Theorem \ref{type2}.
\end{proof}

\begin{proof}[\bf Proof of Theorem \ref{thm0}]
By Proposition \ref{bubble}, we know that the sequence $u_n$ concentrates at finite points. Let the set $\Sigma_\iy=\{q_1,\cdots,q_N\}$ of non-singular concentration points be defined by \eqref{Sigma}. We take a small constant $\kappa_1$ such that
  $$B_{\kappa_1}(q)\cap B_{\kappa_1}(q')=\emptyset,\quad\forall~q,q'\in \Sigma_\iy\cup\DR~s.t.~q\neq q'.$$
By the last line of Proposition \ref{bubble}-(2), we know that there is at least one bubble, so that no matter it is singular or non-singular, we must have $\lim_{n\to+\iy}\la_n=0$ by Theorem \ref{thm0-1} and Theorem \ref{thm0-2}. It follows from \eqref{conv-0} that
\be\lab{612-1} \lim_{n\to+\iy}\beta_n\sbr{\Sigma\setminus\bigcup_{q\in\Sigma_\iy\cup\DR}B_{\kappa_1}(q)}=0,   \ee
where $\beta_n(\cdot)$ is given by \eqref{betadomain}.
Let the set $\DR'$ of singular concentration points be defined by \eqref{DR'}. 
Note that 
  $$\Sigma_\iy\cup\DR=\sbr{\Sigma_\iy\setminus\DR}\sqcup\sbr{\Sigma_\iy\cap\DR\setminus\DR'}\sqcup\DR'\sqcup\sbr{\DR\setminus(\DR'\cup\Sigma_\iy)}.$$
We get from Theorem \ref{thm0-1} that
\be\lab{612-2} \lim_{n\to+\iy}\beta_n\sbr{B_{\kappa_1}(q)}=4\pi M_q, \quad\text{for any}~q\in\Sigma_\iy\setminus\DR,  \ee
with $M_q\in \N_+$, and get from Theorem \ref{thm0-2} that
\be\lab{612-3} \lim_{n\to+\iy}\beta_n\sbr{B_{\kappa_1}(q)}=4\pi M_q', \quad\text{for any}~q\in\Sigma_\iy\cap\DR\setminus\DR',  \ee
with $M_q'\in\N_+$, and that
\be\lab{612-4} \lim_{n\to+\iy}\beta_n\sbr{B_{\kappa_1}(q)}=4\pi(1+\al_q)+4\pi M_q'', \quad\text{for any}~q\in\DR',  \ee
with $M_q''\in\N$ and $\al_q\in(-1,0)$ be the order of the singularity. 
Finally, for any $q\in\DR\setminus(\DR'\cup\Sigma_\iy)$, by \eqref{est-1} we get that
  \[d_{g_0}(\cdot,q)^2\la_nhfu_n^{2(p_n-1)}e^{u_n^{p_n}}=O(1)\quad\text{in}~B_{\kappa_1}(q) \]
for large $n$, and then, since $u_n$ is not blow-up in $B_{\kappa_1}(q)$ and $u_n=O(1)$ uniformly in $\pa B_{\kappa_1}(q)$, by the argument between \eqref{5-47} and \eqref{5-54-1}, we get that $u_n=O(1)$ uniformly in $B_{\kappa_1}(q)$, so that we obtain that
\be\lab{612-5} \lim_{n\to+\iy}\beta_n\sbr{B_{\kappa_1}(q)}=0,\quad\text{for any}~q\in\DR\setminus(\DR'\cup\Sigma_\iy). \ee
Thus, plugging \eqref{612-1}-\eqref{612-5} in \eqref{betan}, we deduce that
  $$\beta=\lim_{n\to+\iy}\beta_n(\Sigma)=4\pi\sum_{q\in\Sigma_\iy\setminus\DR}M_q+4\pi\sum_{q\in\Sigma_\iy\cap\DR\setminus\DR'}M_q'+4\pi\sum_{q\in\DR'}(1+\al_q)+M_q'',$$
which gives \eqref{level-1}. Moreover, if $p_n=p\in(1,2]$ for all $n$, then by using \eqref{612-1}-\eqref{612-5}, we get from Theorem \ref{thm0-1} and Theorem \ref{thm0-2} that
  $$\begin{aligned}
    \beta_n(\Sigma)&\ge\sum_{q\in\Sigma_\iy\setminus\DR}\beta_n\sbr{B_{\kappa_1}(q)}+\sum_{q\in\Sigma_\iy\cap\DR\setminus\DR'}\beta_n\sbr{B_{\kappa_1}(q)}+\sum_{q\in\DR'}\beta_n\sbr{B_{\kappa_1}(q)}\\ 
    &>4\pi\sum_{q\in\Sigma_\iy\setminus\DR}M_q+4\pi\sum_{q\in\Sigma_\iy\cap\DR\setminus\DR'}M_q'+4\pi\sum_{q\in\DR'}(1+\al_q)+M_q''\\
    &=\beta,
  \end{aligned}$$
for large $n$. This finishes the proof of Theorem \ref{thm0}.
\end{proof} 

\begin{proof}[\bf Proof of Theorem \ref{thm1}]
Let $\beta>0$ be given. Assume first that $p$ is given in $(1,2)$. Given a sequence $(u_n)_n\subset \CR_{p,\beta}$, since $u_n$ is a positive critical point of $I_{p,\beta}$, we have that $u_n$ is a solution of \eqref{equ-2} with $\la=\la_n>0$ given by \eqref{la-1} for all $n$. By \eqref{h}, we get that $u_n$ is a solution of \eqref{equ-5} and $\beta_n$ given by \eqref{betan} satisfies $\beta_n\equiv\beta$ for all $n$. Then, by Theorem \ref{thm0}, we get that $u_n$ is uniformly bounded in $C^0(\Sigma)$. Also note from \eqref{boundla} that $\lambda_n$ is uniformly bounded. Then by the elliptic theory, we conclude as before that up to a subsequence, $u_n$ converges to some $u$ in $C^0(\Sigma)\cap C^1_{loc}(\Sigma\setminus \DR)$, and $u\in \CR_{p,\beta}$. This implies that $\CR_{p,\beta}$ is compact in $C^0(\Sigma)$.  The compactness of $\CR_{2,\beta}$ or $\CR_{1,\beta}$ with $\beta\not\in\OR$ can be proved similarly by applying Theorem \ref{thm0}.
\end{proof}

\vskip0.26in
\appendix
\renewcommand{\appendixname}{\appendixname~\Alph{section}}

\section{Standard bubbles}\lab{appe1}

\subsection{Non-singular bubble}\lab{appe1-2}\

Let $\al\in\sbr{-1,0}$ and $(p_n)_n$ be any family of numbers in $[1,2]$, and let $(\tau_n)_n$, $(\nu_n)_n$ be given families of numbers in $[C^{-1},C]$ for some constant $C\ge1$, and let $(\mu_n)_n$ be given families of positive numbers. Let $(\ga_n)_n$ be a positive number sequence satisfying $\ga_n\to+\iy$ as $n\to+\iy$, $(x_n)_n$ be given point sequence in $\R^2$ and $\la_n>0$ be given by
\be\lab{8} \mu_n^{2}\la_np_n^2\tau_n|x_n|^{2\al}\ga_n^{2(p_n-1)}e^{\ga_n^{p_n}}=8, \ee
and let $t_n$, $\bar t_n$ be defined in $\R^2$ by
\be\lab{9} t_n(x)=\ln\sbr{1+\f{|x|^{2}}{\mu_n^{2}}},\quad \bar t_n=t_n+1, \ee
for all large $n$. Let $\eta\in(0,1)$ be fixed. Let also $(\bar r_n)_n$ be any family of positive numbers such that
\be\lab{10} \lim_{n\to+\iy}\f{\mu_n}{\bar r_n}=0, \ee
\be\lab{11} t_n(\bar r_n)\le \eta\f{p_n\ga_n^{p_n}}{2}, \ee
\be\lab{12} \ga_n^{2p_n}\bar r_n^{2}|x_n|^{2\al}=O(1), \ee
\be\lab{13} \bar r_n\le\f{1}{2}|x_n|, \ee
for all large $n$. We study in this section the behavior as $n\to+\iy$ of a family $(\BR_n)_n$ of functions solving
\be\lab{14}\begin{cases}
	-\Delta\BR_n+\nu_n|x_n|^{2\al}\BR_n=\la_np_n\tau_n|x_n|^{2\al}\BR_n^{p_n-1}e^{\BR_n^{p_n}},\\
	\BR_n(0)=\ga_n>0,\\
	\BR_n~\text{is radially symmetric and positive in}~B_{\bar r_n}(0).
\end{cases}\ee
For fixed $n$, \eqref{14} reduces to an ODE with respect to the radial variable $r=|x|$. Let $w_n$ be given by
\be\lab{15} \BR_n=\ga_n\sbr{1-\f{2t_n}{p_n\ga_n^{p_n}}+\f{\omega_n}{\ga_n^{p_n}}}. \ee
Then following the strategy in \cite[Section 2]{MT-blowup-7}, we have the following result.
\begin{proposition}\lab{single-2}
We have $\BR_n\le\ga_n$,
	$$\omega_n=O(\ga_n^{-p_n}t_n),\quad \omega_n'=O(\ga_n^{-p_n}t_n'),$$
and
	$$\la_np_n\tau_n|x_n|^{2\al}\BR_n^{p_n-1}e^{\BR_n^{p_n}}=\f{8e^{-2t_n}}{\mu_n^{2}\ga_n^{p_n-1}p_n}\sbr{1+O\sbr{\f{e^{\ti\eta t_n}}{\ga_n^{p_n}}}},$$
uniformly in $[0,\bar r_n]$ and for all large $n$, where $\ti\eta$ is any fixed constant in $(\eta,1)$.
\end{proposition}
\begin{proof}
Noting that we require \eqref{12} and \eqref{13}, Proposition \ref{single-2} can be proved exactly as \cite[Proposition 2.1]{MT-blowup-7}, and we omit the details here. Remark that, comparing with \cite[Proposition 2.1]{MT-blowup-7}, we have two additional sequences $\tau_n,\nu_n$, but these two sequences have no influence in the proof, because we require $\tau_n,\nu_n\in[C^{-1},C]$ for some constant $C\ge1$. 
\end{proof}

By Proposition \ref{single-2}, we obtain \eqref{11} and \eqref{15} that 
	\be\label{16-66}\ga_n\geq\BR_n(r)\ge (1-\eta+o(1))\ga_n\ee
uniformly in $r\in[0,\bar r_n]$ and for all large $n$, and also obtain from \eqref{8} and \eqref{11} that
\be\lab{16} \BR_n(r)=-\sbr{\f{2}{p_n}-1}\ga_n+\f{2}{p_n\ga_n^{p_n-1}}\ln \f{1}{\la_n\ga_n^{2(p_n-1)}|x_n|^{2\al}(\mu_n^{2}+r^{2})} +O(\ga_n^{1-p_n}) \ee
uniformly in $r\in[0,\bar r_n]$ and for all large $n$.

\subsection{Singular bubble}\lab{appe1-1}\ 

In this subsection, for brevity, we use some of the same symbols as in Section \ref{appe1-2}, which may have different definitions. 
Let $\al\in\sbr{-1,0}$ and $(p_n)_n$ be any family of numbers in $[1,2]$, and let $(\tau_n)_n$, $(\nu_n)_n$ be given families of numbers in $[C^{-1},C]$ for some constant $C\ge1$, and let $(\mu_n)_n$ be given families of positive numbers. Let $(\ga_n)_n$ be a positive number sequence satisfying $\ga_n\to+\iy$ as $n\to+\iy$, and $\la_n>0$ be given by
\be\lab{1} \mu_n^{2(1+\al)}\la_np_n^2\tau_n\ga_n^{2(p_n-1)}e^{\ga_n^{p_n}}=8(1+\al)^2, \ee
and let $t_n$, $\bar t_n$ be defined in $\R^2$ by
\be\lab{2} t_n(x)=\ln\sbr{1+\f{|x|^{2(1+\al)}}{\mu_n^{2(1+\al)}}},\quad \bar t_n=t_n+1, \ee
for all large $n$. Let $\eta\in(0,1)$ be fixed. Let also $(\bar r_n)_n$ be any family of positive numbers such that
\be\lab{3} \lim_{n\to+\iy}\f{\mu_n}{\bar r_n}=0, \ee
\be\lab{4} t_n(\bar r_n)\le \eta\f{p_n\ga_n^{p_n}}{2}, \ee
\be\lab{5} \ga_n^{2p_n}\bar r_n^{2(1+\al)}=O(1), \ee
for all large $n$. We study in this section the behavior as $n\to+\iy$ of a family $(\BR_n)_n$ of functions solving
\be\lab{6}\begin{cases}
	-\Delta\BR_n+|\cdot|^{2\al}\nu_n\BR_n=|\cdot|^{2\al}\la_np_n\tau_n\BR_n^{p_n-1}e^{\BR_n^{p_n}},\\
	\BR_n(0)=\ga_n>0,\\
	\BR_n~\text{is radially symmetric and positive in}~B_{\bar r_n}(0).
\end{cases}\ee
For fixed $n$, \eqref{6} reduces to an ODE with respect to the radial variable $r=|x|$, and since $\al\in\sbr{-1,0}$, we know that $\BR_n$ is continuous at $0$. Let $w_n$ be given by
\be\lab{7} \BR_n=\ga_n\sbr{1-\f{2t_n}{p_n\ga_n^{p_n}}+\f{\omega_n}{\ga_n^{p_n}}}. \ee
Then we have the following result.
\begin{proposition}\lab{single-1}
We have $\BR_n\le\ga_n$,
	$$\omega_n=O(\ga_n^{-p_n}t_n),\quad \omega_n'=O(\ga_n^{-p_n}t_n'),$$
and
	$$\la_np_n\tau_n\BR_n^{p_n-1}e^{\BR_n^{p_n}}=\f{8(1+\al)^2e^{-2t_n}}{\mu_n^{2(1+\al)}\ga_n^{p_n-1}p_n}\sbr{1+O\sbr{\f{e^{\ti\eta t_n}}{\ga_n^{p_n}}}},$$
uniformly in $(0,\bar r_n]$ and for all large $n$, where $\ti\eta$ is any fixed constant in $(\eta,1)$.
\end{proposition}
\begin{proof}
Thanks to $\al\in\sbr{-1,0}$ and $\BR_n$ is radially symmetric, by doing the change of variable $t=\f{r^{1+\al}}{1+\al}$, and setting $\BR_n^*(t)=\BR_n(r)$, $\mu_n^*=\f{\mu_n^{1+\al}}{1+\al}$ and $\bar r_n^*=\f{\bar r_n^{1+\al}}{1+\al}$, equation \eqref{6} reduces to
\[\begin{cases}
  -\Delta\BR_n^*+\nu_n\BR_n^*=\la_np_n\tau_n(\BR_n^*)^{p_n-1}e^{(\BR_n^*)^{p_n}},\\
  \BR_n^*(0)=\ga_n>0,\\
  \BR_n^*~\text{is radially symmetric and positive in}~B_{\bar r_n^*}(0).
\end{cases}\]
namely the situation reduces to that in \cite[Section 2]{MT-blowup-7}, so Proposition \ref{single-1} follows readily from \cite[Proposition 2.1]{MT-blowup-7}.
\end{proof}

We also point out that following the proof of \cite[Proposition 2.1]{MT-blowup-7}, we actually have that
\be\lab{7-1} \lim_{n\to+\iy}\omega_n(\mu_n\cdot)=0\quad\text{in}~C^{0,\delta}_\loc(\R^2)\cap C^{1}_\loc(\R^2\setminus\{0\}), \ee
for any given $\delta\in(0,\min\{1,2(1+\al)\})$.  
By Proposition \ref{single-1}, we obtain from \eqref{4} and \eqref{7} that
	\be\label{7-1-1-1}\ga_n\geq\BR_n(r)\ge (1-\eta+o(1))\ga_n\ee
uniformly in $r\in[0,\bar r_n]$ and for all large $n$, and also obtain from \eqref{1} and \eqref{4} that
\be\lab{7-2} \BR_n(r)=-\sbr{\f{2}{p_n}-1}\ga_n+\f{2}{p_n\ga_n^{p_n-1}}\ln \f{1}{ \la_n\ga_n^{2(p_n-1)}\sbr{\mu_n^{2(1+\al)}+r^{2(1+\al)}} } +O(\ga_n^{1-p_n}) \ee
uniformly in $r\in[0,\bar r_n]$ and for all large $n$.

\vs
\section{Sharper estimates of standard bubbles}\lab{appe4}
\noindent B.1. {\bf Non-singular bubble}

Let $\al\in\sbr{-1,0}$ and $p_n\equiv p\in(1,2]$ for all $n$, and let $(\tau_n)_n$, $(\nu_n)_n$ be given families of numbers in $[C^{-1},C]$ for some constant $C\ge1$, and let $(\mu_n)_n$ be given families of positive numbers. Let $(\ga_n)_n$ be a positive number sequence satisfying $\ga_n\to+\iy$ as $n\to+\iy$, $(x_n)_n$ be given point sequence in $\R^2$ and $\la_n>0$ be given such that \eqref{8} holds. Let $t_n$ be given by \eqref{9}. 
Let also $(\bar r_n)_n$ be a family of positive numbers such that \eqref{10} and \eqref{13} both hold, and such that (instead of \eqref{11}-\eqref{12})
\be\lab{11-1} t_n(\bar r_n)=\sqrt{\ga_n}, \ee
\be\lab{12-1} \ga_n^{4p}\bar r_n^{2}|x_n|^{2\al}=O(1), \ee
for large $n$.
Let $\BR_n$ be functions solving \eqref{14} with $p_n=p$ for all $n$. 
Then, we have the following more precise estimate than in Section \ref{appe1-2}.

\begin{proposition}\lab{single-3}
We have that
\be\lab{521-0}\begin{aligned}
  &\sbr{\f{\la_np^2}{2}\int_{B_{\bar r_n}(0)}\tau_n|x_n|^{2\al}e^{\BR_n^p}\rd x}^{\f{2-p}{p}} \sbr{\f{\la_np^2}{2}\int_{B_{\bar r_n}(0)}\tau_n|x_n|^{2\al}\BR_n^pe^{\BR_n^p}\rd x}^{\f{2(p-1)}{p}}\\
  &=4\pi\sbr{1+\f{4(p-1)}{p^2\ga_n^{2p}}+o\sbr{\f{1}{\ga_n^{2p}}} }
\end{aligned}\ee
for all large $n$.
\end{proposition}

\begin{proof}
Proposition \ref{single-3} can be proved exactly as \cite[Corollary 5.1]{MT-blowup-7}, and we omit the details here. Remark that,
such result was first obtained by Malchiodi-Martinazzi \cite{MT-blowup-4} (see also \cite{MT-blowup-5}) with $p=2$, and then extended to $p\in(1,2]$ in \cite[Corollary 5.1]{MT-blowup-7}. Their strategy of proving \eqref{521-0} is to expand $\BR_n$ more precisely than \eqref{15}:
  $$\BR_n=\ga_n\sbr{1-\f{2}{p\ga_n^{p}}T_0\sbr{\f{\cdot}{\mu_n}}+\f{4(p-1)}{p^2\ga_n^{2p}}\hat w_0\sbr{\f{\cdot}{\mu_n}} +\hat w_1\sbr{\f{\cdot}{\mu_n}}+\f{\hat w_n}{\ga_n^{3p}} },$$
where $T_0=\ln(1+|\cdot|^2)$, $\hat w_0, \hat w_1$ are some special functions (see \cite[(5.5)-(5.10)]{MT-blowup-7}) and $\hat w_n$ is given by this formula. Then, by using the stronger assumption \eqref{11-1}, and by Taylor expansions, they got a sharper estimate of $\BR_n$ (see \cite[Proposition 5.1]{MT-blowup-7}). Eventually, by careful computations, they observed the cancellation of the term $\ga_n^{-p}$ in \eqref{521-0}, and we refer to \cite[Section 5.1]{MT-blowup-7} for detailed proofs. Besides, the term $\ga_n^{-2p}$ vanishes as well for $p=1$. This is the technical reason why the approach does not work for $p=1$ and why $p>1$ is required. 
\end{proof}

\vskip0.1in
\noindent B.2. {\bf Singular bubble}

Let $\al\in\sbr{-1,0}$ and $p_n\equiv p\in(1,2]$ for all $n$, and let $(\tau_n)_n$, $(\nu_n)_n$ be given families of numbers in $[C^{-1},C]$ for some constant $C\ge1$, and let $(\mu_n)_n$ be given families of positive numbers. Let $(\ga_n)_n$ be a positive number sequence satisfying $\ga_n\to+\iy$ as $n\to+\iy$, and $\la_n>0$ be given such that \eqref{1} holds.
Let $t_n$ be given by \eqref{2}. Let also $(\bar r_n)_n$ be a family of positive numbers such that \eqref{3} holds, and such that (instead of \eqref{4}-\eqref{5})
\be\lab{4--1} t_n(\bar r_n)=\sqrt{\ga_n}, \ee
\be\lab{5--1} \ga_n^{4p}\bar r_n^{2+2\al}=O(1), \ee
for large $n$.
Let $\BR_n$ be functions solving \eqref{6} with $p_n=p$ for all $n$. 
Then, we have the following more precise estimate than in Section \ref{appe1-1}.

\begin{proposition}\lab{single-4}
We have that
$$\begin{aligned}
  &\sbr{\f{\la_np^2}{2}\int_{B_{\bar r_n}(0)}\tau_n|x|^{2\al}e^{\BR_n^p}\rd x}^{\f{2-p}{p}} \sbr{\f{\la_np^2}{2}\int_{B_{\bar r_n}(0)}\tau_n|x|^{2\al}\BR_n^pe^{\BR_n^p}\rd x}^{\f{2(p-1)}{p}}\\
  &=4\pi(1+\al)\sbr{1+\f{4(p-1)}{p^2\ga_n^{2p}}+o\sbr{\f{1}{\ga_n^{2p}}} }
\end{aligned}$$ 
for all large $n$.
\end{proposition}
\begin{proof}
Similarly as Proposition \ref{single-1}, thanks to $\al\in\sbr{-1,0}$ and $\BR_n$ is radial symmetric, by doing the variable transformation $t=\f{r^{1+\al}}{1+\al}$, and setting $\BR_n^*(t)=\BR_n(r)$, $\mu_n^*=\f{\mu_n^{1+\al}}{1+\al}$ and $\bar r_n^*=\f{\bar r_n^{1+\al}}{1+\al}$, the situation reduces to that in \cite[Section 5.1]{MT-blowup-7}, so that Proposition \ref{single-4} can be proven by using \cite[Corollary 5.1]{MT-blowup-7}.
\end{proof}

\vs
\section{Linear systems}\lab{appe3}
\begin{lemma}\lab{appe3-0}
Let $v$ be a smooth solution of
\be\begin{cases}
	-\Delta v=\f{8}{\sbr{1+|\cdot|^{2}}^{2}}v \quad\text{in}~\R^2,\\
	|v|\le C(1+|\cdot|) \quad\text{in}~\R^2,\\
	v(0)=0,\quad \nabla v(0)=0,
\end{cases}\ee
then $v\equiv0$. 
\end{lemma}
This result was stated by Chen-Lin \cite[Lemma 2.3]{MF-3}, and was generalized on the growth assumption as above by Laurain \cite[Lemma C.1]{CMC-1}. Here, we repeat the proof of Laurain with additional estimates suitable to the singular case, to get the following result. 
\begin{lemma}\lab{appe3-1}
Let $\al\in(-1,0)$ and $v$ be a $C^0(\R^2)\cap C^2(\R^2\setminus\{0\})$ solution of
\be\lab{label-207}\begin{cases}
	-\Delta v=\f{8(1+\al)^2|\cdot|^{2\al}}{\sbr{1+|\cdot|^{2(1+\al)}}^{2}} v \quad\text{in}~\R^2,\\
	|v|\le C(1+|\cdot|) \quad\text{in}~\R^2,\\
	v(0)=0,
\end{cases}\ee
then $v\equiv0$. 
\end{lemma}
\begin{proof}
We first prove that the Fourier coefficients decrease rapidly. Let $k\ge1$ and
	\[\xi_k+i\eta_k=\int_0^{2\pi}ve^{ik\theta}\rd\theta,\]
where $i=\sqrt{-1}$ is the imaginary unit. Then 
\be\lab{label-200} -\Delta\xi_k=\sbr{8(1+\al)^2r^{2\al}\sbr{1+r^{2(1+\al)}}^{-2}-k^2r^{-2}}\xi_k. \ee
We set $w_k=r^{-k}\xi_k$ on $[1,+\iy)$, and we easily get that
\be\lab{label-201} -\Delta\xi_k=-r^kw_k''-(2k+1)r^{k-1}w_k'-k^2r^{k-2}w_k. \ee
On the other hand, thanks to the hypothesis and \eqref{label-200}, there exists a positive constant $C$ such that
\be\lab{label-202} -\Delta\xi_k\le Cr^{-1-2(1+\al)}-k^2r^{k-2}w_k \quad\text{on}~[1,+\iy). \ee
Hence, thanks to \eqref{label-201} and \eqref{label-202}, we get
\[\begin{aligned}
	-r^kw_k''-(2k+1)r^{k-1}w_k' &\le Cr^{-1-2(1+\al)}\\
	(r^{2k+1}w_k')' &\ge-Cr^{k-2(1+\al)}.
\end{aligned}\]
Then we integrate on $[1,r]$, which gives
\[\begin{aligned}
	r^{2k+1}w_k'(r)-w_k'(1) &\ge C\f{1-r^{k+1-2(1+\al)}}{k+1-2(1+\al)}\\
	w_k'(r) &\ge r^{-(2k+1)}w_k'(1)+\f{C}{k+1-2(1+\al)}\sbr{r^{-(2k+1)}-r^{-k-2(1+\al)}}.
\end{aligned}\]
Note that $w_k(r)=O(r^{1-k})$ and $k\ge1$, we integrate on $[R,\iy)$, which gives
	\[\ti C-w_k(R) \ge \f{1}{2k}R^{-2k}w_k'(1) +\f{C}{k+1-2(1+\al)}\sbr{\f{1}{2k}R^{-2k}-\f{R^{-k+1-2(1+\al)}}{k-1+2(1+\al)}}.\]
It follows that
	\[w_k(R)\le CR^{-k+1-2(1+\al)},\]
which means
	\[\xi_k(r)\le C(1+r)^{1-2(1+\al)} \quad\text{on}~[0,+\iy).\]
Since the equation is linear, we can apply the same argument to $-\xi_k$, and finally we get a improved estimate, for $k\ge1$, 
\be\lab{label-203} |\xi_k(r)|\le C_k(1+r)^{1-2(1+\al)} \quad\text{on}~[0,+\iy).\ee
In this direction, we can go further. For any fixed $\al\in(-1,0)$, let 
	\[j_\al=\mbr{\f{3}{4(1+\al)}} \quad\text{and}\quad \delta_\al=\f{3}{4(1+\al)}-j_\al,\]
where $[x]$ means the integer part of  $x$. We use induction on $j$ to prove that
\be\lab{label-204} |\xi_k(r)|\le C_k(1+r)^{1-2(\delta_\al+j)(1+\al)} \quad\text{on}~[0,+\iy),\ee
for $j=0,1,\cdots,j_\al$. Clearly, for $j=0$, by \eqref{label-203}, we get
	\[|\xi_k(r)|\le Cr^{1-2\delta_\al(1+\al)} \quad\text{on}~[1,+\iy).\]
Now we assume \eqref{label-204} holds for some $j\le j_\al-1$. Then, using \eqref{label-200}, we get
\be\lab{label-205} -\Delta\xi_k\le Cr^{-1-2(\delta_\al+j+1)(1+\al)}-k^2r^{k-2}w_k \quad\text{on}~[1,+\iy). \ee
Hence, thanks to \eqref{label-201} and \eqref{label-205}, we get
	\[(r^{2k+1}w_k')' \ge-Cr^{k-2(\delta_\al+j+1)(1+\al)}.\]
Since $k-2(\delta_\al+j+1)(1+\al)>-1$ by $k\ge1$ and $j\le j_\al-1$, we can integrate on $[1,r]$ to get
	\[w_k'(r) \ge r^{-(2k+1)}w_k'(1) +C\f{r^{-(2k+1)}-r^{-k-2(\delta_\al+j+1)(1+\al)}}{k+1-2(\delta_\al+j+1)(1+\al)}.\]
Note that $w_k(r)=O(r^{1-k-2(1+\al)})$ and $k\ge1$, we integrate on $[R,\iy)$, which gives
	\[-w_k(R) \ge \f{1}{2k}R^{-2k}w_k'(1) +C\f{\f{1}{2k}R^{-2k}-\f{R^{-k+1-2(\delta_\al+j+1)(1+\al)}}{k-1+2(\delta_\al+j+1)(1+\al)}}{k+1-2(\delta_\al+j+1)(1+\al)}.\]
It follows that
	\[\xi_k(r)\le C(1+r)^{1-2(\delta_\al+j+1)(1+\al)} \quad\text{on}~[0,+\iy).\]
Applying the same argument to $-\xi_k$, and we get \eqref{label-204} holds for $j+1$, so that we have proved \eqref{label-204} holds for $j=0,1,\cdots,j_\al$. In particular, we have
\be\lab{label-206} |\xi_k(r)|\le C_k(1+r)^{-\f{1}{2}} \quad\text{on}~[0,\iy).\ee
Of course the same result is ture considering $\eta_k$. 

Now we are going to prove that any solution of \eqref{label-207} which satisfies $|v(x)|\le C(1+|x|)$ is a linear multiple of the elementary solution of \eqref{label-207}, that is to say,
	\[v=c\phi_0 \quad\text{with}\quad \phi_0=\f{1-|\cdot|^{2(1+\al)}}{1+|\cdot|^{2(1+\al)}}.\]
Then, the initial condition will give the result.

In order to show our result it suffices to show that $\xi_k\equiv0$ and $\eta_k\equiv0$ for $k\ge1$. We are going to prove this result for $\xi_k$, and the argument are exactly the same for $\eta_k$. Let
	\[\vp_1(r)=\f{r^{1+\al}}{1+r^{2(1+\al)}}.\]
By easy computation, we see that
\be\lab{label-208} -\Delta\vp_1=\sbr{8(1+\al)^2r^{2\al}\sbr{1+r^{2(1+\al)}}^{-2}-(1+\al)^2r^{-2}}\vp_1.\ee
Suppose $\xi_k\not\equiv0$ for some $k\ge1$. Since $\vp_1>0$ on $(0,+\iy)$, applying Sturm comparison theorem on $\xi_k$ and $\vp_1$, we get that $\xi_k$ never vanishes on $(0,+\iy)$. Without loss of generality, we assume $\xi_k>0$. Thanks to \eqref{label-200} and \eqref{label-208}, we get
	\[\begin{aligned}
		\vp_1(r)\xi_k'(r)r-\vp_1'(r)\xi_k(r)r &=\int_0^r(\vp_1\Delta\xi_k-\xi_k\Delta\vp_1)t\rd t \\
		&=(k^2-(1+\al)^2)\int_0^r\f{\xi_k\vp_1}{t}\rd t
	\end{aligned}\]
Since $\xi_k(r)=O(r^{-1/2})$ as $r\to\iy$, then we get a sequnece $r_n\to\iy$ such that $\xi_k'(r_n)r_n\le Cr_n^{-1/2}$ for some positive constant $C$. Thus, 
	\[0=\lim_{n\to\iy}\vp_1(r_n)\xi_k'(r_n)r_n-\vp_1'(r_n)\xi_k(r_n)r_n =(k^2-(1+\al)^2)\int_0^\iy\f{\xi_k\vp_1}{t}\rd t,\]
thanks to the integrability of $\f{\xi_k\vp_1}{t}$. This gives a contradiction since $\xi_k>0$, and hence we finish  the proof.
\end{proof}

\subsection*{Acknowledgements}  
Z. Chen was supported by National Key R\&D Program of China (Grant 2023YFA1010002) and NSFC (No. 1222109). H. Li is supported by the postdoctoral foundation of BIMSA.


\vskip0.1in
\begin{thebibliography}{10}

\bibitem{SMT-0}
Adimurthi and Sandeep, K.,
A singular Moser-Trudinger embedding and its applications. 
{\it NoDEA Nonlinear Differential Equations Appl.}, {\bf 13}(2007), no. 5-6, 585–603.

\bibitem{MT-blowup-1}
Adimurthi and Struwe, M., 
Global compactness properties of semilinear elliptic equations with critical exponential growth. 
{\it J. Funct. Anal.}, {\bf 175}(2000), no. 1, 125-167.

\bibitem{barycenter-2}
Bartolucci, D., De Marchis, F. and Malchiodi, A.,
Supercritical conformal metrics on surfaces with conical singularities.
{\it Int. Math. Res. Not.}, (2011), no. 24, 5625–5643.

\bibitem{bg-4}
Bartolucci, D., Gui C. F., Jevnikar, A. and Moradifam, A., 
A singular sphere covering Inequality: uniqueness and symmetry of solutions to singular Liouville-type equations. 
{\it Math. Ann.}, {\bf 374}(2019), 1883-1922.

\bibitem{barycenter-4}
Bartolucci, D. and Malchiodi, A.,
An improved geometric inequality via vanishing moments, with applications to singular Liouville equations.
{\it Comm. Math. Phys.}, {\bf 322}(2013), no.2, 415–452.

\bibitem{bg1-1}
Bartolucci, D. and Montefusco, E.,
Blow-up analysis, existence and qualitative properties of solutions for the two-dimensional Emden-Fowler equation with singular potential.
{\it Math. Methods Appl. Sci.}, {\bf 30}(2007), no. 18, 2309–2327.

\bibitem{BT}
Bartolucci D. and Tarantello G., 
Asymptotic blow-up analysis for singular Liouville type equations with applications.
{\it J. Differential Equations}, {\bf 262}(2017), no. 7, 3887–3931.

\bibitem{bg-1}
Bartolucci D. and Tarantello G., 
Liouville Type Equations with Singular Data and Their Applications to Periodic Multivortices for the Electroweak Theory. 
{\it Communications in mathematical physics}, {\bf 229}(2002), 3-47.

\bibitem{barycenter-6}
Battaglia, L., Jevnikar, A., Malchiodi, A. and Ruiz, D.,
A general existence result for the Toda system on compact surfaces.
{\it Adv. Math.}, {\bf 285}(2015), 937–979.

\bibitem{MF-1}
Brezis H. and Merle F.,
Uniform estimates and blow-up behavior for solutions of $-\Delta u=Ve^u$ in two dimensions. 
{\it Comm. Partial Differential Equations}, {\bf 16}(1991), no. 8-9, 1223–1253.

\bibitem{extremal-0}
Carleson, L. and Chang, S.Y.A., 
On the existence of an extremal function for an inequality. 
{\it J. Moser. Bull. Sci. Math.}, {\bf 110}(1986), no. 2, 113–127.

\bibitem{barycenter-3}
Carlotto, A. and Malchiodi, A.,
Weighted barycentric sets and singular Liouville equations on compact surfaces.
{\it J. Funct. Anal.}, {\bf 262}(2012), no.2, 409–450.

\bibitem{exist-s-1}
Csato, G. and Roy, P.,
Extremal functions for the singular Moser-Trudinger inequality in 2 dimensions. 
{\it Calc. Var. Partial Differential Equations}, {\bf 54}(2015), no. 2, 2341–2366.

\bibitem{exist-s-2}
Csato, G. and Roy, P.,
Singular Moser-Trudinger inequality on simply connected domains.
{ \it Comm. Partial Differential Equations}, {\bf 41}(2016), no. 5, 838–847.

\bibitem{Selection-2}
Chen C. C. and Lin C. S., 
Estimate of the conformal scalar curvature equation via the method of moving planes. II. 
{\it Journal of Differential Geometry}, {\bf 49}(1998), no. 1, 115-178.

\bibitem{MF-3}
Chen C. C. and Lin C. S.,
Sharp estimates for solutions of multi-bubbles in compact Riemann surfaces.
{\it Comm. Pure Appl. Math.}, {\bf 55}(2002), no. 6, 728–771.

\bibitem{degree-1}
Chen C. C. and Lin C. S.,
Topological degree for a mean field equation on Riemann surfaces.
{\it Comm. Pure Appl. Math.}, {\bf 56}(2003), no. 12, 1667–1727.

\bibitem{degree-3}
Chen C. C. and Lin C. S.,
Mean field equation of Liouville type with singular data: topological degree.
{\it Comm. Pure Appl. Math.}, {\bf 68}(2015), no. 6, 887–947.

\bibitem{Singular-2}
Chen W. X.,
A Trudinger inequality on surfaces with conical singularities. 
{\it Proceedings of the American Mathematical Society}, {\bf 108}(1990), no. 3, 821-832.

\bibitem{clas-2}
Chen W. X. and Li C. M., 
Classification of solutions of some nonlinear elliptic equations. 
{\it Duke Math. J.}, {\bf 63}(1991), no. 1, 615–622.

\bibitem{clas-0}
Chen W. X. and Li C. M., 
What kinds of singular surfaces can admit constant curvature?
{\it Duke Math. J.}, {\bf 78}(1995), no. 2, 437–451.

\bibitem{bg-5}
Chen X. X., Donaldson, S. and Sun S., 
Kahler-Einstein metrics on Fano manifolds. I: Approximation of metrics with cone singularities. 
{\it Journal of the American Mathematical Society}, {\bf 28}(2015), no. 1, 183-197.


\bibitem{singular-1}
Chou K. S. and Wan T. Y. H., 
Asymptotic radial symmetry for solutions of $\Delta u+e^u=0$ in a punctured disc. 
{\it Pacific J. Math.}, {\bf 166}(1994), no. 2, 269-276.

\bibitem{barycenter-1}
Djadli, Z. and Malchiodi, A.,
Existence of conformal metrics with constant Q-curvature.
{\it Ann. of Math. (2)} {\bf 168}(2008), no. 3, 813–858.

\bibitem{barycenter-5}
De Marchis, F., López-Soriano, R. and Ruiz, D.,
Compactness, existence and multiplicity for the singular mean field problem with sign-changing potentials.
{\it J. Math. Pures Appl.}, {\bf 115}(2018), no. 9, 237–267.

\bibitem{MT-blowup-7}
De Marchis, F., Malchiodi, A., Martinazzi, L. and Thizy, P. D., 
Critical points of the Moser–Trudinger functional on closed surfaces. 
{\it Inventiones Mathematicae}, {\bf 230}(2022), no. 2, 1165-1248.

\bibitem{MT-blowup-2}
Druet, O., 
Multibumps analysis in dimension 2: quantification of blow-up levels. 
{\it Duke Math. J.}, {\bf 132}(2006), no. 2, 217–269.

\bibitem{MT-blowup-6}
Druet, O. and Thizy, P. D., 
Multi-bump analysis for Trudinger–Moser nonlinearities. I. Quantification and location of concentration points. 
{\it J. Eur. Math. Soc.}, {\bf 22}(2020), no. 12, 4025-4096.

\bibitem{MT-3}
Fontana, L.,
Sharp borderline Sobolev inequalities on compact Riemannian manifolds. 
{\it Comment. Math. Helv.}, {\bf 68}(1993), no. 3, 415–454.

\bibitem{GT}
Gilbarg D. and Trudinger N. S., 
Elliptic partial differential equations of second order. 
Vol. 224. No. 2. Berlin: springer (1977).

\bibitem{book-1}
Han, Q. and Lin F. H.,
Elliptic partial differential equations. 
Courant Lect. Notes Math., vol. 1. 
Courant Institute of Mathematical Sciences, New York (1997).

\bibitem{exist-s-3}
Iula, S. and Mancini, G.,
Extremal functions for singular Moser-Trudinger embeddings.
{\it Nonlinear Anal.}, {\bf 156}(2017), 215-248.

\bibitem{barycenter-7}
Jevnikar, A., Kallel, S. and Malchiodi, A.,
A topological join construction and the Toda system on compact surfaces of arbitrary genus.
{\it Anal. PDE}, {\bf 8}(2015), no. 8, 1963–2027.

\bibitem{degree-4}
Jevnikar, A., Wei, J. C. and Yang, W.,
On the topological degree of the mean field equation with two parameters.
{\it Indiana Univ. Math. J.}, {\bf 67}(2018), no. 1, 29–88.

\bibitem{bg-3}
Jost J., Zhou C. Q. and Zhu M. M., 
Vanishing Pohozaev constant and removability of singularities. 
{\it Journal of Differential Geometry}, {\bf 111}(2019), no. 1, 91-144.

\bibitem{EQ-2}
Lamm, T., Robert, F. and Struwe, M.,
The heat flow with a critical exponential nonlinearity.
{\it J. Funct. Anal.}, {\bf 257}(2009), 2951–2998.

\bibitem{CMC-1}
Laurain, P.
Concentration of CMC surfaces in a 3-manifold.
{\it Int. Math. Res. Not.}, {\bf 24}(2012), 5585–5649.

\bibitem{Selection-4}
Li J. Y. and Liu L., 
The qualitative behavior at a vortex point for the Chern-Simon-Higgs equation. 
{\it arXiv preprint}, arXiv:2306.03687, 2023.

\bibitem{MT-4}
Li, Y. X., 
Moser–Trudinger inequality on compact Riemannian manifolds of dimension two. 
{\it J. Partial Differ. Equ.}, {\bf 14}(2001), no. 2, 163–192.

\bibitem{Selection-1}
Li Y. Y.,
Prescribing scalar curvature on $S^n$ and related problems, Part I. 
{\it Journal of Differential Equations}, {\bf 120}(1995), no. 2, 319-410.

\bibitem{MF-2}
Li, Y. Y. and Shafrir, I., 
Blow-up analysis for solutions of $-\Delta u=Ve^u$ in dimension two. 
{\it Indiana Univ. Math. J.}, {\bf 43}(1994), no. 4, 1255–1270.

\bibitem{degree-5}
Lin C. S., Wei J. C. and Wang L. P., 
Topological degree for solutions of fourth order mean field equations.
{\it Math. Z.}, {\bf 268}(2011), no. 3-4, 675–705.

\bibitem{Selection-3}
Lin C. S., Wei J. C. and Zhang L., 
Classification of blowup limits for SU(3) singular Toda systems. 
{\it Analysis \& PDE}, {\bf 8}(2015), no. 5, 807-837.

\bibitem{degree-2}
Lin C. S. and Zhang L., 
A topological degree counting for some Liouville systems of mean field type.
{\it Comm. Pure Appl. Math.}, {\bf 64}(2011), no. 4, 556–590.

\bibitem{MT-blowup-4}
Malchiodi, A. and Martinazzi, L.,
Critical points of the Moser-Trudinger functional on a disk. 
{\it J. Eur. Math. Soc.}, {\bf 16}(2014), no. 5, 893–908.

\bibitem{MT-blowup-8}
Malchiodi, A, Martinazzi, L. and Thizy, P. D., 
Critical points of arbitrary energy for the Trudinger-Moser embedding in planar domains. 
{\it Adv. Math.}, {\bf 442}(2024), Paper No. 109548.

\bibitem{bg-2}
Malchiodi A. and Ruiz D., 
New improved Moser–Trudinger inequalities and singular Liouville equations on compact surfaces. 
{\it Geometric and Functional Analysis}, {\bf 21}(2011), 1196-1217.

\bibitem{MT-blowup-5}
Mancini, G. and Martinazzi, L., 
The Moser-Trudinger inequality and its extremals on a disk via energy estimates. 
{\it Calc. Var. Partial Differ. Equ.}, {\bf 56}(2017),  no. 4, 26.

\bibitem{EQ-3}
Martinazzi, L. and Struwe, M., 
Quantization for an elliptic equation of order $2m$ with critical exponential non-linearity. 
{\it Math Z.}, {\bf 270}(2012), 453–486.

\bibitem{MT-2}
Moser, J.,
A sharp form of an inequality by N.Trudinger. 
{\it Indiana Univ. Math. J.}, {\bf 20}(1970/71), 1077–1092.

\bibitem{clas-1}
Prajapat, J. and Tarantello, G., 
On a class of elliptic problems in $\R^2$: Symmetry and uniqueness results. 
{\it Proc. Royal Soc. Edinb.}, {\bf 131A}(2001), 967–985.

\bibitem{Selection-0}
Schoen R., 
“Topics in differential geometry”, 
lecture notes, Stanford University, 1988.

\bibitem{exist-2}
Struwe, M.,
Critical points of embeddings of $H_0^{1,n}$ into Orlicz spaces. 
{\it Ann. Inst. H. Poincare Anal. Non Lineaire}, {\bf 5}(1988), no. 5, 425–464.

\bibitem{EQ-1}
Struwe, M., 
Quantization for a fourth order equation with critical exponential growth. 
{\it Math. Z.}, {\bf 256}(2007), 397–424.

\bibitem{exist-1}
Thizy, P. D.,
When does a perturbed Moser-Trudinger inequality admit an extremal?
{\it Anal. PDE}, {\bf 13}(2020), no. 5, 1371–1415.

\bibitem{Singular-1}
Troyanov, M., 
Prescribing curvature on compact surfaces with conical singularities. 
{\it Transactions of the American Mathematical Society}, {\bf 324}(1991), no. 2, 793-821.

\bibitem{MT-1}
Trudinger, N. S.,
On imbeddings into Orlicz spaces and some applications.
{\it J. Math. Mech.}, {\bf 17}(1967), 473–483.

\bibitem{MT-blowup-3}
Yang, Y. Y.,
Quantization for an elliptic equation with critical exponential growth on compact Riemannian surface without boundary. 
{\it Calc. Var. Partial Differ. Equ.}, {\bf 53}(2015), no. 3, 901–941.

\end{thebibliography}
\end{document}